\documentclass[11pt]{amsart}
\usepackage{amssymb,euscript,tikz,units}

\usepackage[colorlinks, linkcolor=blue, anchorcolor=blue, citecolor=blue]{hyperref}

\usepackage{subcaption}

\usepackage{verbatim}
\usepackage[nameinlink]{cleveref}
\usepackage{colonequals}
\usepackage{tikz}
\usepackage{enumerate}
\usepackage[justification=centering]{caption}
\usepackage{amsfonts,amssymb,amsmath,amsthm,amscd}
\usepackage{bbm}

\allowdisplaybreaks

\numberwithin{equation}{section}

\newcommand{\calT}{\mathcal{T}}
\newcommand{\T}{\calT}
\newcommand{\M}{\mathcal{M}}
\newcommand{\sM}{\mathcal{M}}

\newcommand{\sL}{\mathcal{L}}

\newcommand{\R}{\mathbb{R}}

\newcommand{\Z}{\mathbb{Z}}
\renewcommand{\H}{\mathbb{H}}

\newcommand{\eps}{\varepsilon}
\newcommand{\wep}{Weil-Petersson}

\newcommand{\sbs}{\subset}

\newcommand{\E}{\mathbb{E}_{\rm WP}^g}
\newcommand{\Egn}{\mathbb{E}_{\rm WP}^{g,n}}

\newcommand{\eg}{\textit{e.g.\@ }}

\newcommand{\ie}{\textit{i.e.\@ }}

\def\sys{\mathop{\rm sys}}
\def\area{\mathop{\rm Area}}
\def\arcsinh{\mathop{\rm arcsinh}}
\def\arccosh{\mathop{\rm arccosh}}

\def\dist{\mathop{\rm dist}}

\def\Spec{\mathop{\rm Spec}}
\def\SpG{\mathop{\rm SpG}}

\def\Prob{\mathop{\rm Prob}\nolimits_{\rm WP}^g}
\def\Probgn{\mathop{\rm Prob}\nolimits_{\rm WP}^{g,n}}

\def\Mod{\mathop{\rm Mod}}

\theoremstyle{plain}
\newtheorem{theorem}{Theorem}[section]
\newtheorem{corollary}[theorem]{Corollary}
\newtheorem{proposition}[theorem]{Proposition}
\newtheorem{lemma}[theorem]{Lemma}

\newtheorem{remark}[theorem]{Remark}
\newtheorem{conjecture}[theorem]{Conjecture}
\newtheorem*{thm*}{Theorem}

\theoremstyle{definition}

\theoremstyle{remark}

\newtheorem*{rem*}{Remark}
\newtheorem*{def*}{Definition}
\newtheorem*{con*}{Construction}
\newtheorem*{definition*}{Definition}

\newcommand{\be}{\begin{equation}}
\newcommand{\ene}{\end{equation}}
\newcommand{\br}{\begin{remark}}
\newcommand{\er}{\end{remark}}
\newcommand{\bl}{\begin{lem}}
\newcommand{\el}{\end{lem}}
\newcommand{\bcor}{\begin{cor}}
\newcommand{\ecor}{\end{cor}}
\newcommand{\bpro}{\begin{pro}}
\newcommand{\epro}{\end{pro}}
\newcommand{\ben}{\begin{enumerate}}
\newcommand{\een}{\end{enumerate}}
\newcommand{\bp}{\begin{proof}}
\newcommand{\ep}{\end{proof}}
\newcommand{\bpo}{\begin{pro}}
\newcommand{\epo}{\end{pro}}
\newcommand{\beq}{\begin{equation*}}
\newcommand{\eeq}{\end{equation*}}
\newcommand{\bear}{\begin{eqnarray}}
\newcommand{\eear}{\end{eqnarray}}
\newcommand{\beqar}{\begin{eqnarray*}}
\newcommand{\eeqar}{\end{eqnarray*}}
\newcommand{\bt}{\begin{theorem}}
\newcommand{\et}{\end{theorem}}
\newcommand{\bex}{\begin{excer}}
\newcommand{\eex}{\end{excer}}

\makeatletter

\newcommand{\Rmnum}[1]{\expandafter\@slowromancap\romannumeral #1@}
\makeatother

\begin{document}

\title[Uniform spectral gaps]{Uniform spectral gaps for random
hyperbolic surfaces with not many cusps}
\author{Yuxin He, Yunhui Wu, and Yuhao Xue}

\vspace{.1in}

\address{Tsinghua University, Beijing, China}
\email[(Y.~H.)]{hyx21@mails.tsinghua.edu.cn}
\email[(Y.~W.)]{yunhui\_wu@tsinghua.edu.cn}

\address{Institut des Hautes \'Etudes Scientifiques (IHES), Bures-sur-Yvette, France}
\email[(Y.~X.)]{xueyh@ihes.fr}

\date{}
\maketitle
\vspace{-.2in}
\begin{abstract}
In this paper, we investigate uniform spectral gaps for Weil-Petersson random hyperbolic surfaces with not many cusps. We show that if $n=O(g^\alpha)$ where $\alpha\in \left[0,\frac{1}{2}\right)$, then for any $\epsilon>0$, a random cusped hyperbolic surface in $\sM_{g,n}$ has no eigenvalues in $\left(0,\frac{1}{4}-\left(\frac{1}{6(1-\alpha)}\right)^2-\epsilon\right)$.  If $\alpha$ is close to $\frac{1}{2}$, this gives a new uniform lower bound $\frac{5}{36}-\epsilon$ for the spectral gaps of Weil-Petersson random hyperbolic surfaces. The major contribution of this work is to reveal a critical phenomenon of ``second order cancellation".
\end{abstract}

\section{Introduction}
Let $X_{g,n}$ be a finite area complete non-compact hyperbolic surface with $g$ genus and $n$ cusps. The spectrum of the Laplacian operator on $X_{g,n}$, denoted by $\Spec(X_{g,n})$, consists of absolutely continuous spectrum $[\frac{1}{4},+\infty)$ and possibly discrete eigenvalues in $[0,+\infty)$. $0$ is always a trivial eigenvalue of multiplicity one with the constant eigenfunction, while it is not known whether any $X_{g,n}$ admits a non-zero eigenvalue \cite{Sar03}. The \emph{spectral gap} of $X_{g,n}$, denoted by $\SpG(X_{g,n})$, is defined to be the gap between the trivial eigenvalue $0$ and the essential lower bound of $\Spec(X_{g,n})\setminus\{0\}$, that is, equal to $\frac{1}{4}$ if there is no discrete eigenvalues in $(0,\frac{1}{4})$ or equal to the smallest non-zero eigenvalue otherwise.

Let $\M_{g,n}$ be the moduli space of finite area complete non-compact hyperbolic surfaces with $g$ genus and $n$ cusps. The \wep \ metric is a finite volume K\"ahler metric on $\M_{g,n}$, and hence induces a probability measure $\Probgn$ by normalizing (see Section \ref{sec WP} for details). In this paper, we study the spectral gap of random finite area hyperbolic surfaces. The randomness here is given by the \wep \ metric. 

The study on spectral gaps of \wep \ random hyperbolic surfaces has been exquisitely explored in recent years. For closed hyperbolic surfaces, the spectrum of Laplacian consists of discrete eigenvalues $$0=\lambda_0<\lambda_1\leq\lambda_2\leq\cdots \to +\infty,$$ and has no continuous spectrum part. The spectral gap is then just the first non-zero eigenvalue $\lambda_1$. In her pioneering paper \cite{mir13}, Mirzakhani showed the first uniform lower bound of spectral gap:
\begin{equation*}
	\lim_{g\to\infty}\Prob\left(X\in \M_{g};\  \lambda_1(X)\geq \frac{1}{4}\left(\frac{\ln(2)}{\ln(2)+2\pi}\right)^2\approx 0.00247\right)=1.
\end{equation*}
Her proof applies the Cheeger inequality \cite{Che70}, which is not expected to achieve the optimal result. By applying the Selberg trace theory \cite{Sel56} and establishing an effective counting result for filling closed geodesics, Wu-Xue \cite{wx22-3/16} significantly extended the lower bound $0.00247$ to $\frac{3}{16}-\epsilon$ for any $\epsilon>0$. That is,
\begin{equation*}
	\lim_{g\to\infty}\Prob\left(X\in \M_{g};\  \lambda_1(X)>\frac{3}{16}-\epsilon\right)=1.
\end{equation*}
This lower bound is also independently obtained by Lipnowski-Wright \cite{Lw24-3/16}. Then continuing the same idea and by more delicate study on non-simple closed geodesics, Anantharaman-Monk improved $\frac{3}{16}-\epsilon$ into $\frac{2}{9}-\epsilon$ in \cite{AM23-2/9}, and finally in \cite{anantharaman2025friedman} they obtained the following remarkable result
\begin{equation*}
	\lim_{g\to\infty}\Prob\left(X\in \M_{g};\  \lambda_1(X)>\frac{1}{4}-\epsilon\right)=1.
\end{equation*}
Using the polynomial method that is quite different from \cite{anantharaman2025friedman},  Hide-Macera-Thomas \cite{hide2025spectral} proved
\begin{equation*}
	\lim_{g\to\infty}\Prob\left(X\in \M_{g};\  \lambda_1(X)>\frac{1}{4}-O\left(\frac{1}{g^c}\right)\right)=1
\end{equation*}
for some universal constant $c>0$. Here $\frac{1}{4}$ is the asymptotically optimal lower bound as $g\to\infty$, since $\limsup_{g\to\infty}\sup_{X_g\in\M_g}\lambda_1(X_g)\leq \frac{1}{4}$  by \cite{ huber1974ersten,cheng1975eigenvalue}.

For the case that the number $n=n(g)$ of cusps also grows as the genus $g$ grows, Hide \cite{hide2023spectral} extended Wu-Xue’s method \cite{wx22-3/16} to random cusped hyperbolic surfaces, and showed that if $n=O(g^\alpha)$ and $\alpha\in[0,\frac{1}{2})$, then for any $\epsilon>0$, 
\begin{equation*}
	\lim_{g\to\infty}\Probgn\left(X\in \M_{g,n}; \ \SpG(X)>
	\frac{1}{4}-\left(\frac{1+2\alpha}{4}\right)^2-\epsilon\right)=1.
\end{equation*}
Notice that $\frac{1}{4}-\left(\frac{1+2\alpha}{4}\right)^2$ equals $\frac{3}{16}$ when $\alpha=0$ and tends to $0$ when $\alpha\to\frac{1}{2}$. 
Shen-Wu \cite{shenwu2022arbitrarily} extended Mirzakhani's method to hyperbolic surfaces with cusps and gave a uniform lower bound, especially when $\alpha\to\frac{1}{2}$. They showed that if $n=o(\sqrt{g})$, then 
\begin{equation*}
	\lim_{g\to\infty}\Probgn\left(X\in \M_{g,n}; \ \SpG(X)\geq \frac{1}{4}\left(\frac{\ln(2)}{\ln(2)+2\pi}\right)^2\approx 0.00247\right)=1.
\end{equation*}
This gives the first uniform positive spectral gap when $n$ is near $\sqrt{g}$. In the regime of $n$ satisfying $\frac{n}{\sqrt{g}}\to\infty$ and $\frac{n}{g}\to0$, Shen-Wu \cite{shenwu2022arbitrarily} showed that for any $\epsilon>0$, 
\begin{equation*}
	\lim_{g\to\infty}\Probgn\left(X\in \M_{g,n}; \ \SpG(X)<\epsilon\right)=1.
\end{equation*}

It is conjectured in \cite{shenwu2022arbitrarily} that
	\begin{conjecture}\label{conj cusp 1/4}
		If $n=o(\sqrt{g})$, then for any $\epsilon>0$, 
		\begin{equation*}
			\lim_{g\to\infty}\Probgn\left(X\in \M_{g,n}; \ \SpG(X)>\frac{1}{4}-\epsilon\right)=1.
		\end{equation*}
	\end{conjecture}

In this paper, we study Conjecture \ref{conj cusp 1/4}. Our result is as follows.
\begin{theorem}\label{thm main-1}
If $n=O(g^\alpha)$ and $\alpha\in[0,\frac{1}{2})$, then for any $\epsilon>0$,
\[\lim\limits_{g\to \infty}\Probgn\left(X\in \sM_{g,n}; \ \SpG(X)>\frac{1}{4}-\left(\frac{1}{6(1-\alpha)}\right)^2-\epsilon\right)=1.\]
\end{theorem}
\noindent Observe that $\frac{1}{4}-\left(\frac{1}{6(1-\alpha)}\right)^2> \frac{5}{36}>0$ for $\alpha\in[0,\frac{1}{2})$, Theorem \ref{thm main-1} gives a new uniform lower bound of the spectral gap when $n$ is near $\sqrt{g}$. Indeed, this is the first such result achieved by using Selberg trace theory. When $\alpha=0$, the gap $\frac{1}{4}-\left(\frac{1}{6}\right)^2 = \frac{2}{9}$ coincides with the result in \cite{AM23-2/9}. See Figure \ref{pic graph} for the comparison between these spectral gaps. 

	\begin{figure}[h]
		\centering
		\includegraphics[scale=0.22]{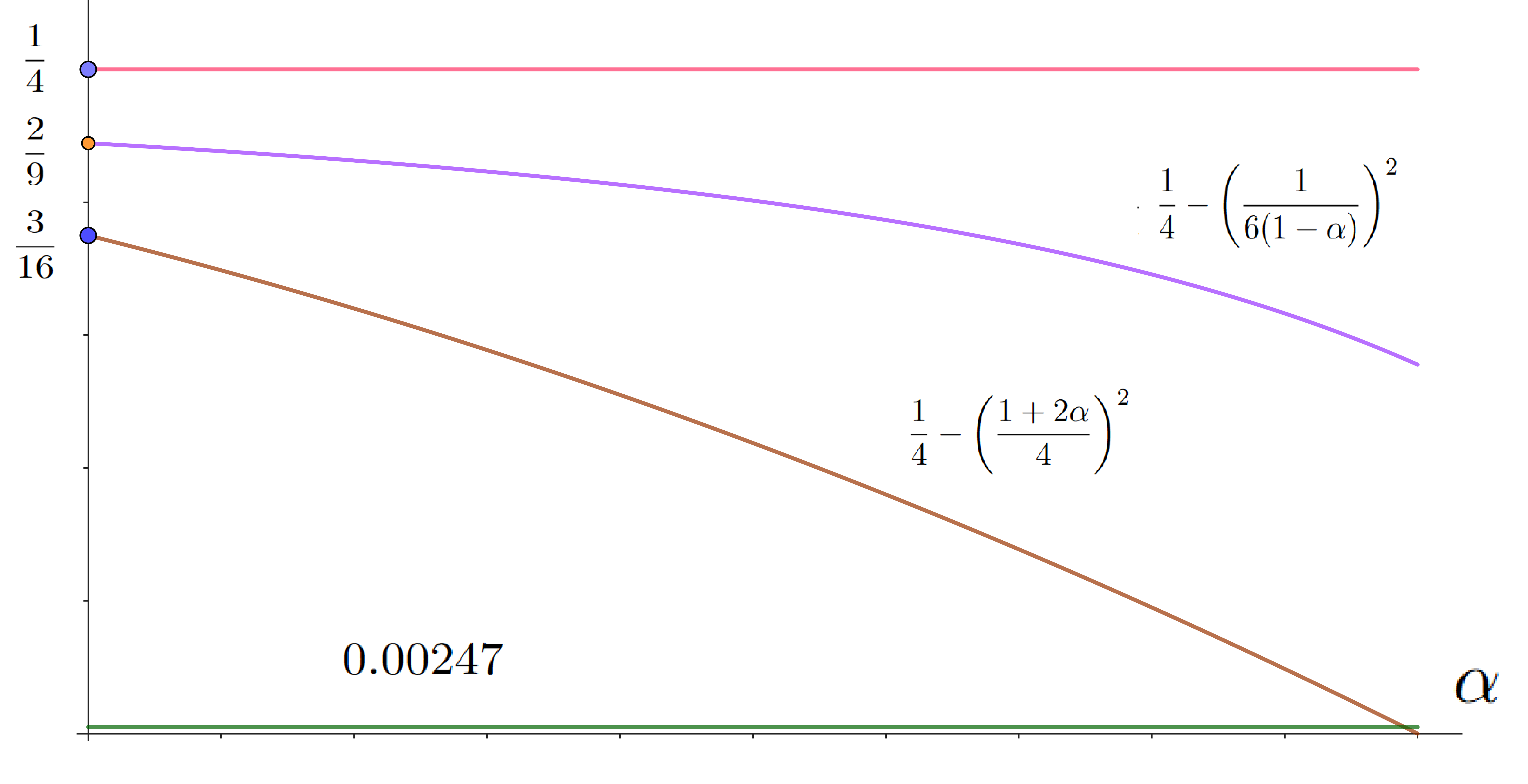}
		\caption{Known results and the conjecture on spectral gaps for random $X\in \M_{g,n}$ with $n=O(g^\alpha)$}
		\label{pic graph}
	\end{figure}

\vspace{2ex}
\noindent \textbf{Related Works.}
Now we introduce some related results of the spectral gap for some specific surfaces and the other two random models of hyperbolic surfaces.

\begin{itemize}
	\item \textbf{Congruence surfaces.}
	For $N\geq 1$, the principal congruence subgroup of $\mathrm{SL}(2,\Z)$ is defined as $\Gamma(N)=\{\gamma\in\mathrm{SL}(2,\Z);\ \gamma\equiv I \ (\mathrm{mod}\ N)\}$. Then $X(N)=\Gamma(N)\backslash\H^2$ is a hyperbolic surface (for $N\geq 2$) with $g(N)$ genus and $n(N)$ cusps, where $g(N)$ grows roughly like $N^3$ and $n(N)$ grows roughly like $N^2$ as $N\to\infty$. The famous Selberg's $\frac{1}{4}$ Conjecture \cite{Sel65} asserts as $\lambda_1(X(N))\geq\frac{1}{4}$, and Selberg \cite{Sel65} proved that $\lambda_1(X(N))\geq\frac{3}{16}$. The best known result until now is $\lambda_1(X(N))\geq\frac{1}{4}-(\frac{7}{64})^2=\frac{975}{4096}$ by Kim-Sarnak \cite{KS03}. See e.g. \cite{GJ78, Iwa89, LRS95, Sar95, Iwa96, KS02} for intermediate results.
	
	\item \textbf{Brooks-Makover model.}
	Brooks-Makover in \cite{BM04} constructed the first random model of hyperbolic surfaces by gluing ideal triangles via 3-regular graphs and then taking compactification. It is also known as random Belyi surfaces \cite{Bel79, Gam06}. In \cite{BM04} they showed that there exists $C>0$ such that $\mathrm{Prob}(\lambda_1(X)\geq C)\to1$ as the number of ideal triangles tends to infinity for their random model. Recently, Shen-Wu \cite{SW25} used the polynomial method (see \eg \cite{CGVTvH24, CGVvH24, MdlS24, MPvH25}) and proved the remarble nearly optimal spectral gap conjecture for the Brooks-Makover model: for any $\epsilon>0$, $\mathrm{Prob}(\lambda_1(X)\geq \frac{1}{4}-\epsilon)\to1$ as the number of ideal triangles tends to infinity.
	
	\item \textbf{Covering model.}
	Fix a finite area hyperbolic surface $X$ (with or without cusps), and consider all degree $N$ coverings of $X$ uniformly at random. For the case that $X$ is closed, Magee-Naud-Puder \cite{magee2022random} showed that for any $\epsilon>0$ there are no new eigenvalues in $(0,\frac{3}{16}-\epsilon)$ with $\mathrm{Prob}\to1$ as $N\to\infty$. Magee-Puder-van Handel \cite{MPvH25} used the polynomial method and showed the remarkable nearly optimal spectral gap conjecture, that is, improving $\frac{3}{16}$ to $\frac{1}{4}$.
	For the case that $X$ is cusped, Hide-Magee \cite{HM23} showed the nearly optimal spectral gap conjecture that there are no new eigenvalues in $(0,\frac{1}{4}-\epsilon)$, which is a first such remarkable theorem for random surfaces. Using this result, they showed that $\lim\limits_{g\to \infty}\sup\limits_{X\in \sM_g}\lambda_1(X_g)=\frac{1}{4}$ confirming an old conjecture of Buser \cite{Buser84}.
\end{itemize}

\vspace{1ex}
\noindent\textbf{Proof Sketch.}
The main ideas of the proof continue those from \cite{wx22-3/16,AM23-2/9}, by taking the integral on both sides of the pre-trace inequality (instead of Selberg's trace formula for the closed surface case). First, we write the pre-trace inequality in the following way (see \eqref{selberg decomposition}):
\begin{equation}\label{selb decomp intr}
	\begin{aligned}
		&C(\epsilon)\!Te^{T(1-\epsilon)\sqrt{\frac{1}{4}-\lambda_1(X)}}\cdot \textbf{1}_{\lambda_1(X)\leq \frac{1}{4}-\epsilon_0}\\
		\leq &\sum_{\gamma\in \mathcal{P}(X)}\sum_{k=1}^\infty \frac{\ell_\gamma(X)}{2\sinh \left(\frac{k\ell_\gamma(X)}{2}\right)}f_T(k\ell_\gamma(X))-\hat{f}_T(\frac{i}{2}) + O\left(\frac{g}{T}\right)\\
		=&\underbrace{\sum_{\gamma\in \mathcal{P}(X)}\sum_{k=2}^\infty \frac{\ell_\gamma(X)f_T(k\ell_\gamma(X))}{2\sinh \left(\frac{k\ell_\gamma(X)}{2}\right)}}_{(\mathrm{\Rmnum{1})}}
		+\underbrace{\sum_{\gamma\in \mathcal{P}_{sep}^s(X)} \frac{\ell_\gamma(X)f_T(\ell_\gamma(X))}{2\sinh \left(\frac{\ell_\gamma(X)}{2}\right)}}_{(\mathrm{\Rmnum{2})}}\\
		+& \underbrace{\sum_{\gamma\in \mathcal{P}_{nsep}^{s}(X)}\frac{\ell_\gamma(X)f_T(\ell_\gamma(X))}{2\sinh \left(\frac{\ell_\gamma(X)}{2}\right)}-\hat{f}_T(\frac{i}{2})}_{\mathrm{(\Rmnum{3})}}
		+\underbrace{\sum_{\gamma\in \mathcal{P}^{ns}(X)}\frac{\ell_\gamma(X)f_T(\ell_\gamma(X))}{2\sinh \left(\frac{\ell_\gamma(X)}{2}\right)}}_{\mathrm{(\Rmnum{4})}}+O\left(\frac{g}{T}\right).
	\end{aligned}
\end{equation}
Here we always assume the spectral gap is below $\frac{1}{4}$, so we write $\lambda_1$ instead of $\SpG(X)$ for simplicity. We choose the test function $f_T$ satisfying Proposition \ref{prop hide2023 ineq}, as used in \cite{magee2022random, wx22-3/16, hide2023spectral}, then take $T=6(1-\alpha)\log g$ to prove Theorem \ref{thm main-1}. The summation in terms (I), (II), (III), and (IV) are taken over all non-primitive closed geodesics, primitive simple separating closed geodesics, primitive simple non-separating closed geodesics, and primitive non-simple closed geodesics, respectively. Notice that in term (III), we put $\sum_{\gamma\in \mathcal{P}_{nsep}^{s}(X)}$ and $-\hat{f}_T(\frac{i}{2})$ together. When computing the expectation $\Egn$, this is the ``\textbf{first order cancellation}" which is an important observation in \cite{wx22-3/16}. An important idea in the proof of Theorem \ref{thm main-1} is that we find a ``\textbf{second order cancellation}" between the expectation of terms (III) and (IV), which will be explained later.

Then we aim to take $\Egn$ over the above inequality. However, similar to what happens in \cite{AM23-2/9, anantharaman2025friedman}, directly computing the expectation $\Egn$ cannot imply what we want. So we need to play an inclusion-exclusion trick (see Section \ref{sec inc-ex}):
\begin{equation}\label{eq inc-ex intr}
	\textbf{1}_{\#N_\ell=0}(X)=1+\sum_{k=1}^j (-1)^k(\#N_\ell)_k(X)+O_j\Big( (\#N_\ell)_{j+1}(X)\Big)
\end{equation}
where $\#N_\ell(X)$ is the number of embedded subsurfaces in $X$ with boundary length shorter than $\ell$ and Euler characteristic equal to $-1$, and $(\#N_\ell)_k=\frac{1}{k!}\#N_\ell(\#N_\ell-1)\cdots(\#N_\ell-k+1)$. Fix $\alpha<\frac{1}{2}$ and $\ell=\kappa\log g$ for a number $0<\kappa<2-4\alpha$, we will show that
\begin{equation}\label{eq P(N=0) intr}
	\lim\limits_{g\to\infty} \Probgn(\#N_\ell\geq 1 )=0 \quad\quad \text{(see \eqref{last section eq 1})}.
\end{equation}
Then we multiply $\textbf{1}_{\#N_\ell=0}(X)$ on both sides of \eqref{selb decomp intr} and compute the expectation $\Egn$. Here, what differs from \cite{AM23-2/9} is that our condition $\textbf{1}_{\#N_\ell=0}(X)$ only needs the boundary length of subsurfaces with Euler characteristic $-1$ larger than $\ell$, while in \cite{AM23-2/9} they also require a lower bound of systole, which makes the discussion much more complicated. Part of the reason is that the geodesic counting result in \cite{AM23-2/9} fails for subsurfaces of Euler characteristic $-1$ with small systole. While we establish several delicate geodesic counting results, Theorem \ref{thm double filling count}, Theorem \ref{thm mono counting}, and Theorem \ref{thm counting s11 sys small}, which are of independent interest, to handle the hard cases. For the closed surface case $n=0$, our proof can simplify the work \cite{AM23-2/9} a lot. A heavy part of this paper is to deal with the new ingredients that come from cusps.

Now take $T=6(1-\alpha)\log g$ and $\ell=\kappa\log g$ for $0<\kappa<2-4\alpha$. We have 
\begin{equation}\label{eq term LHS intr}
	\begin{aligned}
		&\Egn[\textbf{1}_{\#N_\ell=0} \cdot \text{LHS of \eqref{selb decomp intr}}] \succ \log g\cdot g^{6(1-\epsilon)(1-\alpha)\sqrt{\left(\frac{1}{6(1-\alpha)}\right)^2+\epsilon_1}} \\
		& \cdot\left( \Probgn\left(\lambda_1\leq\frac{1}{4}-\left(\frac{1}{6(1-\alpha)}\right)^2-\epsilon_1\right) -\Probgn\left(\#N_\ell \geq 1\right) \right).
	\end{aligned}
\end{equation}
For the right-hand side, the relatively easy part is that 
\begin{equation}\label{eq term 1 intr}
	\Egn[\text{term (I)}] \prec (\log g)^2 g \quad\quad \text{(see Lemma \ref{lemma small term k geq 2})}
\end{equation}
and
\begin{equation}\label{eq term 2 intr}
	\Egn[\text{term (II)}] \prec (\log g)^{22} \quad\quad \text{(see Lemma \ref{lemma small term simple separating})}.
\end{equation}
The hard points are $\Egn[\textbf{1}_{\#N_\ell=0} \cdot \text{term (III)} ]$ and $\Egn[\textbf{1}_{\#N_\ell=0} \cdot \text{term (IV)} ]$. For term (IV), we show that (see Lemma \ref{lemma total int ns for all type})
\begin{equation}\label{eq term 4 intr}
	\begin{aligned}
		\Egn[\textbf{1}_{\#N_\ell=0} \cdot \text{term (IV)} ] \leq & \frac{1}{\pi^2g}\int_0^\infty f_T(x) e^\frac{x}{2}\left(\frac{x^2}{4}+\left(\frac{n}{2}-1\right)x\right)dx \\
		& + O\left((\log g)^{230}g^{1+6\epsilon(1-\alpha)}\right).
	\end{aligned}
\end{equation}
Here, the main term on the right-hand side of \eqref{eq term 4 intr} comes from the contribution of figure-eight and one-sided iterated closed geodesics in the embedded pants that are non-separating and have $0$ or $1$ cusp. The contribution of double-filling geodesics cannot be well controlled if directly computing $\Egn[\text{term (IV)}]$ without multiplying $\textbf{1}_{\#N_\ell=0}$. This is where it benefits. For term (III), applying the inclusion-exclusion formula \eqref{eq inc-ex intr}, we show that for a fixed number $j_0\geq\frac{2-3\alpha}{1-2\alpha}-1\geq1$ (see \eqref{last section eq 9}), 
\begin{equation}\label{eq term 3 intr}
	\begin{aligned}
		\Egn[\textbf{1}_{\#N_\ell=0} \cdot \text{term (III)} ] = & \frac{1}{\pi^2g}\int_0^\infty f_T(x)e^\frac{x}{2}\left(\left(1-\frac{n}{2}\right)x-\frac{x^2}{4}\right)  dx\\
		& +O\left((\log g)^{366j_0+748} g^{1+5(j_0+1)\kappa}\right).
	\end{aligned}
\end{equation}
Here, $\Egn[1\cdot\text{term (III)} ]$ is the main term on the right-hand side of \eqref{eq term 3 intr}, and the other $\Egn[(\#N_\ell)_k \cdot \text{term (III)} ]$ are all relatively smaller. Notice that the main term on the right-hand side of \eqref{eq term 4 intr} and \eqref{eq term 3 intr} cancel with each other, which is the ``\textbf{second order cancellation}". Finally, combining all the above \eqref{selb decomp intr} \eqref{eq P(N=0) intr} \eqref{eq term LHS intr} \eqref{eq term 1 intr} \eqref{eq term 2 intr} \eqref{eq term 4 intr} and \eqref{eq term 3 intr}, for any $\epsilon_1>0$ we can choose suitable $j_0,\kappa,\epsilon>0$ and get
\begin{equation}
	\begin{aligned}
		&\Probgn\left(\lambda_1\leq\frac{1}{4}-\left(\frac{1}{6(1-\alpha)}\right)^2-\epsilon_1\right) \\
		& \prec \Probgn(\#N_\ell\geq1) + O\left(g^{1+7\epsilon-6(1-\epsilon)(1-\alpha)\sqrt{\left(\frac{1}{6(1-\alpha)}\right)^2+\epsilon_1}}\right) \\
		& \to 0,
	\end{aligned}
\end{equation}
which proves Theorem \ref{thm main-1}.

\vspace{2ex}
\noindent\textbf{Notations.}
For any two functions $f(g)$ and $h(g)$, we say $f=O(h)$ if $|f|\leq C h$ for a uniform constant $C>0$ (independent of the variable $g$). We say $f=o(h)$ if $\frac{|f|}{h}\to0$ as $g\to+\infty$. We say $f\prec h$ or $h\succ f$ if $f=O(h)$, and we say $f\asymp h$ if $f\prec h$ and $h\prec f$.

\vspace{2ex}
\noindent\textbf{Plan of the Paper.}
In Section \ref{sec preliminary}, we recall some backgrounds. In Section \ref{sec counting}, we give some effective geodesic counting results, with part of the proof given in Appendix \ref{appendix counting}. In Section \ref{sec Weil-Petersson volumes}, we show estimations of Weil-Petersson volumes, with part of the proof given in Appendix \ref{appendix wp volume}. In Section \ref{sec Pre-trace inequality}, \ref{sec Remainder estimate}, and \ref{sec inc-ex}, we introduce the pre-trace inequality and the inclusion-exclusion formula, and show the relatively easy parts (terms (I) and (II)) of the main proof. Then we show the relatively difficult parts: for term (III) in Section \ref{sec 1 order cancel} and Section \ref{susbsection n ell k}, and for term (IV) in Section \ref{sec Int_ns}. We finish the proof of the main theorem in Section \ref{sec main proof}.

\vspace{2ex}
\noindent\textbf{Acknowledgments.}
We would like to thank all the participants in our seminar on Teichm\"uller theory for helpful discussions on this project. We also thank  Hugo Parlier, Zeev Rudnick, Yang Shen, Weixu Su, Joe Thomas and Haohao Zhang for their valuable discussions and suggestions on this work. Y.~W. is partially supported by the National Key R \& D Program of China (2025YFA1017500) and NSFC grants No. 12361141813 and 12425107.

\tableofcontents

\section{Preliminaries}\label{sec preliminary}
In this section, we will introduce our notations and well-known basic material on the moduli space of hyperbolic Riemann surfaces, the Weil–Petersson volumes, and Mirzakhani’s Integration Formula.
\subsection{Moduli space of hyperbolic Riemann surfaces}
Denote by $S_{g,n}$ an oriented topological
surface with genus $g$ of $n$ punctures or boundaries for $-\chi(S_{g,n})=2g-2+n\geq 1$. The Teichm\"uller space $\T_{g,n}$ of Riemann surfaces homeomorphic to $S_{g,n}$ consists of a diffeomorphism $f$ from $S_{g,n}$ to a complete hyperbolic surface $X$. The map $f$ gives a marking on $X$ by $S_{g,n}$. Two marked surfaces $f:S_{g,n}\to X$ and $h:S_{g,n}\to Y$ are equivalent in $\T_{g,n}$ if and only if $f\circ h^{-1}:Y\to X$ is isotopic to a conformal map. The mapping class group $\Mod_{g,n}:=\textrm{Diff}^+(S_{g,n})/\textrm{Diff}^0(S_{g,n})$ acts on $\T_{g,n}$ by compositing with the marking maps. Elements in $\Mod_{g,n}$ will keep the labels of the cusps and boundaries invariant. The moduli space $\M_{g,n}$ of Riemann surfaces homeomorphic to $S_{g,n}$ is defined by $\T_{g,n}/\Mod_{g,n}$. For $n=0$, let $\T_g=\T_{g,0}$ and $\M_g=\M_{g,0}$. Given $\mathbf{L}=(L_1,\cdots,L_n)\in\mathbb{R}_{\geq 0}^n$, the weighted
Teichm\"uller space $\T_{g,n}(\mathbf{L})$ consists of hyperbolic surfaces $X$ marked by $S_{g,n}$ satisfying:
\begin{enumerate}
    \item if $L_i=0$, the $i$-th puncture of $X$ is a cusp;
    \item if $L_i>0$, the $i$-th boundary of $X$ is a simple closed geodesic of length $L_i$.
\end{enumerate}
The weighted moduli space $\M_{g,n}(\mathbf{L})=\T_{g,n}(\mathbf{L})/\Mod_{g,n}$ consists of unmarked hyperbolic surfaces with fixed boundary lengths. Without ambiguity, let $\T_{g,n}(\mathbf{0})=\T_{g,n}$ and $\M_{g,n}(\mathbf{0})=\M_{g,n}$

\subsection{The Weil-Petersson metric} \label{sec WP} 
The Teichm\"uller space $\T_{g,n}(\mathbf{L})$ carries a natural symplectic form which is invariant under the mapping class group action by the work of Goldman\cite{goldman1984symplectic}. This symplectic form is called the \textit{Weil-Petersson} symplectic form, denoted by $\omega$ or $\omega_{wp}$.           
For a system of pants decomposition $\{\alpha_i\}_{i=1}^{3g-3+n}$ of $S_{g,n}$, the associated \textit{Fenchel-Nielsen} coordinate of $\T_{g,n}(\mathbf{L})$ is given by $X\mapsto (\ell_{\alpha_i}(X),\tau_{\alpha}(X))_{i=1}^{3g-3+n}$, consisting of lengths of geodesics in the pants decomposition and the twisting parameters along them. This coordinate induces an isomorphism $$
\T_{g,n}(\mathbf{L})\simeq \mathbb{R}^{3g-3+n}\times \mathbb{R}_{> 0}^{3g-3+n}.
$$
Wolpert in \cite{wolpert1982fenchel} showed that the Weil-Petersson symplectic form has a simple form under the Fenchel-Nielsen coordinate.
\begin{theorem}
    The Weil-Petersson symplectic form is given by $$
    \omega=\sum_{i=1}^{3g-3+n} d\ell_{\alpha_i}\wedge d\tau_{\alpha_i}.
    $$
\end{theorem}
The Weil-Petersson symplectic form will induce the Weil-Petersson volume form $$
dvol_{wp}=\frac{1}{(3g-3+n)!}\underbrace{\omega\wedge\cdots\wedge\omega}_{3g-3+n \textit{ copies}}.
$$
Both $\omega$ and $dvol_{wp}$ is invariant under the mapping class group, so $dvol_{wp}$ is the lift of a measure $dX$ on $\M_{g,n}(\mathbf{L})$. The total volume of $\M_{g,n}(\mathbf{L})$ is finite and denoted by $V_{g,n}(\mathbf{L}).$ Let $V_{g,n}=V_{g,n}(\mathbf{0})$. See Section \ref{sec Weil-Petersson volumes} and Appendix \ref{appendix wp volume} for details on $V_{g,n}(\mathbf{L})$.

Following \cite{mir13}, a function $f:\M_{g,n}\to \mathbb{R}$ can be viewed as a random variable with respect to the probability measure $\Probgn$, given by normalizing $dX$:
$$
\Probgn(A)=\frac{1}{V_{g,n}}\int_{\M_{g,n}}1_{A}dX,
$$
where $A\subset \M_{g,n}$ is any Borel set and $1_A$ is the characteristic function.

\subsection{Mirzakhani’s Integration Formula} 
Mirzakhani’s Integration Formula is essential for the calculation of random variables over the moduli space.

Given any non-peripheral closed curve $\gamma$ on a topological surface $S_{g,n}$ and $X\in \T_{g,n}$, let $\ell_\gamma(X)$ or $\ell(\gamma)$ denote the hyperbolic length of the unique closed geodesic on $X$ in the homotopy class of $\gamma.$ For any ordered $k$-tuple $\Gamma=(\gamma_1,\cdots,\gamma_k)$ where $\gamma_i$'s
are distinct homotopy classes of non-peripheral disjoint simple closed curves on $S_{g,n}$, let the mapping class group orbit of $\Gamma$ be $$
O_\Gamma=\left\{ (h\cdot\gamma_1,\cdots,h\cdot\gamma_k);h\in \mathrm{Mod}_{g,n}\right\}.
$$
For any function $F:\R_{\geq 0}^k\to \R_{\geq 0}$, define the function $ F^\Gamma$ on $\M_{g,n}$ by:
\begin{align*}
    F^\Gamma:\M_{g,n}&\to \R\\
    X&\mapsto \sum_{(\alpha_1,\cdots,\alpha_k)\in O_\Gamma} F(\ell_{\alpha_1}(X),\cdots,\ell_{\alpha_k}(X) ).
\end{align*}
Set $\mathbf{x}=(x_1\cdots,x_k)\in\R_{\geq 0}^k$. 
If $S_{g,n}-\cup_{i=1}^k \gamma_i\simeq \cup_{i=1}^s S_{g_i,n_i}$, consider the moduli space $\M\left(S_{g,n}(\gamma),\ell_\Gamma=\mathbf{x}\right)$ of hyperbolic Riemann surfaces homeomorphic to $\cup_{i=1}^s S_{g_i,n_i}$ with boundary lengths equal to 0 for the initial cusps of $S_{g,n}$ and $\ell(\gamma_i^1)=\ell(\gamma_i^2)=x_i$ if $\gamma_i^1,\gamma_i^2$ are the two boundary components of $S_{g,n}-\cup_{i=1}^k\gamma_i$ corresponding to $\gamma_i.$ The Weil-Petersson volume of  $\M\left(S_{g,n}(\gamma),\ell_\Gamma=\mathbf{x}\right)$ is the product of the Weil-Petersson volumes of the moduli spaces of all connected components in $S_{g,n}-\cup_{i=1}^k\gamma_i$:
$$
V_{g,n}(\Gamma,\mathbf{x})=Vol\left(  \M\left(S_{g,n}(\gamma),\ell_\Gamma=\mathbf{x}\right)\right)=\prod_{i=1}^s V_{g_i,n_i}(\mathbf{x}^{(i)}),
$$
where $\mathbf{x}^{(i)}$ is the list of $x_j$ in $\mathbf{x}$ corresponding to the curve $\gamma_j$ that is a boundary component of $S_{g_i,n_i}$ and 0 corresponding to the initial cusps of $S_{g,n}$. The following integration formula is due to Mirzakhani. One may see \cite[Theorem 7.1]{mir07-Inv}, \cite[Theorem 2.2]{mir13}, \cite[Theorem 2.2]{mir-petri-19} and \cite[Theorem 4.1]{wright2020tour} for different versions.

\begin{theorem}\label{thm mir int formula}
    For any $\Gamma=(\gamma_1,\cdots,\gamma_k)$, the integral of $F^\Gamma$ over $\M_{g,n}$ with respect to the Weil-Petersson volume form is given by $$
    \int_{\M_{g,n}} F^\Gamma dX=C_\Gamma\int_{\R_{\geq 0}^k} F(x_1,\cdots,x_k)V_{g,n}(\Gamma,\mathbf{x}) \mathbf{x}\cdot d\mathbf{x}
    $$
where the constant $C_\Gamma\in (0,1]$ only depends on $\Gamma$, and $\mathbf{x}\cdot d\mathbf{x}$ represents $x_1\cdots x_k dx_1\wedge\cdots\wedge dx_k.$
\end{theorem}
\begin{rem*}
    The constant $C_\Gamma$ has an explicit expression on the topology of $\Gamma$ in $S_{g,n}$. One may see \cite[Theorem 4.1]{wright2020tour} for more details. We will give the exact value of $C_\Gamma$ only when it is essential in this paper.
\end{rem*}

\section{Counting geodesics in subsurfaces}\label{sec counting}

In this section, we will include several results on counting the number of filling geodesics in subsurfaces.

The following construction of a subsurface implied by a $k$-tuple of closed geodesics has been shown to be very useful as in \eg \cite{mir-petri-19, NWX23,wx22-3/16, Lw24-3/16, hide2023spectral, shenwu2022arbitrarily, Rud23, AM23-2/9, HSWX-23, HW24, anantharaman2025friedman, hide2025spectral}.
\begin{con*}
Let $\Gamma=\left(\gamma_1,\cdots,\gamma_k\right)$ with $\gamma_i\subset X$ be an ordered $k$-tuple where $\gamma_i$’s are nontrivial non-peripheral unoriented primitive closed geodesics on a hyperbolic surface $X$. Let $\gamma=\cup_{i=1}^k\gamma_i$. Consider the $\epsilon$-neighborhood $N_\epsilon(\gamma)$ of $\gamma$ for $\epsilon>0$ small enough such that $N_\epsilon(\gamma)$ is homotopic to $\gamma$. For each boundary component $\xi\subset\partial N_\epsilon(\gamma)$, if $\xi$ is homotopically trivial, we fill the disc bounded by $\xi$ into $N_\epsilon(\gamma).$ If $\xi$ 
is homotopically non-trivial, we deform $\xi$ to the unique simple closed geodesic homotopic to it. Now we will obtain a subsurface $X(\gamma)$ with geodesic boundaries. If two boundary components $\xi_1,\xi_2$ deform to the same simple closed geodesic $\overline{\xi}$, we will not glue them, and $\overline{\xi}$ will appear twice in the geodesic boundaries of $X(\gamma)$.
\end{con*}

For a topological surface $Y$ with possible nonempty boundaries,  let $\mathcal{P}(Y)$ be the set of homotopy classes of primitive closed curves on $Y$. 
We say $\gamma=\cup_{i=1}^k\gamma_i\subset Y$ with $\gamma_i\in \mathcal{P}(Y)$ is \textit{filling} in $Y$ or $\gamma$ \textit{fills} $Y$ if for all choices of representation of homotopy classes $\gamma_i$'s, each connected component of the complement $Y\setminus \gamma$ is either homotopic to a single point or a boundary component of $Y$. See Figure \ref{pic filling} for an illustration.
Notice that when we say a closed geodesic $\gamma$ fills a subsurface, we always assume $\gamma$ to be primitive. More generally, for a subset $A$ of $Y$ with piecewise smooth boundaries, we say $A$ is filling in $Y$ or $A$ fills $Y$ if each component of $Y\setminus A$ is either homotopic to a single point or a boundary component of $Y$. 

\begin{figure}[h]
\centering
\includegraphics[scale=0.15]{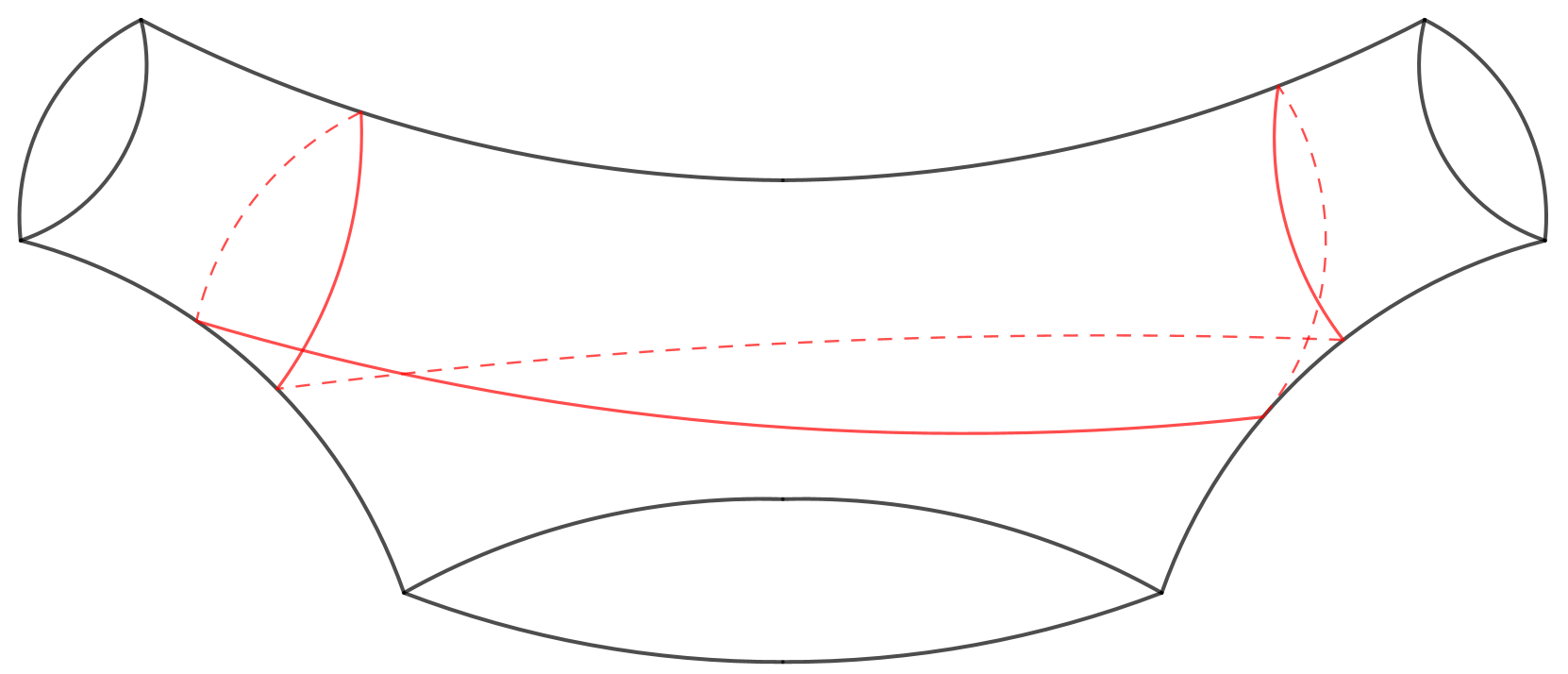}
\caption{A filling geodesic in a pants}
\label{pic filling}
\end{figure}

The following proposition is proved in \cite{NWX23}. 
\begin{proposition}\label{prop x(gamma) bound}\cite[Proposition 59]{NWX23}
     Let $X$ be a hyperbolic surface and $\gamma=\cup_{i=1}^k\gamma_i\subset X$ be the union of some interior primitive closed geodesics. Then the subsurface $X(\gamma)\subset X$ satisfies:\begin{enumerate}
         \item $\gamma\subset X(\gamma)$ is filling;
         \item the boundary $\partial X(\gamma)$ consists of simple closed multi-geodesics with $$
         \ell(\partial X(\gamma))\leq 2\ell_{\gamma}(X),
         $$ 
         where $\ell_{\gamma}(X)=\sum_{i=1}^k\ell_{\gamma_i}(X)$;
         \item the area $$\area(X(\gamma))\leq 4\ell_\gamma(X).$$ 
         \end{enumerate}
\end{proposition}

 \begin{definition*}
 For any $k\geq 1$, $L> 0$ and any hyperbolic surface $X$, we define the counting function $N_k^{fill}(X,L)$  by $$
 N_k^{fill}(X,L)=\#\{\Gamma=\left(\gamma_1,\cdots,\gamma_k\right);\, \gamma_i\in \mathcal{P}(X),\Gamma \textit{ fills } X,\,\, \sum_{i=1}^k\ell_{\gamma_i}(X)\leq L\}.
 $$
 \end{definition*}

The following theorem for $k=1$ is a key theorem in \cite{wx22-3/16} to prove the spectral gap $\frac{3}{16}-\epsilon$. We adopt the following version in \cite{WX2022prime} for general $k$.
\begin{theorem}\label{thm filling count}
\cite[Theorem 18]{WX2022prime} For any $k\geq 1$, $\epsilon>0$  and $m=2g-2+n\geq 1$, there exists a constant $c(k,\epsilon,m)$ only depending on $k,m$ and $\epsilon$ such that for any hyperbolic surface $X\in\mathcal{T}_{g,n}(x_1,\cdots,x_n)$, $$
N_k^{fill}(X,T)\leq c(k,\epsilon,m)\cdot(1+L)^{k-1}\cdot e^{T-\frac{1-\epsilon}{2}\sum_{i=1}^nx_i}.
$$
\end{theorem}
\begin{rem*}
    The decay factor $e^{-\frac{1-\epsilon}{2}\ell(\partial X)}$ is the most important part of this theorem. It also roughly implies the relationship that $\ell(\partial X)\leq 2\ell_{\gamma}(X)$ if $\gamma$ fills $X$.
\end{rem*}

Now consider the inequality $\ell(\partial X(\gamma))\leq 2\ell_{\gamma}(X)$. For any geodesic boundary component $\eta$ of $\partial X(\gamma)$, there is a unique component $C(\eta)$ of $X(\gamma)\setminus \gamma$ such that $\eta\subset \partial C(\eta)$. Since the geodesic $\gamma$ is filling in $X(\gamma)$,
$C(\eta)$ is a cylinder and $\partial C(\eta)=\eta\bigsqcup\tilde{\eta}$, where $\tilde{\eta}$ is a closed piecewise geodesic curve. $\tilde{\eta}$ is a subset of the geodesic $\gamma$ and $\ell_{\eta}(X)\leq \ell(\tilde{\eta})$. Then we have $$
\ell(\partial X(\gamma))=\sum_{\eta\subset \partial X(\gamma)}\ell_{\eta}(X)\leq \sum_{\eta\subset \partial X(\gamma)}\ell(\tilde{\eta})\leq 2\ell_\gamma(X).
$$
The coefficient is $2$  since a geodesic arc in $\gamma$ can contribute to at most two closed piecewise geodesic curves $\tilde{\eta_1},\tilde{\eta_2}$. If any geodesic arc in $\gamma$ only contributes to at most one closed piecewise geodesic curve $\tilde{\eta}$, then $$\ell(\partial X(\gamma))\leq \ell_\gamma(X)$$ and we say $\gamma$ is \textit{double-filling} in $X(\gamma)$ as introduced in \cite[Section 8.1]{AM23-2/9}. More precisely,

\begin{definition*}
    Let $X=S_{g,n}$ be a topological surface and $\partial X=\{\eta_1,\cdots,\eta_n\}$. Assume that $\gamma=\cup_{i=1}^k\gamma_i$ is filling in $X$. For each $\eta_i$, there is a unique cylinder component $C_i$ of $X\setminus\gamma$ such that $\eta_i\subset \partial C_i$. We say $\gamma$ is \textit{double-filling} in $X$ if for all choices of representation of homotopy classes $\gamma_i$'s (with the assumption that all intersections are transversely), the following holds: 
    \begin{enumerate}
        \item for any $i\neq j$, $C_i$ and $C_j$ share no common boundary arcs, or equivalently, $C_i\cap C_j$ is a finite set (maybe empty). 
        \item  for any $1\leq i\leq n$, assume that $\partial C_i=\eta_i\cup \tilde{\eta}_i$. For any parameterization $\phi:S^1\to \tilde{\eta}_i$, the set $$
        \{(t_1,t_2)\in S^1\times S^1 ; t_1\neq t_2, \phi(t_1)=\phi(t_2) \}
        $$
        is a finite set (maybe empty).
    \end{enumerate}
\end{definition*}
 Let $P$ be a pair of pants with three boundary components $\eta_1,\eta_2,\eta_3$. Fix $x\in P$. Then $\pi_1(P,x)$ is generated by the curves $\tilde{\eta}_i$ that go around $\eta_i$ for $i=1,2,3$.
\begin{definition*}
    We say that a non-simple closed curve in $P$ is a figure-eight closed curve if it has exactly one self-intersection point. It must be homotopic to the composition $\tilde{\eta}_i\circ \tilde{\eta}_j$ for some $i\neq j$. Their orientations satisfy that the intersection point at $x$ can not be removed.
    \end{definition*}
\begin{definition*}
     We say that a non-simple closed curve in $P$ is a one-sided iterated eight closed curve if it is homotopic to the composition $\tilde{\eta}_i\circ \tilde{\eta}_j^n$ for some $i\neq j$ with iteration number $n\geq 2$. We take the orientations of $\tilde{\eta}_i$ and $\tilde{\eta}_j$ such that $\tilde{\eta}_i\circ \tilde{\eta}_j$ is a figure-eight closed geodesic curve.
\end{definition*}
Figure \ref{pic f8 one sided} shows a figure-eight closed geodesic and a one-sided iterated eight closed geodesic in $S_{0,3}$ with iteration number $n=3$.

\begin{figure}[h]
\centering
\includegraphics[scale=0.8]{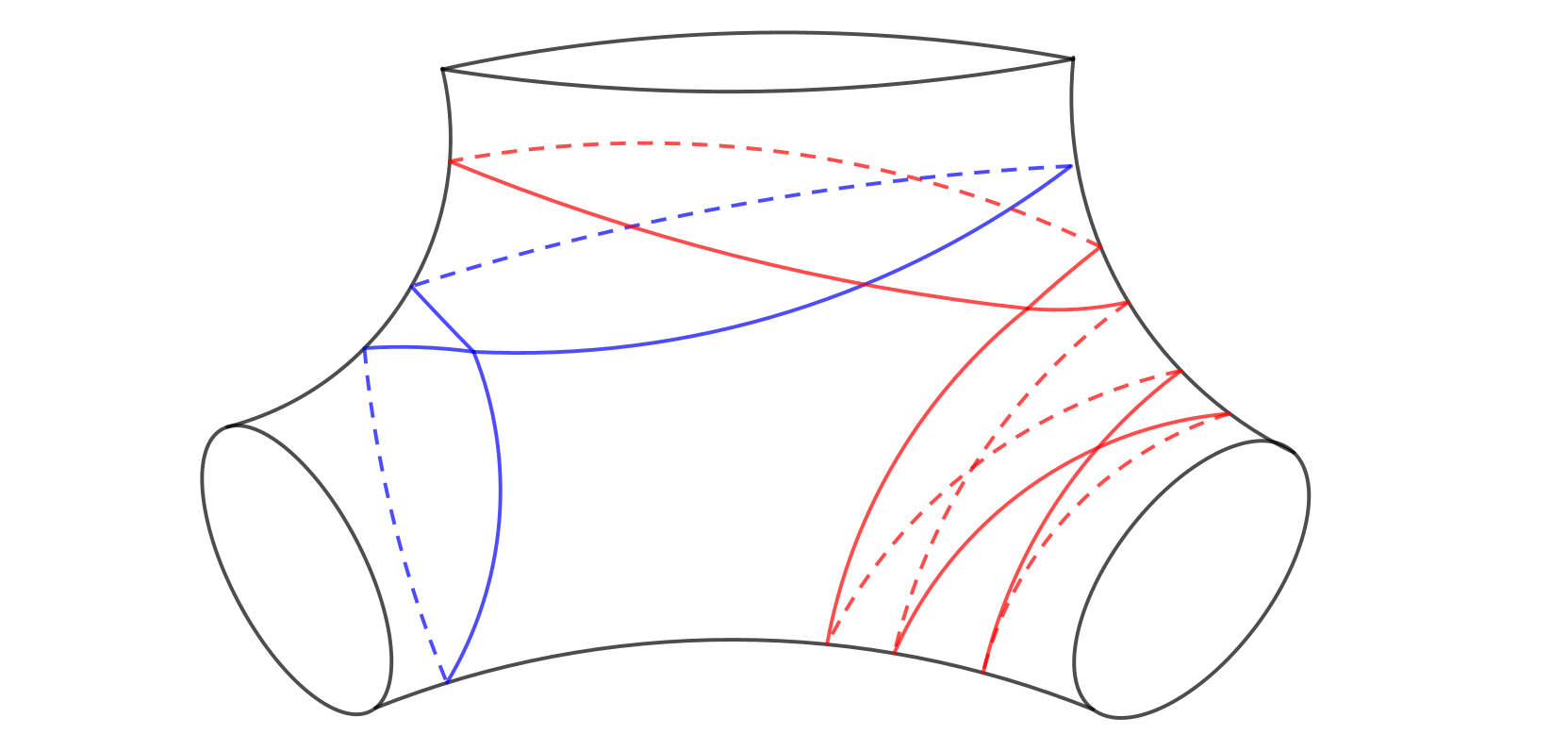}
\caption{ figure-eight and one-sided iterated eight closed geodesics}
\label{pic f8 one sided}
\end{figure}

Let $P$ be a pair of pants with three boundary geodesics of length $x,y,z\geq 0$. The length $l$ of the figure-eight closed geodesic that spirals around the boundaries of length $x$ and $y$ is given by (See \eg \cite[Equation 4.2.3]{buser2010geometry})  \begin{align}\label{eq int f-8 in s03  1}
    \cosh\frac{l}{2}=2\cosh \frac{x}{2}\cosh \frac{y}{2}+\cosh\frac{z}{2}.
\end{align}
 We write \eqref{eq int f-8 in s03  1} as $$
\ell=L_{f-8}(x,y,z).
$$
The length $l$ of a one-sided iterated eight closed geodesic that spirals around the boundary of length $x$ for once and spirals around the boundary of length $y$ for $n\geq 2$ times is given by (See \eg \cite[Section 3]{baribaud1999closed}) \begin{align}\label{eq int one-sided in s03  1}
    \cosh\frac{l}{2}=\frac{\sinh\frac{n+1}{2}y}{\sinh\frac{y}{2}}\cosh\frac{x}{2}+\frac{\sinh\frac{n}{2}y}{\sinh\frac{y}{2}}\cosh\frac{z}{2}.
\end{align}
 We can write \eqref{eq int one-sided in s03  1} as 
$$
\ell_\gamma(X)=L_n(x,y,z).
$$
Anantharaman-Monk in \cite{AM23-2/9} classified the filling geodesics in $S_{1,1}$ and $S_{0,3}$. 

\begin{theorem}\label{thm classify filling}\cite[Proposition 8.8]{AM23-2/9}
    Any filling closed geodesic in $S_{1,1}$ is double-filling. A filling geodesic is double-filling in $S_{0,3}$ if and only if it is neither a figure-eight nor a one-sided iterated eight closed geodesic.
\end{theorem}

 For $k\geq 1$, $L> 0$ and any hyperbolic surface $X$, we define the counting function $N_k^{2-fill}(X,L)$  by $$
 N_k^{2-fill}(X,L)=\#\{\Gamma=\left(\gamma_1,\cdots,\gamma_k\right);\, \Gamma \textit{ double-fills } X,\,\, \sum_{i=1}^k\ell_{\gamma_i}(X)\leq L\}.
 $$

We will prove the following similar counting result of Theorem \ref{thm filling count} for double-filling multi-geodesics.
\begin{theorem}\label{thm double filling count}
 For any $k\geq 1$, $\epsilon>0$  and $m=2g-2+n\geq 1$, there exists a constant $C(k,\epsilon,m)$ only depending on $k,m$ and $\epsilon$ such that for any hyperbolic surface $X\in\mathcal{T}_{g,n}(x_1,\cdots,x_n)$, $$
N_k^{2-fill}(X,T)\leq C(k,\epsilon,m)\cdot(1+L)^{k-1}\cdot e^{T-(1-\epsilon)\sum_{i=1}^nx_i}.
$$
\end{theorem}

Both Theorem \ref{thm filling count} and Theorem \ref{thm double filling count} follow an important rule for a given topological type of filling $k$-tuple $\gamma$ in $X$. As the length of $\partial X$ increases $\delta L$, $\gamma$ will also increase at least approximately $\frac{1-\epsilon}{2}\delta L$ (if $\gamma$ is filling) or $(1-\epsilon)\delta L$ (if $\gamma$ is double-filling). 

The following result is well-known to experts using Thurston's strip deformation. One may see \eg \cite{Thu-strip, PT09, DGK16} for detailed discussions.  For completeness,  we will give a proof that is generalized from \cite{parlier2005}. Both proofs rely on considering the length change of each geodesic when reducing the boundary length of the surface. We put the proofs of Theorem \ref{thm double filling count} and \ref{thm mono length} in Appendix \ref{appendix counting}.
\begin{theorem}\label{thm mono length}
    For any $X\in \mathcal{T}_{g,n}(x_1,\cdots,x_n)$ and for any smaller boundary length $(y_1,\cdots,y_n)$ satisfying $x_i\geq y_i\geq 0$ for all $i=1,\cdots,n$, there exists $Y\in \mathcal{T}_{g,n}(y_1,\cdots,y_n)$ such that any closed curve $\eta\sbs S_{g,n}$ has length
    $$\ell_{\eta}(Y)\leq \ell_{\eta}(X).$$
\end{theorem}

Then this implies the following monotonicity of counting of filling geodesics:
\begin{theorem}\label{thm mono counting}
    For any $X\in \mathcal{T}_{g,n}(x_1,\cdots,x_n)$ and $x_i\geq y_i$ for $i=1,\cdots,n$, there exists $Y\in \mathcal{T}_{g,n}(y_1,\cdots,y_n)$ such that 
$$
N_k^{fill}(X,L)\leq N_k^{fill}(Y,L)
$$
for any $L>0$.
\end{theorem}

\begin{proof}
    Choose $Y\in \mathcal{T}_{g,n}(y_1,\cdots,y_n)$ to be the surface obtained in Theorem \ref{thm mono length}. Let $\Gamma=(\gamma_1,\cdots,\gamma_k)$ be a filling $k$-tuple counted in $N_k^{fill}(X,L)$ (i.e. $\gamma_i\in P(X)$, $\Gamma$ fills $X$ and $\sum \ell_{\gamma_i}(X)\leq L$). Considering $\Gamma$ and $\gamma_i$ as free homotopy classes in $S_{g,n}$, $\Gamma$ is still filling in $Y$. By Theorem \ref{thm mono length}, $\sum \ell_{\gamma_i}(Y)\leq\sum \ell_{\gamma_i}(X)\leq L$. So $\Gamma$ (as the free homotopy class) is also counted in $N_k^{fill}(Y,L)$. Hence $N_k^{fill}(X,L)\leq N_k^{fill}(Y,L)$.
\end{proof}

For a hyperbolic surface $X$, we use $\mathcal{P}(X)$ to represent the set of all oriented primitive geodesics in $X$.
Let $$P(X, L)=\#\{\gamma\in \mathcal{P}(X);\ell_\gamma(X)\leq L\}$$ be the number of primitive closed geodesics of length $\leq L$ on $X$.
The classical \textit{Collar Lemma} (See \eg \cite[Theorem 4.1.6]{buser2010geometry}) states that all simple closed geodesics with lengths $\leq 2\arcsinh 1$ are pairwisely disjoint. It is also known from \cite[Lemma 6.6.4]{buser2010geometry} that for any 
$X\in\M_{g,n}$ and $L>0$, the number of closed geodesics of length $\leq L$ on $X$ which are not iterates of closed geodesics of length $\leq 2\arcsinh 1$ is bounded by $(g-1+\frac{n}{2})e^{L+6}.$  Combining these two estimates, we have the rough estimate below, which is sufficient to approximate the remainder terms.
\begin{lemma}\label{lemma uniform L+6}
    For any $X\in\M_{g,n}(L_1,\cdots,L_n)$ and $L>0$, 
    we have $$
    P(X,L)\leq (g-1+\frac{n}{2})e^{L+6}+3g-3+n.$$
\end{lemma}
The well-known prime geodesic theorem for a finite volume Riemann surface is the following:
\begin{theorem}\label{thm PGT finite}
    Let $X$ be a geometrically finite complete hyperbolic surface without boundary and with finite volume. Then as $L\to \infty$,
    $$
    P(X,L)=\mathrm{Li}(e^L)+\sum_{k=1}^p\mathrm{Li}(e^{\alpha_k L})+O\left(e^{\frac{3}{4}L}\right),
    $$
    where $\mathrm{Li}(x)=\int_2^x\frac{dt}{\log t}$, and $\frac{1}{2}\leq \alpha_p\leq \alpha_{p-1}\leq \cdots<\alpha_0=1$ corresponds to eigenvalues $\{\lambda_k\}_{k=0}^p$ of the Laplacian in $[0,\frac{1}{4}]$ by $\alpha_k(1-\alpha_k)=\lambda_k$. 
\end{theorem}

The precise expansion in the case of the compact surface is due to Hejhal, Huber, and Randol \cite{hejhal1976,huber1961,randol1977}.  The case of noncompact and finite volume is proved by Sarnak \cite{sarnak1980prime}. 
For the case of infinite volume, the asymptotic form is as follows:
\begin{theorem}\label{thm PGT infinite}
 Let $X=\Gamma\backslash\mathbb{H}$ be a geometrically finite complete hyperbolic surface without boundary and with infinite volume. Then 
    $$
\lim_{L\to\infty}    \frac{P(X,L)}{e^{\delta L}/\delta L}=1,
    $$
    where $\delta=\delta(X)$ is the Hausdorff dimension of the limit set of $\Gamma$. 
\end{theorem}
This asymptotic leading term is proved by Guillop\'e and Lalley \cite{guillope1986,lalley1989,guillope1992}. For precise remainder terms, see a result of Naud \cite{naud2005precise}.
A result that is very similar in form to the prime geodesic theorem is the following theorem. See \eg \cite[Corollary 2 of Theorem 4.1.1]{roblin2003ergodicite}.
\begin{theorem}\label{thm  PGT from p to q}
    Let $X=\Gamma\backslash\mathbb{H}$ be a geometrically finite complete hyperbolic surface without boundary and with infinite volume. For  any $p,q\in \mathbb{H}$, there exists a $C_{p,q,X}$ such that 
$$\lim_{L\to \infty}\frac{
\#\{\gamma\in \Gamma;d(p,\gamma q)\leq L\}}{e^{\delta L}}=C_{p,q,X}.
$$    
\end{theorem}
The Hausdorff dimension $\delta$ of the limit set $\Gamma$ is $1$ when $X$ is of finite volume and $0<\delta<1$ when $X$ is of infinite volume. When $X$ contains cusps, we have $\delta>\frac{1}{2}$. See Beardon \cite{beardon1968exponent,beardon1971inequalities}.
We focus on the Hausdorff dimension $\delta_l=\delta(l)$ of the limit set of $\Gamma$ for the hyperbolic surface $X=\Gamma\backslash\mathbb{H}\in \T_{0,3}(l,0,0)$. This function has been extensively studied by \eg \cite{beardon1968exponent,patterson1976,pignataro1984hausdorff,phillips1985laplacian,phillips1985spectrum}. It is known that $\delta_0=1$, $\lim_{l\to\infty}\delta_l=\frac{1}{2}$ and $\delta_l$ is a strictly decreasing Lipschitz continuous function on $l\in[0,\infty)$. So we can define the inversion function $l_\delta=l(\delta)$ for $\delta\in(\frac{1}{2},1]$.

Both Theorem \ref{thm PGT finite} and Theorem \ref{thm PGT infinite} are for fixed hyperbolic surfaces. When considering a class of surfaces, they do not offer a universal estimation of the region on $L$ when the prime geodesic theorem is precise enough. Wu and Xue proved a version of the prime geodesic theorem for Weil-Petersson random closed hyperbolic surfaces, see \cite{WX2022prime}.  
In this paper, we prove the following two counting results for pants or one-holed tori.

\begin{lemma}\label{counting s03 e 1/2 L}
    For any $\delta>\frac{1}{2},$ there exist $C_\delta>0$ such that if $x\geq l_\delta$, and $y,z\geq 0$, then for any $X\in \T_{0,3}(x,y,z)$, the number of filling primitive geodesics of length $\leq L$ on $X$ satisfies
  $$
  N_1^{fill}(X,L)\leq 
  C_\delta \cdot \frac{e^{\delta L}}{L}
  $$
  for any $L>0$.
  \end{lemma}
\begin{proof}
    We combine Theorem \ref{thm mono counting} and Theorem \ref{thm PGT infinite}. 
    By Theorem \ref{thm PGT infinite}, there exists $C_\delta>0$ such that 
 $$
 P(Y,L)\leq C_\delta \cdot \frac{e^{\delta L}}{L}
 $$
 for $Y\in\T_{0,3}(l_\delta,0,0)$ and any $L>0$.
Notice that in $Y$, all primitive closed geodesics are filling except the boundary geodesic of length $l_\delta$. 
Then for $X\in \T_{0,3}(x,y,z)$ with $x\geq l_\delta$, and $y,z\geq 0$, it follows from Theorem \ref{thm mono counting} that $$
 N_1^{fill}(X,L)\leq N_1^{fill}(Y,L)\leq P(Y,L)\leq C_\delta \cdot \frac{e^{\delta L}}{L}.
 $$
 The proof is complete.
\end{proof}

\begin{theorem}\label{thm counting s11 sys small}
    Let 
     $\delta_l$ be the Hausdorff dimension of the limit set of $\Gamma$ with $\Gamma\backslash\mathbb{H}\in \T_{0,3}(l,0,0)$. For any $\epsilon>0$, there exist constants $s_{\epsilon,l},\tilde{C}_{\epsilon,l}>0$ depending on $l,\epsilon$ such that for any $X\in \T_{1,1}(l)$ satisfying $\sys(X)\leq s_{\epsilon,l}$,
    $$
    N_{1}^{fill}(X,L)\leq \tilde{C}_{\epsilon,l}\cdot e^{ (1+\epsilon)\cdot \delta_l L}
    $$
    for any $L>0$.
\end{theorem}
In the rest of this section, we will prove Theorem \ref{thm counting s11 sys small}.

Let $X\in \T_{1,1}(l)$ with $2s=\sys(X)<2\arcsinh 1$.
 By the standard \textit{Collar Lemma}(See \eg \cite{keen1974collars}), the systole is achieved by a unique simple closed geodesic $\eta,$ and there is an embedding collar neighborhood $C(\eta)$ centered at $\eta$ in $X$ isomorphic to $$(\rho,t)\in [-w(s),w(s)]\times \mathbb{R}\backslash\mathbb{Z}$$ with the metric $$
 ds^2=d\rho^2+(2s)^2\cdot\cosh^2\rho\cdot   dt^2, 
 $$
 where \begin{equation}\label{eq estimate w(s)}
     w(s)=\arcsinh \left(\frac{1}{\sinh s}\right)=\log \left(\frac{1}{s}\right)+O(1).
     \end{equation}
     Cut $X$ along $\gamma$, we will get a pair of pants, denoted by $X(l,s)$. 
 The collar $C(\gamma)$ is separated into two half-collars $HC_1$ and $HC_2$.  Set $X(l,s)\setminus \left(HC_1\cup HC_2\right)$ to be the interior part $\textrm{Int}(X(l,s))$ of $X(l,s)$. Then any point in $\textrm{Int}(X(l,s))$ has a distance at least $w(s)$ to $\gamma$.  
 Cut $X(l,s)$ along the three shortest perpendicular geodesic arcs between its different geodesic boundary components, we will get two congruent right-angled hexagons $H_1$ and $H_2$. Denote the six angles by $A_1, B_1, C_1, A_2, B_2, C_2$ respectively such that 
$A_i,B_i$ lie in $HC_i$ for $i=1,2$ and $A_i,C_i$ are adjacent right angles. The boundary $\partial HC_i$ intersects with the geodesic arc $A_iC_i$ at $D_i$ and intersects with the geodesic arc $B_1B_2$ at $E_i$ for each $i=1,2$.
See Figure \ref{pic X(x,s)}. Notice that $A_1$ and $A_2$ are not glued together in $X$ in most cases when the twist parameter along $\gamma$ is not zero. 
\begin{figure}[h]
\centering
\includegraphics[scale=0.5]{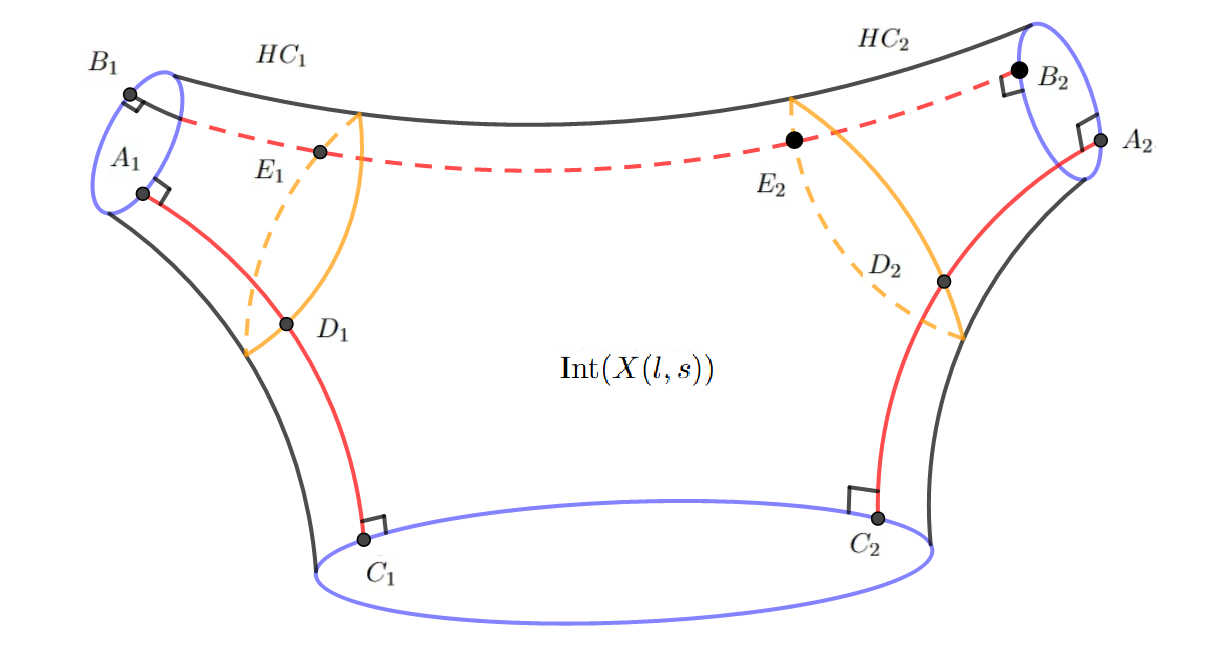}
\caption{$X(l,s)=\textrm{Int}(X(l,s))\cup HC_1\cup HC_2$}
\label{pic X(x,s)}
\end{figure}

Now we estimate the size of $\textrm{Int}(X(l,s))$. Firstly,
 by  \eg \cite[Formula Glossary]{buser2010geometry}
    \begin{align*}
        &d(A_1,C_1)=d(A_2,C_2)\\
        =&\arccosh\left(\frac{\cosh d(C_1,C_2)\cosh d(A_1,B_1)+\cosh d(A_2,B_2)}{\sinh d(C_1,C_2)\sinh d(A_1,B_1)}\right)\\
        =&\arccosh \frac{\cosh s(\cosh \frac{l}{2}+1)}{\sinh s\sinh \frac{l}{2}}\\
        =&\log \frac{1}{s}+O_l(1).\\
    \end{align*}
So we have\begin{equation}\label{eq size int-1}
    \begin{aligned}
         &d(D_1,C_1)=d(D_2,C_2)\\
    =&d(A_1,C_1)-d(A_1,D_1)\\
     =&d(A_1,C_1)-w(s)\\
   =&O_l(1).
    \end{aligned}
\end{equation}

The length of one boundary of $C(\gamma)$ is \begin{equation}\label{eq size int-2}
   \ell(\partial HC_1)=\ell(\partial HC_2)=2s\cdot\cosh\left(\arcsinh \left(\frac{1}{\sinh s}\right)\right)=2+o(1).
\end{equation}
By \eqref{eq size int-1} and \eqref{eq size int-2} we have \begin{equation}\label{eq size int-3}
    d(E_1,E_2)=O_l(1),
\end{equation}
and that for any $P,Q\in \mathrm{Int}(X(l,s))$, 
\begin{equation}\label{eq size int-4}
    d(P,Q)=O_l(1).
\end{equation}

Now we count the number of geodesic arcs in $C(\gamma)$.
\begin{lemma}\label{counting in C(gamma)}
    If $P,Q$ belong to different boundary components of $\partial C(\gamma)$, there exists a uniform constant $R_0>0$ such that the number of geodesic arcs in $C(\gamma)$ connecting $P,Q$ with length $\leq L$ is bounded by $$
    \frac{L-2w(s)+R_0}{s}
    $$ 
    for $L\geq 2w(s)$.
\end{lemma}
\begin{proof}
    Firstly, by \eqref{eq size int-2}, we can assume that the shortest geodesic arc connecting $P$ and $Q$ is of length $2w(s)$ and intersects perpendicularly with $\gamma$. For general $P$ and $Q$, we just move $P$ a finite distance to the symmetric position of $Q$ inside the collar and compare the length of each geodesic arc.

    We lift the geodesic arcs that spiral around $\gamma$ for $n$ times to the upper half-plane $\mathbb{H}$ with the metric $\frac{dx^2+dy^2}{y^2}$.
    Assume that $\gamma$ is lifted to $x=0$ and $P$ is lifted to a point on the circle $|z|=1$. If $n\geq 0$, then $Q$ is lifted to a point on the circle $|z|=e^{2ns}$. See Figure \ref{pic lift H}.
\begin{figure}[h]
\centering
\includegraphics[scale=0.12]{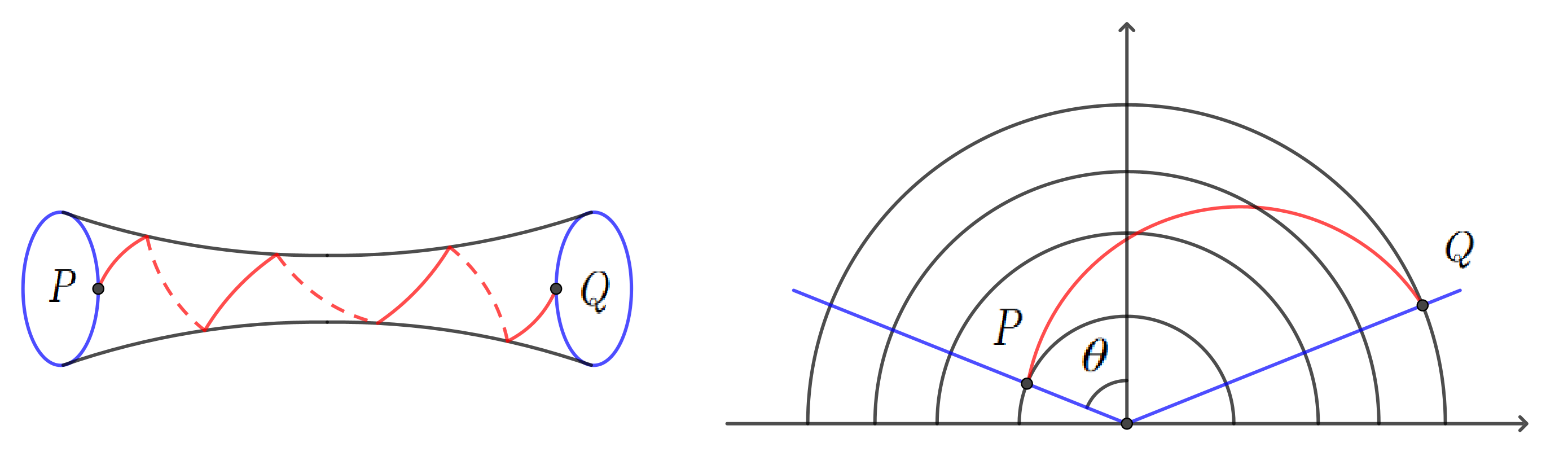}
\caption{A geodesic arc that spirals $3$ times}
\label{pic lift H}
\end{figure}
The angle $\theta$ between the vector from $0$ to $P$ and the lift of $\gamma$ is given by \begin{equation}\label{eq w(s) with theta}
    \cos \theta=\frac{1}{\cosh w(s)}.
\end{equation}
    Notice that as $s\to 0,$ $w(s)\to +\infty$ and $\theta\to \frac{\pi}{2}.$
Now $P$ is lifted to $(-\sin\theta,\cos\theta)$, and $Q$ is lifted to $(e^{2ns}\sin\theta,e^{2ns}\cos\theta).$
It follows from \eqref{eq estimate w(s)} and \eqref{eq w(s) with theta} that the length of the geodesic arc is given by 
    \begin{align*}
        &\cosh d\left((-\sin\theta,\cos\theta),(e^{2ns}\sin\theta,e^{2ns}\cos\theta) \right)\\
       =&1+\frac{ (e^{2ns}+1)^2\sin^2\theta+(e^{2ns}-1)^2\cos^2\theta }{2e^{2ns}\cos^2\theta}\\
       =&\exp \left(O(1)\right)\cdot \frac{e^{2ns}}{\cos^2\theta}
        =\exp \left(O(1)\right)\cdot\cosh^2w(s)\cdot e^{2ns}\\
       =&\exp \left(O(1)\right)\cdot\sinh^2w(s)\cdot e^{2ns}
       =\exp \left(O(1)\right)\cdot\frac{e^{2ns}}{\sinh^2 s}\\
    =&\exp\left(O(1)\right)\cdot\frac{e^{2ns}}{s^2},
    \end{align*}
where all implied constants are uniform for $s<\arcsinh 1$.
So if the geodesic arc that connects $P, Q$ and spirals around $\gamma$ for $n\geq 0$ times has length $\leq L$, we have $$
\exp\left(O(1)\right)\frac{e^{2ns}}{s^2}\leq \cosh L,
$$
which implies $$
n\leq \frac{L-2\log \frac{1}{s} +R^\prime}{2s}
$$
for some uniform $R^\prime>0$. The number of possible negative $n$ has the same upper bound. Finally the lemma follows from \eqref{eq estimate w(s)}.
\end{proof}
Assume that a filling geodesic $\eta$ in $X$ intersects with $\gamma$ for $k\geq 1$ times. Then $\eta$ can be decomposed into $\eta_1\cup \eta_2\cdots \eta_{2k-1}\cup \eta_{2k}$ such that \begin{enumerate}
    \item $\eta_{2i-1}$ is a geodesic arc in $C(\gamma)$ with two ends on different components of $\partial C(\gamma)$ for $i=1,\cdots,k$;
    \item $\eta_{2i}$ is a geodesic arc in $X(l,s)$ with two ends on $\partial C(\gamma)$ for $i=1,\cdots,k;$
    \item  $\eta_{i}$ and $\eta_{i+1}$ are connected for $i=1,\cdots,2k$. $\eta_{2k}$ and $\eta_1$ are connected.
\end{enumerate}

\begin{definition*}
    For any $X\in \T_{1,1}(l)$  of systole $\leq 2\arcsinh 1$, let $\gamma\in \mathcal{P}(X)$ be the unique unoriented geodesic of length $\sys(X)$.
    Define $$
    N_1^{fill}(X,L,k)=\#\{ \eta\in N_1^{fill}(X,L); \eta \textit{ intersects } \gamma \textit{ for } k \textit{ times}   \}
    $$
    for any $k\geq 1$.
\end{definition*}
By the definition, for any $X\in \T_{1,1}(l)$  of systole $\leq 2\arcsinh 1$,
\begin{equation}\label{eq N_1(X,L)=sum N_1(X,L,k)}
    N_{1}^{fill}(X,L)=\sum_{k=1}^\infty N_1^{fill}(X,L,k).
\end{equation}
Now we count $N_1^{fill}(X, L, k)$ as the appetizer of the proof of Theorem \ref{thm counting s11 sys small}.

\begin{lemma}\label{lemma N(X,L,1) count}
     There exist a constant  $C_{l}^\prime>0$ depending on $l$ such that for any $X\in \T_{1,1}(l)$  of systole $2s\leq 2\arcsinh 1$,
    $$
    N_{1}^{fill}(X,L,1)\leq  C_l^\prime \cdot s^{2\delta_l-1}\cdot e^{\delta_l\cdot L} 
    $$
    for any $L>2w$. 
\end{lemma}
\begin{proof}
For $\eta=\eta_1\cup\eta_2\in N_{1}^{fill}(X,L,1),$ we assume $\partial \eta_1=\{p,q\}$ and $p\in \partial HC_1$, $q\in \partial HC_2$.  
 We connect $p$ with $D_1$ and $q$ with $D_2$ by segments in $\partial C(\gamma).$ We can assume that the orientations of these two segments are compatible with $\mathrm{Int}(X(l,s)).$ By passing to the composition with the two segments, $\eta_1$ and $\eta_2$ will deform to smooth geodesic arcs $\tilde{\eta_1}$ and $\tilde{\eta_2}$.
By \eqref{eq size int-2}, there exist a uniform $c>0$ such that
 $\tilde{\eta_1}$ is a geodesic arc in $C(\gamma)$ connecting $D_1$ and $D_2$ of length $2w(s)\leq \ell(\tilde{\eta_1})\leq \ell(\eta_1)+c$, and
 $\tilde{\eta_2}$ is a geodesic arc in $X(l,s)$ connecting $D_1$ and $D_2$ of length $\ell(\tilde{\eta_2})\leq \ell(\eta_2)+c$.
 So $\ell(\tilde{\eta_1})+\ell(\tilde{\eta_2})\leq \ell(\ell)+2c\leq L+2c.$ By the construction, $\tilde{\eta_1}\cup\tilde{\eta_2}$ is homotopic to $\eta$.
 So we have \begin{equation}\label{eq s11 k=1 eq1}\begin{aligned}
      &N_{1}^{fill}(X,L,1)\\
      \leq& \sum_{k=[2w(s)]}^{[L+2c]}\#\{ \tilde{\eta_1}:k\leq  \ell(\tilde{\eta_1})<k+1 \}\#\{\tilde{\eta_2}:\ell(\tilde{\eta_2})<L+2c-k \}.
      \end{aligned}
 \end{equation}
By Lemma \ref{counting in C(gamma)},\begin{equation}\label{eq s11 k=1 eq2}
    \#\{ \tilde{\eta_1}:k\leq  \ell(\tilde{\eta_1})<k+1 \}\leq \frac{k+1-2w(s)+R_0}{s}.
\end{equation}

We composite $\tilde{\eta_2}$ with the geodesic segment $\overline{D_1C_1}$ and $\overline{D_2C_2}$, and deform it to a smooth geodesic arc $\hat{\eta_2}$ with fixed ends $C_1$ and $C_2$.
By \eqref{eq size int-1} there exists some $c_l>0$ depending only on $l$ such that $\ell(\hat{\eta_2})\leq \ell(\tilde{\eta_2})+c_\ell$. Hence we have\begin{equation}\label{eq s11 k=1 eq3.1}
    \#\{\tilde{\eta_2}:\ell(\tilde{\eta_2})<L+2c-k\}\leq \#\{\hat{\eta_2}:\ell(\hat{\eta_2})<L+2c-k+c_l\}.
\end{equation}
Now we use a very similar strategy in the proofs of 
Theorem \ref{thm filling count}, Theorem \ref{thm double filling count}, and Theorem \ref{thm mono counting} to bound the right side of \eqref{eq s11 k=1 eq3.1}. Refer to Appendix \ref{appendix counting}, or \cite{wx22-3/16, WX2022prime} for more details.

Notice that $X(l,s)\in \mathcal{T}_{0,3}(l,2s,2s)$. We can shorten the boundary geodesic adjacent to $HC_1$ to a cuspidal point, and deform $X(l,s)$ to $X^\prime\in \mathcal{T}_{0,3}(l,2s,0)$.
On $X^\prime$, we can define the intersection points $A_i^\prime, B_i^\prime, C_i^\prime$ for $i=1,2$ between the three shortest perpendicular geodesics with boundary geodesics, as $A_i, B_i, C_i$ in $X(l,s)$. Notice that $A_1^\prime$ and $B_1^\prime$ overlap at the cuspidal point.

Cut $X(l,s)$ along $\overline{A_1C_1}$ and $\overline{B_1B_2}$ into an octagon.
For any $\hat{\eta_2}$, we can separate it into successive geodesic arcs by the successive intersection points $p_1,\cdots,p_k$ of $\hat{\eta_2}$ with $\overline{A_1C_1}$ and $\overline{B_1B_2}$.
Assume $p_0=C_1$, $p_{k+1}=C_2$, and $\gamma_i=\hat{\eta_2}|_{[p_{i-1},p_i]}$ for $i=1,\cdots,k+1$.
Now we set $p_0^\prime=C_1^\prime$, $p_{k+1}^\prime=C_2^\prime$. For $i=1,\cdots,k$, if $p_i\in \overline{A_1C_1}$, we set $p_i^\prime\in \overline{A_1^\prime C_1^\prime}$ with $d(p_i,C_1)=d(p_i^\prime,C_1^\prime)$, and  if $p_i\in \overline{B_1B_2}$, we set $p_i^\prime\in \overline{B_1^\prime B_2^\prime}$ with $d(p_i,B_2)=d(p_i^\prime,B_2^\prime)$. Now for each $i$ we can connect $p_{i-1}^\prime$ and $p_i^\prime$ by a geodesic arc $\gamma_i^\prime$ which connects the corresponding edges of  $X^\prime\setminus \{\overline{A_1^\prime C_1^\prime}\cup \overline{B_1^\prime B_2^\prime}\}$ as $\gamma_i$ in the octagon $X(l,s)\setminus\{\overline{A_1 C_1}\cup \overline{B_1 B_2}\}$. 
Now $\gamma_i^\prime (i=1,\cdots,k+1)$ 
  will be connected to a piecewise geodesic arc connecting $C_1^\prime, C_2^\prime$. By the same argument as in Lemma \ref{thm mono length},
  we have $\ell(\gamma_i^\prime)\leq \ell(\gamma_i)$ for all $1\leq i\leq k+1$. Let $\hat{\eta_2}^\prime$ be the smooth geodesic arc in $X'$ that connects $C_1^\prime$ and $C_2^\prime$ homotopic to $\cup_{i=1}^{k+1} \gamma_i^\prime$. Then $\ell(\hat{\eta_2}^\prime) \leq \sum\ell(\gamma_i^\prime) \leq \sum\ell(\gamma_i) = \ell(\hat{\eta_2})$.
  It follows that \begin{equation}\label{eq s11 k=1 eq3.2}
      \begin{aligned}
          &\#\{\hat{\eta_2}:\hat{\eta_2} \textit{ connects } C_1,C_2\textit{ in } X_{l,s},
          \ell(\hat{\eta_2})<L+2c-k+c_l\}\\
          \leq& \#\{\hat{\eta_2}^\prime:\hat{\eta_2}^\prime \textit{ connects } C_1^\prime,C_2^\prime\textit{ in } X^\prime,\ell(\hat{\eta_2}^\prime)<L+2c-k+c_l\}.
      \end{aligned}
  \end{equation}

Similarly,  we can shorten the boundary geodesic adjacent to $HC_2$ to a cuspidal point, and deform $X^\prime$ to $X^{\prime\prime}\in \mathcal{T}_{0,3}(l,0,0)$. By the same argument as above, we have \begin{equation}\label{eq s11 k=1 eq3.3}
      \begin{aligned}
          &\#\{\hat{\eta_2}^\prime:\hat{\eta_2}^\prime \textit{ connects } C_1^\prime,C_2^\prime\textit{ in } X^\prime,
          \ell(\hat{\eta_2}^\prime)<L+2c-k+c_l\}\\
          \leq& \#\{\hat{\eta_2}^{\prime\prime}:\hat{\eta_2}^{\prime\prime} \textit{ connects } C_1^{\prime\prime},C_2^{\prime\prime}\textit{ in } X^{\prime\prime},\ell(\hat{\eta_2}^{\prime\prime})<L+2c-k+c_l\}
      \end{aligned}
  \end{equation}
where $C_1^{\prime\prime},C_2^{\prime\prime}$ are the corresponding points of $C_1,C_2$ in $X^{\prime\prime}$.
Now we can combine \eqref{eq s11 k=1 eq3.1}, \eqref{eq s11 k=1 eq3.2}, \eqref{eq s11 k=1 eq3.3} and Theorem \ref{thm  PGT from p to q} to get \begin{equation}\label{eq s11 k=1 eq3.4}
    \begin{aligned}
        &\#\{\tilde{\eta_2}:\ell(\tilde{\eta_2})<L+2c-k\}\\
        \leq &\#\{\hat{\eta_2}^{\prime\prime}:\hat{\eta_2}^{\prime\prime} \textit{ connects } C_1^{\prime\prime},C_2^{\prime\prime}\textit{ in } X^{\prime\prime},\ell(\hat{\eta_2}^{\prime\prime})<L+2c-k+c_l\}\\
       \leq & \hat{C_l}\cdot   e^{\delta_l(L+2c-k+c_l)}\\
       = & \left(\hat{C_l}e^{\delta_l c_l} \right)\cdot e^{\delta_l(L+2c-k)}.
    \end{aligned}
\end{equation}
If follows from \eqref{eq s11 k=1 eq1}, \eqref{eq s11 k=1 eq2} and \eqref{eq s11 k=1 eq3.4} that \begin{align*}
    &N_1^{fill}(X,L,1)\\
    \leq& \sum_{k=[2w(s)]}^{[L+2c]}\frac{k+1-2w(s)+R_0}{s}
\cdot\left(\hat{C_l}e^{\delta_l c_l} \right)e^{\delta_l(L+2c-k)}\\
    \leq& C_l^{\prime\prime}\cdot \frac{R_0}{s}\cdot  e^{ \delta_l(L+2c-[2w(s)])}\\
    \leq& C_l^\prime\cdot s^{2\delta_l-1}\cdot e^{\delta_l\cdot L}
\end{align*}
for some $C_l^\prime,C_l^{\prime\prime}>0$ which only depends on $l$. In the last inequality we use the fact \eqref{eq estimate w(s)} that $w(s)=\log \left(\frac{1}{s}\right)+O\left(1\right)$.
\end{proof}
\begin{proof}[Proof of Theorem \ref{thm counting s11 sys small}]
    We estimate $N_1^{fill}(X,L,k)$ for $k\geq 2$ using a similar argument as Lemma \ref{lemma N(X,L,1) count}.
Set $\eta=\eta_1\cup\eta_2\cup\cdots\cup \eta_{2k-1}\cup \eta_{2k}\in N_1^{fill}(X,L,k)$. For each segment $\eta_i$, we attach a segment in $\partial C(\gamma)$ to both ends to form a piecewise smooth geodesic arc with endpoints in $\{D_1,D_2\}$.
Then we deform it to a geodesic arc $\tilde{\eta}_i$ with the same endpoints. We can see that $\tilde{\eta}_1\cup \tilde{\eta}_2\cup\cdots\cup\tilde{\eta}_{2k-1}\cup\tilde{\eta}_{2k}$ is homotopic to $\eta$ with total length $\leq L+2ck$ for $c>0$ in the proof of Lemma \ref{lemma N(X,L,1) count}. Then we have \begin{equation}\label{eq N(x,L,k) eq-1}
\begin{aligned}
    &N_1^{fill}(X,L,k)\\
    \leq& \sum_{\substack{n_1+\cdots+n_k\leq L+2ck\\
    n_i\geq [2w(s)]}
    }\#\{(\tilde{\eta}_1,\tilde{\eta}_2):n_1\leq \ell(\tilde{\eta}_1)+\ell(\tilde{\eta}_2)\leq n_1+1   \}\\
    \cdot& \#\{(\tilde{\eta}_3,\tilde{\eta}_4):n_2\leq \ell(\tilde{\eta}_3)+\ell(\tilde{\eta}_4)\leq n_2+1   \}\\
&\cdots\\
\cdot &\#\{(\tilde{\eta}_{2k-1},\tilde{\eta}_{2k}):n_k\leq \ell(\tilde{\eta}_{2k-1})+\ell(\tilde{\eta}_{2k})\leq n_k+1   \}.\\
\end{aligned}
\end{equation}
   By the same argument as in the proof of Lemma \ref{lemma N(X,L,1) count}, we have 
   \begin{equation}\label{eq N(x,L,k) eq-1-1}
       \begin{aligned}
           & \#\{(\tilde{\eta}_{2i-1},\tilde{\eta}_{2i}):n_i\leq \ell(\tilde{\eta}_{2i-1})+\ell(\tilde{\eta}_{2i})\leq n_i+1   \}\\
       \leq& C_\ell^\prime\cdot s^{2\delta_l-1}\cdot e^{\delta_l\dot (n_i+1)}
       \end{aligned}
   \end{equation}
   for $n_i\geq [2w(s)]$.
Take \eqref{eq N(x,L,k) eq-1-1} into \eqref{eq N(x,L,k) eq-1}, we have
\begin{equation}\label{eq N(x,L,k) eq-2}
    \begin{aligned}
        &N_1^{fill}(X,L,k)\\
       \leq& \left(C_l^\prime s^{2\delta_l-1}\right)^{k}
       \cdot \sum_{\substack{n_1+\cdots+n_k\leq L+2ck\\
    n_i\geq [2w(s)]}
    } e^{\delta_l\sum_{j=1}^k (n_j+1)}\\
 \leq &e^{\delta_l\cdot L}\left(C_l^\prime e^{\delta_l(2c+1)} s^{2\delta_l-1} \right)^k
   \cdot \sum_{m=k}^{[L+2ck+k]}
    \sum_{
\substack{n_1+\cdots+n_k=m\\
n_i\geq 1}
}1\\
=&e^{\delta_l\cdot L}\left(C_l^\prime e^{\delta_l(2c+1)} s^{2\delta_l-1} \right)^k
   \cdot \sum_{m=k}^{[L+2ck+k]} {m-1\choose k-1}\\
\leq&e^{\delta_l\cdot L}\left(C_l^\prime e^{\delta_l(2c+1)} s^{2\delta_l-1} \right)^k
   \cdot{[L+2ck+k] \choose k}\\
\leq &e^{\delta_l\cdot L}\left(C_l^\prime e^{\delta_l(2c+1)} s^{2\delta_l-1} \right)^k\cdot \frac{(L+2ck+k)^k}{k!}\\
  \leq &e^{\delta_l\cdot L}\left(C_l^\prime e^{\delta_l(2c+1)} s^{2\delta_l-1} \right)^k \cdot \frac{e^{\epsilon \delta_l(L+2ck+k)}}{(\epsilon \delta_l)^k}.
    \end{aligned}
    \end{equation}
By \eqref{eq N_1(X,L)=sum N_1(X,L,k)} and \eqref{eq N(x,L,k) eq-2},  
    for $s<\arcsinh 1$, we have \begin{equation}\label{eq N(x,L,k) eq-4}
    N_1^{fill}(X,L)\leq e^{\delta_l(1+\epsilon)\cdot L}\cdot \sum_{k=1}^\infty \left(\frac{ C_l^\prime e^{\delta_l(1+\epsilon)(2c+1)}}{\epsilon\delta_l} s^{2\delta_l-1}\right)^k.
\end{equation}
Since $\delta_l>\frac{1}{2}$, there exists $s_{\epsilon,l}>0$ depending on $\epsilon,l$ such that $s_{\epsilon,l}<\arcsinh 1$ and \begin{equation}\label{eq N(x,L,k) eq-5}
\frac{ C_l^\prime e^{\delta_l(1+\epsilon)(2c+1)}}{\epsilon\delta_l} s_{\epsilon,l}^{2\delta_l-1}<1. 
\end{equation}
If we take $$
\tilde{C}_{\epsilon,l}=\frac{1}{1-\frac{ C_l^\prime e^{\delta_l(1+\epsilon)(2c+1)}}{\epsilon\delta_l} s_{\epsilon,l}^{2\delta_l-1}}, 
$$
for $s<s_{\epsilon,l}$, by \eqref{eq N(x,L,k) eq-4} and \eqref{eq N(x,L,k) eq-5} we have 
$$
N_1^{fill}(X,L)\leq \tilde{C}_{\epsilon,l}\cdot  e^{\delta_l(1+\epsilon)\cdot L},
$$
 which finishes the proof.
\end{proof}

\section{Weil-Petersson volumes}\label{sec Weil-Petersson volumes}
In this section, we include results in Weil-Petersson volumes of moduli spaces of hyperbolic Riemann surfaces. Denote $V_{g,n}(x_1,\cdots,x_n)$ as the Weil-Petersson volume of $\M_{g,n}(x_1,\cdots,x_n)$ and $V_{g,n}=V_{g,n}(0,\cdots,0)$ for short.

Firstly, we list the results of Mirzakhani and her co-authors.

\begin{theorem} \label{mir07 poly} \cite[Theorem 1.1]{mir07-Inv}.    
    The volume $V_{g,n}(x_1,\cdots,x_n)$ is a polynomial in $x_1^2,\cdots,x_n^2$ of degree $3g-3+n$. Moreover, if $$    V_{g,n}(x_1,\cdots,x_n)=\sum_{\alpha:|\alpha|\leq 3g-3+n}C_\alpha x_1^{2\alpha_1}\cdots x_n^{2\alpha_n},
    $$
    then $C_\alpha>0$ and $C_\alpha\in \pi^{6g-6+2n-2|\alpha|}$. Here $\alpha=(\alpha_1,\cdots,\alpha_n)\in \mathbb{N}^n$ and $|\alpha|=\alpha_1+\cdots+\alpha_n$.
\end{theorem}

\begin{lemma}\label{mir13 vgn}   \begin{enumerate}
    \item \cite[Lemma 3.2]{mir13}
   $$ V_{g,n}\leq V_{g,n}(x_1,\cdots,x_n)\leq e^{\frac{x_1+\cdots+x_n }{2}} V_{g,n}.$$
   \item\cite[Theorem 3.5]{mir13}
   For fixed $n\geq 0$, as $g\to\infty$, $$
   \frac{V_{g-1,n+2}}{V_{g,n}}=1+O_n\left(\frac{1}{g}\right)
   $$and $$
   \frac{(2g-2+n)V_{g,n}}{V_{g,n+1}}=\frac{1}{4\pi^2}+O_n\left(\frac{1}{g}\right).
   $$
   \item \cite[Corollary 3.7]{mir13} For any fixed $b,k,r\geq 0$ and $0\leq C<2\ln 2$,$$
   \sum_{  
  \substack{ g_1+g_2=g+1-k\\
   r+1\leq g_1\leq g_2}
   } e^{Cg_1}\cdot g_1^b\cdot V_{g_1,k}V_{g_2,k}\asymp \frac{V_{g}}{g^{2r+k}}.
   $$ 
   \end{enumerate}
\end{lemma}
\begin{lemma}\label{appendix mir13 vg-1,n+2 vg,n+1}\cite[Lemma 3.2]{mir13}  
    For any $g,n\geq 0$ with $2g-2+n>2$, $$
    V_{g-1,n+4}\leq V_{g,n+2}
    $$
    and $$
    b_0\leq \frac{4\pi^2(2g-2+n)V_{g,n}}{V_{g,n+1}}\leq b_1
    $$
    where $b_0=\frac{\pi^2}{3}-\frac{\pi^4}{30}$ and $b_1=\cosh{\pi}-\frac{\sinh \pi}{\pi}$.
\end{lemma}

The following asymptotic property is given by Mirzakhani-Zograf.
\begin{theorem}\label{thm mz15 asymp} \cite[Theorem 1.2]{MZ15}
    There exists a universal constant $C\in (0,\infty)$ such that for any $n=n(g)=o(\sqrt{g})$, as $g\to\infty$, $$
    V_{g,n}=C\frac{(2g-3+n)!(4\pi^2)^{2g-3+n} }{\sqrt{g}}\left( 
  1+O\left(\frac{1+n^2}{g}\right)\right).
    $$
    The implied constant here is independent of $n$ and $g$.
\end{theorem}
\begin{rem*}
    Based on numerical data, Zograf made a conjecture that $C=\frac{1}{\sqrt{\pi}}$ in \cite{zograf2008}.
\end{rem*}
\begin{corollary}\label{cor asymp vgn}
    If $n=n(g)=o(\sqrt{g})$, then $$
    V_{g,n}= V_{g-1,n+2}\left(1+O\left(\frac{1+n^2}{g}\right)\right)
    $$
    and $$
    (2g-2+n)V_{g,n}=V_{g,n+1}\left(1+O\left(\frac{1+n^2}{g}\right)\right).
    $$
\end{corollary}
Lemma \ref{appendix mir13 vg-1,n+2 vg,n+1} and Theorem \ref{thm mz15 asymp} directly implies the following corollary:
\begin{corollary}\label{cor vgn+2 leq vg+1n}
     For any $2g+n\geq 1$. 
$$V_{g,n+2}\prec V_{g+1,n}.$$
\end{corollary}
The following estimation was given in \cite{NWX23}. See \cite{mir-petri-19} for a similar version.

\begin{lemma}\label{lemma NWX vgn(x) old version} \cite[Lemma 22]{NWX23} There exists a constant $c(n)>0$ independent of $g,x_1,\cdots,x_n$ but related to $n$ such that $$
\prod_{i=1}^n\!\frac{\sinh(x_i/2)}{x_i/2}\!\left(1\!-\!c(n)\frac{x_1^2+\cdots+x_n^2}{g}\!\right)\!\leq \!\frac{V_{g,n}(x_1,\cdots,x_n)}{V_{g,n}}\!\leq \! \prod_{i=1}^n \frac{\sinh(x_i/2)}{x_i/2}.
$$
\end{lemma} 
These asymptotic properties were applied in the proof of spectral gap  $\frac{3}{16}-\epsilon$ of random closed hyperbolic surfaces $X\in \sM_g$ in \cite{wx22-3/16}.
When lifting the spectral gap from $\frac{3}{16}-\epsilon$ to a higher spectral gap, the remainder terms above is not precise enough.
Mizakhani-Zograf \cite{MZ15} found an algorithm to give more precise asymptotic expansions with remainder terms of finite order with respect to $\frac{1}{g}$ as $g\to\infty$ when $n$ is finite. Actually all the precise asymptotic results about Weil-Petersson volume in the rest part of this paper can be derived from their algorithm. We only list the forms with error terms $O(\frac{1}{g^2})$.
\begin{theorem}\label{thm MZ15 1/g}
\cite{MZ15}
    For any fixed $n\geq 0$:$$
   \begin{aligned}
    \frac{V_{g,n+1}}{8\pi^2g V_{g,n}}&=1+ \left((\frac{1}{2}-\frac{1}{\pi^2})n-\frac{5}{4}+\frac{2}{\pi^2}\right)\cdot\frac{1}{g}+O_n\left(\frac{1}{g^2}\right),\\
    \frac{V_{g-1,n+2}}{V_{g,n}}&=1+\frac{3-2n}{\pi^2}\cdot\frac{1}{g}+O_n\left(\frac{1}{g^2}\right).
   \end{aligned}$$
    The implied constants here are related to $n$ and uniform on $g$.
\end{theorem}
The following asymptotic of $V_{g,n}(x_1,\cdots,x_n)$ was calculated by \cite{MZ15} and recalculated in \cite{AM22-JMP}.
\begin{theorem}\label{thm AM22 1/g} For any $\textbf{x}=(x_1,\cdots,x_n)\in \mathbb{R}_{\geq 0}^n$, as $g\to\infty$,\begin{align*}
&\frac{V_{g,n}(\textbf{x})}{V_{g,n}}=\prod_{i=1}^n \frac{\sinh(x_i/2)}{x_i/2}
+\frac{f_n^1(\textbf{x})}{g}\\
+&O_n\left(\frac{1+(x_1^2+\cdots+x_n^2)^2}{g^2}\exp{\left(\frac{x_1+\cdots+x_n}{2}\right)}\right),
\end{align*}
where \begin{align*}
f_n^1(\textbf{x})=&\frac{1}{\pi^2}\sum_{i=1}^n\left[\cosh\left(x_i/2\right)+1-(\frac{x_i^2}{16}+2)\frac{\sinh(x_i/2)}{x_i/2}\right]\prod_{k\neq i}\frac{\sinh(x_k/2)}{x_k/2}\\
-&\frac{1}{2\pi^2}\sum_{1\leq i<j\leq n}\left[ \cosh\left(x_i/2\right)\cosh\left(x_j/2\right)+1-2\frac{\sinh(x_i/2)}{x_i/2}\frac{\sinh(x_j/2)}{x_j/2}\right]\\
\cdot&\prod_{k\neq i,j}\frac{\sinh(x_k/2)}{x_k/2}.
\end{align*}
 The implied constant here is related to $n$ and uniform on $g,x_1,\cdots,x_n.$
\end{theorem}

In this paper, we extend Lemma \ref{lemma NWX vgn(x) old version}, Theorem \ref{thm MZ15 1/g}, and Theorem \ref{thm AM22 1/g} to the case $n=o(\sqrt{g}).$ That is, we shall see how the error terms involved depend on $n$ in the following results.

\begin{lemma}\label{lemma NWX vgn(x)}  For $n=n(g)=o(\sqrt{g})$, there exists a uniform constant $c>0$ such that $$
\prod_{i=1}^n\!\frac{\sinh(x_i/2)}{x_i/2}\!\left(1\!-\!cn\frac{x_1^2+\cdots+x_n^2}{g}\!\right)\!\leq \!\frac{V_{g,n}(x_1,\cdots,x_n)}{V_{g,n}}\!\leq \! \prod_{i=1}^n \frac{\sinh(x_i/2)}{x_i/2}.
$$
\end{lemma} 

\begin{theorem}\label{thm vgn/vgn+1 n^2+1/g^2}
    For  $n=n(g)=o(\sqrt{g})$ :
    \begin{align}
        &\frac{V_{g,n+1}}{8\pi^2g V_{g,n}}=1+ \left(\left(\frac{1}{2}-\frac{1}{\pi^2}\right)n-\frac{5}{4}+\frac{2}{\pi^2}\right)\cdot\frac{1}{g}+O\left(\frac{1+n^2}{g^2}\right),\\
     &\frac{V_{g-1,n+2}}{V_{g,n}}=1+\frac{3-2n}{\pi^2}\cdot\frac{1}{g}+O\left(\frac{1+n^2}{g^2}\right).
 \end{align}
    The implied constants here are independent of $n$ and $g$.
\end{theorem}

\begin{theorem}\label{thm vgn(x)/vgn 1+n^3/g^2} For any fixed integer $k\geq 1$, if $n=n(g)=o(\sqrt{g})$ and $n\geq k$, then  for any $\textbf{x}=(x_1,\cdots,x_k,\cdots, 0^{n-k})\in \mathbb{R}_{\geq 0}^n$, as $g\to\infty$,\begin{align*}
&\frac{V_{g,n}(\textbf{x})}{V_{g,n}}=\prod_{i=1}^n\!\frac{\sinh(x_i/2)}{x_i/2}
+\frac{f_n^1(\textbf{x})}{g}\\
+&O_k\left(\frac{n^3(1+x_1+\cdots+x_k)^4}{g^2}\exp{\left(\frac{x_1+\cdots+x_k}{2}\right)}\right),
\end{align*}
where $f_n^1(\textbf{x})$ is given in Theorem \ref{thm AM22 1/g} and 
 the implied constant here is independent of  $n,g,x_1,\cdots,x_n.$
\end{theorem}
The proofs of Lemma \ref{lemma NWX vgn(x) old version}, Theorem \ref{thm vgn/vgn+1 n^2+1/g^2}, and Theorem \ref{thm vgn(x)/vgn 1+n^3/g^2} are similar to the proofs when $n$ is fixed. So we left their proofs to Appendix \ref{appendix wp volume}.

\section{Pre-trace inequality}\label{sec Pre-trace inequality}
In this section, we introduce the way to transform the pre-trace inequality into Theorem \ref{thm main eq pre trace}. 

The family of test functions $\{f_T\}$ was constructed in \cite{magee2022random} and adopted in \cite{wx22-3/16}. For hyperbolic surfaces with cusps, adding an extra condition (4) in Proposition \ref{prop hide2023 ineq} to $f_T$ will be convenient. 
\begin{proposition}\label{prop hide2023 ineq} \cite[Proposition 2.4]{hide2023spectral}
    There exists an $f_1\in C_c^\infty(\R)$ such that \begin{enumerate}
        \item $\textbf{supp}(f_1)=[-1,1]$.
        \item $f_1\geq 0$ is non-negative and even.
        \item The Fourier transform $\hat{f}_1$ is non-negative on $\mathbb{R}\cup i\mathbb{R}$. 
        \item $f_1$ is non-increasing on $\mathbb{R}_{\geq 0}$.
    \end{enumerate}
\end{proposition}
For $T\geq 1$, define \begin{equation}
    f_T(x)=f_1\left(\frac{x}{T}\right).
\end{equation}
Then  $f_T$ is non-negative on $\R$, even, smooth and supported on $[-T,T]$. 
And the Fourier transform $\hat{f}_T$ is non-negative on $\R\cup i\R$.
The following lemma gives a lower bound for $\hat{f}_T$ on $i\mathbb{R}$. 
\begin{lemma}\label{lemma mnp22 hat fT(it)}\cite[Lemma 2.4]{magee2022random}
    For any $\epsilon>0$, there exists a constant $C_\epsilon$ such that for all $t\geq 0$ and $T>1$,$$
    \hat{f}_T(it)\geq C_\epsilon T e^{T(1-\epsilon)t}.
    $$
\end{lemma}
In this section, we will prove the following inequality.

\begin{theorem}\label{thm main eq pre trace}
    Assume $g\geq 2$ and $f_T$ be as above for $T\prec \log g$.   Fix $\epsilon_0>0$. For any $\epsilon>0$, there exists a constant $C(\epsilon)>0$ such that for any non-compact finite area surface $X$ of type $X_{g,n}$ with $n=o(g^\frac{1}{2})$ and $\lambda_1(X)\leq \frac{1}{4}-\epsilon_0$,
    \begin{equation}\label{eq main eq pre trace}
    \begin{aligned}
   & C(\epsilon)Te^{T(1-\epsilon)\sqrt{\frac{1}{4}-\lambda_1(X)}}\\
\leq&\sum_{\gamma\in \mathcal{P}(X)}\sum_{k=1}^\infty \frac{\ell_\gamma(X)}{2\sinh \left(\frac{k\ell_\gamma(X)}{2}\right)}f_T(k\ell_\gamma(X))-\hat{f}_T(\frac{i}{2})+ O\left(\frac{g}{T}\right).
  \end{aligned}
    \end{equation}
\end{theorem}

\begin{rem*}We will take $T=6(1-\alpha)\log g$ into Theorem \ref{thm main eq pre trace} to prove Theorem \ref{thm main-1}.  
Theorem \ref{thm main eq pre trace} was proved for $T=4\log g$ in \cite{hide2023spectral} with the remainder term $O(ng)$. 
Here we follow his proof and modify the result to be the theorem above.
In fact, if we apply Selberg's trace formula for non-compact finite-area hyperbolic surface(See \eg \cite[Theorem 10.2]{iwaniec2002spectral}), with detailed analysis on the scattering matrix of $X$, we can establish \eqref{eq main eq pre trace} with error term $O(Tg)$ but do not require $\lambda_1(X)<\frac{1}{4}-\epsilon_0
$.
\end{rem*}

Let $k_T$ be the inverse Abel transform of $f_T$:
\begin{equation}\label{inverse abel transform}
    k_T(\rho)=\frac{-1}{\sqrt{2}\pi}\int_{\rho}^\infty \frac{f_T^\prime(u)}{\sqrt{\cosh u-\cosh \rho}}du.
\end{equation}
Then $k_T$ is non-negative. The Abel transform gives 
\begin{equation}\label{abel transform}
    f_T(u)=\sqrt{2}\int_{|u|}^\infty\frac{k_T(\rho)\sinh \rho}{\sqrt{\cosh\rho-\cosh u}}d\rho.
\end{equation}
For $z,w\in \mathbb{H}$, let $k_T(z,w)=k_T(d(z,w))$. For the normalized eigenfunction $u_j$ of the Laplacian on $X=\Gamma_X\backslash  \mathbb{H}$ with respect to the eigenvalue $\lambda_j$, let \begin{equation*}
r_j(x)=\left\{\begin{aligned}
 &   i\sqrt{\frac{1}{4}-\lambda_j}\quad \textit{ if } 0\leq \lambda_j\leq \frac{1}{4};\\
  &  \sqrt{\lambda_j-\frac{1}{4}}\quad \textit{ if }x>\frac{1}{4}.
\end{aligned}\right.
\end{equation*}
Then we have the following pre-trace inequality. 
\begin{proposition} \cite[Proposition 5.2]{gamburd2002spectral}
    For all $T>0$ and $z\in \mathbb{H}$,
    \begin{equation}\label{eq pre trace formula}
    \sum_{j:\lambda_j<\frac{1}{4}}\hat{f}_T(r_j)u_j^2(z)\leq \sum_{\gamma\in \Gamma_X}k_T(z,\gamma z).
    \end{equation}
\end{proposition}

 For $L>2$, each cusp region $C_i(L)$ with $\ell(C_i(L))=\frac{1}{L}$ of $X$ for $i=1,\cdots,n$ is disjoint by \eg \cite[Lemma 4.4.6]{buser2010geometry}.
Let $$X(L)=X\setminus \cup_{i=1}^n C_i(L).$$

The following lemma implies that $\left|\left| u \right|\right|_{L^2}$ will not concentrate on $C_i(L)$ if $u$ is a Laplacian eigenfunction on $X$ with eigenvalue $0<\lambda<\frac{1}{4}-\epsilon_0$. One can refer to \eg \cite[Lemma 4.1]{gamburd2002spectral}, \cite[Lemma 4.1]{brooks2001riemann} or \cite[Lemma 3.1]{mondal2015topological}.
\begin{lemma}\label{lemma ga02 cusp eigenfunction}
If $u$ is a Laplacian eigenfunction on $X$ with respect to the eigenvalue $0<\lambda<\frac{1}{4}-\epsilon_0$, there exists a constant $c_{\epsilon_0}>1$ only depending on $\epsilon_0$ such that \begin{equation*}
    \frac{\int_{C_i\left(\frac{L}{2}\right)}u^2(z)dz}{\int_{C_i(L)}u^2(z)dz}\geq c_{\epsilon_0}
\end{equation*}
for any $L>2$ and $i$.
\end{lemma}

\begin{proof}[Proof of Theorem \ref{thm main eq pre trace}]
Assume that $\lambda_1(X)\leq \frac{1}{4}-\epsilon_0$.
Let $\mathcal{F}\subset \mathbb{H}$ be a fundamental domain of $X$ with respect to $\Gamma_X$. Let $\mathcal{C}_i(L)$ to be the subset of $\mathcal{F}$ that 
projects to $C_i(L)$, and $$\mathcal{F}(L)=\mathcal{F}\setminus \cup_{i=1}^n \mathcal{C}_i(L).$$

    We integrate the inequality \eqref{eq pre trace formula} over $\mathcal{F}(L)$.
It is easy to compute that $$
\area\left(C_i(L)\right)=\area\left(\mathcal{C}_i(L)\right)=\frac{1}{L}.
$$
Since $u_0$ is a constant function, we have \begin{equation}\label{eq pre inequality 1}
\begin{aligned}
    &\hat{f}_T(r_0(X))\int_{\mathcal{F}(L)} u_0^2(z)dz\\
    =& \hat{f}_T\left(\frac{i}{2}\right) \frac{\area(\mathcal{F}(L))}{\area(\mathcal{F})}=\hat{f}_T\left(\frac{i}{2}\right)\left(1+O\left(\frac{n}{gL}\right)\right).
    \end{aligned}
\end{equation}
For $u_1$, by Lemma \ref{lemma mnp22 hat fT(it)} and Lemma \ref{lemma ga02 cusp eigenfunction}, there exists a constant $C(\epsilon)$ such that \begin{equation}\label{eq pre inequality 2}
     \hat{f}_T(r_1(X))\int_{\mathcal{F}(L)} u_1^2(z)dz\geq C(\epsilon) Te^{(1-\epsilon)\sqrt{\frac{1}{4}-\lambda_1(X)}}.
\end{equation}
We can separate $\Gamma_X$ as $$
\Gamma_X=\left\{\textrm{id}\right\}\cup\left\{\gamma\in \Gamma_X;|\textrm{tr} \gamma|>2 \right\}\cup \left\{\gamma\in \Gamma_X; \gamma\neq \textrm{id},|\textrm{tr} \gamma|=2 \right\}.
$$
Calculation on the Abel transform, see the proof of \cite[Theorem 9.5.3]{buser2010geometry} for example, implies
 \begin{align*}
    k_T(0)=\frac{1}{4\pi}\int_{-\infty}^\infty r\hat{f}_T(r)\tanh (\pi r)dr.
\end{align*}
By \cite[Proposition 23]{wx22-3/16} we have \begin{equation}\label{eq pre inequality 3}
    \int_{\mathcal{F}(L)}  k_T(z,z)dz= \frac{\area(\mathcal{F}(L))}{4\pi}\int_{-\infty}^\infty r\hat{f}_T(r)\tanh (\pi r)dr=O\left(\frac{g}{T}\right).
\end{equation}
For hyperbolic terms, since $k_T$ is non-negative, we have \begin{equation}\label{eq pre inequality 4}\begin{aligned}
     &\sum_{\left\{\gamma\in \Gamma_X;|\textrm{tr} \gamma|>2 \right\}} \int_{\mathcal{F}(L)}  k_T(z,\gamma z)dz\\
  \leq& \sum_{\left\{\gamma\in \Gamma_X;|\textrm{tr} \gamma|>2 \right\}} \int_{\mathcal{F}}  k_T(z,\gamma z)dz\\
  =&\sum_{\gamma\in \mathcal{P}(X)}\sum_{k=1}^\infty\frac{\ell_\gamma(X)}{2\sinh \left(\frac{k\ell_\gamma(X)}{2}\right)}f_T(k\ell_\gamma(X)).
\end{aligned}
\end{equation}
The equality above can be referred to \eg \cite[section 10.2]{iwaniec2002spectral}.

For parabolic terms, any $\gamma\in \Gamma_X\setminus\{\textrm{id}\}$ with $|\textrm{tr}\gamma|=2$ must be conjugate to $\gamma=\gamma_j^m$ with $m\in\mathbb{Z}\setminus\{0\}$ and $\gamma_j$ is one of the two prime elements in $\Gamma_X$ that fix the end $\mathcal{C}_j(L)\cap \partial \mathbb{H}$. There exists $\sigma_j\in \mathrm{SL}_2(\mathbb{R})$ such that $$
\sigma_j^{-1}\gamma_j\sigma_j=\left(\begin{matrix}
    1&1\\
    0&1
\end{matrix}\right).
$$
Notice that the centralizer of $\gamma_j$ in $\Gamma_X$ is the free group $\left<\gamma_j\right>$ generated by $\gamma_j$. Since $k_T$ and $d\mu$ are invariant under the action of $\mathrm{SL}_2(\mathbb{R})$, we have \begin{align*}
    &\sum_{\left\{\gamma\in \Gamma_X\setminus\{\textrm{id}\};|\textrm{tr} \gamma|=2 \right\}} \int_{\mathcal{F}(L)}  k_T(z,\gamma z)dz\\
    =&\sum_{j=1}^n\sum_{m\in\mathbb{Z}\setminus\{0\}}\sum_{\tau\in \left<\gamma_j\right>\setminus\Gamma}\int_{\mathcal{F}(L)}k_T(z,\tau^{-1}\gamma_j^m\tau z)d\mu(z)\\
    =&\sum_{j=1}^n\sum_{m\in\mathbb{Z}\setminus\{0\}}\int_{\cup_{\tau\in \left<\gamma_j\right>\setminus\Gamma}\tau \mathcal{F}(L)}
    k_T(z,\gamma_j^m z)d\mu(z).
\end{align*}
The action of $\left<\gamma_j\right>$ on $\cup_{\gamma\in \Gamma }\gamma \mathcal{F}(L)$ has a fundamental domain $\mathcal{F}_j$ such that $$
\mathcal{F}_j\subset \sigma_j\{z=x+iy\in \mathbb{H}\,\,0<x\leq 1,0<y\leq L\}.
$$ 
Since $k_T$ is non-negative, we have \begin{align*}
    &\sum_{\left\{\gamma\in \Gamma_X\setminus\{\textrm{id}\};|\textrm{tr} \gamma|=2 \right\}} \int_{\mathcal{F}(L)}  k_T(z,\gamma z)d\mu(z)\\
=&\sum_{j=1}^n\sum_{m\in\mathbb{Z}\setminus\{0\}}\int_{\mathcal{F}_j}k_T(z,\gamma_j^mz)d\mu(z)\\
=&\sum_{j=1}^n\sum_{m\in\mathbb{Z}\setminus\{0\}}\int_{\sigma^{-1}\mathcal{F}_j}k_T(z,\sigma^{-1}\gamma_j^m \sigma z)d\mu(z)\\
\leq &n\sum_{m\in\mathbb{Z}\setminus\{0\}}\int_{x=0}^1\int_{y=0}^Lk_T(z,z+m)\frac{1}{y^2}dxdy.
\end{align*}
Since $d(z,z+m)=\arccosh \left(1+\frac{m^2}{2y^2}\right)$ for $z=x+iy$ and $\textbf{supp}(k_T)\subset[0,T]$, by using the change of variable $t=\arccosh \left(1+\frac{m^2}{2y^2}\right)$, we have \begin{equation}\label{eq pre inequality 5}
\begin{aligned}
    &\sum_{\left\{\gamma\in \Gamma_X\setminus\{\textrm{id}\};|\textrm{tr} \gamma|=2 \right\}} \int_{\mathcal{F}(L)}  k_T(z,\gamma z)d\mu(z)\\
\leq & n\sum_{m\in\mathbb{Z}\setminus\{0\}}\int_{0}^Lk_T\left(\arccosh\left(1+\frac{m^2} 
  {2y^2}\right)\right)\frac{1}{y^2} dy\\
=&n\sum_{m\in\mathbb{N}^*}\frac{\sqrt{2}}{m}\int_{\min\{\arccosh\left(1+\frac{m^2}{2L^2}\right),T\}}^T \frac{k_T(t)\sinh t}{\sqrt{\cosh t-1}}dt\\
\leq &n\sum_{m=1}^{\left[\sqrt{L^2e^T}\right]}\frac{\sqrt{2}}{m}\int_0^\infty \frac{k_T(t)\sinh t}{\sqrt{\cosh t-1}}dt\\
\prec &n\log \left(Le^\frac{T}{2}\right)f_T(0)\\
=&n\log \left(Le^\frac{T}{2}\right)f_1(0).
\end{aligned}
\end{equation}
In the last inequality, we use the Abel transform \eqref{abel transform}.
So if we integrate the inequality \eqref{eq pre trace formula} over $\mathcal{F}(L)$, by the estimates \eqref{eq pre inequality 1}, \eqref{eq pre inequality 2}, \eqref{eq pre inequality 3}, \eqref{eq pre inequality 4} and \eqref{eq pre inequality 5}, we have 
\begin{align*}
     &C(\epsilon)\!Te^{T(1-\epsilon)\sqrt{\frac{1}{4}-\lambda_1(X)}}+\hat{f}_T(\frac{i}{2})\left(1+O\left(\frac{n}{gL}\right)\right)\\
     \leq&\sum_{\gamma\in \mathcal{P}(X)}\!\sum_{k=1}^\infty \frac{\ell_\gamma(X)}{2\sinh \left(\frac{k\ell_\gamma(X)}{2}\right)}f_T(k\ell_\gamma(X))\\
     +&O\left( n\log \left(Le^\frac{T}{2}\right)f_1(0)\right)+O\left(\frac{g}{T}\right).
\end{align*}
If $T\leq c\log g$, we have $\hat{f}_T(\frac{i}{2})=\int_{0}^\infty2\cosh\frac{x}{2}f_T(x)dx=O\left(g^{\frac{c}{2}}\right)$. We can take $L=g^\frac{c}{2}$ to get \eqref{eq main eq pre trace}.
\end{proof}

\section{Remainder estimate}\label{sec Remainder estimate}
We classify oriented prime geodesics in $\mathcal{P}(X)$ by $$
\mathcal{P}(X)=\mathcal{P}_{sep}^s(X)\cup \mathcal{P}_{nsep}^{s}(X)\cup\mathcal{P}^{ns}(X),
$$
where \begin{align*}
    \mathcal{P}_{sep}^s(X)&=\{\gamma\in\mathcal{P}(X); \gamma \textit{ is simple and separating} \};\\
    \mathcal{P}_{nsep}^{s}(X)&=\{\gamma\in\mathcal{P}(X); \gamma \textit{ is simple and non-separating}\};\\
    \mathcal{P}^{ns}(X)&=\{\gamma\in\mathcal{P}(X); \gamma \textit{ is non-simple}\}.
\end{align*}
Define $$
H_{X,k}(\gamma)=\frac{\ell_\gamma(X)}{2\sinh \left(\frac{k\ell_\gamma(X)}{2}\right)}f_T(k\ell_\gamma(X))
$$
and $$\overline{H}_{X,k}(t)= \frac{t}{2\sinh \left(\frac{kt}{2}\right)}f_T(kt).$$
Then $H_{X,k}(\gamma)=\overline{H}_{X,k}(\ell_\gamma(X))$ and 
 we can separate the right side of \eqref{eq main eq pre trace} to get
\begin{equation}\label{selberg decomposition}
\begin{aligned}
  &C(\epsilon)\!Te^{T(1-\epsilon)\sqrt{\frac{1}{4}-\lambda_1(X)}}\cdot \textbf{1}_{\lambda_1(X)\leq \frac{1}{4}-\epsilon_0}\\
    \leq &\sum_{\gamma\in \mathcal{P}(X)}\sum_{k=1}^\infty \frac{\ell_\gamma(X)}{2\sinh \left(\frac{k\ell_\gamma(X)}{2}\right)}f_T(k\ell_\gamma(X))-\hat{f}_T(\frac{i}{2}) + O\left(\frac{g}{T}\right)\\
=&\underbrace{\sum_{\gamma\in \mathcal{P}(X)}\sum_{k=2}^\infty H_{X,k}(\gamma)}_{(\Rmnum{1})}
+\underbrace{\sum_{\gamma\in \mathcal{P}_{sep}^s(X)} H_{X,1}(\gamma)}_{(\Rmnum{2})}\\
+& \underbrace{\sum_{\gamma\in \mathcal{P}_{nsep}^{s}(X)}H_{X,1}(\gamma)-\hat{f}_T(\frac{i}{2})}_{(\Rmnum{3})}
+\underbrace{\sum_{\gamma\in \mathcal{P}^{ns}(X)}H_{X,1}(\gamma)}_{(\Rmnum{4})}+O\left(\frac{g}{T}\right).
\end{aligned}
\end{equation}
The expectation of term $(\Rmnum{1})$ was estimated in \cite{wx22-3/16}.
\begin{lemma}\label{lemma small term k geq 2} \cite[Proposition 24]{wx22-3/16}   For $n=o(g^\frac{1}{2})$,
    $$
    \Egn\left[\sum_{\gamma\in \mathcal{P}(X)}\sum_{k=2}^\infty H_{X,k}(\gamma) \right]=O\left(T^2g\right).
    $$
\end{lemma}
We need to estimate the expectation of term $(\Rmnum{2})$ more precisely than in \cite{wx22-3/16} and in the version of surfaces with cusps.

\begin{lemma}\label{lemma small term simple separating}
    For $n=o(g^\frac{1}{2})$ and $T=a\log g$,
    $$
    \Egn\left[\sum_{\gamma\in \mathcal{P}_{sep}^s(X)} H_{X,1}(\gamma) \right]=O_a\left( \log^{3a+4} g   \right),
    $$
    where the implied constant depends only on $a$.
\end{lemma}
\begin{proof}
A simple separating closed geodesic $\gamma$ on $X\simeq S_{g,n}$ will cut $S_{g,n}$ into $S_{g_1,n_1+1}\cup S_{g_2,n_2+1}$ with $g_1+g_2=g,n_1+n_2=n$ and $2g_1+n_1\geq 2,2g_2+n_2\geq 2$.
The mapping class group orbit $\Mod \cdot \gamma$ is determined by the partition of the $n$ labels of cusps into 
$I\cup J=\left\{1,\cdots, n\right\}$ with $\left|I\right|=n_1,\left|J\right|=n_2$. For fixed $g_i,n_i$, there are ${n\choose{n_1}}$ such partitions.
Now we apply Mirzakhani's Integral Formula, \ie, Theorem \ref{thm mir int formula} to get \begin{equation}\label{eq reminder eq 1}
    \begin{aligned}
         &\Egn \left[\sum_{\gamma\in \mathcal{P}_{sep}^s(X)} H_{X,1}(\gamma) \right]=\frac{1}{V_{g,n}}\int_{\M_{g,n}}\sum_{\gamma\in \mathcal{P}_{sep}^s(X)} H_{X,1}(\gamma)  dX\\
    \leq &
    \sum_{\substack{g_1+g_2=g,n_1+n_2=n\\
    2\leq 2g_1+n_1\leq 2g_2+n_2
    }}{n\choose{n_1}}
    \int_0^\infty \frac{V_{g_1,n_1+1}(x,0^{n_1})V_{g_2,n_2+1}(x,0^{n_2})}{V_{g,n}}\\
\cdot&
 \frac{x^2}{\sinh\frac{x}{2}}f_T(x)dx.
    \end{aligned}
\end{equation}
For $2\leq k< \frac{1}{2}(2g+n)$,
we separate the indices in the summation into two parts:
\begin{align*}
   &\{(g_1,g_2,n_1,n_2);g_1+g_2=g,n_1+n_2=n,2\leq 2g_1+n_1\leq 2g_2+n_2\}\\
   =&\{(g_1,g_2,n_1,n_2);g_1+g_2=g,n_1+n_2=n,k+1\leq 2g_1+n_1\leq 2g_2+n_2\}\\
   \cup&\{(g_1,g_2,n_1,n_2);g_1+g_2=g,n_1+n_2=n, 2\leq 2g_1+n_1\leq k\}.
\end{align*}
For the summation over the first part of indices, we have $$
V_{g_i,n_i+1}(x,0^{n_i})\leq V_{g_i,n_i+1}\frac{2\sinh\frac{x}{2}}{x}
$$
 by Lemma \ref{lemma NWX vgn(x)},
and
$$
\frac{4}{V_{g,n}}\left(\sum_{\substack{g_1+g_2=g,n_1+n_2=n\\
    k+1\leq 2g_1+n_1\leq 2g_2+n_2
    }}{n\choose{n_1}}V_{g_1,n_1+1}V_{g_2,n_2+1} \right)
    =O\left(\frac{1+n^{k+1}}{g^{k}}\right)$$
    by Lemma \ref{lemma in hide appendix gen k}.
It follows that 
    \begin{equation}\label{eq reminder eq 2}
    \begin{aligned}
       &\sum_{\substack{g_1+g_2=g,n_1+n_2=n\\
    k+1\leq 2g_1+n_1\leq 2g_2+n_2
    }}{n\choose{n_1}}
    \int_0^\infty \frac{V_{g_1,n_1+1}(x,0^{n_1})V_{g_2,n_2+1}(x,0^{n_2})}{V_{g,n}}\\
    \cdot&\frac{x^2}{\sinh\frac{x}{2}}f_T(x)dx\\
   \leq& \frac{4}{V_{g,n}}\sum_{\substack{g_1+g_2=g,n_1+n_2=n\\
    k+1\leq 2g_1+n_1\leq 2g_2+n_2
    }}{n\choose{n_1}}V_{g_1,n_1+1}V_{g_2,n_2+1} \int_0^\infty \sinh\frac{x}{2}f_T(x)dx\\
   \prec&\frac{(1+n^{k+1}) e^\frac{T}{2}}{g^{k}}.
    \end{aligned}
    \end{equation}
In the last inequality, we use the fact that $\mathrm{supp}(f_T)=[-T, T].$

For the summation over the second part of indices, there are finitely many $(g_1,g_2,n_1,n_2)$ with $2\leq 2g_1+n_1\leq k$. Theorem \ref{mir07 poly} implies that $$
V_{g_1,n_1+1}(x,0^{n_1})\prec 1+x^{6g_1+2n_1-4}\prec 1+x^{3k-4},
$$
and  Lemma \ref{lemma NWX vgn(x)} implies
$$
V_{g_2,n_2+1}(x,0^{n_2})\leq V_{g_2,n_2+1}\frac{2\sinh\frac{x}{2}}{x}.
$$
 By Lemma \ref{lemma in hide appendix gen k}, we have
$$
\frac{1}{V_{g,n}}\left(\sum_{\substack{g_1+g_2=g,n_1+n_2=n\\
    2\leq 2g_1+n_1\leq k
    }}{n\choose{n_1}}V_{g_1,n_1+1}V_{g_2,n_2+1} \right)
    =O\left(\frac{1+n^{2}}{g}\right),$$
giving    \begin{equation}\label{eq reminder eq 3}
\begin{aligned}
       &\sum_{\substack{g_1+g_2=g,n_1+n_2=n\\
    2\leq 2g_1+n_1\leq k
    }}{n\choose{n_1}}
    \int_0^\infty \frac{V_{g_1,n_1+1}(x,0^{n_1})V_{g_2,n_2+1}(x,0^{n_2})}{V_{g,n}}\\
    \cdot&
 \frac{x^2}{\sinh\frac{x}{2}}f_T(x)dx\\
   \prec& \sum_{\substack{g_1+g_2=g,n_1+n_2=n\\
    2\leq 2g_1+n_1\leq k
    }}{n\choose{n_1}}\frac{V_{g_1,n_1+1}V_{g_2,n_2+1}}{V_{g,n}}\int_0^\infty  f_T(x)x(1+x^{3k-4})dx\\
  \prec&\frac{1+n^2}{g}T^{3k-2}.
    \end{aligned}
    \end{equation}
For $T=a\log g$, we take $k=[a+2]>a+1$. Then by \eqref{eq reminder eq 1}, \eqref{eq reminder eq 2} and \eqref{eq reminder eq 3}, we have \begin{align*}
    &\Egn\left[\sum_{\gamma\in \mathcal{P}_{sep}^s(X)}H_{X,1}(\gamma)\right]\\
\prec & g^{\frac{a}{2}+\frac{k+1}{2}-k}+\log^{3k-2} g\\
\prec& \log^{3a+4} g,
\end{align*}
which completes the proof.
\end{proof}

\section{The expansion of the characteristic function}\label{sec inc-ex}
Now we may consider $T$ as $6(1-\alpha)\log g$. 
Unless otherwise emphasized, the subsequent estimates hold for general $ T=a\log g$.

We take $$\mathcal{N}_{\ell}=\{X\in \M_{g,n};\forall \textit{ embedded subsurface } Y\subset X,\chi(Y)=-1,\ell(\partial Y)\geq \ell\}$$
and  $$
N_\ell(X)=\{\textit{embedded }Y\subset X;\chi(Y)=-1,\ell(\partial Y) <\ell\}
$$
for $\ell=\kappa\log g$ and  $0<\kappa<1-2\alpha$. Notice that in the second set, $Y$ should be an embedding subsurface of $X$ with geodesic boundary and hence avoid the case that $Y\simeq S_{0,3}$ is a pants with a pair of overlapped boundary geodesics.
The subset $\mathcal{N}_\ell\subset \M_{g,n}$ can be written in the following way. This idea appeared explicitly in \cite{mir13}. 
\begin{lemma}\label{lemma inc-ex}
    For any $\ell>0$,
    $$1_{X\in\mathcal{N}_\ell}=1_{\#N_\ell=0}(X)=1+\sum_{k=1}^\infty (-1)^k(\#N_\ell)_k(X)$$
    where $(\#N_\ell)_k:=\frac{1}{k!}\#N_\ell(\#N_\ell-1)\cdots(\#N_\ell-k+1)$.
    Besides, for any $j\geq 0$,
     $$1_{X\in\mathcal{N}_\ell}=1_{\#N_\ell=0}(X)=1+\sum_{k=1}^j (-1)^k(\#N_\ell)_k(X)+O_j\Big( (\#N_\ell)_{j+1}(X)\Big).$$
\end{lemma}
\begin{proof}
If $X\in\mathcal{N}_\ell$, then $\#N_\ell(X)=0$. If $\#N_\ell(X)=0$, then for any subsurface $Y$ of $X$ with $\chi(Y)=-1$, we have $\ell(\partial Y)\geq \ell(\partial \overline{Y})\geq \ell$, which implies $X\in \mathcal{N}_\ell$. So $1_{X\in\mathcal{N}_\ell}=1_{\#N_\ell=0}(X)$.

    If $\#N_\ell(X)=0$, then $$1_{X\in\mathcal{N}_\ell}=1_{\#N_\ell=0}(X)=1$$ and $$1+\sum_{k=1}^\infty (-1)^k(\#N_\ell)_k(X)=1.$$ If $\#N_\ell(X)=k_0>0$, then $$1_{X\in\mathcal{N}_\ell}=1_{\#N_\ell=0}(X)=0$$ and \begin{align*}
        &1+\sum_{k=1}^\infty (-1)^k(\#N_\ell)_k(X)\\
       =&1+\sum_{k=1}^{k_0}(-1)^k {k_0\choose k}\\
       =&(1-1)^{k_0}=0.\\
    \end{align*}
    So the first equality always holds.
    For the second equality, if $\#N_\ell(X)=k_0\leq j$, then $$
    1_{X\in\mathcal{N}_\ell}=1_{\#N_\ell=0}(X)=1+\sum_{k=1}^j (-1)^k(\#N_\ell)_k(X)
    $$
    and $$(\#N_\ell)_{j+1}(X)=0.$$
    If $\#N_\ell(X)=k_0> j$, then \begin{align*}
        &\Big|1_{X\in\mathcal{N}_\ell}-\big(1+\sum_{k=1}^j (-1)^k(\#N_\ell)_k(X)\big)\Big|\\
       \leq&\Big|1+\sum_{k=1}^j\frac{k_0(k_0-1)\cdots (k_0-k+1)}{k!}\Big|\\
       \leq & \frac{k_0(k_0-1)\cdots (k_0-j+1)(k_0-j)}{(j+1)!}\\
       \cdot&\sum_{k=0}^j\frac{(j+1) j\cdots(k+1)}{(j+1-j)(j+1-j+1)\cdots( j+1-k)}\\
       =&O_j\left(\frac{k_0(k_0-1)\cdots (k_0-j+1)(k_0-j)}{(j+1)!}\right),
    \end{align*}
    where the implied constant only depends on $j$.
\end{proof}
Now we integrate \eqref{selberg decomposition} over $\mathcal{N}_\ell$ for $\ell=\kappa\log g$ with $\kappa<2-4\alpha$. Then by Lemma \ref{lemma small term k geq 2} and Lemma \ref{lemma small term simple separating} we have 
    \begin{equation}\label{eq ineq on subset T}
    \begin{aligned}
      &\Egn\left[C(\epsilon)\!Te^{T(1-\epsilon)\sqrt{\frac{1}{4}-\lambda_1(X)}}\cdot\textbf{1}_{X\in\mathcal{N}_\ell}\cdot  \textbf{1}_{\lambda_1(X)\leq \frac{1}{4}-\epsilon_0}\right] \\
\leq&\underbrace{\Egn\left[1_{\mathcal{N}_\ell}\cdot\sum_{\gamma\in \mathcal{P}_{nsep}^{s}(X)}H_{X,1}(\gamma)\right]}_{\mathrm{Int}_{nsep}}-\underbrace{\Probgn\left(\mathcal{N}_\ell\right)\hat{f}_T(\frac{i}{2})}_{\mathrm{Int}_{\lambda_0}}\\
+&\underbrace{\Egn\left[1_{\mathcal{N}_\ell}\cdot\sum_{\gamma\in \mathcal{P}^{ns}(X)}H_{X,1}(\gamma)\right]}_{\mathrm{Int}_{ns}}+O\left(\frac{g}{T}+g\log^2g+\log^{3a+4} g\right).
    \end{aligned}
    \end{equation}
\begin{rem*}
    Here we use the non-negativity of $H_{X,k}$ or $f_T$. We want to show that the right side of \eqref{eq ineq on subset T} belongs to $O\left(g^{1+\delta}\right)$  when $T=6(1-\alpha)\log g$ and $\kappa$ is small for any $\delta>0$.
\end{rem*}

\section{First order cancellation between $\mathrm{Int}_{nsep}$ and $\mathrm{Int}_{\lambda_0}$}\label{sec 1 order cancel}
By Lemma \ref{lemma inc-ex}, the term $\mathrm{Int}_{nsep}-\mathrm{Int}_{\lambda_0}$ can be written as 
    \begin{equation}\label{eq expansion formula on N ell}
    \begin{aligned}
&\Egn\left[1_{\mathcal{N}_\ell}\cdot\sum_{\gamma\in \mathcal{P}_{nsep}^{s}(X)}H_{X,1}(\gamma)\right]-\Probgn\left(\mathcal{N}_\ell\right)\hat{f}_T(\frac{i}{2})\\
    =&\Egn\left[\sum_{\gamma\in \mathcal{P}_{nsep}^{s}(X)}H_{X,1}(\gamma)\right]-\hat{f}_T(\frac{i}{2})
+\sum_{k=1}^j(-1)^k\\
\cdot&\left(\Egn\!\left[ (\#N_\ell)_k\cdot\!\sum_{\gamma\in \mathcal{P}_{nsep}^{s}(X)}\!H_{X,1}(\gamma)\right]\!-\!\hat{f}_T(\frac{i}{2})\!\cdot\!\Egn\left[ (\#N_\ell)_k\right]\right)\\
+&O_j\!\left(\Egn\!\left[(\#N_\ell)_{j+1}\!\cdot\!\!\!\sum_{\gamma\in \mathcal{P}_{nsep}^{s}(X)}\!H_{X,1}(\gamma)\right]\!+\!\hat{f}_T(\frac{i}{2})\!\cdot\!\Egn\left[ (\#N_\ell)_{j+1}\right]\right)
    \end{aligned}
    \end{equation}
for any $j\geq 0$.

In this section, we compute $\Egn\left[\sum_{\gamma\in \mathcal{P}_{nsep}^{s}(X)}H_{X,1}(\gamma)\right]-\hat{f}_T(\frac{i}{2})$.

\begin{lemma}\label{lemma nsep zero eigen}
    For $n=n(g)=o(\sqrt{g})$, \begin{align*}
        &\Egn\left[\sum_{\gamma\in \mathcal{P}_{nsep}^{s}(X)}H_{X,1}(\gamma)\right]-\hat{f}_T(\frac{i}{2})\\
        =&\frac{1}{\pi^2g}\int_0^Tf_T(x)\left(\left(1-\frac{n}{2}\right)x-\frac{x^2}{4}\right)e^\frac{x}{2}  dx\\
        +&O\left(1+\frac{(1+n)T^2}{g}+\frac{(1+n^3)T^6e^\frac{T}{2}}{g^2}\right).
    \end{align*}
\end{lemma}
\begin{proof}
We directly apply Mirzakhani's Integration formula Theorem \ref{thm mir int formula}
to obtain \begin{equation}\label{eq int ssep}
\begin{aligned}
        &\Egn\left[\sum_{\gamma\in \mathcal{P}_{nsep}^{s}(X)}H_{X,1}(\gamma)\right]\\
        =&2\cdot\frac{1}{2}\int_0^\infty\frac{x}{2\sinh\frac{x}{2}}f_T(x)\frac{V_{g-1,n+2}(x,x,0^n)}{V_{g,n}}xdx.
        \end{aligned}
\end{equation}
Here $C_\Gamma=\frac{1}{2}$ and the factor $2$ comes from two orientations of a simple non-separating closed geodesic. 
Combining Theorem \ref{thm vgn/vgn+1 n^2+1/g^2} and \ref{thm vgn(x)/vgn 1+n^3/g^2} we obtain 
    \begin{equation}\label{int vg-1,n+2(x,x)/vg2}
    \begin{aligned}
        &\frac{V_{g-1,n+2}(x,x,0^n)}{V_{g,n}}=\frac{V_{g-1,n+2}(x,x,0^n)}{V_{g-1,n+2}}\frac{V_{g-1,n+2}}{V_{g,n}}\\
         \nonumber=&\left[\left(\frac{\sinh\frac{x}{2}}{\frac{x}{2}}\right)^2 +\frac{f_{n+2}^1(x,x,0^n)}{g-1}+O(\frac{(1+n^3)(1+x^4)e^x}{g^2})\right]\\
         \nonumber\cdot&\left[1+\frac{3-2n}{\pi^2g}+O\left(\frac{1+n^2}{g^2}\right)\right]\\
         \nonumber=&\left(\frac{\sinh\frac{x}{2}}{\frac{x}{2}}\right)^2+\frac{1}{\pi^2g}\left[-\cosh^2\frac{x}{2}+(2-n)\frac{\sinh x+2\sinh \frac{x}{2}}{x}\right]\\
        \nonumber +&O\left(\frac{(1+n^3)(1+x^4)e^x}{g^2}\right),
    \end{aligned}
    \end{equation}
since \begin{align*}
    f_{n+2}^1(x,x,0^n)=&\frac{1}{\pi^2}\left(-\cosh^2\frac{x}{2}+\frac{2\sinh x+4\sinh \frac{x}{2}}{x}-\frac{12\sinh^2\frac{x}{2}}{x^2}\right)\\
    -&\frac{n}{\pi^2}\left(\frac{\sinh x+2\sinh\frac{x}{2}}{x}-\frac{8\sinh^2\frac{x}{2}}{x^2}\right).
\end{align*}

Take \eqref{int vg-1,n+2(x,x)/vg2} into \eqref{eq int ssep}. Since $\textbf{supp}(f_T)=[-T,T]$ and $|f_T|$ is bounded, we have 
    \begin{equation}\label{eq int ssep-1}
    \begin{aligned}
         &\Egn\left[\sum_{\gamma\in \mathcal{P}_{nsep}^{s}(X)}H_{X,1}(\gamma)\right]\\
          =&\int_0^T2\sinh \frac{x}{2}f_T(x)dx\\
         +&\frac{1}{\pi^2g}\int_0^Tf_T(x)\left(-\frac{x^2\cosh^2\frac{x}{2}}{2\sinh \frac{x}{2}}+(2-n)(x\cosh\frac{x}{2}+x)\right)dx\\
        +&O\left(\frac{(1+n^3)T^6e^\frac{T}{2}}{g^2}\right)\\
         =&\int_0^Te^\frac{x}{2}f_T(x)dx+\frac{1}{\pi^2g}\int_0^Tf_T(x)\left(\left(1-\frac{n}{2}\right)x-\frac{x^2}{4}\right)e^\frac{x}{2}  dx\\
         +&O\left(1+\frac{(1+n)T^2}{g}+\frac{(1+n^3)T^6e^\frac{T}{2}}{g^2}\right).
          \end{aligned}
    \end{equation}
Since $f_T(x)$ is even, we have \begin{equation}\label{eq zero eigen}
    \hat{f}_T(\frac{i}{2})=\int_{-\infty}^\infty e^\frac{x}{2}f_T(x)dx=\int_0^\infty2\cosh \frac{x}{2}f_T(x)dx.
\end{equation}
The Lemma follows from \eqref{eq int ssep-1} and \eqref{eq zero eigen}.
\end{proof}
We find that the main term $\int_0^\infty f_T(x)e^\frac{x}{2}dx$ in $\Egn\left[\sum_{\gamma\in \mathcal{P}_{nsep}^{s}(X)}H_{X,1}(\gamma)\right]$ and $\hat{f}_T(\frac{i}{2})$ cancels. The second-order terms arise from the expansion of the Weil-Petersson volume in the order $\frac{1}{g}$. If $k\geq 1,$ only first-order leading terms is needed for $\Egn\left[ (\#N_\ell)_k\cdot\sum_{\gamma\in \mathcal{P}_{nsep}^{s}(X)}H_{X,1}(\gamma)\right]$ and $\hat{f}_T(\frac{i}{2})\cdot\Egn\left[ (\#N_\ell)_k\right]$.
We compute them in section \ref{susbsection n ell k}.

\section{Second order cancellation with $\mathrm{Int}_{ns}$}\label{sec Int_ns}
In this section, we will show the following estimate for
$\textbf{Int}_{ns}$.

\begin{lemma}\label{lemma total int ns for all type}
    If $n=n(g)=o(\sqrt{g})$, then for any $\delta>\frac{1}{2}$ and $\epsilon>0$,  \begin{align*}
        &\Egn\left[1_{N_\ell}\cdot \sum_{\gamma\in \mathcal{P}^{ns}(X)}H_{X,1}(\gamma)\right]
        \leq\frac{1}{\pi^2g}\int_0^\infty f_T(t) e^\frac{t}{2}\left(\frac{t^2}{4}+\left(\frac{n}{2}-1\right)t\right)dt \\
        +O&\left( (1+n^2)\frac{T^3e^{(\frac{\epsilon}{2}+\frac{1}{4})T}+T^3e^{(\delta-\frac{1}{4})T}}{g}+\frac{(1+n^3)T^{230}e^{\left(\frac{1}{2}+\epsilon\right)T}}{g^2}+\frac{T^{6}e^{\frac{9}{2}T}}{g^{27}}\right).
    \end{align*}
\end{lemma}
Lemma \ref{lemma total int ns for all type} is a combination of Lemma 
\ref{lemma total type 1}, Lemma \ref{lemma all non double filling type 2}, Lemma \ref{lemma total type 3}, Lemma \ref{lemma total type 4}, Lemma \ref{lemma total type 5}, and Lemma \ref{lemma total type 6}, if we classify $\gamma\in \mathcal{P}^{ns}(X)$ as follows and estimate the total contribution for $\gamma$ in each subset to $\textbf{Int}_{ns}$.
We separate $\mathcal{P}^{ns}(X)$ by 
    \begin{equation}\label{eq classify nonsimple}
    \begin{aligned}
        \mathcal{P}^{ns}(X)=&\{\gamma;\gamma\textit{ fills } Y\simeq S_{0,3},X\setminus Y\simeq S_{g-2,n+3}\}\\
        \sqcup&\{\gamma;\gamma\textit{ fills } Y\simeq S_{1,1},X\setminus Y\simeq S_{g-1,n+1}\}\\
        \sqcup&\{\gamma;\gamma\textit{ fills } Y\simeq S_{0,3}, \overline{Y}\simeq S_{1,1},X\setminus Y\simeq S_{g-1,n+1}\}\\
         \sqcup&\{\gamma;\gamma\textit{ fills } Y\simeq S_{0,3}, 
        \overline{Y}\simeq S_{0,3},
        X\setminus Y\simeq S_{g,n-1}\}\\
       \sqcup&\{\gamma;\gamma\textit{ fills } Y\simeq S_{0,3},\overline{Y}\simeq S_{0,3},X\setminus Y\simeq S_{g-1,n+1}\}\\
         \sqcup&\{\gamma;\gamma\textit{ fills } Y, X\setminus Y \textit{ is disconnected or } |\chi(Y)|\geq 2 \}.
    \end{aligned}
    \end{equation}
We say a non-simple closed geodesic is of type $i$ if it belongs to the $i$-th set on the right side of \eqref{eq classify nonsimple}. See Figure \ref{fig type 123456} for an illustration of the subsurfaces filled by different types of $\gamma$.
 
\begin{rem*}
    Non-simple geodesics of type $4$ exist only when $n\geq 2$, and non-simple geodesics of type $5$ exist only when $n\geq 1$.
\end{rem*}

\begin{figure}[htbp]
  \centering
  \begin{subfigure}[b]{0.45\linewidth}
      \centering
      \includegraphics[width=\linewidth]{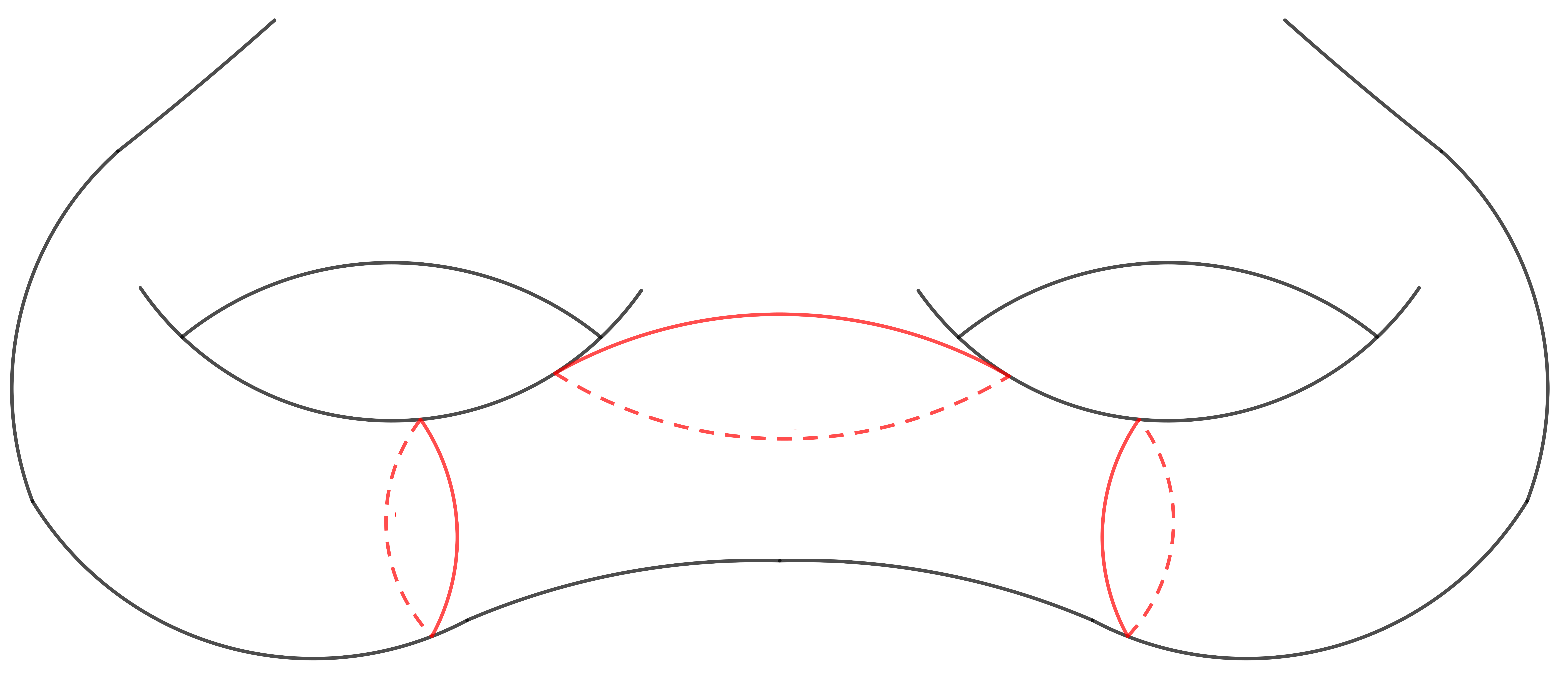}
      \caption{Type 1}
      \label{}
  \end{subfigure}
  \hfill
  \begin{subfigure}[b]{0.45\linewidth}
      \centering
      \includegraphics[width=\linewidth]{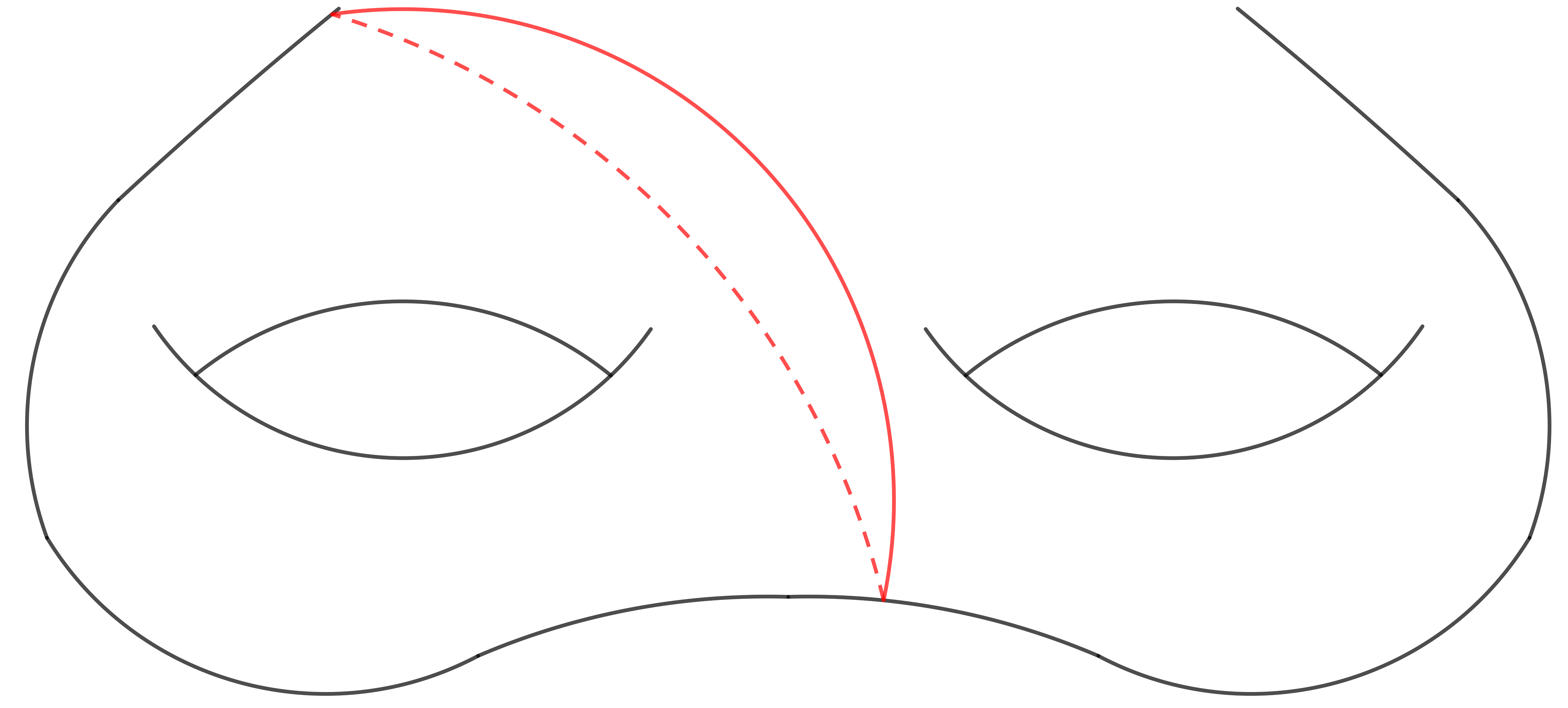}
      \caption{Type 2}
      \label{}
  \end{subfigure}
  \\
  \begin{subfigure}[b]{0.45\linewidth}
      \centering
      \includegraphics[width=\linewidth]{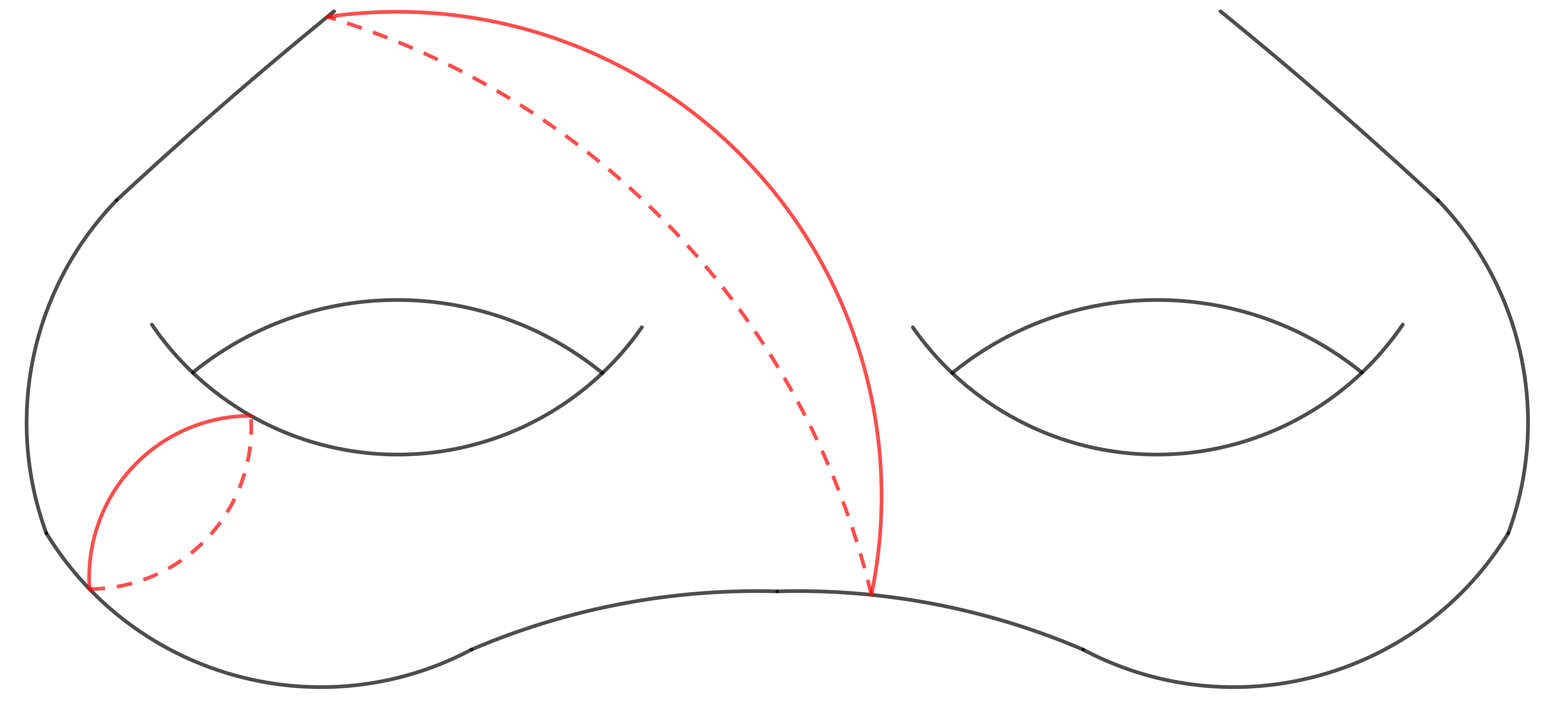}
      \caption{Type 3}
      \label{}
  \end{subfigure}
  \hfill
  \begin{subfigure}[b]{0.45\linewidth}
      \centering
      \includegraphics[width=\linewidth]{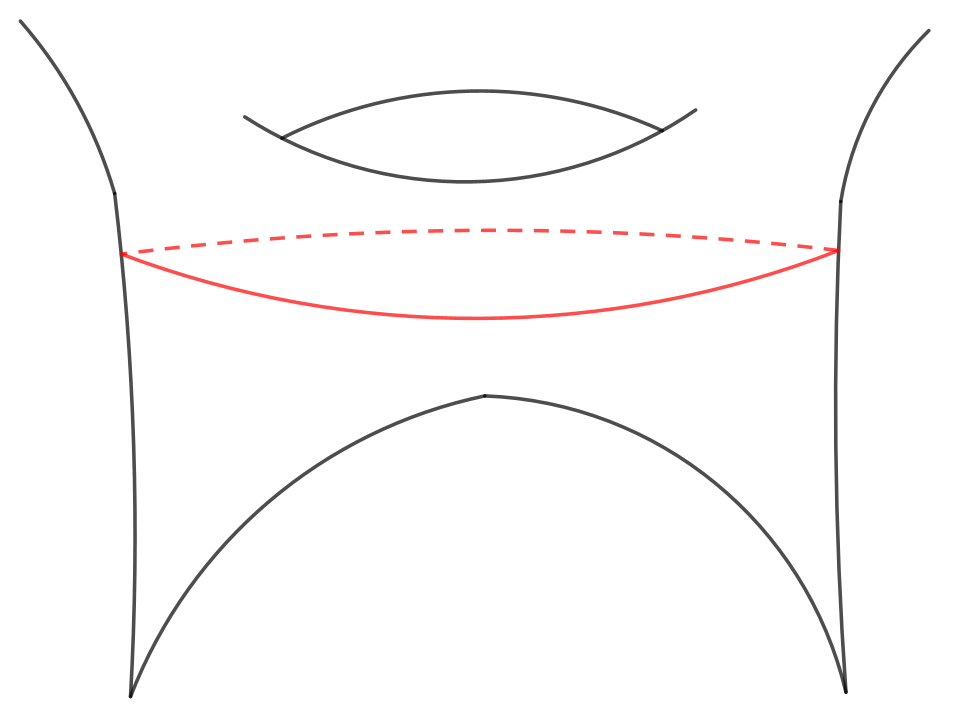}
      \caption{Type 4}
      \label{}
  \end{subfigure}
  \\ \begin{subfigure}[b]{0.45\linewidth}
      \centering
      \includegraphics[width=\linewidth]{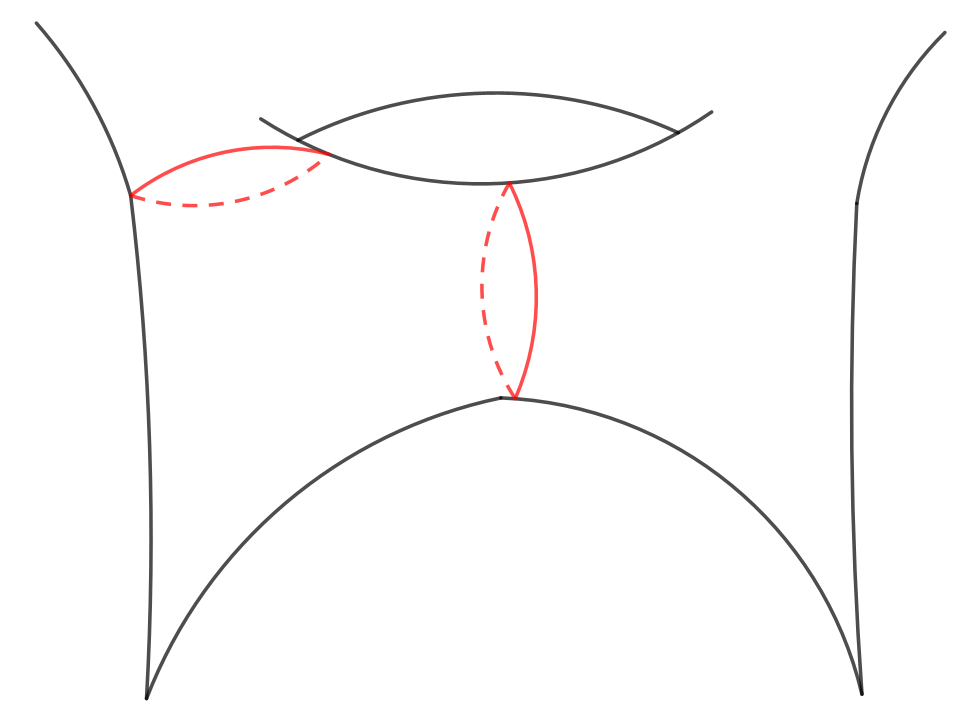}
      \caption{Type 5}
      \label{}
  \end{subfigure}
  \hfill
  \begin{subfigure}[b]{0.45\linewidth}
      \centering
      \includegraphics[width=\linewidth]{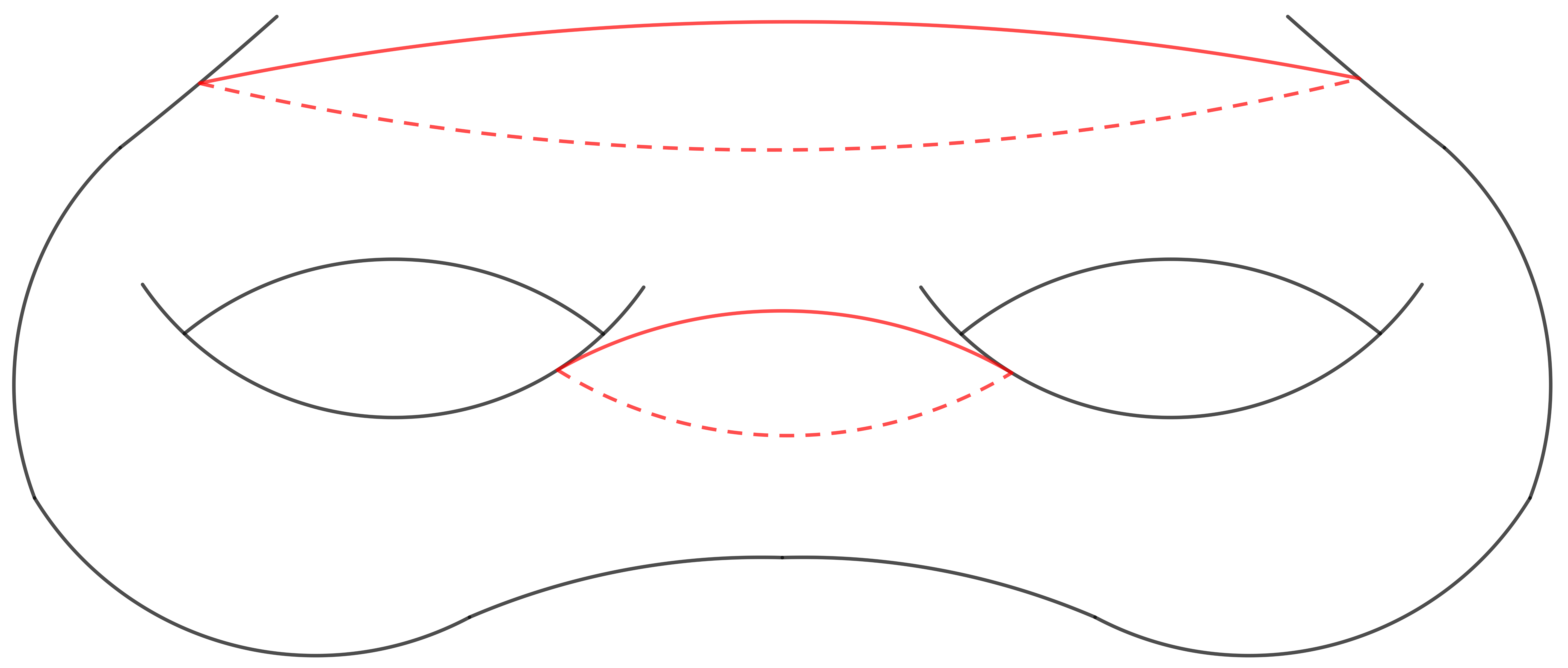}
      \caption{Type 6}
      \label{}
  \end{subfigure}
  \caption{Subsurfaces filled by different types of $\gamma$}
  \label{fig type 123456}
\end{figure}

We compute the contribution of non-simple geodesics belonging to each type to $\mathrm{Int}_{ns}$.

\subsection{Type $1$ in $\mathrm{Int}_{ns}$}
Firstly, we compute the double-filling terms of type 1 for $\textrm{Int}_{ns}$ in \eqref{eq ineq on subset T}. Since $\textbf{supp}(f_T)=[-T,T]$, we only need to consider $\gamma$ with $\ell_\gamma(X)\leq T$.
We have $\ell(\partial Y) \leq \ell_\gamma(X)\leq T$ if $\gamma$ double-fills $Y$ in $X$.

\begin{lemma}\label{lemma double filling type 1}
If $n=n(g)=o(\sqrt{g})$, then for any $\delta>\frac{1}{2}$ and $\epsilon>0$,
   $$\Egn\left[      1_{\mathcal{N}_\ell}\sum_{\substack{
\gamma\textit{ double-fills of type 1}\\
\ell_\gamma(X)\leq T
}}
H_{X,1}(\gamma)\right]
\prec \frac{T^3e^{(\frac{\epsilon}{2}+\frac{1}{4})T}+T^2e^{(\delta-\frac{1}{4})T}}{g},
  $$
    where the implied constant depends on $\epsilon,\delta$ and is independent of $T,g$.
\end{lemma}
\begin{proof}
By Theorem \ref{thm double filling count}, if $Y\simeq S_{0,3}$ and $X\setminus Y\simeq S_{g-2,n+3}$, for any $\epsilon>0$, we have
    \begin{equation}\label{eq type 1-1}
    \begin{aligned}
        &1_{\mathcal{N}_\ell}\sum_{
\substack{
\gamma\textit{ double-fills }Y\\
\ell_\gamma(X) \in[k,k+1]
}}
\frac{\ell_\gamma(X)}{2\sinh\frac{\ell_\gamma(X)}{2}}f_T(\ell_\gamma(X))\\
\leq& C(1,\epsilon,1) e^{k+1-(1-\epsilon)\ell(\partial Y)} \frac{k+1}{2\sinh\frac{k}{2}} \textbf{1}_{\ell(\partial Y)\leq k+1}.
    \end{aligned}
    \end{equation}
If $X\in \mathcal{N}_\ell$ for $\ell=\kappa\log g$, for any $Y\simeq S_{0,3}$ with $X\setminus Y\simeq S_{g-2,n+3}$, we have
$\ell(\partial Y)\geq \ell \geq 3l_\delta$ for any $\delta>\frac{1}{2}$ when $g$ is large enough, where $l_\delta$ satisfies that the Hausdorff dimension of the limit set of $\Gamma$ with $ \Gamma\backslash \mathbb{H} \in \mathcal{T}_{0,3}(l_\delta,0,0)$ is $\delta$.
By Lemma \ref{counting s03 e 1/2 L}, we have 
    \begin{equation}\label{eq type 1-2}
    \begin{aligned}
        &1_{\mathcal{N}_\ell}\sum_{
\substack{
\gamma\textit{ double-fills }Y\\
\ell_\gamma(X) \in[k,k+1]
}}
\frac{\ell_\gamma(X)}{2\sinh\frac{\ell_\gamma(X)}{2}}f_T(\ell_\gamma(X))\\
\leq &C_\delta \frac{e^{\delta(k+1)}}{k+1}\frac{k+1}{2\sinh \frac{k}{2}}\textbf{1}_{\ell(\partial Y)\leq k+1}.
    \end{aligned}
    \end{equation}
For $Y\simeq S_{0,3}$ with $X\setminus Y\simeq S_{g-2,n+3}$,  if  $\ell(\partial Y)\leq\frac{T}{2}$, we have 
    \begin{equation}\label{eq type 1-3}
    \begin{aligned}
        &1_{\mathcal{N}_\ell}\sum_{\substack{
\gamma\textit{ double-fills }Y\\
\ell_\gamma(X)\leq T
}}
\frac{\ell_\gamma(X)}{2\sinh\frac{\ell_\gamma(X)}{2}}f_T(\ell_\gamma(X))\\
\leq & \sum_{k=\max\{1,[\ell(\partial Y)]\}}^{[2\ell(\partial Y)]} 1_{\mathcal{N}_\ell}\sum_{
\substack{
\gamma\textit{ double-fills }Y\\
\ell_\gamma(X) \in[k,k+1]
}}
\frac{\ell_\gamma(X)}{2\sinh\frac{\ell_\gamma(X)}{2}}f_T(\ell_\gamma(X))\\
+&\sum_{k=\max\{1,[2\ell(\partial Y)]\}}^{[T]} 1_{\mathcal{N}_\ell}\sum_{
\substack{
\gamma\textit{ double-fills }Y\\
\ell_\gamma(X) \in[k,k+1]
}}
\frac{\ell_\gamma(X)}{2\sinh\frac{\ell_\gamma(X)}{2}}f_T(\ell_\gamma(X)).
    \end{aligned}
    \end{equation}
Notice that we use the fact that $\ell_\gamma(X)\geq 4\arcsinh 1>1$ here. 
Take the estimations \eqref{eq type 1-1} and \eqref{eq type 1-2} into \eqref{eq type 1-3}. For $Y\simeq S_{0,3}$ with $X\setminus Y\simeq S_{g-2,n+3}$ and $\ell(\partial Y)\leq\frac{T}{2}$,
we have 
\begin{equation}\label{eq type 1-4}
\begin{aligned}
    &1_{\mathcal{N}_\ell}\sum_{\substack{
\gamma\textit{ double-fills }Y\\
\ell_\gamma(X)\leq T
}}
\frac{\ell_\gamma(X)}{2\sinh\frac{\ell_\gamma(X)}{2}}f_T(\ell_\gamma(X))\\
 \prec& \sum_{k=\max\{1,[\ell(\partial Y)]\}}^{[2\ell(\partial Y)]} e^{k-(1-\epsilon)\ell(\partial Y)}\frac{k+1}{2\sinh \frac{k}{2}}\\
+&\sum_{k=\max\{1,[2\ell(\partial Y)]\}}^{[T]} \frac{e^{\delta(k+1)}}{k+1}\frac{k+1}{2\sinh\frac{k}{2}}\\
\prec &\left(1+\ell(\partial Y)\right)e^{\epsilon\ell(\partial Y)}+e^{(\delta-\frac{1}{2})T},
\end{aligned}
\end{equation}
where the implied constants depend on $\epsilon,\delta$ and are uniform on $T, g, \ell(\partial Y)$. 

For $Y\simeq S_{0,3}$ with $X\setminus Y\simeq S_{g-2,n+3}$ and $\frac{T}{2}< \ell(\partial Y)\leq T$, by \eqref{eq type 1-1} we have 
    \begin{equation}\label{eq type 1-5}
    \begin{aligned}
         &1_{\mathcal{N}_\ell}\sum_{\substack{
\gamma\textit{ double-fills }Y\\
\ell_\gamma(X)\leq T
}}
\frac{\ell_\gamma(X)}{2\sinh\frac{\ell_\gamma(X)}{2}}f_T(\ell_\gamma(X))\\
 \leq&1_{\mathcal{N}_\ell}\sum_{\substack{
\gamma\textit{ double-fills }Y\\
 [\ell(\partial Y)]\leq \ell_\gamma(X)\leq [T]+1
}}
\frac{\ell_\gamma(X)}{2\sinh\frac{\ell_\gamma(X)}{2}}f_T(\ell_\gamma(X))\\
\prec& \sum_{k=[\ell(\partial Y)]}^{[T]} e^{k-(1-\epsilon)\ell(\partial Y)}\frac{k+1}{2\sinh\frac{k}{2}}\\
\prec&Te^{\frac{T}{2}-(1-\epsilon)\ell(\partial Y)}
    \end{aligned}
    \end{equation}
where the implied constants depend on $\epsilon$ and are uniform on $T,g,\ell(\partial Y)$.

Combining \eqref{eq type 1-4} and \eqref{eq type 1-5}, we have 
\begin{equation}\label{eq type 1-6}
      1_{\mathcal{N}_\ell}\sum_{\substack{
\gamma\textit{ double-fills }Y\\
\ell_\gamma(X)\leq T
}}
H_{X,1}(\gamma)\prec B_{\epsilon,\delta,T}^1\left(\ell(\partial Y)\right),
\end{equation}
where $$B_{\epsilon,\delta,T}^1\left(x\right)
=\left\{\begin{aligned}
&(1+x)e^{\epsilon x}+e^{(\delta-\frac{1}{2})T}, &\textit{ if }&\quad  x\leq \frac{T}{2};\\
&Te^{\frac{T}{2}-(1-\epsilon)x},&\textit{ if }& \quad\frac{T}{2}<x\leq T;\\
&0&\textit{if }&\quad x> T.\\
\end{aligned}
\right.
$$
Let $\Gamma$ be the union of $3$ closed geodesics such that $S_g\setminus \Gamma\simeq S_{0,3}\cup S_{g-2,n+3}$ and take it into Mirzakhani's Integration Formula Theorem \ref{thm mir int formula}, by \eqref{eq type 1-6} we have 
    \begin{equation}\label{eq int type 1-7}
    \begin{aligned}
&\Egn\left[      1_{\mathcal{N}_\ell}\sum_{\substack{
\gamma\textit{ double-fills of type 1}\\
\ell_\gamma(X)\leq T
}}
H_{X,1}(\gamma)\right]\\
 \prec&\Egn\left[  \sum_{\Gamma^\prime\in\Mod_{g,n} \cdot \Gamma} B_{\epsilon,\delta,T}^1(\ell_{\Gamma^\prime}(X))  \right]\\
 \prec &\frac{1}{V_{g,n}}\int_{\mathbb{R}_+^3} V_{0,3}(x,y,z)V_{g-2,n+3}(x,y,z,0^n) \\
 \cdot &B_{\epsilon,\delta,T}^1(x+y+z)  \cdot xyz\cdot dxdydz.
 \end{aligned}
    \end{equation}
By Theorem \ref{mir07 poly}, Theorem \ref{thm mz15 asymp} and Lemma \ref{lemma NWX vgn(x)} we have\begin{align*}
    &\frac{V_{0,3}(x,y,z)V_{g-2,n+3}(x,y,z,0^n)}{V_{g,n}}\\
    =&\frac{V_{g-2,n+3}}{V_{g,n}}\frac{V_{g-2,n+3}(x,y,z,0^n)}{V_{g-2,n+3}}\\
    \prec& \frac{1}{g}\frac{\sinh \frac{x}{2}\sinh \frac{y}{2}\sinh \frac{z}{2}}{xyz}.
\end{align*}
If follows that 
    \begin{equation}\label{eq int type 1-8}
    \begin{aligned}
        &\int_{0\leq x+y+z\leq \frac{T}{2}} \frac{V_{0,3}(x,y,z)V_{g-2,n+3}(x,y,z,0^n)}{V_{g,n}} \\
        \cdot& B_{\epsilon,\delta,T}^1(x+y+z) xyz dxdydz\\
        \prec &\frac{1}{g}\int_{0\leq x+y+z\leq \frac{T}{2}}
        [(1+T)e^{\epsilon(x+y+z)}+e^{(\delta-\frac{1}{2})T}]e^{\frac{x+y+z}{2}}\cdot dxdydz\\
    \prec&\frac{T^3e^{\left(\frac{\epsilon }{2}+\frac{1}{4}\right)T}+T^2e^{(\delta-\frac{1}{4}) T}}{g},
    \end{aligned}
    \end{equation}
and
    \begin{equation}\label{eq int type 1-9}
    \begin{aligned}
        &\int_{\frac{T}{2}\leq x+y+z\leq T} \frac{V_{0,3}(x,y,z)V_{g-2,n+3}(x,y,z,0^n)}{V_{g,n}} \\
        \cdot&B_{\epsilon,\delta,T}^1(x+y+z) xyz dxdydz\\
        \prec &\frac{1}{g}\int_{\frac{T}{2}\leq x+y+z\leq T}Te^{\frac{T}{2}-(1-\epsilon)(x+y+z)}
         e^{\frac{x+y+z}{2}}\cdot dxdydz\\
        \prec&\frac{T^3e^{(\frac{\epsilon}{2}+\frac{1}{4})T}}{g}.
    \end{aligned}
    \end{equation}
Then the lemma follows from \eqref{eq int type 1-7}, \eqref{eq int type 1-8}, and \eqref{eq int type 1-9}.
\end{proof}

By Theorem \ref{thm classify filling}, if $\gamma$ fills $Y\simeq S_{0,3}$ but is not double-filling, then $\gamma$ is a figure-eight or a one-sided iterated eight closed geodesic. For these terms in $\textrm{Int}_{ns}$, we can remove the condition $X\in \mathcal{N}_\ell$ by the non-negativity of $H_{X,1}$ or $f_T$.

\begin{lemma}\label{lemma figure eight type 1}
If $n=n(g)=o(\sqrt{g})$, then 
    \begin{align*}
   & \Egn\left[      \sum_{\substack{
\gamma\textit{ figure-eight of type 1}\\
\ell_\gamma(X)\leq T
}}
H_{X,1}(\gamma)\right]\\
= &\frac{1}{\pi^2g}\int_0^\infty f_T(t) e^\frac{t}{2}\left[\frac{t^2}{4}-\left(\frac{1}{2}+\ln 2\right)t\right]dt+O\left(\frac{T^3}{g}+\frac{(1+n)T^4e^\frac{T}{2}}{g^2}\right).
    \end{align*}
    where the implied constant is independent of $g$ and $T$.
\end{lemma}
\begin{proof}
In a pair of pants with boundary lengths $x,y,z\geq 0$, there are three figure-eight closed geodesics of lengths $L_{f-8}(x,y,z)$, $L_{f-8}(y,z,x)$ and $L_{f-8}(z,x,y)$ that spiral around different boundary components, where $L_{f-8}$ is given by \eqref{eq int f-8 in s03  1}.  Let $\Gamma$ be the union of three simple closed geodesics that separates $S_g$ into $S_{0,3}\cup S_{g-2,n+3}$, which serves as the boundary of the pair of pants that $\gamma$ fills. We take $\Gamma$ into Mirzakhani's Integration Formula Theorem \ref{thm mir int formula} to get
\begin{equation}\label{eq int f-8 in s03  2}
\begin{aligned}
    &\Egn\left[\sum_{\substack{
\gamma\textit{ figure-eight of type 1}\\
\ell_\gamma(X)\leq T
}}
H_{X,1}(\gamma)
    \right]\\
   =&2\times \frac{1}{6}\int_{\mathbb{R}_+^3}\left[ \overline{H}_{X,1}(L_{f-8}(x,y,z))+\overline{H}_{X,1}(L_{f-8}(y,z,x))\right.\\
   +& \left.\overline{H}_{X,1}(L_{f-8}(z,x,y)) \right]
 \cdot\frac{V_{0,3}(x,y,z)V_{g-2,n+3}(x,y,z,0^n)}{V_{g,n}}xyzdxdydz\\
=&\int_{\mathbb{R}_+^3}\overline{H}_{X,1}(L_{f-8}(x,y,z))\frac{V_{0,3}(x,y,z)V_{g-2,n+3}(x,y,z,0^n)}{V_{g,n}}xyzdxdydz.
\end{aligned}
\end{equation}
Notice that the factor $\frac{1}{6}$ offsets the multiplicity of the ordering of geodesics in $\Gamma$, and the factor $2$ comes from the two orientations of a figure-eight closed geodesic. The second equality is given by the symmetry of $V_{0,3}(x,y,z)$ and $V_{g-2,n+3}(x,y,z,0^n)$ with respect to $x,y,z$.

By Lemma \ref{lemma NWX vgn(x)} and Theorem \ref{thm vgn/vgn+1 n^2+1/g^2}, for $n=o(\sqrt{g})$ we have \begin{equation}\label{eq int f-8 in s03 volume}
\begin{aligned}
    &\frac{V_{0,3}(x,y,z)V_{g-2,n+3}(x,y,z,0^n)}{V_{g,n}}\\
   =&\frac{V_{g-2,n+3}}{V_{g,n}} \frac{V_{g-2,n+3}(x,y,z,0^n)}{V_{g-2,n+3}}\\
   =&\frac{1}{8\pi^2g}\left(1+O\left(\frac{1+n}{g}\right)\right)\\
    \cdot& \frac{2\sinh \frac{x}{2}}{x}\frac{2\sinh \frac{y}{2}}{y}\frac{2\sinh \frac{z}{2}}{z}\left(1+O\left(\frac{(1+n)(x^2+y^2+z^2)}{g}\right)\right)\\
    =&\frac{1}{\pi^2g}\frac{\sinh \frac{x}{2}}{x}\frac{\sinh \frac{y}{2}}{y}\frac{\sinh \frac{z}{2}}{z}\left(1+O\left(\frac{(1+n)(1+x^2+y^2+z^2)}{g}\right)\right).
\end{aligned}
\end{equation}
Take \eqref{eq int f-8 in s03 volume} into \eqref{eq int f-8 in s03  2}. By \eqref{eq int f-8 in s03  1} we have $x,y,z\prec T$ if $L_{f-8}(x,y,z)\leq T$. Since $\textbf{supp}(\overline{H}_{X,1})=\textbf{supp}(f_T)=[-T,T]$, we have 
\begin{equation}\label{eq int f-8 in s03  3}
\begin{aligned}
 & \Egn\left[\sum_{\substack{
\gamma\textit{ figure-eight of type 1}\\
\ell_\gamma(X)\leq T
}}
H_{X,1}(\gamma)
    \right]\\  
  =&\frac{1+O\left(\frac{(1+n)T^2}{g}\right)}{\pi^2g}\int_{\substack{x,y,z\geq 0\\ L_{f-8}(x,y,z)\leq T}} \overline{H}_{X,1}(L_{f-8}(x,y,z))\\
   \cdot&\sinh\frac{x}{2}\sinh\frac{y}{2}\sinh\frac{z}{2}dxdydz.
\end{aligned}
\end{equation}
We can simplify this integration by substituting the variables $(x,y,z)$ to $(x,y,t)$  where $t=L_{f-8}(x,y,z)$.
We can differentiate \eqref{eq int f-8 in s03  1} to get  \begin{align*}
    \sinh\frac{t}{2}dt=*dx+*dy+\sinh\frac{z}{2}dz,
\end{align*} 
where both $*$ represent some functions on $x,y,z$ which are not important. The domain of integration is substituted by \begin{equation*}
    \left\{ 
    \begin{aligned}
        &x,y,z\geq 0\\
        &t=L_{f-8}(x,y,z)\leq T
    \end{aligned}
    \right.\Leftrightarrow \left\{
    \begin{aligned}
        &x,y,t\geq 0\\
        &t\leq T\\
        &2\cosh\frac{x}{2}\cosh\frac{y}{2}+1\leq \cosh\frac{t}{2},
    \end{aligned}
    \right.
\end{equation*}
which is equivalent to \begin{equation*}
    \left\{\begin{aligned}
        &x,y,t\geq 0\\
        & 2\arccosh 3\leq t\leq T\\
        &2\cosh\frac{y}{2}+1\leq \cosh\frac{t}{2}\\
        &\cosh\frac{x}{2}\leq \frac{\cosh\frac{t}{2}-1}{2\cosh\frac{y}{2}}.
    \end{aligned}\right.
\end{equation*} 
So we have\begin{align*}\label{eq int f-8 in s03  4}
&\int_{\substack{x,y,z\geq 0\\ L_{f-8}(x,y,z)\leq T}} \overline{H}_{X,1}(L_{f-8}(x,y,z))\sinh\frac{x}{2}\sinh\frac{y}{2}\sinh\frac{z}{2}dxdydz\\
=&\int_{2\arccosh3}^T\overline{H}_{X,1}(t)\sinh\frac{t}{2}\left[ \int_{2\cosh\frac{y}{2}+1\leq \cosh\frac{t}{2}}\sinh \frac{y}{2}   \right.\\
&\left.\left(\int_{\cosh\frac{x}{2}\leq \frac{\cosh\frac{t}{2}-1}{2\cosh\frac{y}{2}}}\sinh\frac{x}{2}dx\right)dy\right]dt\\
=&\int_{2\arccosh3}^T\overline{H}_{X,1}(t)\sinh\frac{t}{2}\left[ \int_{2\cosh\frac{y}{2}+1\leq \cosh\frac{t}{2}}2\sinh \frac{y}{2} \left(\frac{\sinh^2\frac{t}{4}}{\cosh\frac{y}{2}}-1\right)  dy\right]dt\\
=&\int_{2\arccosh3}^T\overline{H}_{X,1}(t)\sinh\frac{t}{2}\left[ 4\sinh^2\frac{t}{4}\log\left(\sinh^2\frac{t}{4}\right)-4\sinh^2\frac{t}{4}+4\right]dt\\
=&\int_{2\arccosh3}^T \frac{1}{2}tf_T(t)\left[ 4\sinh^2\frac{t}{4}\log\left(\sinh^2\frac{t}{4}\right)-4\sinh^2\frac{t}{4}+4\right]dt\\
=&\int_{2\arccosh3}^T \frac{1}{2}tf_T(t)\left[ \left(e^\frac{t}{2}+O(1)\right)\left(\frac{t}{2}-2\ln2+2\log \left(1-e^{-\frac{t}{2}}\right)\right)\right.\\
&\left.-(e^\frac{t}{2}+O(1))+4 \right]dt\\
=&\int_0^\infty f_T(t)e^\frac{t}{2}\left[\frac{t^2}{4}-\left(\frac{1}{2}+\ln 2\right)t\right]dt+O\left(T^3\right),
\end{align*}
which together with \eqref{eq int f-8 in s03  3} gives Lemma \ref{lemma figure eight type 1}.
\end{proof}

\begin{lemma}\label{lemma one side iterated type 1}
If $n=n(g)=o(\sqrt{g})$, then for any $k\geq 2$,
    \begin{align*}
   & \Egn\left[      \sum_{\substack{
\gamma\textit{ one-sided iterated}\\
\textit{eight of type 1}\\
\textit{iteration number =k}\\
\ell_\gamma(X)\leq T
}}
H_{X,1}(\gamma)\right]\\
= &\frac{1}{\pi^2g}\left(\int_{0}^\infty \frac{\sinh^3\frac{y}{2}}{\sinh\frac{k}{2}y\sinh\frac{k+1}{2}y}dy\right)\left(\int_0^Ttf_T(t)e^\frac{t}{2}dt\right)\\
  +&O\left(\frac{\ln k+T^2}{k^3g}+\frac{(1+n)T^3e^\frac{T}{2}}{k^4g^2}+\textbf{1}_{k=2}\frac{T^3}{g} \right),
    \end{align*}
    where the implied constant is independent of $g,k,T$.
\end{lemma}

\begin{proof}
In a pair of pants with boundary lengths $x,y,z\geq 0$, there are $6$ one-sided iterated eight closed geodesics of iteration number $k\geq 2$ of length $L_k(x,y,z)$, $L_k(x,z,y)$, $L_k(y,x,z)$, $L_k(y,z,x)$, $L_k(z,x,y)$ and  $L_k(z,y,x)$, where $L_k$ is given by \eqref{eq int one-sided in s03  1}.
We can use a similar integration method by substitution as in the Lemma \ref{lemma figure eight type 1}. 
Let $\Gamma$ be the union of three simple closed geodesics that separates $S_g$ into $S_{0,3}\cup S_{g-2,n+3}$, which serves as the boundary of the pair of pants that $\gamma$ fills, and take it into Mirzakhani's Integration Formula Theorem \ref{thm mir int formula}.  By the symmetry of $V_{0,3}(x,y,z)$ and $V_{g-2,n+3}(x,y,z,0^n)$ with respect to $x,y,z$ and the estimate \eqref{eq int f-8 in s03 volume}, we have \begin{equation}\label{eq int one-sided in s03  2}
\begin{aligned}
    &\Egn\left[ \sum_{\substack{
\gamma\textit{ one-sided iterated}\\
\textit{eight of type 1}\\
\textit{iteration number =k}\\
\ell_\gamma(X)\leq T
}}
H_{X,1}(\gamma)
    \right]\\
    =&2\!\times\! \frac{1}{6}\!\int_{\mathbb{R}_+^3}\!\left[ 
    \overline{H}_{X,1}(L_k(x,y,z))\!+\!\overline{H}_{X,1}(L_k(x,z,y))\!+\!\overline{H}_{X,1}(L_k(y,x,z))\right.\\
    &\left.+\overline{H}_{X,1}(L_k(y,z,x))+\overline{H}_{X,1}(L_k(z,x,y))+\overline{H}_{X,1}(L_k(z,y,x))
    \right]\\
   \cdot&\frac{V_{0,3}(x,y,z)V_{g-2,n+3}(x,y,z,0^n)}{V_{g,n}}xyzdxdydz\\
=&2\!\int_{\mathbb{R}_+^3}\!\overline{H}_{X,1}(L_k(x,y,z))\!\cdot\!\frac{V_{0,3}(x,y,z)V_{g-2,n+3}(x,y,z,0^n)}{V_{g,n}}xyzdxdydz\\
   =&\frac{2+O\left(\frac{(1+n)T^2}{g}\right)}{\pi^2g}\\
   \cdot&\int_{\substack{x,y,z\geq 0\\ L_k(x,y,z)\leq T}} \overline{H}_{X,1}(L_k(x,y,z))\sinh\frac{x}{2}\sinh\frac{y}{2}\sinh\frac{z}{2}dxdydz.
\end{aligned}
\end{equation}
We substitute the variables $(x,y,z)$ to $(x,y,t)$ with $t=L_k(x,y,z)$.
We can differentiate \eqref{eq int one-sided in s03  1} to get $$
\sinh\frac{t}{2}dt=*dx+*dy+\frac{\sinh\frac{k}{2}y}{\sinh\frac{y}{2}}\sinh\frac{z}{2}dz,
$$
where  both $*$ represent some functions on $x,y,z$ which are not important. The domain of integration is substituted by \begin{equation*}
    \left\{\begin{aligned}
        &x,y,z\geq 0\\
        &t=L_k(x,y,z)\leq T
    \end{aligned}\right.
    \Leftrightarrow
    \left\{
    \begin{aligned}
        &x,y\geq 0, 0\leq t\leq T\\
        &\cosh\frac{t}{2}\geq \frac{\sinh\frac{k+1}{2}y}{\sinh\frac{y}{2}}\cosh\frac{x}{2}+\frac{\sinh\frac{k}{2}y}{\sinh\frac{y}{2}},
    \end{aligned}
    \right.
\end{equation*} 
which is equivalent to \begin{equation*}
    \left\{\begin{aligned}
        &x,y\geq 0\\
        &2\arccosh(2k+1) \leq t\leq T\\
        &\cosh\frac{t}{2}\geq \frac{\sinh\frac{k+1}{2}y}{\sinh\frac{y}{2}}+\frac{\sinh\frac{k}{2}y}{\sinh\frac{y}{2}}\\
        &\cosh\frac{x}{2}\leq \frac{\cosh\frac{t}{2}\sinh\frac{y}{2}-\sinh\frac{k}{2}y}{\sinh\frac{k+1}{2}y}.
    \end{aligned}\right.
\end{equation*}
So we have \begin{equation}\label{eq int one-sided in s03  3}
\begin{aligned}
    &2\int_{\substack{x,y,z\geq 0\\ L_k(x,y,z)\leq T}} \overline{H}_{X,1}(L_k(x,y,z))\sinh\frac{x}{2}\sinh\frac{y}{2}\sinh\frac{z}{2}dxdydz\\
=&2\int_{2\arccosh(2k+1)}^T\overline{H}_{X,1}(t)\sinh\frac{t}{2} \left[\int_{\cosh\frac{t}{2}\geq \frac{\sinh\frac{k+1}{2}y}{\sinh\frac{y}{2}}+\frac{\sinh\frac{k}{2}y}{\sinh\frac{y}{2}}} \frac{\sinh^2\frac{y}{2}}{\sinh\frac{k}{2}y}\right. \\
  &
  \left.
  \left( \int_{\cosh\frac{x}{2}\leq \frac{\cosh\frac{t}{2}\sinh\frac{y}{2}-\sinh\frac{k}{2}y}{\sinh\frac{k+1}{2}y}}  \sinh\frac{x}{2}dx \right)dy
  \right]dt\\
  =&\int_{2\arccosh(2k+1)}^T tf_T(t) \left[\int_{\cosh\frac{t}{2}\geq \frac{\sinh\frac{2k+1}{4}y}{\sinh\frac{y}{4}}} \frac{\sinh^2\frac{y}{2}}{\sinh\frac{k}{2}y}\right. \\
  &
  \left.
 2 \left( \frac{\cosh\frac{t}{2}\sinh\frac{y}{2}-\sinh\frac{k}{2}y}{\sinh\frac{k+1}{2}y}-1 \right)dy
  \right]dt.
\end{aligned}
\end{equation}
Note that for $m\geq 3$, we have \begin{equation}\label{eq int one-sided in s03  4}
\begin{aligned}
&\int_{0}^\infty\frac{\sinh^2\frac{y}{2}}{\sinh\frac{m}{2}y}dy\\
\prec &\int_{0}^\frac{1}{m}\frac{y}{m}dy+\int_{\frac{1}{m}}^1y^2e^{-\frac{m}{2}y}dy+\int_{1}^\infty e^{-\frac{m-2}{2}y}dy\\
\prec& \frac{1}{m^3}.
\end{aligned}
\end{equation}
Take \eqref{eq int one-sided in s03  4} into \eqref{eq int one-sided in s03  3}. It follows that \begin{equation}\label{eq int one-sided in s03  5}
\begin{aligned}
    &2\int_{\substack{x,y,z\geq 0\\ L_k(x,y,z)\leq T}} \overline{H}_{X,1}(L_k(x,y,z))\sinh\frac{x}{2}\sinh\frac{y}{2}\sinh\frac{z}{2}dxdydz\\
  =&\int_{2\arccosh(2k+1)}^T tf_T(t) \left[\int_{\cosh\frac{t}{2}\geq \frac{\sinh\frac{2k+1}{4}y}{\sinh\frac{y}{4}}} \frac{2\cosh\frac{t}{2}\sinh^3\frac{y}{2}}{\sinh\frac{k}{2}y\sinh\frac{k+1}{2}y}dy\right. \\
  +&
  \left.O\left(\frac{1}{k^3}\right)-\textbf{1}_{k=2}\int_{\cosh\frac{t}{2}\geq \frac{\sinh\frac{2k+1}{4}y}{\sinh\frac{y}{4}}}\tanh\frac{y}{2}dy
  \right]dt.
\end{aligned}
\end{equation}
Notice that for $k\geq 3$ we have \begin{equation}\label{eq int one-sided in s03  6}
\begin{aligned}
    &\int_{\cosh\frac{t}{2}\leq  \frac{\sinh\frac{2k+1}{4}y}{\sinh\frac{y}{4}}} \frac{2\cosh\frac{t}{2}\sinh^3\frac{y}{2}}{\sinh\frac{k}{2}y\sinh\frac{k+1}{2}y}dy\\
    \prec&\int_0^\infty  \frac{\sinh\frac{2k+1}{4}y}{\sinh\frac{y}{4}}
    \frac{\sinh^3\frac{y}{2}}{\sinh\frac{k}{2}y\sinh\frac{k+1}{2}y} dy\\
    \prec&\int_0^\frac{1}{k} \frac{y}{k}dy+\int_\frac{1}{k}^1 y^2e^{-\frac{2k+1}{4}y}dy+\int_1^\infty e^{(1-\frac{k}{2})y}dy\\
    \prec&\frac{1}{k^3}.
\end{aligned}
\end{equation}
While for $k=2$ we have  \begin{equation}\label{eq int one-sided in s03  7}
\begin{aligned}
    &\int_{\cosh\frac{t}{2}\leq  \frac{\sinh\frac{2k+1}{4}y}{\sinh\frac{y}{4}}} \frac{2\cosh\frac{t}{2}\sinh^3\frac{y}{2}}{\sinh\frac{k}{2}y\sinh\frac{k+1}{2}y}dy\\
    \prec&\int_0^1 \frac{\sinh\frac{5}{4}y}{\sinh\frac{y}{4}}
    \frac{\sinh^3\frac{y}{2}}{\sinh y\sinh\frac{3}{2}y} dy\\
    +&\cosh\frac{t}{2}\int_1^\infty \textbf{1}_{\cosh\frac{t}{2}\leq  \frac{\sinh\frac{5}{4}y}{\sinh\frac{y}{4}}}
    \frac{\sinh^3\frac{y}{2}}{\sinh y\sinh\frac{3}{2}y} dy\\
   \prec&1+\cosh\frac{t}{2}\int_{\max\{1,\frac{t}{2}-100\}}^\infty e^{-y}dy
  \prec 1,
\end{aligned}
\end{equation}
and  \begin{align}\label{eq int one-sided in s03  8}
    &\int_{\cosh\frac{t}{2}\geq \frac{\sinh\frac{2k+1}{4}y}{\sinh\frac{y}{4}}}\tanh\frac{y}{2}dy\prec t+1.
\end{align}
Take \eqref{eq int one-sided in s03  6}, \eqref{eq int one-sided in s03  7}, and \eqref{eq int one-sided in s03  8} into \eqref{eq int one-sided in s03  5}. We have \begin{equation}\label{eq int one-sided in s03  9}
\begin{aligned}
   &2\int_{\substack{x,y,z\geq 0\\ L_k(x,y,z)\leq T}} \overline{H}_{X,1}(L_k(x,y,z))\sinh\frac{x}{2}\sinh\frac{y}{2}\sinh\frac{z}{2}dxdydz\\
  =&\int_{2\arccosh(2k+1)}^T tf_T(t)\left[\int_{0}^\infty \frac{2\cosh\frac{t}{2}\sinh^3\frac{y}{2}}{\sinh\frac{k}{2}y\sinh\frac{k+1}{2}y}dy\right. \\
 +&\left.O\left(\frac{1}{k^3}+\textbf{1}_{k=2}t\right)
  \right]dt\\
  =&\left(\int_{0}^\infty \frac{\sinh^3\frac{y}{2}}{\sinh\frac{k}{2}y\sinh\frac{k+1}{2}y}dy\right)\left(\int_0^T2tf_T(t)\cosh\frac{t}{2}dt\right.\\
  -&\left.\int_0^{2\arccosh(2k+1)} 2tf_T(t)\cosh\frac{t}{2}dt\right)+O\left(\frac{T^2}{k^3}+\textbf{1}_{k=2} T^3 \right)\\
 =&\left(\int_{0}^\infty \frac{\sinh^3\frac{y}{2}}{\sinh\frac{k}{2}y\sinh\frac{k+1}{2}y}dy\right)\left(\int_0^T2tf_T(t)\cosh\frac{t}{2}dt+O(k\ln k)\right)\\
 +&O\left(\frac{T^2}{k^3}+\textbf{1}_{k=2} T^3 \right).
\end{aligned}
\end{equation}
Note that for $k\geq 2$ we have\begin{equation}\label{eq int one-sided in s03  10}\begin{aligned}
    &\int_{0}^\infty \frac{\sinh^3\frac{y}{2}}{\sinh\frac{k}{2}y\sinh\frac{k+1}{2}y}dy\\
    \prec&\int_{0}^\frac{1}{k}\frac{y}{k^2}dy+\int_{\frac{1}{k}}^1 y^3e^{-(k+\frac{1}{2})y}dy+\int_{1}^\infty e^{-(k-1)y}dy\\
    \prec&\frac{1}{k^4}.
\end{aligned}
\end{equation}
By combing\eqref{eq int one-sided in s03  2} \eqref{eq int one-sided in s03  9} and \eqref{eq int one-sided in s03  10}
we have \begin{align*}
&\Egn\left[ \sum_{\substack{
\gamma\textit{ one-sided iterated}\\
\textit{eight of type 1}\\
\textit{iteration number =k}\\
\ell_\gamma(X)\leq T
}}
H_{X,1}(\gamma)
    \right]\\
    =&\frac{1+O\left(\frac{(1+n)T^2}{g}\right)}{\pi^2g}\left(\int_{0}^\infty \frac{\sinh^3\frac{y}{2}}{\sinh\frac{k}{2}y\sinh\frac{k+1}{2}y}dy\right)\left(\int_0^T2tf_T(t)\cosh\frac{t}{2}dt\right)\\
  \nonumber+&O\left(\frac{\ln k}{k^3g}+\frac{T^2}{k^3g}+\textbf{1}_{k=2}\frac{T^3}{g} \right)\\
  =\nonumber&\frac{1}{\pi^2g}\left(\int_{0}^\infty \frac{\sinh^3\frac{y}{2}}{\sinh\frac{k}{2}y\sinh\frac{k+1}{2}y}dy\right)\left(\int_0^Ttf_T(t)e^\frac{t}{2}dt\right)\\
  \nonumber+&O\left(\frac{\ln k+T^2}{k^3g}+\frac{(1+n)T^3e^\frac{T}{2}}{k^4g^2}+\textbf{1}_{k=2}\frac{T^3}{g} \right),
\end{align*}
which completes the proof.
\end{proof}

Now we can calculate the total contribution of filling geodesics of type $1$ that are not double-filling to $\textrm{Int}_{ns}$ in \eqref{eq ineq on subset T}. It is incredibly magical how it perfectly cancels out the main term in Lemma \ref{lemma nsep zero eigen} when $n=0$. When $n\neq 0$, the filling geodesics of type $5$ will also participate in this cancellation. 
\begin{lemma}\label{lemma non double filling type 1}
If $n=n(g)=o(\sqrt{g})$, then 
    \begin{align*}
   & \Egn\left[      \sum_{\substack{
\gamma\textit{ is not double-filling}\\
\textit{ of type 1, }
\ell_\gamma(X)\leq T
}}
H_{X,1}(\gamma)\right]\\
=& \frac{1}{\pi^2g}\int_0^\infty f_T(t)e^\frac{t}{2}\left(\frac{t^2}{4}-t\right)dt+O\left(
\frac{T^3}{g}+\frac{(1+n)T^4e^\frac{T}{2}}{g^2}
\right),
    \end{align*}
    where the implied constant is independent of $g$ and $T$.
\end{lemma}
\begin{proof}
    By Lemma \ref{lemma figure eight type 1} and Lemma \ref{lemma one side iterated type 1}, it suffices to show that \begin{align}\label{summation one sided iterated}
        \sum_{k=2}^\infty\int_0^\infty\frac{\sinh^3\frac{y}{2}}{\sinh\frac{k}{2}y\sinh\frac{k+1}{2}y}dy=\ln 2-\frac{1}{2}.
    \end{align}
By substituting the variable through $t=e^\frac{y}{2}$, we have \begin{align*}
    &\int_0^\infty\frac{\sinh^3\frac{y}{2}}{\sinh\frac{k}{2}y\sinh\frac{k+1}{2}y}dy\\
 =&\int_1^\infty\frac{\left(t-\frac{1}{t}\right)^3}{t\left(t^k-\frac{1}{t^k}\right)\left(t^{k+1}-\frac{1}{t^{k+1}}\right)}dt\\
=&\int_1^\infty t^{2k-3}\frac{\left(t^2-1\right)^3}{\left(t^{2k}-1\right)\left(t^{2k+2}-1\right)}dt\\
=&\int_1^\infty \frac{t^2-1}{t^3}\frac{t^{2k}}{\left(1+t^2+\cdots+t^{2k-2}\right)\left(1+t^2+\cdots+t^{2k}\right)}dt\\
=&\int_1^\infty  \frac{t^2-1}{t^3}\left(
    \frac{1}{1+t^2+\cdots+t^{2k-2}}-\frac{1}{1+t^2+\cdots+t^{2k}}
    \right)dt.
\end{align*}
According to the dominated convergence theorem, we have \begin{align*}
       &\sum_{k=2}^\infty\int_0^\infty\frac{\sinh^3\frac{y}{2}}{\sinh\frac{k}{2}y\sinh\frac{k+1}{2}y}dy\\
    =&\int_1^\infty\frac{t^2-1}{t^3}\sum_{k=2}^\infty\left(
    \frac{1}{1+t^2+\cdots+t^{2k-2}}-\frac{1}{1+t^2+\cdots+t^{2k}}
    \right)dt\\
    =&\int_1^\infty\frac{t^2-1}{t^3(1+t^2)}dt\\
    =&\left.\left(\ln\frac{t^2}{1+t^2}+\frac{1}{2t^2}\right)\right|_{t=1}^{\infty}\\
    =&\ln 2-\frac{1}{2},
\end{align*} 
    which is \eqref{summation one sided iterated}.
\end{proof}

Now we combine  Lemma \ref{lemma double filling type 1} and Lemma \ref{lemma non double filling type 1} to get the following estimate.

\begin{lemma}\label{lemma total type 1}
If $n=n(g)=o(\sqrt{g})$, then for any $\delta>\frac{1}{2}$ and $\epsilon>0$, 
    \begin{align*}
   & \Egn\left[  1_{N_\ell}\cdot    \sum_{\substack{
\gamma
\textit{ of type 1}\\
\ell_\gamma(X)\leq T
}}
H_{X,1}(\gamma)\right]
\leq \frac{1}{\pi^2g}\int_0^\infty f_T(t)e^\frac{t}{2}\left(\frac{t^2}{4}-t\right)dt\\
+&O\left(
\frac{T^3e^{(\frac{\epsilon}{2}+\frac{1}{4})T} +T^2e^{(\delta-\frac{1}{4})T} }{g}+\frac{(1+n)T^4e^\frac{T}{2}}{g^2}
\right),
    \end{align*}
    where the implied constant is independent of $g$ and $T$.
\end{lemma}

\subsection{Type $2$ in $\mathrm{Int}_{ns}$}

The contribution of filling geodesics of type $2$ to $\textrm{Int}_{ns}$ in \eqref{eq ineq on subset T} can be estimated similarly to the contribution of double-filling geodesics of type $1$ as Lemma \ref{lemma double filling type 1},
since by Theorem \ref{thm classify filling}, any filling geodesic $\gamma$ in $Y\simeq S_{1,1}$ is double-filling. In particular, $\ell(\partial Y)\leq \ell_\gamma(X)$. For geodesics of type $2$, Theorem \ref{thm double filling count} also works, but Lemma \ref{counting s03 e 1/2 L} should be replaced by Theorem \ref{thm counting s11 sys small}.

\begin{lemma}\label{lemma all non double filling type 2}
    If $n=n(g)=o(\sqrt{g})$, then for any $\delta>\frac{1}{2}$ and $\epsilon>0$,  
   $$\Egn\left[      1_{\mathcal{N}_\ell}\sum_{\substack{
\gamma\textit{ of type 2}\\
\ell_\gamma(X)\leq T
}}
H_{X,1}(\gamma)\right]
\prec \frac{T^3e^{\left(\frac{\epsilon}{2}+\frac{1}{4}\right)T}+T^3e^{\left(\delta-\frac{1}{4}\right)T   }   }{g},
  $$
    where the implied constant depends on $\epsilon,\delta,\kappa$ and $a$, if $T=a\log g$.
\end{lemma}

\begin{proof}
By Theorem \ref{thm double filling count}, if $Y\simeq S_{1,1}$ and $X\setminus Y\simeq S_{g-1,n+1}$, then for any $\epsilon>0$,
    \begin{equation}\label{eq type 2-1}
    \begin{aligned}
        &1_{\mathcal{N}_\ell}\sum_{
\substack{
\gamma\textit{ double-fills }Y\\
\ell_\gamma(X)\leq \in[k,k+1]
}}
\frac{\ell_\gamma(X)}{2\sinh\frac{\ell_\gamma(X)}{2}}f_T(\ell_\gamma(X))\\
\leq& C(1,\epsilon,1) e^{k+1-(1-\epsilon)\ell(\partial Y)} \frac{k+1}{2\sinh\frac{k}{2}} \textbf{1}_{\ell(\partial Y)\leq k+1}.
    \end{aligned}
    \end{equation}
Consider $l_{\frac{1}{4}+\frac{\delta}{2}}>0$ satisfying that the Hausdorff dimension of the limit set of $\Gamma$ with $\Gamma\backslash \mathbb{H}\in \T_{0,3}(0,0,l_{\frac{1}{4}+\frac{\delta}{2}})$ is $\frac{1}{4}+\frac{\delta}{2}$. Take $\epsilon_\delta=\frac{2\delta-1}{1+2\delta}$ into Theorem \ref{thm counting s11 sys small}. Then there exists $s_\delta=s_{\epsilon_\delta,l_{\frac{1}{4}+\frac{\delta}{2}}}>0$ and $\tilde{C}_{\epsilon_\delta,l_{\frac{1}{4}+\frac{\delta}{2}}}$ such that 
for any $Y_0\in \T_{1,1}(l_{\frac{1}{4}+\frac{\delta}{2}})$ with $\sys(Y_0)\leq s_\delta$, we have
\begin{align}\label{eq type 2-1.25}
    N_1^{fill}(Y_0,L)\leq \tilde{C}_{\epsilon,l_{\frac{1}{4}+\frac{\delta}{2}}} \cdot e^{(1+\epsilon_\delta)(\frac{1}{4}+\frac{\delta}{2})L}=\tilde{C}_{\epsilon,l_{\frac{1}{4}+\frac{\delta}{2}}} \cdot e^{\delta \cdot L}.
\end{align}
Assume $X\in \mathcal{N}_\ell$ for $\ell=\kappa\log g$. When $g$ is large enough, for any $Y\simeq S_{1,1}$ with $X\setminus Y\simeq S_{g-1,n+1}$, we have $\ell(\partial Y)\geq l_{\frac{1}{4}+\frac{\delta}{2}}$. By Theorem \ref{thm mono counting}, there exists some $Y^\prime\in \T_{1,1}(l_{\frac{1}{4}+\frac{\delta}{2}})$ such that \begin{align}\label{eq type 2-1.5}
    N_1^{fill}(Y,L)\leq N_1^{fill}(Y^\prime,L)
\end{align}
for any $L>0$.
If $\sys(Y^\prime)\leq s_\delta$, 
the estimate \eqref{eq type 2-1.25} holds for $Y_0=Y^\prime$. So by \eqref{eq type 2-1.5}, if $\sys(Y^\prime)\leq s_\delta$,we have \begin{equation}\label{eq type 2-2}
\begin{aligned}
   & \textbf{1}_{\mathcal{N}_\ell}\sum_{\substack{\gamma \textit{ fills Y}\\
    \ell_\gamma(X)\in[k,k+1]
    }}
    \frac{\ell_\gamma(X)}{2\sinh\frac{\ell_\gamma(X)}{2}}f_T(\ell_\gamma(X))\\
  \leq & N_1^{fill}(Y,k+1) \frac{k+1}{2\sinh\frac{k}{2}}\cdot  \textbf{1}_{\ell(\partial Y)\leq k+1}\\
  \leq & N_1^{fill}(Y^\prime,k+1) \frac{k+1}{2\sinh\frac{k}{2}}\cdot  \textbf{1}_{\ell(\partial Y)\leq k+1} \\
\leq&\tilde{C}_{\epsilon,l_{\frac{1}{4}+\frac{\delta}{2}}} e^{\delta (k+1)}\frac{k+1}{2\sinh\frac{k}{2}}\cdot \textbf{1}_{\ell(\partial Y)\leq k+1}.
\end{aligned}
\end{equation}
If $\sys(Y^\prime)> s_\delta$, and $k\leq T=a\log g$, \eqref{eq type 2-1.5} and \cite[Lemma 10.5]{AM23-2/9} imply
\begin{equation}\label{eq type 2-3}
\begin{aligned}
   & \textbf{1}_{\mathcal{N}_\ell}\sum_{\substack{\gamma \textit{ fills Y}\\
    \ell_\gamma(X)\in[k,k+1]
    }}
    \frac{\ell_\gamma(X)}{2\sinh\frac{\ell_\gamma(X)}{2}}f_T(\ell_\gamma(X))\\
      \leq & N_1^{fill}(Y,k+1) \frac{k+1}{2\sinh\frac{k}{2}}\cdot  \textbf{1}_{\ell\leq \ell(\partial Y)\leq k+1}\\
      \leq & N_1^{fill}(Y^\prime,k+1) \frac{k+1}{2\sinh\frac{k}{2}}\cdot  \textbf{1}_{\ell\leq \ell(\partial Y)\leq k+1} \\
  \prec&\frac{k+1}{\kappa
 \log g}\left(1+\frac{k+1}{\min\{\kappa,s_\delta\}}\right)^{36\frac{k+1}{\kappa \log g}}    \frac{k+1}{2\sinh\frac{k}{2}}\textbf{1}_{\ell(\partial Y)\leq k+1}\\
   \prec & e^{\delta(k+1)} \frac{k+1}{2\sinh\frac{k}{2}}\textbf{1}_{\ell(\partial Y)\leq k+1},
\end{aligned}
\end{equation}
where  the implied constant depends only on $\kappa, A,\delta$.
Using the same argument as \eqref{eq type 1-3} and \eqref{eq type 1-4}, for $Y\simeq S_{1,1}$ with $X\setminus Y\simeq S_{g-1,n+1}$ and $\ell(\partial Y)\leq \frac{T}{2},$  by \eqref{eq type 2-1}, \eqref{eq type 2-2} and \eqref{eq type 2-3},
we have \begin{equation}\label{eq type 2-4}
    \begin{aligned}
&\textbf{1}_{\mathcal{N}_\ell}\sum_{\substack{\gamma \textit{ fills Y}\\
    \ell_\gamma(X)\leq T
    }}
    \frac{\ell_\gamma(X)}{2\sinh\frac{\ell_\gamma(X)}{2}}f_T(\ell_\gamma(X))\\
    \leq & \sum_{k=\max\{1,[\ell(\partial Y)]\}}^{[2\ell(\partial Y)]} 1_{\mathcal{N}_\ell}\sum_{
\substack{
\gamma\textit{ fills }Y\\
\ell_\gamma(X)\leq \in[k,k+1]
}}
\frac{\ell_\gamma(X)}{2\sinh\frac{\ell_\gamma(X)}{2}}f_T(\ell_\gamma(X))\\
+&\sum_{k=\max\{1,[2\ell(\partial Y)]\}}^{[T]} 1_{\mathcal{N}_\ell}\sum_{
\substack{
\gamma\textit{ fills }Y\\
\ell_\gamma(X)\leq \in[k,k+1]
}}
\frac{\ell_\gamma(X)}{2\sinh\frac{\ell_\gamma(X)}{2}}f_T(\ell_\gamma(X))\\
\prec& \sum_{k=\max\{1,[\ell(\partial Y)]\}}^{[2\ell(\partial Y)]}  e^{k-(1-\epsilon)\ell(\partial Y)}\frac{k}{2\sinh \frac{k}{2}}
+\sum_{k=\max\{1,[2\ell(\partial Y)]\}}^{[T]}   e^{\delta k}\frac{k}{\sinh\frac{k}{2}} \\
\prec&\left(1+\ell(\partial Y)\right)e^{\epsilon\ell(\partial Y)}+T e^{(\delta-\frac{1}{2})T}.
\end{aligned}
\end{equation}
Using the same argument as \eqref{eq type 1-5}, for $Y\simeq S_{1,1}$ with $X\setminus Y\simeq S_{g-1,n+1}$ and 
$\frac{T}{2}\leq \ell(\partial Y)\leq T,$ by \eqref{eq type 2-1} we have \begin{equation}\label{eq type 2-5}
\begin{aligned}
     &1_{\mathcal{N}_\ell}\sum_{\substack{
\gamma\textit{ fills }Y\\
\ell_\gamma(X)\leq T
}}
\frac{\ell_\gamma(X)}{2\sinh\frac{\ell_\gamma(X)}{2}}f_T(\ell_\gamma(X))\\
\leq &
\sum_{k=[\partial Y]}^{[T]}\textbf{1}_{\mathcal{N}_\ell}\sum_{
\substack{
\gamma\textit{ double-fills }Y\\
\ell_\gamma(X)\leq \in[k,k+1]
}}
\frac{\ell_\gamma(X)}{2\sinh\frac{\ell_\gamma(X)}{2}}f_T(\ell_\gamma(X))\\
\prec &\sum_{k=[\partial Y]}^{[T]}e^{k+1-(1-\epsilon)\ell(\partial Y)}\frac{k+1}{\sinh\frac{k}{2}}\\
\prec& Te^{\frac{T}{2}-(1-\epsilon)\ell(\partial Y)}.
\end{aligned}
\end{equation}
Combing \eqref{eq type 2-4} and \eqref{eq type 2-5} we have \begin{align}\label{eq type 2-6}
    1_{\mathcal{N}_\ell}\sum_{\substack{
\gamma\textit{ fills }Y\\
\ell_\gamma(X)\leq T
}} H_{X,1}(\gamma)\prec B_{\epsilon,\delta,T}^2(\ell(\partial Y)),
\end{align}
where  $$B_{\epsilon,\delta,T}^2\left(x\right)
=\left\{\begin{aligned}
&(1+x)e^{\epsilon x}+Te^{(\delta-\frac{1}{2})T}, &\textit{ if }&\quad  x\leq \frac{T}{2};\\
&Te^{\frac{T}{2}-(1-\epsilon)x},&\textit{ if }& \quad\frac{T}{2}<x\leq T;\\
&0&\textit{if }&\quad x> T.\\
\end{aligned}
\right.
$$
Notice that $V_{1,1}(x)\prec 1+x^2$ by Theorem \ref{mir07 poly}. Using similar computations with \eqref{eq int type 1-7}, \eqref{eq int type 1-8} and \eqref{eq int type 1-9}, by \eqref{eq type 2-6} we have \begin{align*}
    &\Egn\left[      1_{\mathcal{N}_\ell}\sum_{\substack{
\gamma\textit{ of type 2}\\
\ell_\gamma(X)\leq T
}}
H_{X,1}(\gamma)\right]\\
\prec&\int_0^\infty \frac{V_{1,1}(x)V_{g-1,n+1}(x,0^n)}{V_{g,n}}B_{\epsilon,\delta,T}^2\left(x\right) xdx\\
\prec&\frac{1}{g}\int_{0}^\infty\left(1+x^2\right)\sinh\frac{x}{2}B_{\epsilon,\delta,T}^2\left(x\right)dx\\
=&\frac{1}{g}\int_{0}^{\frac{T}{2}}\left(1+x^2\right)\sinh\frac{x}{2}\left[(1+x)e^{\epsilon x} +Te^{\left(\delta-\frac{1}{2}\right)T}\right]dx\\
+&\frac{1}{g}\int_{\frac{T}{2}}^T \left(1+x^2\right)\sinh\frac{x}{2}\left[ Te^{\frac{T}{2}-(1-\epsilon)x}\right]dx\\
\prec&\frac{T^3e^{\left(\frac{\epsilon}{2}+\frac{1}{4}\right)T}+T^3e^{\left(\delta-\frac{1}{4}\right)T   }   }{g},
\end{align*}
where the implied constant depends on $\kappa,A,\delta,\epsilon$.
\end{proof}

\subsection{Type $3$ in $\mathrm{Int}_{ns}$}
The length of any filling geodesic of type $3$ is determined by the pair of pants $Y$ that it fills, as in type $1$. 
For type $3$, the pair of pants $Y$ will have two boundary components glued together. So we can describe $Y$ by $\gamma_b\cup \gamma_{in}$, where $\gamma_b$ is the boundary geodesic of $Y$ such that $X\setminus\gamma_b\simeq S_{1,1}\cup S_{g-1,n+1}$, and $\gamma_{in}$ is a simple closed geodesic in the $S_{1,1}$ part such that $S_{1,1}\setminus \gamma_{in}=Y$.

We have a similar estimation for the contribution of double-filling terms of type $3$ to $\textrm{Int}_{ns}$ in \eqref{eq ineq on subset T} as in Lemma \ref{lemma double filling type 1}.

\begin{lemma}\label{lemma double filling type 3}
If $n=n(g)=o(\sqrt{g})$, then for any $\delta>\frac{1}{2}$ and $\epsilon>0$, 
   $$\Egn\left[      1_{\mathcal{N}_\ell}\sum_{\substack{
\gamma\textit{ double-fills of type 3}\\
\ell_\gamma(X)\leq T
}}
H_{X,1}(\gamma)\right]
\prec \frac{T^3e^{(\frac{\epsilon}{2}+\frac{1}{4})T}+T^2e^{(\delta-\frac{1}{4})T}}{g},
  $$
    where the implied constant depends on $\epsilon,\delta$ and is independent of $T,g$.
\end{lemma}

\begin{proof}
Using the same argument as for the estimation \eqref{eq type 1-6} for double-filling terms, for $Y\simeq S_{0,3}$ and $\overline{Y}\simeq S_{1,1}$ with $X\setminus Y\simeq S_{g-1,n+1}$,
we have \begin{equation}\label{eq type 3-1}
      1_{\mathcal{N}_\ell}\sum_{\substack{
\gamma\textit{ double- fills }Y\\
\ell_\gamma(X)\leq T
}}
H_{X,1}(\gamma)\prec B_{\epsilon,\delta,T}^1\left(\ell_{\gamma_b}(X)+2\ell_{\gamma_{in}}(X)\right),
\end{equation}
where $$B_{\epsilon,\delta,T}^1\left(x\right)
=\left\{\begin{aligned}
&(1+x)e^{\epsilon x}+e^{(\delta-\frac{1}{2})T}, &\textit{ if }&\quad  x\leq \frac{T}{2};\\
&Te^{\frac{T}{2}-(1-\epsilon)x},&\textit{ if }& \quad\frac{T}{2}<x\leq T;\\
&0&\textit{if }&\quad x> T.\\
\end{aligned}
\right.
$$ 
Take \eqref{eq type 3-1} over the mapping class group orbit of $\gamma_b\cup \gamma_{in}$ into Mirzakhani's Integration Formula Theorem \ref{thm mir int formula}. By Theorem \ref{mir07 poly}, Theorem \ref{thm mz15 asymp} and Lemma \ref{lemma NWX vgn(x)},
 we have \begin{equation}\label{eq type 3-2}
 \begin{aligned}
&\Egn\left[ 1_{\mathcal{N}_\ell}\sum_{\substack{
\gamma\textit{ double- fills of type }3\\
\ell_\gamma(X)\leq T
}}
H_{X,1}(\gamma)\right]\\
\prec&\frac{1}{V_{g,n}}\int_{\mathbb{R}_+^2}V_{0,3}(x,x,y)V_{g-1,n+1}(y,0^n)B_{\epsilon,\delta,T}^1\left(2x+y\right)xy dxdy\\
\prec&  \frac{V_{g-1,n+1}}{V_{g,n}}\int_{\mathbb{R}_+^2}B_{\epsilon,\delta,T}^1\left(2x+y\right)x\sinh\frac{y}{2} dxdy \\
\prec&\frac{1}{g}\int_{\mathbb{R}_+^2}B_{\epsilon,\delta,T}^1\left(2x+y\right)x\sinh\frac{y}{2}dxdy.
 \end{aligned}
 \end{equation}
Direct computation shows that 
\begin{equation}\label{eq type 3-3}
\begin{aligned}
    & \int_{2x+y\leq \frac{T}{2}}B_{\epsilon,\delta,T}^1\left(2x+y\right)x\sinh\frac{y}{2}dxdy\\
\prec& \int_{2x+y\leq \frac{T}{2}}\left[(1+2x+y)e^{\epsilon(2x+y)}+e^{\left(\delta-\frac{1}{2}\right)T} \right]x\sinh\frac{y}{2}dxdy\\
\prec&T^3e^{\left(\frac{\epsilon}{2}+\frac{1}{4}\right)T}+T^2e^{\left(\delta-\frac{1}{4}\right)T},
 \end{aligned}
 \end{equation}
 and \begin{equation}\label{eq type 3-4}
 \begin{aligned}
     & \int_{\frac{T}{2}\leq 2x+y\leq T}B_{\epsilon,\delta,T}^1\left(2x+y\right)x\sinh\frac{y}{2}dxdy\\
   \prec &\int_{\frac{T}{2}\leq 2x+y\leq T} Te^{\frac{T}{2}-(1-\epsilon)(2x+y)}x\sinh\frac{y}{2}dxdy\\
   \prec&T^3e^{\left(\frac{\epsilon}{2}+\frac{1}{4}\right)T}.
 \end{aligned}
 \end{equation}
The lemma follows from \eqref{eq type 3-2}, \eqref{eq type 3-3} and \eqref{eq type 3-4}.
\end{proof}
The contribution of geodesics of type $3$ that are not double-filling to $\textrm{Int}_{ns}$ in \eqref{eq ineq on subset T} can be accurately estimated as in Lemma \ref{lemma figure eight type 1} and Lemma \ref{lemma one side iterated type 1}, but it will be a small term since the Weil-Petersson volume of the moduli space $V_{0,3}(\ell_{\gamma_{in}}(X),\ell_{\gamma_{in}}(X),\ell_{\gamma_{b}}(X))\cdot V_{g-1,n+1}(\ell_{\gamma_{b}}(X),0^n)$ is not exponential on $\ell_{\gamma_{in}}(X)$. We only provide an effective upper bound for it.
If we view a figure-eight closed geodesic as a one-sided iterated closed geodesic with iteration number $k=1$, then we have the following estimates.
\begin{lemma}\label{lemma one-sided iterated type 3}
If $n=n(g)=o(\sqrt{g})$, then for any $k\geq 1$,
    \begin{align*}
    \Egn\left[      \sum_{\substack{
\gamma\textit{ one-sided iterated}\\
\textit{eight of type 3}\\
\textit{iteration number =k}\\
\ell_\gamma(X)\leq T
}}
H_{X,1}(\gamma)\right]\prec\frac{T^4}{gk^2},
    \end{align*}
    where the implied constant is independent of $g,k$, and $T$.
\end{lemma}
\begin{proof}
    In $Y=S_{1,1}\setminus \gamma_{in}$, there are $3$ different kinds of one-sided iterated eight closed geodesics.  The first kind will spiral around the boundary geodesics of length $\ell_{\gamma_{b}}(X)$ for $1$ time and spiral around the boundary geodesics of length $\ell_{\gamma_{in}}(X)$ for $k$ times. The second kind will spiral around the boundary geodesics of length $\ell_{\gamma_{b}}(X)$ for $k$ times and spiral around the boundary geodesics of length $\ell_{\gamma_{in}}(X)$ for $1$ time. The third kind will spiral around the two boundary geodesics of length $\ell_{\gamma_{in}}(X)$ for $1$ and $k$ times separately.   Notice that for $k=1$ the first two kinds of geodesics will coincide. They have lengths $L_k(\ell_{\gamma_{b}}(X),\ell_{\gamma_{in}}(X),\ell_{\gamma_{in}}(X))$, $L_k(\ell_{\gamma_{in}}(X),\ell_{\gamma_{b}}(X),\ell_{\gamma_{in}}(X))$ and $L_k(\ell_{\gamma_{in}}(X),\ell_{\gamma_{in}}(X),\ell_{\gamma_{b}}(X))$, where $L_k$ is given by \eqref{eq int one-sided in s03  1}.
    Taking the mapping class group orbit of $\gamma_{in}\cup\gamma_{b}$ into Mirzakhani's Integration Formula Theorem \ref{thm mir int formula},  by Theorem \ref{mir07 poly}, Theorem \ref{thm mz15 asymp} and Lemma \ref{lemma NWX vgn(x)},
 we have\begin{equation}\label{eq int one sided type 3-1}
 \begin{aligned}
    &  \Egn\left[      \sum_{\substack{
\gamma\textit{ one-sided iterated}\\
\textit{eight of type 3}\\
\textit{iteration number =k}\\
\ell_\gamma(X)\leq T
}}
H_{X,1}(\gamma)\right]\\
\prec&\frac{1}{V_{g,n}}\int_{\mathbb{R}_+^2}\!\left[\overline{H}_{X,1}(L_k(y,x,x))\!+\!\overline{H}_{X,1}(L_k(x,y,x))\!+\!\overline{H}_{X,1}(L_k(x,x,y)) \right]\\
\cdot&V_{0,3}(x,x,y)V_{g-1,n+1}(y,0^n)xydxdy\\
\prec&\frac{1}{g}\int_{\mathbb{R}_+^2}\left[\overline{H}_{X,1}(L_k(y,x,x))+\overline{H}_{X,1}(L_k(x,y,x))+\overline{H}_{X,1}(L_k(x,x,y)) \right]\\
\cdot&x\sinh\frac{y}{2}dxdy.
 \end{aligned}
 \end{equation}
 Notice that by \eqref{eq int one-sided in s03  1} we have $$
 \cosh\frac{L_k(x,y,z)}{2}\geq\cosh\frac{ky}{2}\cosh\frac{x}{2}+\cosh\frac{z}{2}\geq \cosh\frac{x+ky}{2}+\cosh\frac{z}{2}.
 $$
 So $$
 L_k(x,y,z)\geq \max\{x+ky,z\}.
 $$
Since $\textbf{supp}(f_T)\subset[-T,T]$, we have \begin{equation}\label{eq int one sided type 3-2}
\begin{aligned}
    &\int_{\mathbb{R}_+^2}\overline{H}_{X,1}(L_k(y,x,x))x\sinh\frac{y}{2}dxdy\\
    \prec&\int_{kx,y\leq L_k(y,x,x)\leq T} \frac{L_k(y,x,x)}{\sinh\frac{L_k(y,x,x)}{2}}x\sinh\frac{y}{2}dxdy\\
    \prec&\int_{x\leq \frac{T}{k}}\int_{y\leq T}\frac{y}{\sinh\frac{y}{2}}x\sinh\frac{y}{2}dxdy\\
    \prec &\frac{T^4}{k^2},
\end{aligned}
\end{equation}
\begin{equation}\label{eq int one sided type 3-3}
\begin{aligned}
    &\int_{\mathbb{R}_+^2}\overline{H}_{X,1}(L_k(x,y,x))x\sinh\frac{y}{2}dxdy\\
    \prec &\int_{x+ky\leq L_k(x,y,x)\leq T}\frac{L_k(x,y,x)}{\sinh\frac{L_k(x,y,x)}{2}}x\sinh\frac{y}{2}dxdy\\
   \leq& \int_{x+ky\leq T}\frac{x+ky}{\sinh\frac{x+ky}{2}} x\sinh\frac{y}{2}dxdy\\
   \prec&\int_{x+ky\leq T}\frac{x+ky}{\sinh\frac{x+ky}{2}} \sinh\frac{x}{2}\frac{\sinh\frac{k}{2}y}{k}dxdy\\
    \prec&\frac{1}{k}\int_{x+ky\leq T}(x+ky)dxdy\\
    \prec&\frac{T^3}{k^2},
\end{aligned}
\end{equation}
and \begin{equation}\label{eq int one sided type 3-4}
\begin{aligned}
    &\int_{\mathbb{R}_+^2}\overline{H}_{X,1}(L_k(x,x,y))x\sinh\frac{y}{2}dxdy\\
   \prec&\int_{kx,y\leq L_k(x,x,y)\leq T} \frac{L_k(x,x,y)}{\sinh\frac{L_k(x,x,y)}{2}}x\sinh\frac{y}{2}dxdy\\
    \prec&\int_{x\leq \frac{T}{k}}\int_{y\leq T}\frac{y}{\sinh\frac{y}{2}}x\sinh\frac{y}{2}dxdy\\
    \prec &\frac{T^4}{k^2}.
\end{aligned}
\end{equation}

Now the lemma follows from \eqref{eq int one sided type 3-1}, \eqref{eq int one sided type 3-2}, \eqref{eq int one sided type 3-3} and \eqref{eq int one sided type 3-4}.
\end{proof}
Now we can bound the total contribution of filling geodesics of type $3$ to $\textrm{Int}_{ns}$ in \eqref{eq ineq on subset T}. 
\begin{lemma}\label{lemma total type 3}
If $n=n(g)=o(\sqrt{g})$, then for any $\delta>\frac{1}{2}$ and $\epsilon>0$, 
   $$\Egn\left[      1_{\mathcal{N}_\ell}\sum_{\substack{
\gamma\textit{ of type 3}\\
\ell_\gamma(X)\leq T
}}
H_{X,1}(\gamma)\right]
\prec \frac{T^3e^{(\frac{\epsilon}{2}+\frac{1}{4})T}+T^2e^{(\delta-\frac{1}{4})T}}{g},
  $$
    where the implied constant depends on $\epsilon,\delta$ and is independent of $T,g$.
\end{lemma}
\begin{proof}
    By Lemma \ref{lemma double filling type 3} and Lemma \ref{lemma one-sided iterated type 3}, we have \begin{align*}
&\Egn\left[      1_{\mathcal{N}_\ell}\sum_{\substack{
\gamma\textit{ of type 3}\\
\ell_\gamma(X)\leq T
}}
H_{X,1}(\gamma)\right]\\
\prec& \frac{T^3e^{(\frac{\epsilon}{2}+\frac{1}{4})T}+T^2e^{(\delta-\frac{1}{4})T}}{g}+\frac{T^4}{g}\sum_{k=1}^\infty\frac{1}{k^2}\\
\prec& \frac{T^3e^{(\frac{\epsilon}{2}+\frac{1}{4})T}+T^2e^{(\delta-\frac{1}{4})T}+T^4}{g}\\
\prec& \frac{T^3e^{(\frac{\epsilon}{2}+\frac{1}{4})T}+T^2e^{(\delta-\frac{1}{4})T}}{g},
\end{align*}
  which completes the proof.
\end{proof}

\subsection{Type $4$ in $\mathrm{Int}_{ns}$}
In this type, the nonsimple closed geodesic $\gamma$ will fill a subsurface $Y\subset X$ with two cusps. The subsurface is bounded by a simple closed geodesic $\eta$ with $X\setminus \eta=Y\cup S_{g,n-1}$.

\begin{lemma}\label{lemma type 4 double filling}
If $n=n(g)=o(\sqrt{g})\geq 2$, then for any $\delta>\frac{1}{2}$ and $\epsilon>0$, 
   $$\Egn\left[      1_{\mathcal{N}_\ell}\sum_{\substack{
\gamma\textit{ double-fills of type 4}\\
\ell_\gamma(X)\leq T
}}
H_{X,1}(\gamma)\right]
\prec n^2 \frac{Te^{(\frac{\epsilon}{2}+\frac{1}{4})T}+e^{(\delta-\frac{1}{4})T}}{g},
  $$
    where the implied constant depends on $\epsilon,\delta$ and is independent of $T,g$.
\end{lemma}

\begin{proof}
    Using the same argument as in the estimation \eqref{eq type 1-6} for double-filling terms, for $Y\simeq S_{0,3}$ and $\overline{Y}\simeq S_{0,3}$ with $X\setminus Y\simeq S_{g,n-1}$, if $\eta$ is the simple closed geodesic on $X$ that cuts off $Y$,
we have \begin{equation}\label{eq double fill type 4-1}
      1_{\mathcal{N}_\ell}\sum_{\substack{
\gamma\textit{ double-fills }Y\\
\ell_\gamma(X)\leq T
}}
H_{X,1}(\gamma)\prec B_{\epsilon,\delta,T}^1\left(\ell_{\eta}(X)\right),
\end{equation}
where $$B_{\epsilon,\delta,T}^1\left(x\right)
=\left\{\begin{aligned}
&(1+x)e^{\epsilon x}+e^{(\delta-\frac{1}{2})T}, &\textit{ if }&\quad  x\leq \frac{T}{2};\\
&Te^{\frac{T}{2}-(1-\epsilon)x},&\textit{ if }& \quad\frac{T}{2}<x\leq T;\\
&0&\textit{if }&\quad x> T.\\
\end{aligned}
\right.
$$ 
Taking \eqref{eq double fill type 4-1} over each mapping class group orbit of $\eta$ into Mirzakhani's Integration Formula Theorem \ref{thm mir int formula}, by Theorem \ref{mir07 poly}, Theorem \ref{thm mz15 asymp} and Lemma \ref{lemma NWX vgn(x)},
 we have \begin{align*}
&\Egn\left[ 1_{\mathcal{N}_\ell}\sum_{\substack{
\gamma\textit{ double-fills of type }4\\
\ell_\gamma(X)\leq T
}}
H_{X,1}(\gamma)\right]\\
\prec&\frac{{n\choose 2}}{V_{g,n}}\int_{\mathbb{R}_+}V_{0,3}(0,0,x)V_{g,n-1}(x,0^{n-2})B_{\epsilon,\delta,T}^1\left(x\right)x dx\\
\prec&\frac{n^2}{g}\int_{\mathbb{R}_+}B_{\epsilon,\delta,T}^1\left(x\right)\sinh\frac{x}{2}dx\\
\leq& \frac{n^2}{g}\int_{0}^{\frac{T}{2}}\left[(1+x)e^{\epsilon x}+e^{(\delta-\frac{1}{2})T}\right]\sinh\frac{x}{2}dx\\
+&\frac{n^2}{g}\int_{\frac{T}{2}}^\infty\left[Te^{\frac{T}{2}-(1-\epsilon)x}\right]\sinh\frac{x}{2}dx\\
\prec&\frac{n^2}{g}\left(Te^{\left(\frac{1}{4}+\frac{\epsilon}{2}\right)T} +e^{\left(\delta-\frac{1}{4}\right)T}\right),
 \end{align*}
 where the factor $n\choose 2$ originates from the possible choices of the two cusp boundaries of $Y$ among the $n$ boundary components, since the orbit of $\eta$ can be determined by the labeling of the two cusp components it bounds.
\end{proof}

Similar to geodesics of type $3$ that are not double-filling, the contribution to $\textrm{Int}_{ns}$ by geodesics of type $4$ that are not double-filling does not need to be precisely estimated. We also view a figure-eight closed geodesic as a one-sided iterated closed geodesic with iteration number $k=1$. Then we have the following estimates.

\begin{lemma}\label{lemma type 4 one sided iterated}
If $n=n(g)=o(\sqrt{g})\geq 2$, then for any $k\geq 1$, 
    \begin{align*}
    \Egn\left[      \sum_{\substack{
\gamma\textit{ one-sided iterated}\\
\textit{eight of type 4}\\
\textit{iteration number =k}\\
\ell_\gamma(X)\leq T
}}
H_{X,1}(\gamma)\right]\prec \frac{n^2T^2}{gk}\cdot \textbf{1}_{k\leq \frac{\cosh\frac{T}{2}-1}{2}},
    \end{align*}
    where the implied constant is independent of $g,k$, and $T$.
\end{lemma}
\begin{proof}
   In $\overline{Y}\simeq S_{0,3}$ separated by $\eta$ from $X$ with two cusp boundaries, there are $3$ different types of one-sided iterated eight closed geodesics.  The first kind will spiral around both cusp boundaries. The second kind will spiral around $\eta$ for $k$ times and spiral around one cusp boundary for $1$ time. The third kind will spiral around $\eta$ for $1$ time and spiral around one cusp boundary for $k$ times.  Notice that for $k=1$ the last two kinds of geodesics will coincide.
   They have lengths $L_k(0,0,\ell_{\eta}(X))$, $L_k(0,\ell_{\eta}(X),0)$ and $L_k(\ell_{\eta}(X),0,0)$, where $L_k$ is given by \eqref{eq int one-sided in s03  1}.
   Taking all mapping class group orbits of $\eta$ into Mirzakhani's Integration Formula Theorem \ref{thm mir int formula},  by Theorem \ref{mir07 poly}, Theorem \ref{thm mz15 asymp} and Lemma \ref{lemma NWX vgn(x)},   
 we have \begin{equation}\label{eq int one sided iterated type 4-1}
 \begin{aligned}
    &  \Egn\left[      \sum_{\substack{
\gamma\textit{ one-sided iterated}\\
\textit{eight of type 4}\\
\textit{iteration number =k}\\
\ell_\gamma(X)\leq T
}}
H_{X,1}(\gamma)\right]\\
\prec&\frac{{n\choose 2}}{V_{g,n}}\int_{\mathbb{R}_+}\!\left[\overline{H}_{X,1}(L_k(0,0,x))\!+\!\overline{H}_{X,1}(L_k(0,x,0))\!+\!\overline{H}_{X,1}(L_k(x,0,0)) \right]\\
   \cdot&V_{0,3}(0,0,x)V_{g-1,n+1}(x,0^{n-2})xdx\\
\prec&\frac{n^2}{g}\int_{\mathbb{R}_+}\left[\overline{H}_{X,1}(L_k(0,0,x))+\overline{H}_{X,1}(L_k(0,x,0))+\overline{H}_{X,1}(L_k(x,0,0)) \right]\\
\cdot&\sinh\frac{x}{2}dx,
 \end{aligned}
 \end{equation}
where the factor $n\choose 2$ originates from the possible choices of the two cusp boundaries among the $n$ boundary components. Notice that $\textbf{supp}(\overline{H}_{X,1})\subset [-T,T]$.

For terms of the first kind of one-sided
iterated eight closed geodesics, notice that $$
\cosh\frac{L_k(0,0,x)}{2}=k+1+k\cosh\frac{x}{2}.
$$
Under the change of variable $x\mapsto t=L_k(0,0,x)$,
we have $$
\sinh\frac{t}{2}dt=k\sinh\frac{x}{2}dx
$$
and $$t\geq 2\arccosh(2k+1).$$
It follows that  \begin{equation}\label{eq int one sided iterated type 4-2}
\begin{aligned}
    &\int_{\mathbb{R}_+}\overline{H}_{X,1}(L_k(0,0,x))\sinh\frac{x}{2}dx\\
    \prec&\int_{L_x(0,0,x)\leq T}\frac{L_x(0,0,x)}{\sinh\frac{L_x(0,0,x)}{2}}\sinh\frac{x}{2}dx\\
    =&\int_{2\arccosh(2k+1)}^T \frac{t}{\sinh\frac{t}{2}}\cdot \frac{1}{k}\sinh\frac{t}{2}dt \cdot \textbf{1}_{k\leq \frac{\cosh\frac{T}{2}-1}{2}}\\
    \prec & \frac{T^2}{k}\textbf{1}_{k\leq \frac{\cosh\frac{T}{2}-1}{2}}.
\end{aligned}
\end{equation}
For terms of the second kind of one-sided
iterated eight closed geodesics, we have 
$$
L_k(0,x,0)\geq \max\{kx,2\arccosh(2k+1)\}.
$$
It follows that \begin{equation}\label{eq int one sided iterated type 4-3}
\begin{aligned}
    &\int_{\mathbb{R}_+}\overline{H}_{X,1}(L_k(0,x,0))\sinh\frac{x}{2}dx\\
   \prec&\int_{kx\leq L_k(0,x,0)\leq T}\frac{L_k(0,x,0)}{\sinh\frac{L_k(0,x,0)}{2}}\sinh\frac{x}{2}dx\cdot \textbf{1}_{k\leq \frac{\cosh\frac{T}{2}-1}{2}}\\
   \prec&\int_{kx\leq T}\frac{kx}{\sinh\frac{k}{2}x}\sinh\frac{x}{2}dx\cdot \textbf{1}_{k\leq \frac{\cosh\frac{T}{2}-1}{2}}\\
    \prec&\int_{kx\leq T} xdx   \cdot \textbf{1}_{k\leq \frac{\cosh\frac{T}{2}-1}{2}} \\
    \prec &\frac{T^2}{k^2}\cdot \textbf{1}_{k\leq \frac{\cosh\frac{T}{2}-1}{2}}.
\end{aligned}
\end{equation}
Terms of the third kind of one-sided iterated eight closed geodesics can be estimated by the same method as the first kind.
We  have $$\cosh\frac{L_k(x,0,0)}{2}=(k+1)\cosh\frac{x}{2}+k.$$
Under the change of variable $x\mapsto t=L_k(x,0,0)$ we have $$
\sinh\frac{t}{2}dt=(k+1)\sinh\frac{x}{2}dx
$$
and $$
t\geq 2\arccosh(2k+1).
$$
It follows that  \begin{equation}\label{eq int one sided iterated type 4-4}
\begin{aligned}
    &\int_{\mathbb{R}_+}\overline{H}_{X,1}(L_k(x,0,0))\sinh\frac{x}{2}dx\\
 \prec&\int_{L_k(x,0,0)\leq T}\frac{L_x(x,0,0)}{\sinh\frac{L_x(x,0,0)}{2}}\sinh\frac{x}{2}dx\\
    =&\int_{2\arccosh(2k+1)}^T \frac{t}{\sinh\frac{t}{2}}\cdot \frac{1}{k+1}\sinh\frac{t}{2}dt \cdot \textbf{1}_{k\leq \frac{\cosh\frac{T}{2}-1}{2}}\\
   \prec & \frac{T^2}{k+1}\textbf{1}_{k\leq \frac{\cosh\frac{T}{2}-1}{2}}.
\end{aligned}
\end{equation}
The lemma follows from \eqref{eq int one sided iterated type 4-1}, \eqref{eq int one sided iterated type 4-2}, \eqref{eq int one sided iterated type 4-3} and \eqref{eq int one sided iterated type 4-4}.
\end{proof}
Now we can bound the total contribution of filling geodesics of type $4$ to $\textrm{Int}_{ns}$ in \eqref{eq ineq on subset T}. 
\begin{lemma}\label{lemma total type 4}
If $n=n(g)=o(\sqrt{g})\geq 2$, then for any $\delta>\frac{1}{2}$ and $\epsilon>0$, 
   $$\Egn\left[      1_{\mathcal{N}_\ell}\sum_{\substack{
\gamma\textit{ of type 4}\\
\ell_\gamma(X)\leq T
}}
H_{X,1}(\gamma)\right]
\prec n^2 \frac{Te^{(\frac{\epsilon}{2}+\frac{1}{4})T}+e^{(\delta-\frac{1}{4})T}}{g},
  $$
    where the implied constant depends on $\epsilon,\delta$ and is independent of $T,g$.
\end{lemma}
\begin{proof}
    By Lemma \ref{lemma type 4 double filling} and Lemma \ref{lemma type 4 one sided iterated}, we have \begin{align*}
        &\Egn\left[      1_{\mathcal{N}_\ell}\sum_{\substack{
\gamma\textit{ of type 4}\\
\ell_\gamma(X)\leq T
}}
H_{X,1}(\gamma)\right]\\
\prec&n^2 \frac{Te^{(\frac{\epsilon}{2}+\frac{1}{4})T}+e^{(\delta-\frac{1}{4})T}}{g}+\frac{n^2T^2}{g}\sum_{k=2}^{\left[\frac{\cosh\frac{T}{2}-1}{2}\right]}\frac{1}{k}\\
\prec&n^2 \frac{Te^{(\frac{\epsilon}{2}+\frac{1}{4})T}+e^{(\delta-\frac{1}{4})T}}{g}+\frac{n^2T^2}{g}\log\left(\frac{\cosh\frac{T}{2}-1}{2}\right)\\
\prec&n^2 \frac{Te^{(\frac{\epsilon}{2}+\frac{1}{4})T}+e^{(\delta-\frac{1}{4})T}}{g}+\frac{n^2T^3}{g}\\
\prec&n^2 \frac{Te^{(\frac{\epsilon}{2}+\frac{1}{4})T}+e^{(\delta-\frac{1}{4})T}}{g},
    \end{align*}
    which ends the proof.
\end{proof}

\subsection{Type $5$ in $\mathrm{Int}_{ns}$}
In this type, the nonsimple closed geodesic $\gamma$ will fill a subsurface $Y\subset X$ with one cusp.
It is bounded by two simple closed geodesics $\eta_1\cup\eta_2$ with $X\setminus \{\eta_1\cup\eta_2\}=Y\cup S_{g-1,n+1}$. We hereby remind readers in advance that in the problems we are considering, unlike type $3$, the total contribution of terms of geodesics of type $4$ is not a minor term. 

\begin{lemma}\label{lemma double fill type 5}
If $n=n(g)=o(\sqrt{g})\geq 1$, then for any $\delta>\frac{1}{2}$ and $\epsilon>0$, 
   $$\Egn\left[      1_{\mathcal{N}_\ell}\sum_{\substack{
\gamma\textit{ double-fills of type 5}\\
\ell_\gamma(X)\leq T
}}
H_{X,1}(\gamma)\right]
\prec n \frac{T^2e^{\left(\frac{1}{4}+\frac{\epsilon}{2}\right)T} +Te^{\left(\delta-\frac{1}{4}\right)T}}{g},
  $$
    where the implied constant depends on $\epsilon,\delta$ and is independent of $T,g$.
\end{lemma}

\begin{proof}
    Using the same argument as for the estimation \eqref{eq type 1-6} for double-filling terms, for $Y\simeq S_{0,3}$ and $\overline{Y}\simeq S_{0,3}$ with $X\setminus Y\simeq S_{g-1,n+1}$, if $\eta_1,\eta_2$ are the two simple closed geodesics on $X$ that cuts off $Y$,
we have \begin{equation}\label{eq double fill type 5-1}
      1_{\mathcal{N}_\ell}\sum_{\substack{
\gamma\textit{ double- fills }Y\\
\ell_\gamma(X)\leq T
}}
H_{X,1}(\gamma)\prec B_{\epsilon,\delta,T}^1\left(\ell_{\eta_1}(X)+\ell_{\eta_2}(X)\right),
\end{equation}
where $$B_{\epsilon,\delta,T}^1\left(x\right)
=\left\{\begin{aligned}
&(1+x)e^{\epsilon x}+e^{(\delta-\frac{1}{2})T}, &\textit{ if }&\quad  x\leq \frac{T}{2};\\
&Te^{\frac{T}{2}-(1-\epsilon)x},&\textit{ if }& \quad\frac{T}{2}<x\leq T;\\
&0&\textit{if }&\quad x> T.\\
\end{aligned}
\right.
$$ 

Taking \eqref{eq double fill type 5-1} over all mapping class group orbits of $\eta_1\cup\eta_2$ into Mirzakhani's Integration Formula Theorem \ref{thm mir int formula}, by Theorem \ref{mir07 poly}, Theorem \ref{thm mz15 asymp} and Lemma \ref{lemma NWX vgn(x)},
 we have \begin{align*}
&\Egn\left[ 1_{\mathcal{N}_\ell}\sum_{\substack{
\gamma\textit{ double-fills of type }5\\
\ell_\gamma(X)\leq T
}}
H_{X,1}(\gamma)\right]\\
\prec&\frac{{n\choose 1}}{V_{g,n}}\int_{\mathbb{R}_+^2}V_{0,3}(0,x,y)V_{g-1,n+1}(x,y,0^{n-1})B_{\epsilon,\delta,T}^1\left(x+y\right)xy dxdy\\
\prec&\frac{n}{g}\int_{\mathbb{R}_+^2}B_{\epsilon,\delta,T}^1\left(x+y\right)\sinh\frac{x}{2}\sinh\frac{y}{2}dxdy\\
\leq& \frac{n}{g}\int_{x+y\leq \frac{T}{2}}\left[(1+x+y)e^{\epsilon (x+y)}+e^{(\delta-\frac{1}{2})T}\right]\sinh\frac{x}{2}\sinh\frac{y}{2}dxdy\\
+&\frac{n}{g}\int_{\frac{T}{2}\leq x+y\leq T}\left[Te^{\frac{T}{2}-(1-\epsilon)(x+y)}\right]\sinh\frac{x}{2}\sinh\frac{y}{2}dxdy\\
\prec&\frac{n}{g}\left(T^2e^{\left(\frac{1}{4}+\frac{\epsilon}{2}\right)T} +Te^{\left(\delta-\frac{1}{4}\right)T}\right),
 \end{align*}
 where the factor ${n\choose 1}$ originates from the possible choices of the cusp boundary for $Y$ among the $n$ boundary components of $X$.
\end{proof}

\begin{lemma}\label{lemma figure eight type 5}
    If $n=n(g)=o(\sqrt{g})\geq 1$, then
   \begin{align*}
&\Egn\left[      \sum_{\substack{
\gamma\textit{ figure-eight of type }5,  \\
\ell_\gamma(X)\leq T
}}
H_{X,1}(\gamma)\right]\\
=&\frac{n}{\pi^2g}\left(\int_{0}^\infty f_T(t)e^{\frac{t}{2}}\cdot \frac{t}{4} dt+O\left(T^3\right)\right)\left(1+O\left(\frac{nT^2}{g}\right)\right),
\end{align*}
    where the implied constant is uniform on $g$ and $T$.
\end{lemma}
\begin{proof}
In $\overline{Y}\simeq S_{0,3}$ separated by $\eta_1\cup \eta_2$ from $X$ with one cusp boundary, two different kinds of figure-eight closed geodesics exist. One will spiral around both $\eta_1$ and $\eta_2$, and the other will spiral around the cusp boundary and one of $\{\eta_1,\eta_2\}$. They have lengths $L_{f-8}(\ell_{\eta_1}(X),\ell_{\eta_2}(X),0)$ and $L_{f-8}(0,\ell_{\eta_i}(X),\ell_{\eta_{3-i}}(X))$, where $L_{f-8}$ is given by the formula \eqref{eq int f-8 in s03  1}.

  By Lemma \ref{lemma NWX vgn(x)} and Theorem \ref{thm vgn/vgn+1 n^2+1/g^2}(or Lemma \ref{lemma weak vgn}), for $n=o(\sqrt{g})$ we have
\begin{equation}\label{eq frac v03 vg-1n+1 1cusp}
\begin{aligned}
    &\frac{V_{0,3}(x,y,0)V_{g-1,n+1}(x,y,0^{n-1})}{V_{g,n}}\\
 =&\frac{1}{2\pi^2g}\frac{\sinh\frac{x}{2}}{x}\frac{\sinh\frac{y}{2}}{y}\left(1+O\left(\frac{n(1+x^2+y^2)}{g} \right)\right).
\end{aligned}
\end{equation}
 Taking all mapping class group orbits of $\eta_1\cup\eta_2$ into Mirzakhani's Integration Formula Theorem \ref{thm mir int formula}, we have
 \begin{equation}\label{eq int figure-eight type 5-1}
 \begin{aligned}
 &\Egn\left[     \sum_{\substack{
\gamma\textit{ figure-eight of type }5,  \\
\ell_\gamma(X)\leq T
}}
H_{X,1}(\gamma)\right]\\
=&2\cdot\frac{1}{2}\cdot\frac{{n\choose 1}}{V_{g,n}}\int_{\mathbb{R}_+^2}\left[\overline{H}_{X,1}(L_{f-8}(x,0,y))+\overline{H}_{X,1}(L_{f-8}(y,0,x))\right.\\
+&\left.\overline{H}_{X,1}(L_{f-8}(x,y,0))\right]\cdot V_{0,3}(x,y,0)V_{g-1,n+1}(x,y,0^{n-1})xydxdy\\
=&\frac{n}{2\pi^2g}\int_{\mathbb{R}_+^2} \left[ 2\overline{H}_{X,1}(L_{f-8}(x,0,y))+\overline{H}_{X,1}(L_{f-8}(x,y,0))\right] \\
\cdot&\sinh\frac{x}{2}\sinh\frac{y}{2}\left(1+O\left(\frac{n(1+x^2+y^2)}{g}\right)\right)dxdy\\
=&\frac{n}{2\pi^2g}\int_{\mathbb{R}_+^2} \left[ 2\overline{H}_{X,1}(L_{f-8}(x,0,y))+\overline{H}_{X,1}(L_{f-8}(x,y,0))\right] \\
\cdot&\sinh\frac{x}{2}\sinh\frac{y}{2}dxdy\left(1+O\left(\frac{nT^2}{g}\right)\right).
\end{aligned}
\end{equation}
Here on the second line, the factor $\frac{1}{2}$ offsets the multiplicity of the ordering of $\{\gamma_1,\gamma_2\}$, and the factor $2$ comes from the two orientations of any $\gamma$. On the last line, we use the fact $x,y,z\leq L_{f-8}(x,y,z)$ and $\textbf{supp}(f_T)=[-T,T]$. 

Since $L_{f-8}(x,y,z)\geq \max\{x+y,z\}$, we have 
 \begin{equation}\label{eq int figure-eight type 5-2}
 \begin{aligned}
 &\int_{\mathbb{R}_+^2}\overline{H}_{X,1}(L_{f-8}(x,y,0))\sinh\frac{x}{2}\sinh\frac{y}{2}dxdy\\
\prec&\int_{x+y\leq L_{f-8}(x,y,0)\leq T}\frac{L_{f-8}(x,y,0)}{\sinh\frac{L_{f-8}(x,y,0)}{2}}\sinh\frac{x}{2}\sinh\frac{y}{2}dxdy\\
\prec&\int_{x+y\leq T}\frac{x+y}{\sinh\frac{x+y}{2}}\sinh\frac{x}{2}\sinh\frac{y}{2}dxdy\\
\prec& \int_{x+y\leq T} (x+y)dxdy \\
\prec& T^3.
 \end{aligned}
 \end{equation}
If
we change the variables $(x,y)\mapsto (x,t)$ with $t=L_{f-8}(x,0,y)$
by $$
\cosh\frac{L_{f-8}(x,0,y)}{2}=2\cosh\frac{x}{2}+\cosh\frac{y}{2},
$$
we will have $$
\sinh\frac{t}{2}dt=2\sinh\frac{x}{2}dx+\sinh\frac{y}{2}dy,
$$
and $$
2\cosh\frac{x}{2}\leq \cosh\frac{t}{2}-1.
$$
Then \begin{equation}\label{eq int figure-eight type 5-3}
\begin{aligned}
&\int_{\mathbb{R}_+^2}\overline{H}_{X,1}(L_{f-8}(x,0,y))\sinh\frac{x}{2}\sinh\frac{y}{2}dxdy\\
=&\int_{t=2\arccosh 3}^T\int_{2\cosh\frac{x}{2}\leq \cosh\frac{t}{2}-1}\frac{t}{2\sinh\frac{t}{2}}f_T(t)\sinh\frac{x}{2}\sinh\frac{t}{2}dxdt\\
=&\int_{t=2\arccosh 3}^T \frac{t}{2} \left(\cosh\frac{t}{2}-3\right)f_T(t)dt\\
=&\int_0^\infty f_T(t) e^{\frac{t}{2}}\cdot\frac{t}{4}\cdot dt+O\left(T^2\right).
\end{aligned}
\end{equation}
Taking \eqref{eq int figure-eight type 5-2} and \eqref{eq int figure-eight type 5-3} into \eqref{eq int figure-eight type 5-1}, we will get the lemma.
\end{proof}

    \begin{lemma}\label{lemma one sided type 5}
If $n=n(g)=o(\sqrt{g})$, then for any $k\geq 2$,
    \begin{align*}
   & \Egn\left[      \sum_{\substack{
\gamma\textit{ one-sided iterated}\\
\textit{eight of type 5}\\
\textit{iteration number =k}\\
\ell_\gamma(X)\leq T
}}
H_{X,1}(\gamma)\right]\\
= &\frac{n}{2\pi^2g}\left(\int_0^\infty\frac{tf_T(t)e^{\frac{t}{2}}}{k(k+1)}dt+O\left(\frac{T^3}{k}\right)\right)\left(1+O\left(\frac{nT^2}{g}\right)\right)\cdot  \textbf{1}_{k\leq \frac{\cosh\frac{T}{2}-1}{2}},
    \end{align*}
    where the implied constant is independent of $g,k,T$.
\end{lemma}
\begin{proof}
    For any $k\geq 2$, in $\overline{Y}\simeq S_{0,3}$ separated by $\eta_1\cup \eta_2$ from $X$ with one cusp boundary, three different kinds of one-sided iterated eight closed geodesics with iteration number $k$ exist. The first kind will spiral around one of $\eta_i$ for $k$ times and spiral around $\eta_{3-i}$ for $1$ times. The second kind will spiral around one of $\eta_i$ for $k$ times and spiral around the cusp boundary for $1$ time.  The third kind will spiral around one of $\eta_i$ for $1$ time and spiral around the cusp boundary for $k$ times. They have lengths $L_k(\ell_{\eta_{3-i}}(X),\ell_{\eta_i}(X),0)$, $L_k(0,\ell_{\eta_{i}}(X),\ell_{\eta_{3-i}}(X))$ and $L_k(\ell_{\eta_{i}}(X),0,\ell_{\eta_{3-i}}(X))$, where $L_k$ is given by the formula \eqref{eq int one-sided in s03  1}.

 Taking all mapping class group orbits of $\eta_1\cup\eta_2$ into Mirzakhani Integration Formula Theorem \ref{thm mir int formula}, by \eqref{eq frac v03 vg-1n+1 1cusp}, we have \begin{equation}\label{eq int one-sided type 5-1}
 \begin{aligned}
      &\Egn\left[      \sum_{\substack{
\gamma\textit{ one-sided iterated}\\
\textit{eight of type 5}\\
\textit{iteration number =k}\\
\ell_\gamma(X)\leq T
}}
H_{X,1}(\gamma)\right]\\
= &2\cdot \frac{1}{2}\cdot\frac{{n\choose 1}}{V_{g,n}}\int_{\mathbb{R}_+^2}\left[  \overline{H}_{X,1}(L_k(x,y,0))+\overline{H}_{X,1}(L_k(y,x,0))\right.\\
+&\overline{H}_{X,1}(L_k(x,0,y))+\overline{H}_{X,1}(L_k(y,0,x))+\overline{H}_{X,1}(L_k(0,x,y))\\
+&\left.\overline{H}_{X,1}(L_k(0,y,x))\right]\cdot V_{0,3}(x,y,0)V_{g-1,n+1}(x,y,0^{n-1})xydxdy\\
=&\frac{n}{\pi^2g}\int_{\mathbb{R}_+^2}\left[ \overline{H}_{X,1}(L_k(x,y,0))+ \overline{H}_{X,1}(L_k(x,0,y))\right.\\
+&\left. \overline{H}_{X,1}(L_k(0,x,y))\right]\sinh\frac{x}{2}\sinh\frac{y}{2}dxdy\left(1+O\left(\frac{nT^2}{g}\right)\right).
 \end{aligned}
 \end{equation}
 Here on the second line, the factor $\frac{1}{2}$ offsets the multiplicity of the ordering
of $\{\eta_1,\eta_2\}$, and the factor $2$ comes from the two orientations of any $\gamma$. On the last line, we use the fact $x,y,z\leq L_{k}(x,y,z)$ and $\textbf{supp}(f_T)=[-T,T]$.
 
 Since $L_k(x,y,0)\geq x+ky$, we have \begin{equation}\label{eq int one-sided type 5-2}
 \begin{aligned}
     &\int_{\mathbb{R}_+^2}\overline{H}_{X,1}(L_k(x,y,0))\sinh\frac{x}{2}\sinh\frac{y}{2}dxdy\\
 \prec&\int_{x+ky\leq L_k(x,y,0)\leq T }\frac{L_k(x,y,0)}{\sinh\frac{L_k(x,y,0)}{2}}\sinh\frac{x}{2}\sinh\frac{y}{2}dxdy\\
     \prec&\int_{x+ky\leq L_k(x,y,0)\leq T }\frac{x+ky}{\sinh\frac{x+ky}{2}}\sinh\frac{x}{2}\frac{\sinh\frac{ky}{2}}{k}dxdy\\
     \prec&\frac{1}{k}\int_{x+ky\leq T}(x+ky)dxdy\\
    \prec&\frac{T^3}{k^2}. 
 \end{aligned}
 \end{equation}
If we use the change of variable $(x,y)\mapsto(x,t)$ with $t=L_k(0,x,y)$ by $$
\cosh\frac{L_k(0,x,y)}{2}=\frac{\sinh\frac{k+1}{2}x}{\sinh\frac{x}{2}}+\frac{\sinh\frac{k}{2}x}{\sinh\frac{x}{2}}\cosh\frac{y}{2},
$$
we will have $$
\sinh\frac{t}{2}dt=*dx+\frac{\sinh\frac{k}{2}x}{\sinh\frac{x}{2}}\sinh\frac{y}{2}dy,
$$
 where $*$ is some function on $x,y$. Then \begin{equation}\label{eq int one-sided type 5-3}
 \begin{aligned}
      &\int_{\mathbb{R}_+^2}\overline{H}_{X,1}(L_k(0,x,y))\sinh\frac{x}{2}\sinh\frac{y}{2}dxdy\\
      \prec&\int_{L_k(0,x,y)\leq T }\frac{L_k(0,x,y)}{\sinh\frac{L_k(0,x,y)}{2}}\sinh\frac{x}{2}\sinh\frac{y}{2}dxdy\\
      \prec&\int_{t=0}^T\int_{kx\leq t}t\frac{\sinh^2\frac{x}{2}}{\sinh\frac{k}{2}x}dxdt.
 \end{aligned}
 \end{equation}
For $k=2$ we have $$
\int_{kx\leq t}\frac{\sinh^2\frac{x}{2}}{\sinh\frac{k}{2}x}dx\leq t.
$$
For $k\geq 3$ we have \begin{align*}
   & \int_{kx\leq t}\frac{\sinh^2\frac{x}{2}}{\sinh\frac{k}{2}x}dx\\
   \leq& \int_{0}^{\frac{1}{k}}\frac{x}{k}dx+\int_{\frac{1}{k}}^1 x^2 e^{-\frac{k}{2}x}dx+\int_1^\infty e^{(1-\frac{k}{2})x}dx\\
   \prec &\frac{1}{k^3}.
\end{align*}
So by \eqref{eq int one-sided type 5-3} we have \begin{equation}\label{eq int one-sided type 5-4}
\begin{aligned}
    &\int_{\mathbb{R}_+^2}\overline{H}_{X,1}(L_k(0,x,y))\sinh\frac{x}{2}\sinh\frac{y}{2}dxdy\\
    \prec& \textbf{1}_{k\geq 3}\frac{T^2}{k^3}+\textbf{1}_{k=2}T^3\\
    \prec&\frac{T^3}{k^3}.
\end{aligned}
\end{equation}
If we use the change of variable $(x,y)\mapsto(y,t)$ with $t=L_k(x,0,y)$ by $$
\cosh\frac{L_k(x,0,y)}{2}=(k+1)\cosh\frac{x}{2}+k\cosh\frac{y}{2},
$$
we will have $$
\sinh\frac{t}{2}dt=(k+1)\sinh\frac{x}{2}dx+k\sinh\frac{y}{2}dy
$$
and $$
\cosh\frac{t}{2}\geq k+1+k\cosh\frac{y}{2}.
$$
It follows that \begin{equation}\label{eq int one-sided type 5-5}
\begin{aligned}
     &\int_{\mathbb{R}_+^2}\overline{H}_{X,1}(L_k(x,0,y))\sinh\frac{x}{2}\sinh\frac{y}{2}dxdy\\
     =&\int_{t=2\arccosh(2k+1)}^T \frac{t}{2(k+1)}f_T(t) \left(\int_{\cosh\frac{t}{2}\geq k+1+k\cosh\frac{y}{2}} \sinh\frac{y}{2}dy\right)dt\\
   =&\int_{2\arccosh(2k+1)}^T\frac{t}{k+1}f_T(t)\frac{\cosh\frac{t}{2}-2k-1}{k}dt\cdot \textbf{1}_{k\leq \frac{\cosh\frac{T}{2}-1}{2}}\\
     =&\int_{2\arccosh(2k+1)}^T\frac{1}{k(k+1)}tf_T(t)\cosh\frac{t}{2}dt \cdot \textbf{1}_{k\leq \frac{\cosh\frac{T}{2}-1}{2}}\\
     +& O\left(\frac{T^2}{k}\right)\cdot \textbf{1}_{k\leq \frac{\cosh\frac{T}{2}-1}{2}}\\
     =&\int_{0}^T\frac{1}{k(k+1)}tf_T(t)\cosh\frac{t}{2}dt \cdot \textbf{1}_{k\leq \frac{\cosh\frac{T}{2}-1}{2}}+O\left(\frac{T^2}{k}\right)\cdot \textbf{1}_{k\leq \frac{\cosh\frac{T}{2}-1}{2}}.
\end{aligned}
\end{equation}
The lemma follows from \eqref{eq int one-sided type 5-1}, \eqref{eq int one-sided type 5-2}, \eqref{eq int one-sided type 5-4} and \eqref{eq int one-sided type 5-5}.
\end{proof}

Now we can compute the total contribution of filling geodesics of type $5$ to $\textrm{Int}_{ns}$ in \eqref{eq ineq on subset T}. It completely offsets the part containing $n$ in the main term in Lemma \ref{lemma nsep zero eigen}.
\begin{lemma}\label{lemma total type 5}
If $n=n(g)=o(\sqrt{g})\geq 1$, then for any $\delta>\frac{1}{2}$ and $\epsilon>0$,  \begin{align*}
&\Egn\left[      1_{\mathcal{N}_\ell}\sum_{\substack{
\gamma\textit{ of type 5}\\
\ell_\gamma(X)\leq T
}}
H_{X,1}(\gamma)\right]\\
\leq& \frac{n}{2\pi^2g}\int_0^\infty t f_T(t)e^{\frac{t}{2}}dt+O\left(n \frac{T^2e^{\left(\frac{1}{4}+\frac{\epsilon}{2}\right)T} +Te^{\left(\delta-\frac{1}{4}\right)T}}{g}+\frac{n^2T^3e^{\frac{T}{2}}}{g^2}\right),
\end{align*}
    where the implied constant depends on $\epsilon,\delta$ and is independent of $T,g$.
\end{lemma}
\begin{proof}
    By lemma \ref{lemma double fill type 5}, Lemma \ref{lemma figure eight type 5} and Lemma \ref{lemma one sided type 5}, we have \begin{align*}
&\Egn\left[      1_{\mathcal{N}_\ell}\sum_{\substack{
\gamma\textit{ of type 5}\\
\ell_\gamma(X)\leq T
}}
H_{X,1}(\gamma)\right]\\
 \leq&O\left(  n \frac{T^2e^{\left(\frac{1}{4}+\frac{\epsilon}{2}\right)T} +Te^{\left(\delta-\frac{1}{4}\right)T}}{g}   \right)+\frac{n}{\pi^2 g}\left(1+O\left(\frac{nT^2}{g}\right)\right)\\
 \cdot&\left[\int_{0}^T tf_T(t)e^{\frac{t}{2}}dt\cdot \left(\frac{1}{4}+\sum_{k=2}^{\left[\frac{\cosh\frac{T}{2}-1}{2}\right]}\frac{1}{2k(k+1)}\right)+O\left(\sum_{k=1}^{\left[\frac{\cosh\frac{T}{2}-1}{2}\right]} \frac{T^3}{k}\right)\right]  \\
 =&O\left(n \frac{T^2e^{\left(\frac{1}{4}+\frac{\epsilon}{2}\right)T} +Te^{\left(\delta-\frac{1}{4}\right)T}}{g} \right)  +\frac{n}{\pi^2 g}\left(1+O\left(\frac{nT^2}{g}\right)\right)\\
 \cdot&\left[\int_{0}^T tf_T(t)e^{\frac{t}{2}}dt\cdot \left(\frac{1}{2}+O\left(\frac{1}{\cosh\frac{T}{2}}\right)\right)+O\left(T^3\log \left(\frac{\cosh\frac{T}{2}-1}{2}\right)\right)\right]  \\
 =&\frac{n}{2\pi^2g}\int_0^\infty t f_T(t)e^{\frac{t}{2}}dt+O\left(n \frac{T^2e^{\left(\frac{1}{4}+\frac{\epsilon}{2}\right)T} +Te^{\left(\delta-\frac{1}{4}\right)T}}{g}+\frac{n^2T^3e^{\frac{T}{2}}}{g^2}\right),
    \end{align*}
    which completes the proof.
\end{proof}

\subsection{Type $6$ in $\mathrm{Int}_{ns}$}
For non-simple geodesics of type $6$, the counting functions in Theorem \ref{thm filling count} and Lemma \ref{lemma uniform L+6} are sufficient to give the following proposition. \begin{proposition}\label{prop wx prop 30}
\cite[Proposition 30]{wx22-3/16}
    For any $\epsilon>0$, as $g\to \infty$, 
\begin{align*}
  \sum_{\substack{
\gamma \textit{ of type } 6\\
\ell_{\gamma}(X)\leq T
    }}H_{X,1}(\gamma) \prec& T^2e^T  \sum_{\substack{
 Y\textit{ subsurface of } X\\
56\leq  \left|\chi(Y)\right|\leq \frac{2T}{\pi}
 }}   e^{-\frac{1}{4}\ell(\partial Y)}\textbf{1}_{[0,2T]}(\ell(\partial Y))\\
 +& T\sum_{\substack{
 Y\textit{ subsurface of } X,\\
 2\leq  \left|\chi(Y)\right|\leq 55,\textit{ or}\\
\chi(X)=-1, \textit{and}\\
Y\setminus X \textit{ is disconnected}
 }}e^{\frac{T}{2}-\frac{1-\epsilon}{2}\ell(\partial Y)  }\textbf{1}_{[0,2T]}(\ell(\partial Y)).
\end{align*}
\end{proposition}

\begin{lemma}\label{lemma type 6 k geq 2}
If $n=o(\sqrt{g})$, then for any $\epsilon>0$,  \begin{align*}
    &\Egn\left[T^2e^T  \sum_{\substack{
 Y\textit{ subsurface of } X\\
56\leq  \left|\chi(Y)\right|\leq \frac{2T}{\pi}
 }}   e^{-\frac{1}{4}\ell(\partial Y)}\textbf{1}_{[0,2T]}(\ell(\partial Y))\right]\\
 +&\Egn\left[ T\sum_{\substack{
 Y\textit{ subsurface of } X\\
 2\leq  \left|\chi(Y)\right|\leq 55
 }}e^{\frac{T}{2}-\frac{1-\epsilon}{2}\ell(\partial Y)  }\textbf{1}_{[0,2T]}(\ell(\partial Y))\right]\\
 \prec &\frac{T^6e^{\frac{9}{2}T}}{g^{27}}+T^{230}e^{(\frac{1}{2}+\epsilon)T}\frac{1+n^3}{g^2}.
\end{align*}
\end{lemma}\begin{proof}
We classify the possible orbits of $Y$. Assume $X\setminus Y$ has $q$ connected components of type $S_{g_i,n_i+a_i}$ for $i=1,\cdots,q$, where $a_i$ is the number of cusps and $n_i$ is the number of boundary geodesics. Assume that $Y$ is of type $S_{g_0,k+a_0}$, where $a_0$ is the number of cusps and $k$ is the number of boundary geodesics, with $n_0\geq 0$ pairs of boundary geodesics coinciding on $X$. See Figure \ref{figure: subsurface chi geq 2 in X} for an illustration. Then we have \begin{enumerate}
    \item $\sum_{i=1}^q (2g_i+n_i+a_i-2)=2g+n-2-\left|\chi(Y)\right| $;
    \item $\sum_{i=1}^q n_i=k-2n_0$;
    \item $\sum_{i=1}^q a_i=n-a_0$.
\end{enumerate}
\begin{figure}[h]
\centering
\includegraphics[scale=0.12]{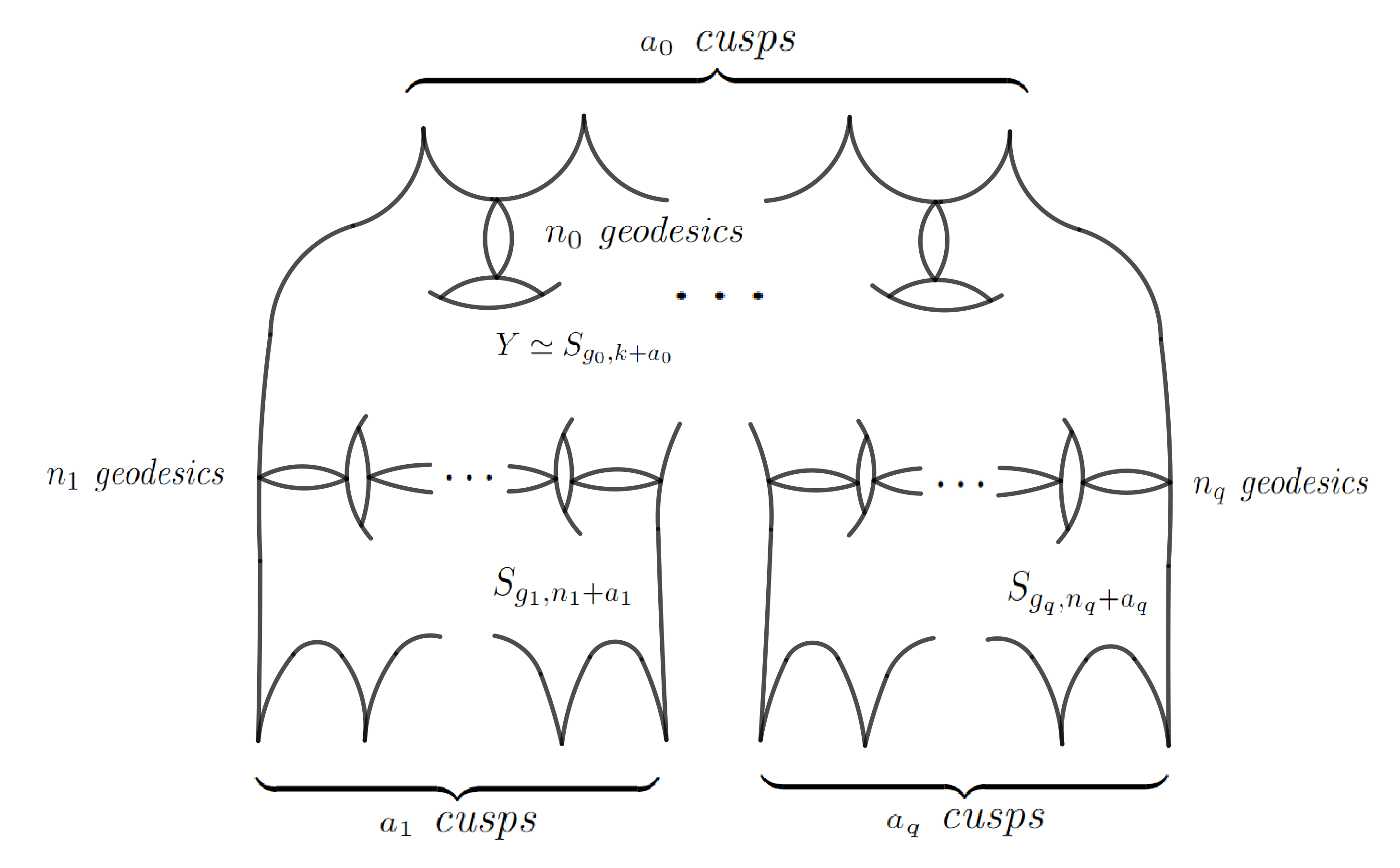}
\caption{$Y\subset X$ with $\left|\chi(Y)\right|\geq 2$}
\label{figure: subsurface chi geq 2 in X}
\end{figure}
We denote $\mathcal{A}$ to be the index set of all possible $\{(g_i,n_i,a_i)\}_{i=1}^q$ for fixed $g_0,k,a,n_0,q$. For $56\leq 2g_0+a_0+k-2\leq \frac{2T}{\pi}$, by the same argument as in \cite[Proposition 31]{wx22-3/16},
we have \begin{equation}\label{int type 6 geq k+1 eq 1}
\begin{aligned}
    &\Egn\left[\sum_{\substack{
    Y\simeq S_{g_0,k+a_0}\subset X\\
    \textit{with } a_0 \textit{ cusps}\\
    }}e^{-\frac{1}{4}\ell(\partial Y)} \textbf{1}_{[0,2T]}(\ell(\partial Y))
    \right]\\
   \prec &e^{\frac{7}{2}T}\sum_{n_0=0}^{[\frac{k-1}{2}]}\sum_{q=1}^{k-2n_0}\sum_{\{(g_i,n_i,a_i)\}_{i=1}^q\in\mathcal{A}}{n\choose a_0,\cdots,a_q}\frac{V_{g_0,a_0+k}\prod_{i=1}^qV_{g_i,n_i+a_i}}{V_{g,n}n_0!n_1!\cdots n_q!}.
\end{aligned}
\end{equation}
Since $q\prec T=O(\log g)$, we can apply Lemma \ref{lemma admissible triple in hide} to get \begin{equation}\label{int type 6 geq k+1 eq 2}
\begin{aligned}
    &\Egn\left[T^2 e^T\sum_{\substack{
    Y\simeq S_{g_0,k+a_0}\subset X\\
    \textit{with } a_0 \textit{ cusps}\\
    }}e^{-\frac{1}{4}\ell(\partial Y)} \textbf{1}_{[0,2T]}(\ell(\partial Y))
    \right]\\
    \prec&\sum_{n_0=0}^{[\frac{k-1}{2}]}\sum_{q=1}^{k-2n_0} T^2 e^{\frac{9}{2}T}(2g_0+k+a_0-3)!\frac{1+n^{a_0}}{g^{2g_0+a_0+k-2}}\\
    \prec& T^2 k^2 e^{\frac{9}{2}T}(2g_0+k+a_0-3)!\frac{1+n^{a_0}}{g^{2g_0+a_0+k-2}}.
\end{aligned}
\end{equation}
Summing \eqref{int type 6 geq k+1 eq 2} over possible $(g_0,a_0,k)$, we have \begin{equation}\label{int type 6 geq k+1 eq 3}
\begin{aligned}
    &\Egn\left[T^2 e^T\sum_{\substack{
   56\leq |\chi(Y)|\leq \frac{2T}{\pi}
    }}e^{-\frac{1}{4}\ell(\partial Y)} \textbf{1}_{[0,2T]}(\ell(\partial Y))
    \right]\\
   \prec&T^2 e^{\frac{9}{2}T}\sum_{a_0\leq \frac{2T}{\pi}+1}\sum_{1\leq k\leq \frac{2T}{\pi}+2-a_0}\\
    \cdot&\sum_{56\leq 2g_0+k+a_0-2\leq \frac{2T}{\pi}} k^2 \frac{(1+n^{a_0}) (2g_0+k+a_0-3)! }{g^{2g_0+a_0+k-2}}\\
     \prec&T^2 e^{\frac{9}{2}T}\sum_{a_0\leq \frac{2T}{\pi}+1}\sum_{1\leq k\leq \frac{2T}{\pi}+2-a_0}\sum_{56\leq 2g_0+k+a_0-2\leq \frac{2T}{\pi}} k^2 \frac{ (2g_0+k+a_0-3)! }{g^{2g_0+\frac{a_0}{2}+k-2}}\\
     \prec&T^2 e^{\frac{9}{2}T}\sum_{a_0\leq \frac{2T}{\pi}+1}\sum_{1\leq k\leq \frac{2T}{\pi}+2-a_0} k^2 \frac{ 1 }{g^{27}}\\
     \prec&\frac{T^6e^{\frac{9}{2}T}}{g^{27}}.
\end{aligned}
\end{equation}

If $2\leq 2g_0+a_0+k-2\leq 55$, by the same argument as in \cite[Proposition 34]{wx22-3/16} and Lemma \ref{lemma admissible triple in hide}, for any $\epsilon>0$, we have \begin{equation}\label{int type 6 geq k+1 eq 4}
\begin{aligned}
    &\Egn\left[T\sum_{\substack{
    Y\simeq S_{g_0,k+a_0}\subset X\\
    \textit{with } a_0 \textit{ cusps}\\
    }}e^{\frac{T}{2}-\frac{1-\epsilon}{2}\ell(\partial Y)} \textbf{1}_{[0,2T]}(\ell(\partial Y))
    \right]\\
\prec &T^{230}\! e^{(\frac{1}{2}+\epsilon)T}\!\sum_{n_0=0}^{[\frac{k-1}{2}]}\sum_{q=1}^{k-2n_0}\!\!
  \sum_{\{(g_i,n_i,a_i)\}_{i=1}^q}\!\!\! \!\!\!{n\choose a_0,\cdots,a_q}\!\frac{V_{g_0,a_0+k}\prod_{i=1}^q V_{g_i,n_i+a_i}}{V_{g,n}n_0!n_1!\cdots n_q!}\\
   \prec&T^{230} e^{(\frac{1}{2}+\epsilon)T}\sum_{n_0=0}^{[\frac{k-1}{2}]}\sum_{q=1}^{k-2n_0}(2g_0+k+a_0-3)!\frac{1+n^{a_0}}{g^{2g_0+a_0+k-2}}\\
   \prec&T^{230} e^{(\frac{1}{2}+\epsilon)T}\frac{1+n^{a_0}}{g^{2g_0+{a_0}+k-2}}.
\end{aligned}
\end{equation}
So \begin{equation}\label{int type 6 geq k+1 eq 5}
\begin{aligned}
    &\Egn\left[T\sum_{\substack{
   2\leq |\chi(Y)|\leq 55
    }}e^{\frac{T}{2}-\frac{1-\epsilon}{2}\ell(\partial Y)} \textbf{1}_{[0,2T]}(\ell(\partial Y))
    \right]\\
  \prec &\sum_{2\leq 2g_0+a_0+k-2\leq 55} T^{230} e^{(\frac{1}{2}+\epsilon)T}\frac{1+n^3}{g^2}\\
      \prec & T^{230} e^{(\frac{1}{2}+\epsilon)T}\frac{1+n^3}{g^2}.
\end{aligned}
\end{equation}
Now the lemma follows from \eqref{int type 6 geq k+1 eq 3} and \eqref{int type 6 geq k+1 eq 5}.
\end{proof}
If $\left|\chi(Y)\right|=1$ and $X\setminus Y$ is disconnected, $Y$ must be of type $S_{0,3}$ with at most $1$ cusps. We have:
\begin{lemma}\label{lemma type 6 disconnected}
If $n=o(\sqrt{g})$, then for any $\epsilon>0$,
    \begin{align*}
        \Egn\left[T\sum_{\substack{
    Y\simeq S_{0,3}\subset X\\
    X\setminus Y \textit{ is disconnected}\\
    }}e^{\frac{T}{2}-\frac{1-\epsilon}{2}\ell(\partial Y)} \textbf{1}_{[0,2T]}(\ell(\partial Y))
    \right]\prec T^{230}e^{(\frac{1}{2}+\epsilon)T} \frac{1+n^3}{g^2}.
    \end{align*}
\end{lemma}
\begin{proof}
  There are three cases. 
For the first case, $X\setminus Y$ has three components of type $S_{g_i,a_i+1}$ with $a_i$ cusps for $i=1,2,3$, and each component shares a common boundary geodesic with $Y$. In this case, $\sum_{i=1}^3g_i=g$ and $\sum_{i=1}^n a_i=n$. 
For the second case, $X\setminus Y$ has two components of type $S_{g_1,a_1+1}$ and $S_{g_2,a_2+2}$, where $S_{g_1,a_1+1}$ shares a common boundary geodesic with $Y$ and $S_{g_2,a_1+2}$ shares two common boundary geodesics with $Y$. In this case, $g_1+g_2=g-1$ and $a_1+a_2=n$.
For the third case, $Y$ has a cusp and $X\setminus Y$ has two components of type $S_{g_i,a_i+1}$ with $a_i$ cusps, for $i=1,2$, and each component shares a common boundary geodesic with $Y$.
In this case, $g_1+g_2=g$ and $a_1+a_2=n-1$.
By the same argument as in \cite[Proposition 34]{wx22-3/16} we have \begin{equation}\label{type 6 dis connected 1}
\begin{aligned}
     &\Egn\left[T\sum_{\substack{
    Y\simeq S_{0,3}\subset X\\
    X\setminus Y \textit{ is disconnected}\\
    }}e^{\frac{T}{2}-\frac{1-\epsilon}{2}\ell(\partial Y)} \textbf{1}_{[0,2T]}(\ell(\partial Y))
    \right]\\
   \prec &\left[ \sum_{\substack{
g_1+g_2+g_3=g   \\
   a_1+a_2+a_3=n\\
   2\leq 2g_1+a_1\leq 2g_2+a_2\leq 2g_3+a_3
   }}{n\choose a_1,a_2,a_3}\frac{V_{0,3}V_{g_1,a_1+1}V_{g_2,a_2+1}V_{g_3,a_3+1}}{V_{g,n}}
   \right.\\
   +&\sum_{\substack{
g_1+g_2=g-1\\
   a_1+a_2=n\\
   2g_1+a_1\geq 2,2g_2+a_2\geq 1
   }} {n\choose a_1}\frac{V_{0,3}V_{g_1,a_1+1}V_{g_2,a_2+2}}{V_{g,n}} \\
   +&\left. \sum_{
   \substack{
g_1+g_2=g\\
   a_1+a_2=n-1\\
   2\leq 2g_1+a_1\leq 2g_2+a_2\\
   }
   }{n\choose 1,a_1,a_2}\frac{V_{0,3}V_{g_1,a_1+1}V_{g_2,a_2+1}}{V_{g,n}} \right]\cdot T^{230} e^{(\frac{1}{2}+\epsilon)T}.
\end{aligned}
\end{equation}
For the first summation in \eqref{type 6 dis connected 1}, we have \begin{align*}
    &\sum_{\substack{
g_1+g_2+g_3=g   \\
   a_1+a_2+a_3=n\\
   2\leq 2g_1+a_1\leq 2g_2+a_2\leq 2g_3+a_3
   }}{n\choose a_1,a_2,a_3}\frac{V_{0,3}V_{g_1,a_1+1}V_{g_2,a_2+1}V_{g_3,a_3+1}}{V_{g,n}}\\
   \prec &\sum_{1\leq 2g_1+a_1-1\leq\frac{2g+n-3}{3}}{n\choose a_1}\frac{V_{g_1,a_1+1}V_{g-g_1,n-a_1}}{V_{g,n}}\\
   &\sum_{\substack{
g_2+g_3=g-g_1\\
   a_2+a_3=n-a_1\\
   2g_1+a_1\leq 2g_2+a_2\leq 2g_3+a_3
   }}{n-a_1\choose a_2}\frac{V_{g_2,a_2+1}V_{g_3,a_3+1}}{V_{g-g_1,n-a_1}}.
\end{align*}
Since $g-g_1=g_2+g_3\asymp g$, by Lemma \ref{lemma in hide appendix gen k} for $k=1$,
we have \begin{align*}
    \sum_{\substack{
g_2+g_3=g-g_1\\
   a_2+a_3=n-a_1\\
   2g_1+a_1\leq 2g_2+a_2\leq 2g_3+a_3
   }}{n-a_1\choose a_2}\frac{V_{g_2,a_2+1}V_{g_3,a_3+1}}{V_{g-g_1,n-a_1}}\prec \frac{1+(n-a_1)^2}{g}.
\end{align*}
By Theorem \ref{thm mz15 asymp} and Lemma \ref{lemma in hide appendix gen k} for $k=1$, we have \begin{align*}
    &\sum_{1\leq 2g_1+a_1-1\leq\frac{2g+n-3}{3}}{n\choose a_1}\frac{V_{g_1,a_1+1}V_{g-g_1,n-a_1}}{V_{g,n}}\\
    \prec&\frac{1}{g} \sum_{1\leq 2g_1+a_1-1\leq\frac{2g+n-3}{3}}{n\choose a_1}\frac{V_{g_1,a_1+1}V_{g-g_1,n-a_1+1}}{V_{g,n}}\\
    \prec& \frac{1+n^2}{g^2}.
\end{align*}
So \begin{equation}\label{type 6 dis connected 2}
\begin{aligned}
     \sum_{\substack{
g_1+g_2+g_3=g   \\
   a_1+a_2+a_3=n\\
   2\leq 2g_1+a_1\leq 2g_2+a_2\leq 2g_3+a_3
   }}\!\!\!\!\!\!\!\!\!\!\!\!\!\!\!  {n\choose a_1,a_2,a_3}\frac{V_{0,3}V_{g_1,a_1+1}V_{g_2,a_2+1}V_{g_3,a_3+1}}{V_{g,n}}
  \!\! \prec\!\! \frac{1+n^4}{g^3}.
\end{aligned}
\end{equation}
For the second summation in \eqref{type 6 dis connected 1}, we have \begin{align*}
    &\sum_{\substack{
g_1+g_2=g-1\\
   a_1+a_2=n\\
    2g_1+a_1\geq 2,2g_2+a_2\geq 1
   }} {n\choose a_1}\frac{V_{0,3}V_{g_1,a_1+1}V_{g_2,a_2+2}}{V_{g,n}} \\
   \prec& \sum_{\substack{
g_1+g_2=g-1\\
   a_1+a_2=n\\
   g_1\geq g_2\\
   2g_2+a_2\geq 1
   }} {n\choose a_1}\frac{V_{g_1,a_1+1}V_{g_2,a_2+2}}{V_{g,n}} \\
   +&\sum_{\substack{
g_1+g_2=g-1\\
   a_1+a_2=n\\
   g_1\leq g_2\\
    2g_1+a_1\geq 2
   }} {n\choose a_1}\frac{V_{g_1,a_1+1}V_{g_2,a_2+2}}{V_{g,n}}
\end{align*}
By Theorem \ref{thm mz15 asymp} and Lemma \ref{appendix product lemma  2i+j geq k} for $k=1$, we have \begin{align*}
    &\sum_{\substack{
g_1+g_2=g-1\\
   a_1+a_2=n\\
   g_1\geq g_2\\
   2g_2+a_2\geq 1
   }} {n\choose a_1}\frac{V_{g_1,a_1+1}V_{g_2,a_2+2}}{V_{g,n}}\\
   \prec&\frac{1}{g}\sum_{\substack{
g_1+g_2=g-1\\
   a_1+a_2=n\\
   g_1\geq g_2\\ 
   2g_2+a_2\geq 1
   }} {n\choose a_1}\frac{V_{g_1,a_1+2}V_{g_2,a_2+2}}{V_{g-1,n+2}}
   \prec\frac{1+n}{g^2}.
\end{align*}
By Theorem \ref{thm mz15 asymp} and Lemma \ref{lemma in hide appendix gen k} for $k=1$, we have \begin{align*}
   & \sum_{\substack{
g_1+g_2=g-1\\
   a_1+a_2=n\\
   g_1\leq g_2\\
    2g_1+a_1\geq 2
   }} {n\choose a_1}\frac{V_{g_1,a_1+1}V_{g_2,a_2+2}}{V_{g,n}}
   \prec g \sum_{\substack{
g_1+g_2=g-1\\
   a_1+a_2=n\\
   g_1\leq g_2\\
    2g_1+a_1\geq 2
   }} {n\choose a_1}\frac{V_{g_1,a_1+1}V_{g_2,a_2+1}}{V_{g,n}}\\
    \prec& \frac{1}{g}\sum_{\substack{
g_1+g_2=g-1\\
   a_1+a_2=n\\
   g_1\leq g_2\\
    2g_1+a_1\geq 2
   }} {n\choose a_1}\frac{V_{g_1,a_1+1}V_{g_2,a_2+1}}{V_{g-1,n}}
   \prec \frac{1+n^2}{g^2}.
\end{align*}
So \begin{align}\label{type 6 dis connected 3}
    \sum_{\substack{
g_1+g_2=g-1\\
   a_1+a_2=n\\
    2g_1+a_1\geq 2,2g_2+a_2\geq 1
   }} {n\choose a_1}\frac{V_{0,3}V_{g_1,a_1+1}V_{g_2,a_2+2}}{V_{g,n}}\prec  \frac{1+n^2}{g^2}.
\end{align}
For the third summation in \eqref{type 6 dis connected 1}, by Theorem \ref{thm mz15 asymp} and Lemma \ref{lemma in hide appendix gen k} for $k=1$, we have 
\begin{equation}\label{type 6 dis connected 4}
\begin{aligned}
   & \sum_{
   \substack{
g_1+g_2=g\\
   a_1+a_2=n-1\\
   2\leq 2g_1+a_1\leq 2g_2+a_2\\
   }
   }{n\choose 1,a_1,a_2}\frac{V_{0,3}V_{g_1,a_1+1}V_{g_2,a_2+1}}{V_{g,n}}\\
\prec& \frac{1}{g}\sum_{
   \substack{
g_1+g_2=g\\
   a_1+a_2=n-1\\
   2\leq 2g_1+a_1\leq 2g_2+a_2\\
   }} n {n-1\choose a_1}\frac{V_{g_1,a_1+1}V_{g_2,a_2+1}}{V_{g,n-1}}
   \prec\frac{1+n^3}{g^2}.
\end{aligned}
\end{equation}
Now the Lemma follows from \eqref{type 6 dis connected 1}, \eqref{type 6 dis connected 2}, \eqref{type 6 dis connected 3} and \eqref{type 6 dis connected 4}.
\end{proof}

Combing Proposition \ref{prop wx prop 30}, Lemma \ref{lemma type 6 k geq 2} and Lemma \ref{lemma type 6 disconnected}, we have the following estimation for the total contribution of filling geodesics of type $6$ to $\textrm{Int}_{ns}$ in \eqref{eq ineq on subset T}.
\begin{lemma}\label{lemma total type 6}
If $n=n(g)=o(\sqrt{g})\geq 1$, then for any $\epsilon>0$, \begin{align*}
&\Egn\left[     \sum_{\substack{
\gamma\textit{ of type 6}\\
\ell_\gamma(X)\leq T
}}
H_{X,1}(\gamma)\right]\prec \frac{T^6e^{\frac{9}{2}T}}{g^{27}}+T^{230}e^{(\frac{1}{2}+\epsilon)T}\frac{1+n^3}{g^2},
\end{align*}
    where the implied constant depends on $\epsilon$ and is independent of $T,g$.
\end{lemma}

\section{Remainder terms on $(\mathcal{N}_\ell)^c$}\label{susbsection n ell k}
In this section, we estimate $\Egn\left[ (\#N_\ell)_k\cdot\sum_{\gamma\in \mathcal{P}_{nsep}^{s}(X)}H_{X,1}(\gamma)\right]$ and $\hat{f}_T(\frac{i}{2})\cdot\Egn\left[ (\#N_\ell)_k\right]$ for fixed $k\geq 1$, which will complete the estimation of \eqref{eq expansion formula on N ell}.
By the definition, we have 
\begin{equation}
\begin{aligned}
    (\#N_\ell(X))_k=&\#\{\textit{unorder pairs of }(Y_1,\cdots,Y_k)\\
    &\in N_\ell(X)^k; 
    Y_i\textit{'s are different} \},
    \end{aligned}
    \end{equation}
where \begin{align*}
    N_\ell(X)=\{\textit{embedded } Y\subset X;\chi(Y)=-1,\ell(\partial Y)< \ell\}.
\end{align*}

We will prove the following four lemmas.
\begin{lemma}\label{thm Nellk with H k geq 2}
    If $n=o(g^{\frac{1}{2}-\epsilon})$ for some $\epsilon>0$, then for $1\prec \ell\prec T\prec \log g$ and any fixed $k\geq 2$,
    \begin{align*}
        &\Egn\left[ (\#N_\ell)_k\cdot\sum_{\gamma\in \mathcal{P}_{nsep}^{s}(X)}H_{X,1}(\gamma)\right]\\
        =&\frac{1}{k!}{n\choose  2,\cdots,2,n-2k}\!\left(\frac{\cosh \frac{\ell}{2}-1}{2\pi^2g}\right)^k\!\int_0^T\!2\sinh\frac{y}{2}f_T(y)dy \!\left(1\!+\!O\left(\frac{nT^2}{g}\right)\!\right)\\
        +&O\left(\ell^{6k+1}e^{5k\ell} \right)+O\left(T^{366k+365}e^{\frac{T}{2}+5k\ell}\frac{\log^{14}g(\log g+n)^3}{g^2}\right)\\
        +&O\left( T^{6k+4}e^{\frac{9}{2}T+5k\ell}\frac{\log^{427}g(\log g+n)^{62}}{g^{61}}\right).
    \end{align*}
    \end{lemma}
    
    \begin{lemma}\label{thm Nellk with H k equal 1}
    If $n=o(g^{\frac{1}{2}-\epsilon})$ for some $\epsilon>0$, then for $1\prec \ell\prec T\prec \log g$, \begin{align*}
         &\Egn\left[ (\#N_\ell)\cdot\sum_{\gamma\in \mathcal{P}_{nsep}^{s}(X)}H_{X,1}(\gamma)\right]\\
         =&\frac{1}{\pi^2g}\left[ \frac{n(n-1)}{4}\left(\cosh\frac{\ell}{2}-1\right)+\frac{n}{4}\int_{x_1+x_2\leq \ell}\sinh\frac{x_1}{2}\sinh\frac{x_2}{2}dx_1dx_2\right.\\
        +&\frac{1}{6} \int_{x_1+x_2+x_3\leq \ell}\sinh\frac{x_1}{2}\sinh\frac{x_2}{2}\sinh\frac{x_3}{2}dx_1dx_2dx_3 \\
       +& \left.\frac{1}{8}\int_{0}^\ell V_{1,1}(x)\sinh\frac{x}{2}dx
  \right] \cdot \int_{0}^T 2\sinh\frac{y}{2}f_T(y)dy 
         +O\left(T^{7}e^{5\ell} \right)\\
         +&O\left(T^{731}e^{\frac{T+7\ell}{2}}\frac{\log^{14}g(\log g+n)^3}{g^2}\right)
        +O\left( T^{10}e^{\frac{9}{2}T+5\ell}\frac{\log^{427}g(\log g+n)^{62}}{g^{61}}\right).
    \end{align*}
\end{lemma}

\begin{lemma}\label{thm Nellk without H k geq 2}
      If $n=o(g^{\frac{1}{2}-\epsilon})$ for some $\epsilon>0$, then for $1\prec \ell\prec\log g$ and any fixed $k\geq 2$,  \begin{align*}
        &\Egn\left[ (\#N_\ell)_k\right]\\
        =&\frac{1}{k!}{n\choose  2,\cdots,2,n-2k}\left(\frac{\cosh \frac{\ell}{2}-1}{2\pi^2g}\right)^k\left(1+O\left(\frac{n\ell^2}{g}\right)\right)\\
        +&O\left( \ell^{6k+3} e^{5k\ell}\frac{\log^{14}g (\log g+n)^3}{g^2} \right).
    \end{align*}
\end{lemma}
\begin{lemma}\label{thm Nellk without H k equal 1}
      If $n=o(g^{\frac{1}{2}-\epsilon})$ for some $\epsilon>0$, then for $1\prec \ell\prec\log g$,  \begin{align*}
        &\Egn\left[ \#N_\ell\right]\\
        =&\frac{1}{\pi^2g}\left[ \frac{n(n-1)}{4}\left(\cosh\frac{\ell}{2}-1\right)+\frac{n}{4}\int_{x_1+x_2\leq \ell}\sinh\frac{x_1}{2}\sinh\frac{x_2}{2}dx_1dx_2\right.\\
        +&\frac{1}{6} \int_{x_1+x_2+x_3\leq \ell}\sinh\frac{x_1}{2}\sinh\frac{x_2}{2}\sinh\frac{x_3}{2}dx_1dx_2dx_3 \\
       +& \left.\frac{1}{8}\int_{0}^\ell V_{1,1}(x)\sinh\frac{x}{2}dx
  \right]
        +O\left(\frac{(1+n^3)\ell^4 e^\frac{\ell}{2}}{g^2}\right).
    \end{align*}
\end{lemma}

\subsection{Estimate of $\Egn\left[ (\#N_\ell)_k\cdot\sum_{\gamma\in \mathcal{P}_{nsep}^{s}(X)}H_{X,1}(\gamma)\right]$}\label{subsec est N ell k with H}
 For $(\gamma,Y_1,\cdots,Y_k)\in \mathcal{P}_{nsep}^s(X)\times N_\ell(X)^k$ with $\ell_\gamma(X)\leq T$, there is a special case where every $Y_i$ has two cusps and is disjoint from each other and $\gamma$. Then $X\setminus\bigsqcup_{i=1}^k Y_i\simeq S_{g,n-k}$ and $\gamma$ cuts the $S_{g,n-k}$ part into $S_{g-1,n-k+2}$. We say that such $(k+1)-$tuples of $(\gamma,Y_1,\cdots,Y_k)$ is of the leading type.

\begin{lemma}\label{lemma N ell k with H leading term}
    If $n=o(\sqrt{g})$, then for any fixed $k$,
    \begin{align*}
        &\E\left[\sum_{\substack{
       \gamma\in \mathcal{P}_{nsep}^s(X)\\
       (Y_1,\cdots,Y_k)\in  N_\ell(X)^k\\
      Y_i \textit{ unordered, leading type} }}H_{X,1}(\gamma) \right]\\
      =&\frac{1}{k!}{n\choose  2,\cdots,2,n-2k}\left(\frac{\cosh \frac{\ell}{2}-1}{2\pi^2g}\right)^k\int_0^T2\sinh\frac{y}{2}f_T(y)dy\\
    \cdot&\left(1+O\left(\frac{nT^2}{g}\right)\right).
    \end{align*}
\end{lemma}
\begin{proof}
    We apply Mirzakhani's Integration Formula Theorem \ref{thm mir int formula} to each mapping class group orbit of $(\gamma, Y_1,\cdots, Y_k)$ of the leading type. Notice that there are $\frac{1}{k!}{n\choose 2}{n-2\choose 2}\cdots {n-2k+2\choose 2}$ orbits, determined by the grouping of cusps. By Lemma \ref{lemma NWX vgn(x)} and Lemma \ref{lemma weak vgn}, we have \begin{align*}
        &\E\left[\sum_{\substack{
       \gamma\in \mathcal{P}_{nsep}^s(X)\\
       (Y_1,\cdots,Y_k)\in  N_\ell(X)^k\\
      Y_i \textit{ unordered, leading type} }}H_{X,1}(\gamma) \right]\\
      =&\frac{1}{k!}{n\choose 2}{n-2\choose 2}\cdots {n-2k+2\choose 2}\cdot2 \cdot\frac{1}{2}\frac{1}{V_{g,n}}\int_{x_1,\cdots,x_k\leq \ell,y\leq T} H_{X,1}(y)\\
      &\prod_{i=1}^kV_{0,3}(0,0,x_i)V_{g-1,n-k+2}(x_1,\cdots,x_k,y,y,0^{n-2k})x_1\cdots x_k y dx_1\cdots dx_kdy\\
      =&\frac{1}{k!}{n\choose  2,\cdots,2,n-2k}\frac{V_{g-1,n-k+2}}{V_{g,n}}
    \int_{x_1,\cdots,x_k\leq \ell,y\leq T}
    \prod_{i=1}^k\left(2\sinh\frac{x_i}{2}\right)\\
    \cdot &2\sinh\frac{y}{2}f_T(y)\left(1+O\left(\frac{n(\sum_{i=1}^kx_i^2+y^2)}{g}\right)\right)dx_1\cdots dx_kdy\\
    =&\frac{1}{k!}{n\choose  2,\cdots,2,n-2k}\left(\frac{1}{8\pi^2g}\right)^k
    \int_{x_1,\cdots,x_k\leq \ell,y\leq T}
    \prod_{i=1}^k\left(2\sinh\frac{x_i}{2}\right)\\
    \cdot &2\sinh\frac{y}{2}f_T(y)dx_1\cdots dx_kdy \left(1+O\left(\frac{n(T^2+\ell^2)}{g}\right)\right)\\
    =&\frac{1}{k!}{n\choose  2,\cdots,2,n-2k}\left(\frac{\cosh \frac{\ell}{2}-1}{2\pi^2g}\right)^k\int_0^T2\sinh\frac{y}{2}f_T(y)dy\\
    \cdot&\left(1+O\left(\frac{n(T^2+\ell^2)}{g}\right)\right),
    \end{align*}
where the implied constant only depends on $k$. Since $\ell=\kappa\log g$ for small $\kappa$ and $T\prec\log g$, we have the lemma.
\end{proof}

We introduce the following notation.
\begin{def*}
Let $X_{g,n}$ be a hyperbolic surface with genus $g$ and $n$ cusps.
   For $2g_i+a_i+n_i-2\geq 1$ and $n_i\geq 1$, assume $S_i\simeq S_{g_i,n_i+a_i}$ is a topology surface of genus $g_i$ with cusps and nonempty boundaries, where $n_i$ is the number of boundary loops and $a_i$ is the number of punctures. We also use $(g_i,n_i,a_i)=(0,2,0)$ to represent that $S_i$ is a simple curve.
Define the admissible set $\mathcal{B}_S^{s,(g,n)}$ for $S=\{(g_i,n_i,a_i)\}_{i=1}^s$ to be the set of possible non-overlapped orbits of $S_1,\cdots, S_s$ in $X_{g,n}$. More precisely, $\mathcal{B}_S^{s,(g,n)}$ consists of the equivalent classes of $F=\left(\{f_i\}_{i=1}^s,\{V_j\}_{j=1}^r,\{h_j,m_j,b_j\}_{j=1}^r\right)$, where  
 \begin{enumerate}
     \item If $2g_i+a_i+n_i-2\geq 1$, the local homeomorphism $f_i:S_{g_i,n_i+a_i}\to X_{g,n}$ is injective on $\mathring{S}_{g_i,n_i+a_i}$, and $f_i(\partial S_{g_i,n_i+a_i})$ consists of simple closed geodesics in $X_{g,n}$.
     \item If $(g_i,n_i,a_0)=(0,2,0)$, the image of the embedding $f_i:S_{g_i,n_i+a_i}\to X_{g,n}$ is a simple closed geodesic. 
     \item For different $i$, $f_i( S_{g_i,n_i+a_i})$ has disjoint interior. 
     \item The complementary $X_{g,n}\setminus\cup_{i=1}^s f_i( S_{g_i,n_i+a_i})$ has $r$ components, denoted by $V_j$ for $j=1,\cdots,r$, where $V_j$ has genus $h_j$, $b_j$ cusps and $m_j$ geodesic boundary components.
     \item $2h_i+m_i+b_i\leq 2h_j+m_j+b_j$ if $i< j$.
     \item $F=\left(\{f_i\}_{i=1}^s,\{V_j\}_{j=1}^r,\{h_j,m_j,b_j\}_{j=1}^r\right)$ is said to be equivalent to $E=\left(\{\tilde{f}_i\}_{i=1}^s,\{\tilde{V}_j\}_{j=1}^r,\{\tilde{h}_j,\tilde{m}_j,\tilde{b}_j\}_{j=1}^r\right)$ if there is a homeomorphism $H: X_{g,n}\to X_{g,n}$ which may not fix the label of cusps and permutation elements $\sigma_1\in S_s, \sigma_2\in S_r$ (here $S_s$ and $S_r$ are the permutation group of $s$ and $r$ elements) such that $H\circ f_i=\tilde{f}_{\sigma_1(i)}$ and $H(V_j)=V_{\sigma_2(j)}$.
 \end{enumerate}
 We say $f_i(S_{g_i,n_i+a_i})$ for $i=1,\cdots,s$ and $V_j$ for $j=1,\cdots,r$ are all the components subordinated to $F$.
\end{def*}
\begin{rem*}
    In the definition of $\mathcal{B}_{S}^{s,(g,n)}$, if in the condition (6) the homeomorphism $H$ is required to fixing the cusps, then $F$ will counts for ${n\choose a_1,\cdots,a_s,b_1,\cdots,b_r}$ times.
\end{rem*}

Now we assume that $(\gamma,Y_1,\cdots,Y_k)$ is not of the leading type. Assume $Y_1,\cdots, Y_k$ fills the union $\tilde{S}_1\cup \cdots \tilde{S}_{t}$ of subsurfaces with geodesic boundaries, where every $\tilde{S}_i$ is connected and the interiors are disjoint. Each $\tilde{S}_i$ may have some overlapped pairs of boundary geodesics. 
We have \begin{align}\label{eq esti total ell partial Si}
\sum_{i=1}^t \ell(\partial \tilde{S_i})\leq \sum_{i=1}^k \ell(\partial Y_i)\leq k\ell
\end{align}
for $\ell=\kappa\log g$.
By Isoperimetric Inequality on hyperbolic surfaces (for example, see \cite{buser2010geometry} or \cite{wu2022small}), 
\begin{align}\label{eq esti total area partial Si}
\area(Y_1)\leq \sum_{i=1}^t\area(\tilde{S}_i)\leq 2\pi k+2\sum_{i=1}^k\ell(\partial Y_i)\leq 2\pi k+2k\ell.\end{align} 
So \begin{align}\label{eq esti total charaster partial Si}
1\leq \sum_{i=1}^t\left|\chi(\tilde{S}_i)\right|\leq k+\frac{k\ell}{\pi}=O(\log g),\end{align}
and $\sum_{i=1}^t\left|\chi(\tilde{S}_i)\right|=1$ can only happen when $k=1$.
Now we classify the relative position of $\gamma$ with $\cup_{i=1}^t\tilde{S}_i.$

\subsection{Case 1: $\gamma\subset \mathring{\tilde{S}}_{i_0}$ for some $i_0$}\label{subsection case 1 gamma subset S i_0}\label{subsec est N ell k with H case 1}
In this case, the subsurface $\tilde{S}_{i_0}\setminus\gamma$ may be disconnected. Then we can assume that $\cup_{i=1}^t \tilde{S}_i\setminus\gamma=\cup_{i=1}^s S_i$, where every $S_i$ is connected and the interiors are disjoint. For $k\geq 2$,
we have $s\leq t+1\leq k+1$, $2\leq \sum_{i=1}^s \left|\chi(S_i)\right|\leq k+\frac{k\ell}{\pi}$, and $\gamma$ is one of the boundary geodesics of $\cup_{i=1}^s S_i$. For the length of the boundary, we have $\ell(\gamma)\leq T$ and $\sum_{i=1}^s \ell(\partial S_i-\gamma)\leq k\ell$ by \eqref{eq esti total ell partial Si}.
Now for any fixed $\gamma\cup\left(\cup_{i=1}^s S_i\right)$, we can bound the number of $(Y_1,\cdots,Y_k)$ that generates $\cup_{i=1}^s S_i$ with $\gamma$ of case 1. A very rough estimate is sufficient for us. The boundaries of each $Y_i$ are simple closed geodesics with length $\leq \ell$. By Lemma \ref{lemma uniform L+6}, the map$$
(\gamma,Y_1,\cdots,Y_k)\mapsto (\gamma,\cup_{i=1}^s S_i) 
$$
 is at most $\left(\left(k+\frac{k\ell}{\pi}\right)\frac{ e^{\ell+6}+3}{2}\right)^{3k}$ to $1$. Assume that $S_i\simeq S_{g_i,n_i,a_i}$, where $n_i$ is the number of boundary geodesics and $a_i$ is the number of cusps for $S_i$. Let $S=\{(g_i,n_i,a_i)\}_{i=1}^s$.
We can view $\cup_{i=1}^sS_i$ as an element $\mathcal{F}$ in $\mathcal{B}_S^{s,(g,n)}.$ Notice that not all $S_i$ can be a pair of pants with two cusps, or they will be disjoint and $\gamma$ can not serve as an interior geodesic for $\tilde{S}_{i_0}$. 
Let $\Gamma_0=\cup_{i=1}^sf_i(\partial S_{g_i,n_i+a_i})$. Then $\gamma$ will serve as one of the simple closed geodesic in $\Gamma_0$. Following \eqref{eq esti total charaster partial Si}, the choice of number of $\gamma$ is bounded by $\sum_{i=1}^s n_i\leq 3\sum_{i=1}^s (2g_i+a_i+n_i-2)\leq 3k+\frac{3k\ell}{\pi}$. Let $\overline{\Mod}_{g,n}$ to be the generalized mapping class group of $S_{g,n}$ allowing the permutation of the labels of cusps.
Then for $k\geq 2$ we have \begin{equation}\label{N ell k case 1 main ineq}
\begin{aligned}
    &\sum_{\substack{\gamma\in \mathcal{P}_{nsep}^s(X),\ell_\gamma(X)\leq T\\
       (Y_1,\cdots,Y_k)\in  N_\ell(X)^k\\
     \textit{of case 1}}}H_{X,1}(\gamma)\\
 \prec&\left(\left(k+\frac{k\ell}{\pi}\right)\frac{ e^{\ell+6}+3}{2}\right)^{3k}\sum_{1\leq s\leq k+1}\sum_{S\in \textbf{Ind}_1}\\
    &\sum_{F\in\mathcal{B}_S^{s,(g,n)} }\sum_{\gamma_0\subset\Gamma_0}\sum_{(\gamma,\Gamma)\in \overline{\Mod}_{g,n}\cdot (\gamma_0,\Gamma_0)}H_{X,1}(\gamma)\textbf{1}_{[0,k\ell]}(\ell(\Gamma-\gamma)),
\end{aligned}
\end{equation}
where $S$ is over \begin{align*}
\textbf{Ind}_1=\left\{\right.S&=\{(g_i,n_i,a_i)\}_{i=1}^s\neq (0,1,2)^s;
1\leq 2g_i+n_i+a_i-2,\\
&1\leq n_i,\left.2\leq \sum_{i=1}^s(2g_i+n_i+a_i-2)\leq k+\frac{k\ell}{\pi}   \right\}.
\end{align*}
\begin{lemma}\label{lemma N ell k for lemma for case 1}
   Fix $k\geq 2$ and $1\leq s\leq k+1$. For any $S\in \textbf{Ind}_1$, $F\in\mathcal{B}_S^{s,(g,n)}$, and $\Gamma_0=\cup_{i=1}^sf_i(\partial S_{g_i,n_i+a_i})$, the following estimate holds:
   \begin{align*}
&\Egn\left[\sum_{\gamma_0\subset\Gamma_0}\sum_{(\gamma,\Gamma)\in \overline{\Mod}_{g,n}\cdot (\gamma_0,\Gamma_0)}H_{X,1}(\gamma)\textbf{1}_{[0,k\ell]}(\ell(\Gamma-\gamma))\right]\\
        \prec &\frac{{n\choose a_1,\cdots,a_s,b_1,\cdots b_r}}{V_{g,n}}\prod_{i=1}^s V_{g_i,n_i+a_i}\prod_{j=1}^r V_{h_j,m_j+b_j} \cdot 
   e^{\frac{T}{2}+2k\ell}\cdot k\ell.
    \end{align*}
\end{lemma}
\begin{proof}
For each $\gamma_0\subset \Gamma_0$, we divide $\overline{\Mod}_{g,n}\cdot (\gamma_0,\Gamma_0)$ into ${n\choose a_1\cdots,a_s,b_1,\cdots,b_r}$ mapping class group orbits and apply each orbit to Mirzakhani's Integration Formula Theorem \ref{thm mir int formula}. Let $x_p$ be the length of the simple closed geodesic $\gamma_p$ in $\Gamma\setminus \gamma$ for $p$ in the index set $I$, and $x_0$ be the length of $\gamma$. 
Then we have \begin{equation}\label{over N ell k case 1  eq 2}
\begin{aligned}
    &\Egn\left[\sum_{(\gamma,\Gamma)\in \overline{\Mod}_{g,n}\cdot (\gamma_0,\Gamma_0)} H_{X,1}(\gamma)\textbf{1}_{[0,k\ell]}(\ell(\Gamma-\gamma))\right]\\
    \prec& \frac{{n\choose a_1,\cdots,a_s,b_1,\cdots b_r}}{V_{g,n}} \int_{\substack{ 
x_0\leq T\\
    \sum_{p\in I}x_p\leq k\ell
    }} H_{X,1}(x_0) \prod_{i=1}^s V_{g_i,n_i+a_i}(x_{p^i_1},\cdots ,x_{p^{i}_{n_i}},0^{a_i})\\
    \cdot&\prod_{j=1}^r V_{h_j,m_j+b_j}(x_{q^{j}_1},\cdots,x_{q^{j}_{m_j}},0^{b_j})x_0\prod_{p\in I} x_p\cdot  dx_0\prod_{p\in I}dx_p,
\end{aligned}
\end{equation}
where $p_1^i,\cdots,p_{n_i}^i$ are the labels of geodesics on $\partial S_i$, and $q_1^j,\cdots,q_{m_j}^j$ are the labels of geodesics on $\partial V_j$. 
By Lemma \ref{lemma NWX vgn(x)}, we have \begin{align*}
   & V_{g_i,n_i+a_i}(x_{p^i_1},\cdots ,x_{p^{i}_{n_i}},0^{a_i})
    \leq V_{g_i,n_i+a_i} \prod_{j=1}^{n_i}\frac{2\sinh\frac{x_{p_j^i}}{2}}{x_{p_j^i}},
\end{align*}
and \begin{align*}
    V_{h_j,m_j+b_j}(x_{q^{j}_1},\cdots,x_{q^{j}_{m_j}},0^{b_j})\leq  V_{h_j,m_j+b_j}\prod_{i=1}^{m_j}\frac{2\sinh\frac{x_{q_{m_i}^j}}{2}}{x_{q_{m_i}^j}}.
\end{align*}
Since each component of $\Gamma$ is adjoint to two subsurfaces separated by $\Gamma$, we have 
\begin{equation}\label{over N ell k case 1  eq 3}
\begin{aligned}
     & \prod_{i=1}^s V_{g_i,n_i+a_i}(x_{p^i_1},\cdots ,x_{p^{i}_{n_i}},0^{a_i})\prod_{j=1}^r V_{h_j,m_j+b_j}(x_{q^{j}_1},\cdots,x_{q^{j}_{m_j}},o^{b_j})\\
 \leq & \prod_{i=1}^s V_{g_i,n_i+a_i}\prod_{j=1}^r V_{h_j,m_j+b_j} \prod_{p\in I}\left( \frac{2\sinh\frac{x_p}{2}}{x_p}\right)^2\cdot\left(\frac{2\sinh\frac{x_0}{2}}{x_0}\right)^2\\
\leq &\prod_{i=1}^s V_{g_i,n_i+a_i}\prod_{j=1}^r V_{h_j,m_j+b_j} \prod_{p\in I}\cosh^2\frac{x_p}{2}\cdot\left(\frac{2\sinh\frac{x_0}{2}}{x_0}\right)^2.
\end{aligned}
\end{equation}
Now we can combine \eqref{over N ell k case 1  eq 2} and \eqref{over N ell k case 1  eq 3} to get \begin{equation}\label{over N ell k case 1  eq 4}
\begin{aligned}
    &\Egn\left[\sum_{(\gamma,\Gamma)\in \overline{\Mod}_{g,n}\cdot (\gamma_0,\Gamma_0)} H_{X,1}(\gamma)\textbf{1}_{[0,k\ell]}(\ell(\Gamma-\gamma))\right]\\
    \prec&\frac{{n\choose a_1,\cdots,a_s,b_1,\cdots b_r}}{V_{g,n}}\prod_{i=1}^s V_{g_i,n_i+a_i}\prod_{j=1}^r V_{h_j,m_j+b_j}\\
   \cdot&\int_{x_0\leq T} 2\sinh \frac{x_0}{2}f_T(x_0)dx_0   
   \cdot  \int_{\substack{ \sum_{p\in I}x_p\leq k\ell}}
    e^{\sum_{p\in I}x_p} \prod_{p\in I}x_p \prod_{p\in I}dx_p\\
   \prec & \frac{{n\choose a_1,\cdots,a_s,b_1,\cdots b_r}}{V_{g,n}}\prod_{i=1}^s V_{g_i,n_i+a_i}\prod_{j=1}^r V_{h_j,m_j+b_j} 
   e^{\frac{T}{2}+2k\ell}.
\end{aligned}
\end{equation}
Here we use the fact that \begin{align}\label{over N ell k case 1  eq 5}
\int_{\substack{ \sum_{p\in I}x_p\leq k\ell}}
    \prod_{p\in I}x_p \prod_{p\in I}dx_p=\frac{(k\ell)^{2\left|I\right|}}{(2\left|I\right|)!}\leq e^{k\ell}.
\end{align}
The lemma follows from \eqref{over N ell k case 1  eq 4}, and the fact that $\gamma_0\subset\Gamma_0$ has at most $3k+\frac{3k\ell}{\pi}$ possible choices.
\end{proof}
\begin{lemma}\label{lemma N ell k final for case 1}
    If $n=o(g^{\frac{1}{2}-\epsilon})$ for some $\epsilon>0$, then for any $k\geq 2$, \begin{align*}
         &\Egn\left[\sum_{\substack{\gamma\in \mathcal{P}_{nsep}^s(X),\ell_\gamma(X)\leq T\\
       (Y_1,\cdots,Y_k)\in  N_\ell(X)^k\\
     \textit{of case 1}}}H_{X,1}(\gamma)\right]\\
 \nonumber    \prec&\left(\left(k+\frac{k\ell}{\pi}\right)\frac{ e^{\ell+6}+3}{2}\right)^{3k}(k\ell)^{3k+4} \frac{\log^{14} g(\log g+n)^3}{g^2}e^{\frac{T}{2}+2k\ell}.
    \end{align*}
\end{lemma}
\begin{proof}
    By \eqref{N ell k case 1 main ineq} and Lemma \ref{lemma N ell k for lemma for case 1}, we have \begin{equation}\label{lemma N ell k final for case 1 eq 1}
    \begin{aligned}
        &\Egn\left[\sum_{\substack{\gamma\in \mathcal{P}_{nsep}^s(X),\ell_\gamma(X)\leq T\\
       (Y_1,\cdots,Y_k)\in  N_\ell(X)^k\\
     \textit{of case 1}}}H_{X,1}(\gamma)\right]\\
  \prec&\left(\left(k+\frac{k\ell}{\pi}\right)\frac{ e^{\ell+6}+3}{2}\right)^{3k}e^{\frac{T}{2}+2k\ell}\cdot k\ell\cdot\sum_{1\leq s\leq k+1}\sum_{S\in \textbf{Ind}_1}\\
   &\frac{1}{V_{g,n}}\sum_{F\in\mathcal{B}_S^{s,(g,n)} }{n\choose a_1,\cdots,a_s,b_1,\cdots,b_r} \prod_{i=1}^sV_{g_i,n_i+a_i}\prod_{j=1}^rV_{h_j,m_j+b_j}.
    \end{aligned}
    \end{equation}
    Notice that $2n_i\leq 7(2g_i+n_i+a_i-2)$. So for $s=1$, by Lemma \ref{appendix lemma k Si product} we have\begin{equation}\label{lemma N ell k final for case 1 eq 2}
    \begin{aligned}
    &  
    \sum_{\substack{S=(g_1,n_1,a_1),n_1\geq 1
    \\2\leq 2g_1+n_1+a_1-2\leq k+\frac{k\ell}{\pi}}} \sum_{ \mathcal{B}_{S}^{1,(g,n)}} \! {n\choose a_1,b_1\cdots,b_r}\frac{V_{g_1,n_1+a_1}\!\prod_{j=1}^r\! V_{h_j,m_j+b_j}}{V_{g,n}}
    \\
     \prec &\sum_{N=2}^{k+[\frac{k\ell}{\pi}]}\sum_{\substack{2g_1+n_1+a_1-2=N\\
        n_1\geq 1
        }}\left((N+2)!\right)^2(C\log g)^{7N}\frac{\left(C\log g+Cn\right)^{N+1}}{g^N}\\
       \prec&\sum_{N=2}^{k+[\frac{k\ell}{\pi}]} N^3\left((N+2)!\right)^2(C\log g)^{7N}\frac{\left(C\log g+Cn\right)^{N+1}}{g^N}\\
    \prec& \frac{\log^{14}g\left(\log g+n\right)^3}{g^2}.
    \end{aligned}
    \end{equation}
    For $s\geq 2$, since $\{(g_i,n_i,a_i)\}_{i=1}^s\neq (0,1,2)^s$,  there exist some $i_0$ such that $2g_{i_0}+n_{i_0}+a_{i_0}-2=1$ with $a_{i_0}\leq 1$ or $2g_{i_0}+n_{i_0}+a_{i_0}-2\geq 2$. We assume that $i_0=1$. By Lemma \ref{appendix lemma k Si product} we have
    \begin{equation}\label{lemma N ell k final for case 1 eq 3}
    \begin{aligned}
        &\sum_{s=2}^{k+1}\sum_{S\in \textbf{Ind}_1}\sum_{F\in\mathcal{B}_S^{s,(g,n)} }  \!\!\!{n\choose a_1,\cdots,a_s,b_1,\cdots,b_r}\!\! 
\cdot\!\!\prod_{i=1}^sV_{g_i,n_i+a_i}\!\!\prod_{j=1}^rV_{h_j,m_j+b_j}\\
   \prec&\sum_{s=2}^{k+1}\!\!\!\!\!\!\sum_{\substack{
2g_1+n_1+a_1-2\geq 2\\
1\leq n_i,1\leq 2g_i+n_i+a_i-2\\
 \sum_{i=1}^s(2g_i+n_i+a_i-2)\leq k+\frac{k\ell}{\pi}     
     }}\!\!\!\!\!\!
     \left((2g_1+a_1+n_1)!\right)^2(C\log g)^{2n_1}\\
    \cdot&\frac{(Cn+C\log g)^{a_1}}{g^{2g_1+n_1+a_1-2}}
     +\sum_{s=2}^{k+1}\sum_{\substack{
(2g_1+n_1+a_1-2)=1,a_1\leq 1\\
1\leq n_i,1\leq 2g_i+n_i+a_i-2\\
 \sum_{i=1}^s(2g_i+n_i+a_i-2)\leq k+\frac{k\ell}{\pi}     
     }}\frac{\log^7 g(\log g+n)^{a_1}}{g}  \\
  \cdot&\left((2g_2+a_2+n_2)!\right)^2(C\log g)^{2n_2}\frac{(Cn+C\log g)^{a_2}}{g^{2g_2+n_2+a_2-2}}\\
     \prec&\sum_{s=2}^{k+1}(k\ell)^{3s}\sum_{N=2}^{k+[\frac{k\ell}{\pi}]}\left((N+2)!\right)^2(C\log g)^{7N}\frac{(Cn+C\log g)^{N+1}}{g^N}\\
    +&\sum_{s=2}^{k+1}(k\ell)^{3s}\frac{\log^7 g (\log g+n)}{g}\\
    \cdot&\sum_{N=1}^{k+[\frac{k\ell}{\pi}]}\left((N+2)!\right)^2(C\log g)^{7N}\frac{(Cn+C\log g)^{N+1}}{g^N}\\
  \prec & (k\ell)^{3k+3} \frac{\log^{14} g(\log g+n)^3}{g^2}.
    \end{aligned}
    \end{equation}
  Here we use the fact that $
\left|\textbf{Ind}_1\right|\prec(k\ell)^{3s}$
and $s\leq k+1$.
Combining the above estimate with \eqref{lemma N ell k final for case 1 eq 1} and \eqref{lemma N ell k final for case 1 eq 2}, we prove the lemma.
\end{proof}

For $k=1$, we have
\begin{lemma}\label{lemma N ell k=1 final for case 1}
    If $n=o(g^{\frac{1}{2}-\epsilon})$ for some $\epsilon>0$, then \begin{align*}
\Egn\left[\sum_{\substack{(\gamma,Y_1)\in \mathcal{P}_{nsep}^s(X)\times N_\ell(X),\\
         \ell_\gamma(X)\leq T,
     \textit{ of case 1}}}H_{X,1}(\gamma)\right]
   \prec \frac{T^2e^{\frac{\ell}{2}}}{g}.
    \end{align*}
\end{lemma}
\begin{proof}
    If $\gamma\subset \mathring{Y}_1$ and $\chi(Y_1)=-1$, then $Y_1\simeq S_{1,1}$, $X\setminus Y_1\simeq S_{g-1,n+1}$, and $Y_1\setminus \gamma\simeq S_{0,3}$.
   By Mirzakhani's Integral Formula Theorem \ref{thm mir int formula},  Theorem \ref{thm mz15 asymp} and Lemma \ref{lemma NWX vgn(x)}, we have \begin{align*}
        &\Egn\left[\sum_{\substack{(\gamma,Y_1)\in \mathcal{P}_{nsep}^s(X)\times N_\ell(X),\\
         \ell_\gamma(X)\leq T,
     \textit{ of case 1}}}H_{X,1}(\gamma)\right]\\
     \prec&\frac{1}{V_{g,n}}\int_{x\leq \ell.y\leq T}V_{0,3}(x,y,y)V_{g-1,n+1}(x,0^n)xydxdy\\
     \prec & \frac{V_{g-1,n+1}}{V_{g,n}}\int_{x\leq \ell.y\leq T} \sinh\frac{x}{2}ydxdy\\
     \prec &\frac{T^2e^{\frac{\ell}{2}}}{g},
   \end{align*}
   which makes up for the omitted estimate for the case $k=1$.
\end{proof}

\subsection{Case 2: $\gamma\subset \partial \tilde{S}_{i_0}$ for some $i_0$}\label{subsec est N ell k with H case 2}
In this case, we just set $s=t\leq k$ and $S_i=\tilde{S}_i$. The simple geodesic $\gamma$ is one of the boundary geodesics of $\cup_{i=1}^s S_i$. By the same argument as case 1 in subsection \ref{subsection case 1 gamma subset S i_0},
if \begin{align*}
\textbf{Ind}_2=\left\{S=\{(g_i,n_i,a_i)\}_{i=1}^s;\right. &2g_i+n_i+a_i\geq 3, n_i\geq 1,\\
1\leq \sum_{i=1}^s&\left.(2g_i+n_i+a_i-2)\leq k+\frac{k\ell}{\pi}\right\},
\end{align*}
for any $k\geq 1$ we will have  \begin{equation}\label{N ell k case 2 main ineq}
\begin{aligned}
    &\sum_{\substack{\gamma\in \mathcal{P}_{nsep}^s(X),\ell_\gamma(X)\leq T\\
       (Y_1,\cdots,Y_k)\in  N_\ell(X)^k\\
     \textit{of case 2}}}H_{X,1}(\gamma)\\
  \prec&\left(\left(k+\frac{k\ell}{\pi}\right)\frac{ e^{\ell+6}+3}{2}\right)^{3k}\sum_{1\leq s\leq k}\sum_{S\in \textbf{Ind}_2}\\
   &\sum_{F\in\mathcal{B}_S^{s,(g,n)} }\sum_{\gamma_0\subset \Gamma_0}\sum_{\substack{\
(\gamma,\Gamma)\in\overline{\Mod}_{g,n}\cdot(\gamma_0,\Gamma_0)
     }}H_{X,1}(\gamma)\textbf{1}_{[0,k\ell]}(\ell(\Gamma)),
\end{aligned}
\end{equation}
where $\Gamma_0=\cup_{i=1}^s f_i(\partial S_{g_i,n_i+a_i})$ for $F\in \mathcal{B}_S^{s,(g,n)}$.
The only differences here are the conditions that $\ell(\Gamma-\gamma)\leq k\ell-\ell(\gamma)$ and $S_i$ can be pairs of pants with two cusps simultaneously.
\begin{lemma}\label{lemma N ell k for lemma for case 2}
   Fix $k\geq 1$ and $1\leq s\leq k$. For any $S\in \textbf{Ind}_2$, $F\in\mathcal{B}_S^{s,(g,n)}$ and $\Gamma_0=\cup_{i=1}^sf_i(\partial S_{g_i,n_i+a_i})$, the following estimate holds: \begin{align*}
        &\Egn\left[\sum_{\gamma_0\subset \Gamma_0}\sum_{\substack{\
    (\gamma,\Gamma)\in\overline{\Mod}_{g,n}\cdot(\gamma_0,\Gamma_0)
     }}H_{X,1}(\gamma)\textbf{1}_{[0,k\ell]}(\ell(\Gamma))\right]\\
        \prec &\frac{{n\choose a_1,\cdots,a_s,b_1,\cdots b_r}}{V_{g,n}}\prod_{i=1}^s V_{g_i,n_i+a_i}\prod_{j=1}^r V_{h_j,m_j+b_j} \cdot 
   e^{2k\ell}\cdot k\ell.
    \end{align*}
\end{lemma}
\begin{proof}
    Similar to the proof of Lemma \ref{lemma N ell k for lemma for case 1}, for each $\gamma_0\subset \Gamma_0$, we let $x_p$ be the length of the simple closed geodesic $\gamma_p$ in $\Gamma\setminus \gamma$ for $p$ in the index set $I$, and let $x_0$ be the length of $\gamma$. 
    Let $p_1^i,\cdots,p_{n_i}^i$ be the labels of geodesics on $\partial S_i$, and $q_1^j,\cdots,q_{m_j}^j$ be the labels of geodesics on $\partial V_j$. Then we have  \begin{align*}
    &\Egn\left[\sum_{\substack{\
    (\gamma,\Gamma)\in\overline{\Mod}_{g,n}\cdot(\gamma_0,\Gamma_0)
     }} H_{X,1}(\gamma)\textbf{1}_{[0,k\ell]}(\ell(\Gamma))\right]\\
    \nonumber\prec& \frac{{n\choose a_1,\cdots,a_s,b_1,\cdots b_r}}{V_{g,n}} \int_{\substack{ 
x_0+\sum_{p\in I}x_p\leq k\ell
    }} H_{X,1}(x_0) \prod_{i=1}^s V_{g_i,n_i+a_i}(x_{p^i_1},\cdots ,x_{p^{i}_{n_i}},0^{a_i})\\
    \nonumber\cdot&\prod_{j=1}^r V_{h_j,m_j+b_j}(x_{q^{j}_1},\cdots,x_{q^{j}_{m_j}},0^{b_j})x_0\prod_{p\in I} x_p\cdot  dx_0\prod_{p\in I}dx_p\\
   \nonumber\prec&\frac{{n\choose a_1,\cdots,a_s,b_1,\cdots b_r}}{V_{g,n}}\prod_{i=1}^s V_{g_i,n_i+a_i}\prod_{j=1}^r V_{h_j,m_j+b_j}\\
    \nonumber\cdot&\int_{x_0\leq k\ell} 2\sinh \frac{x_0}{2}f_T(x_0)  
   \left( \int_{\substack{ \sum_{p\in I}x_p\leq k\ell-x_0}}
    e^{\sum_{p\in I}x_p} \prod_{p\in I}x_p \prod_{p\in I}dx_p \right)dx_0 \\
   \nonumber\prec & \frac{{n\choose a_1,\cdots,a_s,b_1,\cdots b_r}}{V_{g,n}}\prod_{i=1}^s\! V_{g_i,n_i+a_i}\!\prod_{j=1}^r \!V_{h_j,m_j+b_j}\! 
  \int_{x_0\leq k\ell}2\sinh\frac{x_0}{2}f_T(x_0)\cdot e^{2k\ell-2x_0}dx_0\\
  \nonumber\prec & \frac{{n\choose a_1,\cdots,a_s,b_1,\cdots b_r}}{V_{g,n}}\prod_{i=1}^s V_{g_i,n_i+a_i}\prod_{j=1}^r V_{h_j,m_j+b_j}\cdot e^{2k\ell}.
\end{align*}
Here we use the estimates \eqref{over N ell k case 1  eq 3} and \eqref{over N ell k case 1  eq 5} again.
    The lemma follows since $\gamma_0\subset\Gamma_0$ has at most $3k+\frac{3k\ell}{\pi}$ possible choices.
\end{proof}
\begin{lemma}\label{lemma N ell k final for case 2}
    If $n=o(g^{\frac{1}{2}-\epsilon})$ for some $\epsilon>0$, then for $k\geq 1$,  \begin{align*}
        &\Egn\left[\sum_{\substack{\gamma\in \mathcal{P}_{nsep}^s(X),\ell_\gamma(X)\leq T\\
       (Y_1,\cdots,Y_k)\in  N_\ell(X)^k\\
     \textit{of case 2}}}H_{X,1}(\gamma)\right]\\
 \nonumber    \prec&\left(\left(k+\frac{k\ell}{\pi}\right)\frac{ e^{\ell+6}+3}{2}\right)^{3k}(k\ell)^{3k+1}e^{2k\ell}.
    \end{align*}
\end{lemma}
\begin{proof}
    This Lemma directly follows from \eqref{N ell k case 2 main ineq}, Lemma \ref{lemma N ell k for lemma for case 2}, Lemma \ref{appendix lemma k Si product}, and the fact that $\left|\textbf{Ind}_2\right|\prec (k\ell)^{3s}\prec (k\ell)^{3k}$.
\end{proof}

\subsection{Case 3: $\gamma$ intersects with $\cup_{i=1}^t\partial\tilde{S}_i$ transversely.}\label{subsec est N ell k with H case 3} In this case, we assume that $\cup_{i=1}^t\partial\tilde{S}_i\cup \gamma$ fills the union $\cup_{i=1}^{t^\prime}S_i^\prime$ of subsurfaces with geodesic boundaries, where every $S_i^\prime$ is connected and the interior are disjoint. We have $t^\prime\leq t$ and $\gamma$ is an interior simple closed geodesic in $\cup_{i=1}^{t^\prime}S_i^\prime$. So we can assume that $\cup_{i=1}^{t^\prime}S_i^\prime\setminus \gamma=\cup_{i=1}^s S_i$ with $s\leq t^\prime+1\leq t+1\leq k+1$. By \eqref{eq esti total ell partial Si} we have \begin{align*}
&\ell(\cup_{i=1}^s \partial S_i\setminus \gamma)= 
\sum_{i=1}^s\ell(\partial S_i)-2\ell(\gamma)=\sum_{i=1}^{t^\prime}\ell(\partial S_i^\prime)\\
&\leq \sum_{i=1}^t\ell(\partial \tilde{S}_i)+2\ell(\gamma)\leq k\ell+2\ell(\gamma),
\end{align*}
and by the Isoperimetric Inequality on hyperbolic surfaces along with \eqref{eq esti total ell partial Si} and \eqref{eq esti total area partial Si}, we have \begin{align*}
&\sum_{i=1}^t \area(\tilde{S}_i)<\sum_{i=1}^s\area( S_i)=\sum_{i=1}^{t^\prime}\area(S_i^\prime)\\
&\leq \sum_{i=1}^t \area(\tilde{S}_i)+2\left(\sum_{i=1}^t\ell(\partial \tilde{S}_i)+2T\right)\leq 2\pi k+4k\ell+4T.\\
\end{align*}
It follows that $$
2\leq \sum_{i=1}^{t} \left|\chi(S_i)\right|\leq k+\frac{2k\ell+2T}{\pi}=O(\log g).
$$
 By Lemma \ref{lemma uniform L+6}, the map$$
(\gamma,Y_1,\cdots,Y_k)\mapsto (\gamma,\cup_{i=1}^s S_i) 
$$
 is at most $\left(\left(k+\frac{2k\ell+2T}{\pi}\right)\frac{ e^{\ell+6}+3}{2}\right)^{3k}$ to $1$. Assume that $S_i\simeq S_{g_i,n_i,a_i}$, where $n_i$ is the number of boundary geodesics and $a_i$ is the number of cusps of $S_i$. Let $S=\{(g_i,n_i,a_i)\}_{i=1}^s$. 
 We also view $\cup_{i=1}^s S_i$ as an element $\mathcal{F}$ in $\mathcal{B}_{S}^{s,(g,n)}$. In this case, $S\neq (0,1,2)^s$ and $\gamma$ is not a boundary geodesic of $X\setminus \cup_{i=1}^s S_i$. Let $\Gamma_0=\cup_{i=1}^s f_i(\partial S_{g_i,n_i+a_i})$, then $\gamma$ will serve as one of simple closed geodesics in $\Gamma_0$ and $\cup_{j=1}^r \partial V_j\subset \Gamma_0-\gamma$. Similar to case 1, we have \begin{equation}\label{N ell k case 2.5 main ineq}
 \begin{aligned}
     &\sum_{\substack{
\gamma\in \mathcal{P}_{nsep}^s(X),\ell_\gamma(X)\leq T\\
(Y_1,\cdots,Y_k)\in \mathcal{N}_\ell^k\\
\textit{of case 3}
     }} H_{X,1}(\gamma)\\
\prec& \left(\left(k+\frac{2k\ell+2T}{\pi}\right)\frac{ e^{\ell+6}+3}{2}\right)^{3k}\sum_{1\leq s\leq k+1}\sum_{S\in \textbf{Ind}_3}\sum_{F\in \mathcal{B}_S^{s,(g,n)}}\\
&\sum_{\gamma_0\subset \Gamma_0\setminus\cup_{j=1}^r \partial V_j}
\sum_{(\gamma,\Gamma)\in \overline{\Mod}_{g,n}\cdot(\gamma_0,\Gamma_0)} H_{X,1}(\gamma)\textbf{1}_{\ell(\Gamma-\gamma)\leq 2\ell(\gamma)+k\ell},
 \end{aligned}
 \end{equation}
where $S$ is over \begin{align*}
\textbf{Ind}_3=&\left\{S=(g_i,a_i,n_i)_{i=1}^s\neq (0,1,2)^s\right.;1\leq 2g_i+n_i+a_i-2,\\
&1\leq n_i,\left.2\leq \sum_{i=1}^s(2g_i+n_i+a_i-2)\leq k+\frac{2k\ell+2T}{\pi}\right\}.
\end{align*}

\begin{lemma}\label{lemma N ell k for lemma for case 2.5}
   Fix $1\leq s\leq k+1$. For any $S\in \textbf{Ind}_3$, $F\in\mathcal{B}_S^{s,(g,n)}$ and $\Gamma_0=\cup_{i=1}^sf_i(\partial S_{g_i,n_i+a_i})$,
   the following estimate holds: \begin{align*}
    &\Egn\left[\sum_{\gamma_0\subset \Gamma_0\setminus\cup_{j=1}^r \partial V_j}
\sum_{(\gamma,\Gamma)\in \overline{\Mod}_{g,n}\cdot(\gamma_0,\Gamma_0)} H_{X,1}(\gamma)\textbf{1}_{\ell(\Gamma-\gamma)\leq 2\ell(\gamma)+k\ell}\right]\\
        \prec &\frac{{n\choose a_1,\cdots,a_s,b_1,\cdots b_r}}{V_{g,n}}\prod_{i=1}^s V_{g_i,n_i+a_i}\prod_{j=1}^r V_{h_j,m_j+b_j}\cdot e^{2k\ell+\frac{9}{2}T}\cdot (T+k\ell).
    \end{align*}
    If $\sum_{i=1}^s(2g_i+n_i+a_i-2)\leq 60(k+1)$, this estimate can be improved to be \begin{align*}
         &\Egn\left[\sum_{\gamma_0\subset \Gamma_0\setminus\cup_{j=1}^r \partial V_j}
\sum_{(\gamma,\Gamma)\in \overline{\Mod}_{g,n}\cdot(\gamma_0,\Gamma_0)} H_{X,1}(\gamma)\textbf{1}_{\ell(\Gamma-\gamma)\leq 2\ell(\gamma)+k\ell}\right]\\
        \prec &\frac{{n\choose a_1,\cdots,a_s,b_1,\cdots b_r}}{V_{g,n}}\prod_{i=1}^s V_{g_i,n_i+a_i}\prod_{j=1}^r V_{h_j,m_j+b_j}\cdot e^{\frac{T+k\ell}{2}}(T+k\ell)^{360k+362}.
    \end{align*}
\end{lemma}
\begin{proof}
     Similar to the proof of Lemma \ref{lemma N ell k for lemma for case 1}, for any fixed choice of $\gamma_0\subset \Gamma_0\setminus\cup_{j=1}^r \partial V_j$, we let $x_p$ be the length of the simple closed geodesic $\gamma_p$ in $\Gamma\setminus \gamma$ for $p$ in the index set $I$, and let $x_0$ be the length of $\gamma$. We have $\sum_{p\in I}x_p\leq 2x_0+k\ell$.
     Let $p_1^i,\cdots,p_{n_i}^i$ to be the labels of geodesics on $\partial S_i$, and $q_1^j,\cdots,q_{m_j}^j$ to be the labels of geodesics on $\partial V_j$. 
     Take the estimate \eqref{over N ell k case 1  eq 3} induced by Lemma \ref{lemma NWX vgn(x)} into Mirzakhani's Integral Formula Theorem \ref{thm mir int formula}, we have \begin{align*}
         &\Egn\left[
\sum_{(\gamma,\Gamma)\in \overline{\Mod}_{g,n}\cdot(\gamma_0,\Gamma_0)} H_{X,1}(\gamma)\textbf{1}_{\ell(\Gamma-\gamma)\leq 2\ell(\gamma)+k\ell}\right]\\
\prec&\frac{{n\choose a_1,\cdots,a_s,b_1,\cdots,b_r}}{V_{g,n}}\int_{\substack{x_0\leq T\\
\sum_{p\in I}x_p\leq 2x_0+k\ell}}H_{X,1}(x_0)\prod_{i=1}^sV_{g_i,n_i+a_i}(x_{p_1^i},\cdots,x_{p_{n_i}^i},0^{a_i})\\
\cdot&\prod_{j=1}^r V_{h_r,m_j+b_j}(x_{q_1^j},\cdots,x_{q_{m_j}^j},0^{b_j}) x_0\prod_{p\in I}x_p \cdot dx_0\prod_{p\in I} dx_p\\ 
\nonumber\prec&\frac{{n\choose a_1,\cdots,a_s,b_1,\cdots b_r}}{V_{g,n}}\prod_{i=1}^s V_{g_i,n_i+a_i}\prod_{j=1}^r V_{h_j,m_j+b_j}\\
    \nonumber\cdot&\int_{x_0\leq T} 2\sinh \frac{x_0}{2}f_T(x_0)  
   \left( \int_{\substack{ \sum_{p\in I}x_p\leq k\ell+2x_0}}
    e^{\sum_{p\in I}x_i} \prod_{p\in I}x_p \prod_{p\in I}dx_p \right)dx_0 \\
   \nonumber\prec & \frac{{n\choose a_1,\cdots,a_s,b_1,\cdots b_r}}{V_{g,n}}\prod_{i=1}^s V_{g_i,n_i+a_i}\prod_{j=1}^r V_{h_j,m_j+b_j}\\ 
\cdot&  \int_{x_0\leq T}2\sinh\frac{x_0}{2}f_T(x_0)\cdot e^{2k\ell+4x_0}dx_0\\
  \nonumber\prec & \frac{{n\choose a_1,\cdots,a_s,b_1,\cdots b_r}}{V_{g,n}}\prod_{i=1}^s V_{g_i,n_i+a_i}\prod_{j=1}^r V_{h_j,m_j+b_j}\cdot e^{2k\ell+\frac{9}{2}T}.
     \end{align*}
     Since $\gamma_0\subset \Gamma_0\setminus\cup_{j=1}^r \partial V_j$ has at most $\sum_{i=1}^s n_i\leq 3\sum_{i=1}^s(2g_i+n_i+a_i-2)\leq 3k+\frac{6k\ell+6T}{\pi}$ choices, the first estimate in the lemma follows.

Now if $\sum_{i=1}^s(2g_i+n_i+a_i-2)\leq 60(k+1)$, then by Theorem \ref{mir07 poly} we have $$
V_{g_i,n_i+a_i}(x_{x_{p_1}^i},\cdots,x_{p_{n_i}^i},0^{a_i})\prec V_{g_i,n_i+a_i}\left(1+x_0+\sum_{p\in I}x_p  \right)^{6g_i-6+2n_i+2a_i},
$$
and thus \begin{equation}\label{N ell k for lemma for case 2.5 eq 1}
\begin{aligned}
   & \prod_{i=1}^s V_{g_i,n_i+a_i}(x_{p_1}^i,\cdots,x_{p_{n_i}^i},0^{a_i})\\
  \prec& \prod_{i=1}^s V_{g_i,n_i+a_i}\cdot \left(1+x_0+\sum_{p\in I}x_p\right)^{\sum_{i=1}^s(6g_i-6+2n_i+2a_i)}\\
   \prec &\prod_{i=1}^s V_{g_i,n_i+a_i}\cdot\left(1+x_0+\sum_{p\in I}x_p\right)^{180k+180}.
\end{aligned}
\end{equation}
Since different $V_j$ have no common boundary geodesics in $\Gamma_0$, and $\gamma_0\subset \Gamma_0\setminus \cup_{j=1}^r \partial V_j$, by Lemma \ref{lemma NWX vgn(x)} we have \begin{equation}\label{N ell k for lemma for case 2.5 eq 2}
\begin{aligned}
   & \prod_{j=1}^r V_{h_r,m_j+b_j}(x_{q_1^j},\cdots,x_{q_{m_j}^j},0^{b_j})\\
   \leq &\prod_{j=1}^r \left( V_{h_r,m_j+b_j}\cdot \prod_{i=1}^{m_j}\frac{2\sinh\frac{x_{q_i^j}}{2}}{x_{q_i^j}}\right)\\
   \prec &\prod_{j=1}^r V_{h_r,m_j+b_j}\cdot \prod_{p\in I}\frac{2\sinh\frac{x_p}{2}}{x_p}.
\end{aligned}
\end{equation}
If $\sum_{i=1}^s(2g_i+n_i+a_i-2)\leq 60(k+1)$, we have
 $\left|I\right|\leq 3\sum_{i=1}^s (2g_i+n_i+a_i-2)\leq 180k+180$.
It follows from \eqref{N ell k for lemma for case 2.5 eq 1} and \eqref{N ell k for lemma for case 2.5 eq 2} that \begin{align*}
    &\Egn\left[
\sum_{(\gamma,\Gamma)\in \overline{\Mod}_{g,n}\cdot(\gamma_0,\Gamma_0)} H_{X,1}(\gamma)\textbf{1}_{\ell(\Gamma-\gamma)\leq 2\ell(\gamma)+k\ell}\right]\\
\prec&\frac{{n\choose a_1,\cdots,a_s,b_1,\cdots,b_r}}{V_{g,n}}\int_{\substack{x_0\leq T\\
\sum_{p\in I}x_p\leq 2x_0+k\ell}}H_{X,1}(x_0)\prod_{i=1}^sV_{g_i,n_i+a_i}(x_{p_1^i},\cdots,x_{p_{n_i}^i},0^{a_0})\\
\cdot&\prod_{j=1}^r V_{h_r,m_j+b_j}(x_{q_1^j},\cdots,x_{q_{m_j}^j},0^{b_j}) x_0\prod_{p\in I}x_p \cdot dx_0\prod_{p\in I} dx_p\\ 
\prec&\frac{{n\choose a_1,\cdots,a_s,b_1,\cdots b_r}}{V_{g,n}}\prod_{i=1}^s V_{g_i,n_i+a_i}\prod_{j=1}^r V_{h_j,m_j+b_j}
\cdot\int_{\substack{x_0\leq T\\
\sum_{p\in I}\leq x_p\leq 2x_0+k\ell}}
\frac{x_0f_T(x_0)}{\sinh\frac{x_0}{2}}\\
\cdot&\left(1+x_0+\sum_{p\in I}x_p\right)^{180k+180}\cdot x_0\cdot \prod_{p\in I}\left(2\sinh\frac{x_p}{2}\right)\cdot dx_0\prod_{p\in I}dx_p\\
\prec&\frac{{n\choose a_1,\cdots,a_s,b_1,\cdots b_r}}{V_{g,n}}\prod_{i=1}^s V_{g_i,n_i+a_i}\prod_{j=1}^r V_{h_j,m_j+b_j}\cdot (T+k\ell)^{180k+180}\\
\cdot&\int_{x_0\leq T} \frac{x_0^2f_T(x_0)}{\sinh\frac{x_0}{2}} e^{\frac{2x_0+k\ell}{2}}(2x_0+k\ell)^{\left|I\right|}  dx_0\\
\prec&\frac{{n\choose a_1,\cdots,a_s,b_1,\cdots b_r}}{V_{g,n}}\prod_{i=1}^s V_{g_i,n_i+a_i}\prod_{j=1}^r V_{h_j,m_j+b_j}\cdot e^{\frac{T+k\ell}{2}}(T+k\ell)^{360k+362}.
\end{align*}
     Since $\gamma_0\subset \Gamma_0\setminus\cup_{j=1}^r \partial V_j$ has at most $180k+180$ choices, the second estimate in the lemma follows.
\end{proof}

\begin{lemma}\label{lemma N ell k final for case 2.5}
    If $n=o(g^{\frac{1}{2}-\epsilon})$ for some $\epsilon>0$, then for $k\geq 1$,  \begin{align*}
         &\Egn\left[\sum_{\substack{\gamma\in \mathcal{P}_{nsep}^s(X),\ell_\gamma(X)\leq T\\
       (Y_1,\cdots,Y_k)\in  N_\ell(X)^k\\
     \textit{of case 3}}}H_{X,1}(\gamma)\right]\\
 \nonumber    \prec& \left(\left(k+\frac{2k\ell+2T}{\pi}\right)\frac{ e^{\ell+6}+3}{2}\right)^{3k}\left[e^{\frac{T+k\ell}{2}}(T+k\ell)^{363k+365}\frac{\log^{14} g(\log g+n)^3}{g^2}  \right.\\
 +&\left.e^{2k\ell+\frac{9}{2}T}(T+k\ell)^{3k+4}\frac{\log^{427} g(\log g+n)^{62}}{g^{61}}\right].
    \end{align*}
\end{lemma}
\begin{proof}
    By \eqref{N ell k case 2.5 main ineq} we have \begin{equation}\label{lemma N ell k final for case 2.5 eq 1}
    \begin{aligned}
         &\Egn\left[\sum_{\substack{
\gamma\in \mathcal{P}_{nsep}^s(X),\ell_\gamma(X)\leq T\\
(Y_1,\cdots,Y_k)\in \mathcal{N}_\ell^k\\
\textit{of case 3}
     }} H_{X,1}(\gamma)\right]\\
\prec& \left(\left(k+\frac{2k\ell+2T}{\pi}\right)\frac{ e^{\ell+6}+3}{2}\right)^{3k}\!\!\!\sum_{1\leq s\leq k+1}\sum_{\substack{S\in \textbf{Ind}_3\\\sum_{i=1}^s(2g_i+n_i+a_i-2)\leq 60k+60}}\\
&\sum_{ \mathcal{B}_S^{s,(g,n)}}
\Egn\left[\sum_{\gamma_0\subset \Gamma_0\setminus\cup_{j=1}^r \partial V_j}
\sum_{(\gamma,\Gamma)\in \overline{\Mod}_{g,n}\cdot(\gamma_0,\Gamma_0)}\!\!\!\!\!\!\!\!\!\!\!\!\! H_{X,1}(\gamma)\textbf{1}_{\ell(\Gamma-\gamma)\leq 2\ell(\gamma)+k\ell}\right]\\
+&\left(\left(k+\frac{2k\ell+2T}{\pi}\right)\frac{ e^{\ell+6}+3}{2}\right)^{3k}\!\!\!\sum_{1\leq s\leq k+1}\sum_{\substack{S\in \textbf{Ind}_3\\\sum_{i=1}^s(2g_i+n_i+a_i-2)> 60k+60}}\\
&\sum_{ \mathcal{B}_S^{s,(g,n)}}
\Egn\left[\sum_{\gamma_0\subset \Gamma_0\setminus\cup_{j=1}^r \partial V_j}
\sum_{(\gamma,\Gamma)\in \overline{\Mod}_{g,n}\cdot(\gamma_0,\Gamma_0)}\!\!\!\!\!\!\!\!\!\!\!\!\! H_{X,1}(\gamma)\textbf{1}_{\ell(\Gamma-\gamma)\leq 2\ell(\gamma)+k\ell}\right].
 \end{aligned}
 \end{equation}

By Lemma \ref{lemma N ell k for lemma for case 2.5}, we have \begin{equation}\label{lemma N ell k final for case 2.5 eq 2}
\begin{aligned}
&\sum_{1\leq s\leq k+1}\sum_{\substack{S\in \textbf{Ind}_3\\\sum_{i=1}^s(2g_i+n_i+a_i-2)\leq 60k+60}}\sum_{F\in \mathcal{B}_S^{s,(g,n)}}\\
&\Egn\left[\sum_{\gamma_0\subset \Gamma_0\setminus\cup_{j=1}^r \partial V_j}
\sum_{(\gamma,\Gamma)\in \overline{\Mod}_{g,n}\cdot(\gamma_0,\Gamma_0)} H_{X,1}(\gamma)\textbf{1}_{\ell(\Gamma-\gamma)\leq 2\ell(\gamma)+k\ell}\right]\\
\prec&e^{\frac{T+k\ell}{2}}(T+k\ell)^{360k+362}\cdot \sum_{1\leq s\leq k+1}\sum_{\substack{S\in \textbf{Ind}_3\\\sum_{i=1}^s(2g_i+n_i+a_i-2)\leq 60k+60}}\\
&\sum_{F\in \mathcal{B}_S^{s,(g,n)}}\frac{{n\choose a_1,\cdots,a_s,b_1,\cdots b_r}}{V_{g,n}}\prod_{i=1}^s V_{g_i,n_i+a_i}\prod_{j=1}^r V_{h_j,m_j+b_j}.
\end{aligned}
\end{equation}
Since $$
\{S\in \textbf{Ind}_3;\sum_{i=1}^s(2g_i+n_i+a_i-2)\leq 60 k+60\}\subset \textbf{Ind}_1
$$
for large $g$, by the estimates in the proof of Lemma \ref{lemma N ell k final for case 1}, we have \begin{equation}\label{lemma N ell k final for case 2.5 eq 3}
\begin{aligned}
    &\sum_{1\leq s\leq k+1}\sum_{\substack{S\in \textbf{Ind}_3\\\sum_{i=1}^s(2g_i+n_i+a_i-2)\leq 60k+60}}\\
&\sum_{F\in \mathcal{B}_S^{s,(g,n)}}\frac{{n\choose a_1,\cdots,a_s,b_1,\cdots b_r}}{V_{g,n}}\prod_{i=1}^s V_{g_i,n_i+a_i}\prod_{j=1}^r V_{h_j,m_j+b_j}\\
\prec&(k\ell)^{3k+3}\frac{\log^{14} g(\log g+n)^3}{g^2}.
\end{aligned}
\end{equation}
Again by Lemma \ref{lemma N ell k for lemma for case 2.5}, we have \begin{equation}\label{lemma N ell k final for case 2.5 eq 4}
\begin{aligned}
&\sum_{1\leq s\leq k+1}\sum_{\substack{S\in \textbf{Ind}_3\\\sum_{i=1}^s(2g_i+n_i+a_i-2)> 60k+60}}\sum_{F\in \mathcal{B}_S^{s,(g,n)}}\\
&\Egn\left[\sum_{\gamma_0\subset \Gamma_0\setminus\cup_{j=1}^r \partial V_j}
\sum_{(\gamma,\Gamma)\in \overline{\Mod}_{g,n}\cdot(\gamma_0,\Gamma_0)} H_{X,1}(\gamma)\textbf{1}_{\ell(\Gamma-\gamma)\leq 2\ell(\gamma)+k\ell}\right]\\
\prec&e^{2k\ell+\frac{9}{2}T}\cdot (T+k\ell)\cdot \sum_{1\leq s\leq k+1}\sum_{\substack{S\in \textbf{Ind}_3\\\sum_{i=1}^s(2g_i+n_i+a_i-2)> 60k+60}}\\
&\sum_{F\in \mathcal{B}_S^{s,(g,n)}}\frac{{n\choose a_1,\cdots,a_s,b_1,\cdots b_r}}{V_{g,n}}\prod_{i=1}^s V_{g_i,n_i+a_i}\prod_{j=1}^r V_{h_j,m_j+b_j}.
\end{aligned}
\end{equation}
If $\sum_{i=1}^s(2g_i+n_i+a_i-2)>60k+60$ and $s\leq k+1$, there exists some $i_0$ such that $2g_{i_0}+n_{i_0}+a_{i_0}-2> 60$. We can assume $i_0=1$. It follows Lemma \ref{appendix lemma k Si product} that \begin{equation}\label{lemma N ell k final for case 2.5 eq 5}
\begin{aligned}
   & \sum_{s=1}^{k+1}\sum_{\substack{S\in \textbf{Ind}_3\\\sum_{i=1}^s(2g_i+n_i+a_i-2)> 60k+60}}\\
&\sum_{F\in \mathcal{B}_S^{s,(g,n)}}\frac{{n\choose a_1,\cdots,a_s,b_1,\cdots b_r}}{V_{g,n}}\prod_{i=1}^s V_{g_i,n_i+a_i}\prod_{j=1}^r V_{h_j,m_j+b_j}\\
\prec&\sum_{s=1}^{k+1}\sum_{\substack{2g_1+n_1+a_1-2>60\\
1\leq n_i,1\leq 2g_i+n_i+a_i-2\\
\sum_{i=1}^s(2g_i+n_i+a_i-2)\leq k+\frac{2k\ell+2T}{\pi}
}}\left((2g_1+n_1+a_1)!\right)^2\\
\cdot&(C\log g)^{2n_1}\frac{(Cn+C\log g)^{a_1}}{g^{2g_1+n_1+a_1-2}}\\
\prec&\sum_{s=1}^{k+1}\left(k+\frac{2k\ell+2T}{\pi}\right)^{3s}\\
\cdot&\sum_{N=61}^{k+\left[\frac{2k\ell+2T}{\pi}\right]}\left((N+2)!\right)^2(C\log g)^{7N}\frac{(Cn+C\log g)^{N+1}}{g^N}\\
\prec&(T+k\ell)^{3k+3}\frac{\log^{427} g(\log g+n)^{62}  }{g^{61}}.
\end{aligned}
\end{equation}
The lemma is proved by combining \eqref{lemma N ell k final for case 2.5 eq 1}, \eqref{lemma N ell k final for case 2.5 eq 2}, \eqref{lemma N ell k final for case 2.5 eq 3}, \eqref{lemma N ell k final for case 2.5 eq 4} and \eqref{lemma N ell k final for case 2.5 eq 5}.
\end{proof}

\subsection{Case 4: $\gamma\cap \tilde{S}_{i}=\emptyset$ for all $i$} \label{subsec est N ell k with H case 4}
Firstly, we assume $k\geq 2$. 
In this case, we set $s=t+1\geq 2$, $S_i=\tilde{S}_i$ for $i=1,\cdots,t$ and $S_{t+1}=\gamma$ of type $S_{0,2,0}$. If each $S_i$ for $i\leq s-1$ is of type $S_{0,1,2}$, since $\gamma$ is non-separating in $X_{g,n}$, then $\gamma$ is also non-separating in $X_{g,n}\setminus\cup_{i=1}^{s-1}S_i\simeq X_{g,n-s+1}$ and $(\gamma,Y_1,\cdots,Y_k)$ is of the leading type. By the same argument as case 1 in subsection \ref{subsection case 1 gamma subset S i_0}, for $F\in \mathcal{B}_S^{s,(g,n)}$, if
we set $\Gamma_0=\cup_{i=1}^{s-1}f_i(\partial S_{g_i,n_i+a_i})$ for $i\leq s-1$ and $\gamma_0=f_s(S_{0,2,0})$, we will have  \begin{equation}\label{N ell k case 3 main ineq}
\begin{aligned}
    &\sum_{\substack{\gamma\in \mathcal{P}_{nsep}^s(X),\ell_\gamma(X)\leq T\\
       (Y_1,\cdots,Y_k)\in  N_\ell(X)^k\\
     \textit{of case 4, not leading type}}}H_{X,1}(\gamma)\\
 \prec&\left(\left(k+\frac{k\ell}{\pi}\right)\frac{ e^{\ell+6}+3}{2}\right)^{3k}\sum_{2\leq s\leq k+1}\sum_{S\in \textbf{Ind}_4^{\geq 2}}\\
     &\sum_{F\in\mathcal{B}_S^{s,(g,n)} }\sum_{\substack{\
(\gamma,\Gamma)\in\overline{\Mod}_{g,n}\cdot (\gamma_0,\Gamma_0)
     }}H_{X,1}(\gamma)\textbf{1}_{[0,k\ell]}(\ell(\Gamma)),
\end{aligned}
\end{equation}
where $S$ is over \begin{align*}
&\textbf{Ind}_4^{\geq 2}=\left\{\right.S=\{(g_i,n_i,a_i)\}_{i=1}^s; (g_s,n_s,a_s)=(0,2,0),1\leq n_i,\\
&\{(g_i,n_i,a_i)\}_{i=1}^{s-1} \neq (0,1,2)^{s-1};
1\leq 2g_i+n_i+a_i-2, \textit{ for }i\leq s-1,\\
&\left.2\leq \sum_{i=1}^{s-1}(2g_i+n_i+a_i-2)\leq k+\frac{k\ell}{\pi}   \right\}.  
\end{align*}

\begin{lemma}\label{lemma N ell k for lemma for case 3}
   Fix $2\leq s\leq k+1$ and $k\geq 2$. For any $S\in \textbf{Ind}_4^{\geq 2}$, $F\in\mathcal{B}_S^{s,(g,n)}$, $\Gamma_0=\cup_{i=1}^{s-1}f_i(\partial S_{g_i,n_i+a_i})$ and $\gamma_0=f_s(S_{0,2,0})$, the following estimate holds: \begin{align*}
        &\Egn\left[\sum_{(\gamma,\Gamma)\in \overline{\Mod}_{g,n}\cdot(\gamma_0,\Gamma_0)}H_{X,1}(\gamma)\textbf{1}_{[0,k\ell]}(\ell(\Gamma))\right]\\
        \prec &\frac{{n\choose a_1,\cdots,a_s,b_1,\cdots b_r}}{V_{g,n}}\prod_{i=1}^{s-1} V_{g_i,n_i+a_i}\prod_{j=1}^r V_{h_j,m_j+b_j} \cdot 
   e^{\frac{T}{2}+2k\ell}.
    \end{align*}
\end{lemma}
\begin{proof}
    Similar to the proof of Lemma \ref{lemma N ell k for lemma for case 1}, let $x_p$ be the length of the simple closed geodesic $\gamma_p$ in $\Gamma$ for $p$ in the index set $I$, and let $x_0$ be the length of $\gamma$. Let $p_1^i,\cdots,p_{n_i}^i$ be the labels of geodesics on $\partial S_i$, and let $q_1^j,\cdots,q_{m_j}^j$ be the labels of geodesics on $\partial V_j$. Then we have  \begin{align*}
    &\Egn\left[\sum_{(\gamma,\Gamma)\in \overline{\Mod}_{g,n}\cdot (\gamma_0,\Gamma_0)} H_{X,1}(\gamma)\textbf{1}_{[0,k\ell]}(\ell(\Gamma))\right]\\
    \nonumber\prec& \frac{{n\choose a_1,\cdots,a_s,b_1,\cdots b_r}}{V_{g,n}} \int_{\substack{ 
    x_0\leq T\\
\sum_{p\in I}x_p\leq k\ell
    }} H_{X,1}(x_0) \prod_{i=1}^{s-1} V_{g_i,n_i+a_i}(x_{p^i_1},\cdots ,x_{p^{i}_{n_i}},0^{a_i})\\
    \nonumber\cdot&\prod_{j=1}^r V_{h_j,m_j+b_j}(x_{q^{j}_1},\cdots,x_{q^{j}_{m_j}},0^{b_j})x_0\prod_{p\in I} x_p\cdot  dx_0\prod_{p\in I}dx_p\\
   \nonumber\prec&\frac{{n\choose a_1,\cdots,a_s,b_1,\cdots b_r}}{V_{g,n}}\prod_{i=1}^{s-1} V_{g_i,n_i+a_i}\prod_{j=1}^r V_{h_j,m_j+b_j}\\
    \nonumber\cdot&\int_{x_0\leq T} 2\sinh \frac{x_0}{2}f_T(x_0)  dx_0 
   \int_{\substack{ \sum_{p\in I}x_p\leq k\ell}}
    e^{\sum_{p\in I}x_p} \prod_{p\in I}x_p \prod_{p\in I}dx_p \\
   \nonumber\prec & \frac{{n\choose a_1,\cdots,a_s,b_1,\cdots b_r}}{V_{g,n}}\prod_{i=1}^{s-1} V_{g_i,n_i+a_i}\prod_{j=1}^r V_{h_j,m_j+b_j} e^{\frac{T}{2}+2k\ell}.
\end{align*}
Here we use the estimates \eqref{over N ell k case 1  eq 3} and \eqref{over N ell k case 1  eq 5} again.
\end{proof}

\begin{lemma}\label{lemma N ell k geq 2 final for case 3}
    If $k\geq 2$ and $n=o(g^{\frac{1}{2}-\epsilon})$ for some $\epsilon>0$, then \begin{align*}
        &\Egn\left[\sum_{\substack{\gamma\in \mathcal{P}_{nsep}^s(X),\ell_\gamma(X)\leq T\\
       (Y_1,\cdots,Y_k)\in  N_\ell(X)^k\\
     \textit{of case 4, not leading type}}}H_{X,1}(\gamma)\right]\\
 \nonumber    \prec&\left(\left(k+\frac{k\ell}{\pi}\right)\frac{ e^{\ell+6}+3}{2}\right)^{3k}\frac{\log ^{14}g(n+\log g)^3}{g^2}(k\ell)^{3k}e^{\frac{T}{2}+2k\ell}.
    \end{align*}
\end{lemma}
\begin{proof}
    For $S\in \textbf{Ind}_4^{\geq 2}$, 
    if $2g_i+n_i+a_i-2\geq 2$ for some $i$, by Lemma \ref{lemma N ell k for lemma for case 3} Lemma \ref{appendix lemma k Si product}, we have \begin{equation}\label{eq N ell k final for case 3-1}
    \begin{aligned}
    &\Egn\left[\sum_{F\in\mathcal{B}_S^{s,(g,n)} }\sum_{\substack{\
(\gamma,\Gamma)\in\overline{\Mod}_{g,n}\cdot (\gamma_0,\Gamma_0)
     }}H_{X,1}(\gamma)\textbf{1}_{[0,k\ell]}(\ell(\Gamma))\right]\\
  \prec& e^{\frac{T}{2}+2k\ell}\cdot((2g_i+n_i+a_i)!)^2 (C\log g)^{2n_i}\frac{(Cn+C\log g)^{a_i}}{g^{2g_i+n_i+a_i-2}}\\
  \prec& e^{\frac{T}{2}+2k\ell}\cdot  ((N+2)!)^2 (C\log g)^{7N}\frac{(Cn+C\log g)^{N+1}}{g^{N}}
    \end{aligned}
    \end{equation}
    for $N=2g_i+n_i+a_i-2\geq 2$.
    If $2g_i+n_i+a_i-2=1$ for all $1\leq i\leq s-1$, then  there exists some $i_0$ such that $a_{i_0}\leq 1$. Assume $i_0=1$, then by Lemma \ref{lemma N ell k for lemma for case 3} and Lemma \ref{appendix lemma k Si product}, we have \begin{equation}\label{eq N ell k final for case 3-2}
         \begin{aligned}&\Egn\left[\sum_{F\in\mathcal{B}_S^{s,(g,n)} }\sum_{\substack{\
     (\gamma,\Gamma)\in\overline{\Mod}_{g,n}\cdot (\gamma_0,\Gamma_0)
     }}H_{X,1}(\gamma)\textbf{1}_{[0,k\ell]}(\ell(\Gamma))\right]\\
   \prec& e^{\frac{T}{2}+2k\ell}\cdot \frac{\log^{14} g(n+\log g)^{a_1+a_2}}{g^2}\\
   \prec&  e^{\frac{T}{2}+2k\ell}\cdot \frac{\log^{14} g(n+\log g)^{3}}{g^2}.
    \end{aligned}
    \end{equation}
     Combining \eqref{eq N ell k final for case 3-1} and \eqref{eq N ell k final for case 3-2}, we have \begin{equation}\label{eq N ell k final for case 3-3}
     \begin{aligned}
        & \Egn\left[\sum_{2\leq s\leq k+1}\sum_{S\in \textbf{Ind}_4^{\geq 2}}\sum_{F\in\mathcal{B}_S^{s,(g,n)} }\sum_{\substack{
(\gamma,\Gamma)\in\overline{\Mod}_{g,n}\cdot (\gamma_0,\Gamma_0)
     }} \!\!\!\!\!\!\!\!\!\!\!\!\!\!\!\!  H_{X,1}(\gamma)\textbf{1}_{[0,k\ell]}(\ell(\Gamma))\right]\\
    \prec &e^{\frac{T}{2}+2k\ell}\sum_{2\leq s\leq k+1}\left|\textbf{Ind}_4^{\geq 2}\right|\left(\sum_{N=2}^{\left[k+\frac{k\ell}{\pi}\right]} ((N+2)!)^2 (C\log g)^{7N}\right.\\
\cdot&\left.\frac{(Cn+C\log g)^{N+1}}{g^{N}}+ \frac{\log^{14} g(n+\log g)^3}{g^2}\right)\\
 \prec &e^{\frac{T}{2}+2k\ell} \cdot(k\ell)^{3k}\cdot \frac{\log^{14} g(n+\log g)^3}{g^2}.
     \end{aligned}
     \end{equation}
The lemma follows from \eqref{N ell k case 3 main ineq} and \eqref{eq N ell k final for case 3-3}.
\end{proof}

When $k=1$, the contribution of terms for $(\gamma, Y_1)$ in case $4$ is not a small remainder term for our purpose to prove Theorem \ref{thm main-1}.  We classify the possible mapping class group orbits of $(\gamma, Y_1)$. Firstly, if $Y_1$ has two cusps, then $X\setminus Y_1\simeq S_{g,n-1}$. Since $\gamma$ is non-separating in $X$, it is also non-separating in $X\setminus Y_1$. This corresponds to the leading type. So $Y_1$ has at most one cusp.

\begin{lemma}\label{lemma n ell k=1 with nsep s11}
    If $n=o(g^{\frac{1}{2}})$, then  \begin{align*}
        &\Egn\left[\sum_{\substack{(\gamma,Y_1)\in \mathcal{P}_{nsep}^s(X)\times N_\ell(X)\\
     \textit{of case 4}, Y_1\simeq S_{1,1}
     }}H_{X,1}(\gamma)\right]\\
     =&\frac{1}{4\pi^2g}\int_{x\leq \ell} \sinh\frac{x}{2}V_{1,1}(x)dx\int_{y\leq T} \sinh\frac{y}{2}f_T(y)dy \left(1+O\left(\frac{(1+n)(T^2)}{g}\right)\right).
    \end{align*}
\end{lemma}
\begin{proof}
    If $Y_1\simeq S_{1,1}$, then $X\setminus Y_1\simeq S_{g-1,n+1}$, and $\gamma$ is non-separating in $X\setminus Y_1$. So $X\setminus(\gamma\cup Y_1)\simeq S_{g-2,n+3}$. By Mirzakhani's Integral Formula Theorem \ref{thm mir int formula}, Theorem \ref{thm MZ15 1/g} and  Lemma \ref{lemma NWX vgn(x)}, we have  \begin{align*}
       &\Egn\left[\sum_{\substack{(\gamma,Y_1)\in \mathcal{P}_{nsep}^s(X)\times N_\ell(X)\\
     \textit{of case 4}, Y_1\simeq S_{1,1}
     }}H_{X,1}(\gamma)\right]\\
     =&2\cdot \frac{1}{4}\frac{1}{V_{g,n}}\int_{x\leq \ell,y\leq T} V_{1,1}(x)V_{g-2,n+3}(x,y,y,0^n) \frac{yf_T(y)}{2\sinh \frac{y}{2}}xydxdy\\
     =&\frac{V_{g-2,n+3}}{2V_{g,n}}\int_{x\leq \ell,y\leq T} \frac{2\sinh\frac{x}{2}}{x}\left(\frac{2\sinh\frac{y}{2}}{y}\right)^2 \left(1+O\left(\frac{(1+n)(x^2+y^2)}{g}\right)\right)\\
     \cdot&V_{1,1}(x) \frac{yf_T(y)}{2\sinh \frac{y}{2}}xydxdy\\
     =&\frac{1}{4\pi^2g}\int_{x\leq \ell} \sinh\frac{x}{2}V_{1,1}(x)dx\int_{y\leq T} \sinh\frac{y}{2}f_T(y)dy \left(1+O\left(\frac{(1+n)(T^2)}{g}\right)\right).
    \end{align*}
The factor $2$ here comes from the two orientations of $\gamma$ and $\frac{1}{4}$ is $C_\Gamma$ in Theorem \ref{thm mir int formula}.
\end{proof}

\begin{lemma}\label{lemma n ell k=1 with nsep s03 1 cusp}
      If $n=o(g^{\frac{1}{2}})$, then\begin{align*}
        &\Egn\left[\sum_{\substack{(\gamma,Y_1)\in \mathcal{P}_{nsep}^s(X)\times N_\ell(X)\\
     \textit{of case 4}, Y_1\simeq S_{0,3} \textit{ with one cusp}
     }}H_{X,1}(\gamma)\right]\\
     =&\frac{n}{2\pi^2g}\int_{x_1+x_2\leq \ell}\sinh\frac{x_1}{2}\sinh\frac{x_2}{2}dx_1dx_2\int_{0}^T\sinh\frac{y}{2}f_T(y)dy\\
     \cdot &\left(1+O\left(\frac{nT^2}{g}\right)\right)+O\left(\frac{n^3}{g^2} \cdot\ell\cdot e^{\frac{\ell+T}{2}}\right).
    \end{align*}
\end{lemma}

\begin{proof}
    If $Y_1\simeq S_{0,3}$ with $1$ cusp and $X\setminus Y_1$ is connected, then $X\setminus Y_1\simeq S_{g-1,n+1}$. If $\gamma$ is separating in $X\setminus Y_1$, since $\gamma$ is non-separating in $X$, the two boundary geodesics of $X\setminus Y_1$ will not be in the same component of $X\setminus(\gamma\cup Y_1)$. We can assume $X\setminus(\gamma\cup Y_1)\simeq S_{g_1,n_1+2}\cup S_{g_2,n_2+2}$ for some $g_1+g_2=g-1$ and $n_1+n_2=n-1$. Here $n_i$ is the number of cusps for the component $S_{g_i,n_i+2}$.
    If $\gamma$ is non-separating in $X\setminus Y_1$, then $X\setminus(\gamma\cup Y_1)\simeq S_{g-2,n+3}$.  By Mirzakhani Integral Formula Theorem \ref{thm mir int formula}, Theorem \ref{thm MZ15 1/g} and Lemma \ref{lemma NWX vgn(x)}, we have  \begin{equation}\label{n ell k=1 with nsep s03 1 cusp eq 1}
       \begin{aligned}&\Egn\left[\sum_{\substack{(\gamma,Y_1)\in \mathcal{P}_{nsep}^s(X)\times N_\ell(X)\\
     \textit{of case 4}, Y_1 \textit{ has one cusp}\\
     X\setminus Y_1\cup \gamma\textit{ is connected}
     }}H_{X,1}(\gamma)\right]\\
     =&\frac{1}{4}\cdot 2\cdot \frac{{n\choose 1}}{V_{g,n}}\int_{\substack{x_1+x_2\leq \ell\\y\leq T}} V_{0,3}(x_1,x_2,0)V_{g-2,n+3}(x_1,x_2,y,y,0^{n-1}) \\
   \cdot&\frac{yf_T(y)}{2\sinh\frac{y}{2}} x_1x_2ydx_1dx_2dy\\
  =&\frac{n V_{g-2,n+3}}{2V_{g,n}}\int_{\substack{x_1+x_2\leq \ell\\y\leq T}}\left(1+O\left(\frac{(1+n)(x_1^2+x_2^2+y^2)}{g}\right)\right)\\
   \cdot& 8\sinh\frac{x_1}{2}\sinh\frac{x_2}{2}\sinh\frac{y}{2} f_T(y)dx_1dx_2dy\\
    =&\frac{n}{2\pi^2g}\int_{x_1+x_2\leq \ell}\sinh\frac{x_1}{2}\sinh\frac{x_2}{2}dx_1dx_2\\
    \cdot&\int_{0}^T\sinh\frac{y}{2}f_T(y)dy
  \left(1\!+\!O\left(\frac{nT^2}{g}\right)\right),
\end{aligned}
\end{equation}
where ${n\choose 1}$ is the number of mapping class group orbits for $(\gamma,Y_1)$ of this case and $C_\Gamma=\frac{1}{4}$ in Theorem \ref{thm mir int formula}.
And by Lemma \ref{appendix product lemma  2i+j geq k} for $k=1$ additionally, we have
     \begin{equation}\label{n ell k=1 with nsep s03 1 cusp eq 2}\begin{aligned}&\Egn\left[\sum_{\substack{(\gamma,Y_1)\in \mathcal{P}_{nsep}^s(X)\times N_\ell(X)\\
     \textit{of case 4}, Y_1 \textit{ has one cusp}\\
     X\setminus Y_1 \textit{ is connected}\\
     X\setminus Y_1\cup \gamma\textit{ is not connected}
     }}H_{X,1}(\gamma)\right]\\
  \prec&\sum_{\substack{g_1+g_2=g-1\\n_1+n_2=n-1\\2g_i+n_i\geq1}}  \frac{{n\choose 1}{n-1\choose n_1}}{V_{g,n}}V_{g_1,n_1+2}V_{g_2,n_2+2}\\
  \cdot&\int_{\substack{x_1+x_2\leq \ell\\
     y\leq T}}\sinh\frac{x_1}{2}\sinh\frac{x_2}{2}\sinh\frac{y}{2}f_T(y)dx_1dx_2dy\\
     \prec&\frac{n}{g}\sum_{\substack{g_1+g_2=g-1\\n_1+n_2=n-1\\2g_i+n_i\geq1}}  \frac{{n-1\choose n_1}}{V_{g-1,n+1}}V_{g_1,n_1+2}V_{g_2,n_2+2}\cdot \ell\cdot e^{\frac{T+\ell}{2}}\\
     \prec& \frac{n^2}{g^2}\cdot \ell\cdot e^{\frac{T+\ell}{2}},
     \end{aligned}
     \end{equation}
where ${n\choose 1}{n-1\choose n_1}$ is the number of mapping class group orbits for $(\gamma,Y_1)$ for given $(g_1,g_2,n_1,n_2)$.
If $Y_1$ has one cusps and $X\setminus Y_1\simeq X_1\cup X_2$ for $X_i\simeq S_{g_i,n_i+1}$ is not connected, where $g_1+g_2=g$, $n_1+n_2=n-1$ and $\gamma\subset X_1$, then $g_1\geq 1$ and $X_1\setminus \gamma\simeq S_{g_1-1,n_1+3}$ is connected, since $\gamma$ is non-separating in $X$. Here $n_i$ is the number of cusps for $X_i$.
So by Mirzakhani's Integral Formula Theorem \ref{thm mir int formula}, Theorem \ref{thm mz15 asymp}, Corollary \ref{cor vgn+2 leq vg+1n}, Lemma \ref{lemma NWX vgn(x)} and Lemma \ref{lemma in hide appendix gen k} we have  \begin{equation}\label{n ell k=1 with nsep s03 1 cusp eq 3}
  \begin{aligned}     &\Egn\left[\sum_{\substack{(\gamma,Y_1)\in \mathcal{P}_{nsep}^s(X)\times N_\ell(X)\\
     \textit{of case 4}, Y_1\textit{ has one cusp}\\
     X\setminus Y_1 \textit{ is not connected}
     }}H_{X,1}(\gamma)\right]\\
    \prec&\sum_{\substack{g_1+g_2=g,g_1\geq 1\\
     n_1+n_2=n-1\\
     2g_i+n_i\geq 2}}\frac{{n\choose 1}{n-1\choose n_1}}{V_{g,n}}\int_{\substack{x_1+x_2\leq \ell\\y\leq T}} V_{g_1-1,n_1+3}(x_1,y,y,0^{n_1})\\
    \cdot&V_{g_2,n_2+1}(x_2,0^{n_2})V_{0,3}(x_1,x_2,0) \frac{yf_T(y)}{2\sinh\frac{y}{2}} \cdot x_1x_2y\cdot dx_1dx_2dy\\
    \prec&\sum_{\substack{g_1+g_2=g,g_1\geq 1\\
     n_1+n_2=n-1\\
     2g_i+n_i\geq 2}}\frac{{n\choose 1}{n-1\choose n_1}}{V_{g,n}} V_{g_1-1,n_1+3}V_{g_2,n_2+1}\\
     \cdot& \int_{\substack{x_1+x_2\leq \ell\\y\leq T}}\sinh\frac{x_1}{2}\sinh\frac{x_2}{2}\sinh\frac{y}{2}f_T(y)dx_1dx_2dy\\
    \prec &\frac{n}{g} \cdot\ell \cdot e^{\frac{\ell+T}{2}}\sum_{\substack{g_1+g_2=g,g_1\geq 1\\
     n_1+n_2=n-1\\
     2g_i+n_i\geq 2}}\frac{{n-1\choose n_1}}{V_{g,n-1}} V_{g_1,n_1+1}V_{g_2,n_2+1}\\
    \prec&\frac{n^3}{g^2} \cdot\ell\cdot e^{\frac{\ell+T}{2}},
     \end{aligned}
     \end{equation}
where ${n\choose 1}{n-1\choose n_1}$ is the number of mapping class group orbits for $(\gamma,Y_1)$ for given $(g_1,g_2,n_1,n_2)$. The lemma follows from \eqref{n ell k=1 with nsep s03 1 cusp eq 1}, \eqref{n ell k=1 with nsep s03 1 cusp eq 2} and \eqref{n ell k=1 with nsep s03 1 cusp eq 3}.
\end{proof}

\begin{lemma}\label{lemma n ell k=1 with nsep s03 0 cusp}
      If $n=o(g^{\frac{1}{2}})$, then  \begin{align*}
        &\Egn\left[\sum_{\substack{(\gamma,Y_1)\in \mathcal{P}_{nsep}^s(X)\times N_\ell(X)\\
     \textit{of case 4}, Y_1\simeq S_{0,3} \textit{ with no cusp}
     }}H_{X,1}(\gamma)\right]\\
     =&\frac{1}{3\pi^2g}\int_{x_1+x_2+x_3\leq \ell}\sinh\frac{x_1}{2}\sinh\frac{x_2}{2}\sinh\frac{x_3}{2}dx_1dx_2dx_3\\
 \nonumber    \cdot &\int_{0}^T \sinh\frac{y}{2}f_T(y)dy\left(1+O\left(\frac{(1+n)T^2}{g}\right)\right)+O\left(\frac{1+n^2}{g^2}\cdot \ell^2\cdot e^{\frac{\ell+T}{2}}\right).
    \end{align*}
\end{lemma}

\begin{proof}
     If $Y_1\simeq S_{0,3}$ has no cusp, then $X\setminus Y_1$ has at most three components. If $\gamma$ belongs to some component $X_0$ and $X_0\setminus \gamma=\tilde{X}_1\cup\tilde{X}_2$ is not connected, then both $\tilde{X}_i$ contains at least one boundary geodesic of $X_0$. In particular, if $X_0$ has only one common boundary geodesic with $Y_1$, then $\gamma$ is non-separating in $X_0$.
     
     For the case that $Y_1\simeq S_{0,3}$ with no cusp, $X\setminus Y_1\simeq S_{g-2,n+3}$, and $\gamma$ is non-sepratating in $X\setminus Y_1$,
     by Mirzakhani's Integral Formula Theorem \ref{thm mir int formula}, Theorem \ref{thm MZ15 1/g}, and Lemma \ref{lemma NWX vgn(x)}, we have
     \begin{equation}\label{n ell k=1 with nsep s03 0 cusp eq 1}
           \begin{aligned}&\Egn\left[\sum_{\substack{(\gamma,Y_1)\in \mathcal{P}_{nsep}^s(X)\times N_\ell(X)\\
     \textit{of case 4}, Y_1\simeq S_{0,3}\textit{ has no cusp}\\
     X\setminus (Y_1\cup\gamma)\simeq S_{g-3,n+5}
     }}H_{X,1}(\gamma)\right]\\
    =&2\cdot \frac{1}{12V_{g,n}}\int_{\substack{x_1+x_2+x_3\leq \ell\\
     y\leq T}}V_{0,3}(x_1,x_2,x_3)V_{g-3,n+5}(x_1,x_2,x_3,y,y,0^{n})\\
   \cdot&\frac{yf_T(y)}{2\sinh\frac{y}{2}}x_1x_2x_3ydx_1dx_2dx_3dy\\
  =&\frac{V_{g-3,n+5}}{6V_{g,n}}\int_{x_1+x_2+x_3\leq \ell}8\sinh\frac{x_1}{2}\sinh\frac{x_2}{2}\sinh\frac{x_3}{2}dx_1dx_2dx_3\\
    \cdot &\int_{0}^T 2\sinh\frac{y}{2}f_T(y)dy\left(1+O\left(\frac{(1+n)T^2}{g}\right)\right)\\
    =&\frac{1}{3\pi^2g}\int_{x_1+x_2+x_3\leq \ell}\sinh\frac{x_1}{2}\sinh\frac{x_2}{2}\sinh\frac{x_3}{2}dx_1dx_2dx_3\\
    \cdot &\int_{0}^T \sinh\frac{y}{2}f_T(y)dy\left(1+O\left(\frac{(1+n)T^2}{g}\right)\right).
      \end{aligned}
      \end{equation}
 If $X\setminus Y_1\simeq S_{g-2,n+3}$ is connected, and $\gamma$ separates $X\setminus Y_1$ into $X_1\cup X_2$, then we can assume that $X_1$ shares one boundary geodesic with $Y_1$ and $X_2$ shares two boundary geodesics with $Y_1$. Then we can assume $X_1\simeq S_{g_1,n_1+2}$ and $X_2\simeq S_{g_2,n_2+3}$ with $g_1+g_2=g-2$ and $n_1+n_2=n$. Here $n_i$ is the number of cusps for $X_i$. By Mirzakhani's Integral Formula Theorem \ref{thm mir int formula}, Theorem \ref{thm MZ15 1/g}, and Lemma \ref{lemma NWX vgn(x)}, we have
\begin{equation}\label{n ell k=1 with nsep s03 0 cusp eq 2}
\begin{aligned}
 &\Egn\left[\sum_{\substack{(\gamma,Y_1)\in \mathcal{P}_{nsep}^s(X)\times N_\ell(X)\\
     \textit{of case 4}, Y_1\simeq S_{0,3}\textit{ has no cusp}\\
     X\setminus Y_1\textit{ is connected}\\
     X\setminus (Y_1\cup\gamma)\textit{ is not connected}
     }}H_{X,1}(\gamma)\right]\\
          \prec&\sum_{\substack{g_1+g_2=g-2\\n_1+n_2=n\\2g_1+n_1\geq 1}}
     \frac{{n\choose n_1}}{V_{g,n}}\int_{\substack{x_1+x_2+x_3\leq \ell\\
     y\leq T}}V_{g_1,n_1+2}(x_1,y,0^{n_1})\\
     \cdot &V_{g_2,n_2+3}(x_2,x_3,y,0^{n_2})V_{0,3}(x_1,x_2,x_3)\cdot x_1x_2x_3y \cdot dx_1dx_2dx_3dy\\
   \prec &\sum_{\substack{g_1+g_2=g-2\\n_1+n_2=n\\2g_1+n_1\geq 1}}
     \frac{{n\choose n_1}}{V_{g,n}}V_{g_1,n_1+2}V_{g_2,n_2+3}\\
    \cdot&\int_{\substack{x_1+x_2+x_3\leq \ell\\
     y\leq T}}\sinh\frac{x_1}{2}\sinh\frac{x_2}{2}\sinh\frac{x_3}{2}\sinh\frac{y}{2}f_T(y)dx_1dx_2dx_3dy\\
   \prec& \ell^2\cdot e^{\frac{T+\ell}{2}}\cdot \sum_{\substack{g_1+g_2=g-2\\n_1+n_2=n\\2g_1+n_1\geq 1}}\frac{{n\choose n_1}V_{g_1,n_1+2}V_{g_2,n_2+3}}{V_{g,n}}\\
  \prec & \frac{1+n^2}{g^2}\cdot \ell^2\cdot e^{\frac{T+\ell}{2}}.
     \end{aligned}
     \end{equation}
Here we use Corollary \ref{cor vgn+2 leq vg+1n} and the estimate \eqref{type 6 dis connected 3} for the last inequality.

 If $X\setminus Y_1=X_1\cup X_2\cup X_3$ has three connected components with $\gamma\subset X_1$ and $X_i\simeq S_{g_i,n_i+1}$, where $g_1+g_2+g_3=g$, $n_1+n_2+n_3=n$, then $\gamma$ is non-separating in $X_1$. Here $n_i$ is the number of cusps for $X_i$. So $X_1\setminus\gamma\simeq S_{g_1-1,n_1+3}$. By similar estimates in \eqref{n ell k=1 with nsep s03 0 cusp eq 2} along with the estimate \eqref{type 6 dis connected 2}, we have \begin{equation}\label{n ell k=1 with nsep s03 0 cusp eq 3}
      \begin{aligned}&\Egn\left[\sum_{\substack{(\gamma,Y_1)\in \mathcal{P}_{nsep}^s(X)\times N_\ell(X)\\
     \textit{of case 4}, Y_1\simeq S_{0,3}\textit{ has no cusp}\\
     X\setminus Y_1\textit{ has three components}
     }}H_{X,1}(\gamma)\right]\\
    \prec&  \ell^2\cdot e^{\frac{T+\ell}{2}}\cdot \sum_{\substack{g_1+g_2+g_3=g,g_1\geq 1\\n_1+n_2+n_3=n\\2g_i+n_i\geq 2}}\frac{{n\choose n_1,n_2,n_3}V_{g_1-1,n_1+3}V_{g_2,n_2+1}V_{g_3,n_3+1}}{V_{g,n}}\\
     \prec& \frac{1+n^4}{g^3} \cdot\ell^2\cdot e^{\frac{T+\ell}{2}}.
 \end{aligned}
 \end{equation}
Now we assume that $X\setminus Y_1=X_1\cup X_2$ has two components, $X_1$ has one boundary geodesic, and $X_2$ has two boundary geodesics. if $\gamma\subset X_1$, then $X_1\setminus \gamma$ must be connected, so we can assume $X_1\simeq S_{g_1,n_1+1}$, $X_2\simeq S_{g_2,n_2+2}$ and $X_1\setminus \gamma\simeq S_{g_1-1,n_1+3}$ with $g_1\geq 1$, $n_1+n_2=n$ and $g_1+g_2=g-1$. If $\gamma\subset X_2$ is non-separating, we can assume that $X_1\simeq S_{g_1,n_1+1}$, $X_2\simeq S_{g_2,n_2+2}$ and $X_2\setminus \gamma\simeq S_{g_2-1,n_2+4}$ with $g_2\geq 1$, $n_1+n_2=n$ and $g_1+g_2=g-1$. Here $n_i$ is the number of cusps for $X_i$.
By similar estimates in \eqref{n ell k=1 with nsep s03 0 cusp eq 2}, we have \begin{equation}\label{n ell k=1 with nsep s03 0 cusp eq 4}
      \begin{aligned}&\Egn\left[\sum_{\substack{(\gamma,Y_1)\in \mathcal{P}_{nsep}^s(X)\times N_\ell(X)\\
     \textit{of case 4}, Y_1\simeq S_{0,3}\textit{ has no cusp}\\
     X\setminus Y_1\textit{ has two components}\\
   X\setminus (Y_1\cup\gamma)\textit{ has two components}
     }}H_{X,1}(\gamma)\right]\\
    \prec &\ell^2\cdot e^{\frac{T+\ell}{2}}\cdot\left( 
     \sum_{\substack{g_1+g_2=g-1,g_1\geq 1\\n_1+n_2=n\\
     2g_2+n_2\geq 1}} {n\choose n_1}\frac{V_{g_1-1,n_1+3}V_{g_2,n_2+2}}{V_{g,n}}
     \right)\\
     + &\ell^2\cdot e^{\frac{T+\ell}{2}}\cdot\left( 
     \sum_{\substack{g_1+g_2=g-1,g_2\geq 1\\n_1+n_2=n\\
     2g_1+n_1\geq 2}} {n\choose n_1}\frac{V_{g_1,n_1+1}V_{g_2-1,n_2+4}}{V_{g,n}}
     \right)\\
     \prec &\ell^2\cdot e^{\frac{T+\ell}{2}}\cdot \frac{1+n^2}{g^2}.
\end{aligned}
\end{equation}
Here we also use Corollary \ref{cor vgn+2 leq vg+1n} and the estimate \eqref{type 6 dis connected 3} for the last inequality.

Now we assume that $X\setminus Y_1=X_1\cup X_2$ has two components, $X_1$ has one boundary geodesic, and $X_2$ has two boundary geodesics. 
If $\gamma\subset X_2$ is separating, we can assume that $X_2\setminus \gamma=\tilde{X}_2^1\cup \tilde{X}_2^2$, $X_1\simeq S_{g_1,n_1+1}$, $\tilde{X}_2^1\simeq S_{g_2,n_2+2}$ and $\tilde{X}_2^2\simeq S_{g_3,n_3+2}$ with $n_1+n_2+n_3=n$, $g_1+g_2+g_3=g-1$ and $g_2\geq g_3$. Here $n_i$ is the number of cusps for each component.
By the same estimates as in \eqref{n ell k=1 with nsep s03 0 cusp eq 2}, we have
\begin{equation}\label{n ell k=1 with nsep s03 0 cusp eq 6}
\begin{aligned}
&\Egn\left[\sum_{\substack{(\gamma,Y_1)\in \mathcal{P}_{nsep}^s(X)\times N_\ell(X)\\
     \textit{of case 4}, Y_1\simeq S_{0,3}\textit{ has no cusp}\\
     X\setminus Y_1\textit{ has two components}\\
   X\setminus (Y_1\cup\gamma)\textit{ has three components}
     }}H_{X,1}(\gamma)\right]\\
    + &\ell^2\cdot e^{\frac{T+\ell}{2}}\cdot 
     \sum_{\substack{g_1+g_2+g_3=g-1,g_2\geq g_3\\n_1+n_2+n_3=n\\
     2g_1+n_1\geq 2,2g_i+n_i\geq 1}} {n\choose n_1,n_2,n_3}\frac{V_{g_1,n_1+1}V_{g_2,n_2+2}V_{g_3,n_3+2}}{V_{g,n}}.
\end{aligned}
\end{equation}
 By Corollary \ref{cor vgn+2 leq vg+1n}, \eqref{type 6 dis connected 3}, and Lemma \ref{appendix product lemma  2i+j geq k} for $k=1$,
 we have \begin{equation}\label{n ell k=1 with nsep s03 0 cusp eq 7}
 \begin{aligned}
     &\sum_{\substack{g_1+g_2+g_3=g-1,g_2\geq g_3\\n_1+n_2+n_3=n\\
     2g_1+n_1\geq 2,2g_i+n_i\geq 1}} {n\choose n_1,n_2,n_3}\frac{V_{g_1,n_1+1}V_{g_2,n_2+2}V_{g_3,n_3+2}}{V_{g,n}}\\
    \prec& \sum_{\substack{g_0+g_3=g-1,g_0\geq g_3\\n_0+n_3=n,2g_3+n_3\geq 1\\
    }} \sum_{\substack{g_1+g_2=g_0\\n_1+n_2=n_0\\2g_1+n_1\geq 2,2g_2+n_2\geq 1
    }}\!\!\!\!\!\!\!\!\!\!\!\!\!\!  {n\choose n_1,n_2,n_3}\frac{V_{g_1,n_1+1}V_{g_2,n_2+2}V_{g_3,n_3+2}}{V_{g,n}}\\
    \prec& \sum_{\substack{g_0+g_3=g-1,g_0\geq g_3\\n_0+n_3=n,2g_3+n_3\geq 1\\
    }}{n\choose n_0} \frac{V_{g_0+1,n_0}V_{g_3,n_3+2}}{V_{g,n}}\\
\cdot&\left(\sum_{\substack{g_1+g_2=g_0\\n_1+n_2=n_0\\2g_1+n_1\geq 2,2g_2+n_2\geq 1
    }}{n_0\choose n_1}\frac{V_{g_1,n_1+1}V_{g_2,n_2+2}}{V_{g_0+1,n_0}}\right)\\
    \prec&\frac{1+n^2}{g^2}\sum_{\substack{g_0+g_3=g-1,g_0\geq g_3\\n_0+n_3=n\\
    }}{n\choose n_0} \frac{V_{g_0+1,n_0}V_{g_3,n_3+2}}{V_{g,n}}\\
    \prec&\frac{1+n^2}{g^2}\sum_{\substack{g_0+g_3=g-1,g_0\geq g_3\\n_0+n_3=n\\
    }}{n\choose n_0} \frac{V_{g_0,n_0+2}V_{g_3,n_3+2}}{V_{g,n}}\\
    \prec&\frac{1+n^3}{g^3}.
 \end{aligned}
 \end{equation}
The lemma follows from \eqref{n ell k=1 with nsep s03 0 cusp eq 1}, \eqref{n ell k=1 with nsep s03 0 cusp eq 2}, \eqref{n ell k=1 with nsep s03 0 cusp eq 3}, \eqref{n ell k=1 with nsep s03 0 cusp eq 4}, \eqref{n ell k=1 with nsep s03 0 cusp eq 6}, and \eqref{n ell k=1 with nsep s03 0 cusp eq 7}.
\end{proof}

As a combination of Lemma \ref{lemma n ell k=1 with nsep s11}, Lemma \ref{lemma n ell k=1 with nsep s03 1 cusp}, and Lemma \ref{lemma n ell k=1 with nsep s03 0 cusp}, we get the following result.
\begin{lemma}\label{lemma N ell k=1 final for case 3}
    If $n=o(g^\frac{1}{2})$, then for $k=1$, \begin{align*}
       & \Egn\left[\sum_{\substack{\gamma\in \mathcal{P}_{nsep}^s(X),\ell_\gamma(X)\leq T\\
       Y_1\in  N_\ell(X)
     \textit{, of case 4}}}H_{X,1}(\gamma)\right]\\
 =&\frac{1}{\pi^2g}\left(\frac{1}{4}\int_{x\leq \ell} \sinh\frac{x}{2}V_{1,1}(x)dx+\frac{n}{2}\int_{x_1+x_2\leq \ell}\sinh\frac{x_1}{2}\sinh\frac{x_2}{2}dx_1dx_2\right.\\
 +&\left.\frac{1}{3}\int_{x_1+x_2+x_3\leq \ell}\sinh\frac{x_1}{2}\sinh\frac{x_2}{2}\sinh\frac{x_3}{2}dx_1dx_2dx_3\right)\\
 \cdot&\int_0^T \sinh\frac{y}{2}f_T(y)dy\cdot \left(1+O\left(\frac{(1+n)T^2}{g}\right)\right)+O\left(\frac{1+n^3}{g^2}\cdot \ell^2\cdot e^{\frac{\ell+T}{2}}\right).
    \end{align*}
\end{lemma}

\subsection{Proof of Lemma \ref{thm Nellk with H k geq 2} and \ref{thm Nellk with H k equal 1}}\label{subsec est N ell k with H proof 10.1 10.2}
Based on all the previous estimates, we provide a brief summary of the proofs for Lemma \ref{thm Nellk with H k geq 2} and Lemma \ref{thm Nellk with H k equal 1}. Lemma \ref{thm Nellk with H k geq 2} is a combination of Lemma \ref{lemma N ell k with H leading term}, Lemma \ref{lemma N ell k final for case 1}, Lemma \ref{lemma N ell k final for case 2}, Lemma \ref{lemma N ell k final for case 2.5}, and Lemma \ref{lemma N ell k geq 2 final for case 3}. Lemma \ref{thm Nellk with H k equal 1} is a combination of Lemma \ref{lemma N ell k with H leading term}, Lemma \ref{lemma N ell k=1 final for case 1}, Lemma \ref{lemma N ell k final for case 2}, Lemma \ref{lemma N ell k final for case 2.5} and Lemma \ref{lemma N ell k=1 final for case 3}.

\subsection{Estimate of $\Egn\left[ (\#N_\ell)_k\right]$}\label{subsec est N ell k without H} We say $(Y_1,\cdots, Y_k)$ is of leading type if each $Y_i$ has two cusps and is disjoint from each other, similar to the classification of $(\gamma, Y_1,\cdots, Y_k)$ in the estimate of $\Egn\left[ (\#N_\ell)_k\sum_{\gamma\in \mathcal{P}_{nsep}^{s}(X)}H_{X,1}(\gamma)\right]$. In this case, $X\setminus \sqcup_{i=1}^k Y_i\simeq S_{g,n-k}$.

\begin{lemma}\label{lemma N ell k without H leading term}
    If $n=o(\sqrt{g})$, then for any fixed $k$, 
    \begin{align*}
        &\E\left[\#\left\{(Y_1,\cdots,Y_k)\in  N_\ell(X)^k;
   \textit{ unordered, leading type}\right\} \right]\\
      =&\frac{1}{k!}{n\choose  2,\cdots,2,n-2k}\left(\frac{\cosh \frac{\ell}{2}-1}{2\pi^2g}\right)^k\left(1+O\left(\frac{n\ell^2}{g}\right)\right).
    \end{align*}
\end{lemma}
\begin{proof}
    We apply Mirzakhani's Integration Formula Theorem \ref{thm mir int formula} to all mapping class group orbits of $(Y_1,\cdots, Y_k)$ of the leading type. There are exactly $\frac{1}{k!}{n\choose 2,2,\cdots,2,n-2k}$ orbits, determined by the grouping of cusps. By Lemma \ref{lemma NWX vgn(x)} and Lemma \ref{lemma weak vgn}, we have \begin{align*}
             &\E\left[\#\left\{(Y_1,\cdots,Y_k)\in  N_\ell(X)^k;
   \textit{ unordered, leading type}\right\}\right]\\
      =&\frac{1}{k!}{n\choose 2,2,\cdots,2,n-2k}\frac{1}{V_{g,n}}\int_{x_1,\cdots,x_k\leq \ell} \prod_{i=1}^kV_{0,3}(x_i,0,0)\\
      \cdot &V_{g,n-k}(x_1,\cdots,x_k,0^{n-2k})x_1\cdots x_k dx_1\cdots dx_k\\
      =&\frac{1}{k!}{n\choose 2,2,\cdots,2,n-2k}\frac{V_{g,n-k}}{V_{g,n}}\int_{x_1,\cdots,x_k\leq \ell}\prod_{i=1}^k \left(2\sinh\frac{x_i}{2}\right)\\
      \cdot&\left(1+O\left(\frac{n(\sum_{i=1}^k x_i^2)}{g}\right)\right)dx_1\cdots dx_k\\
      =&\frac{1}{k!}{n\choose 2,2,\cdots,2,n-2k}\left(\frac{1}{8\pi^2g}\right)^k\left(1+O\left(\frac{1+n}{g}\right)\right)\\
      \cdot& \left(\int_{0}^\ell 2\sinh\frac{x}{2}dx\right)^k\left(1+O\left(\frac{n\ell^2}{g}\right)\right)\\
      =&\frac{1}{k!}{n\choose 2,2,\cdots,2,n-2k}\left(\frac{\cosh\frac{\ell}{2}-1}{2\pi^2g}\right)^k\left(1+O\left(\frac{1+n(1+\ell^2)}{g}\right)\right).
    \end{align*}
    Since $\ell=\kappa\log g$ for small $\kappa>0$, we have the lemma.
\end{proof}

Now we compute $\E\left[\#\left\{(Y_1,\cdots,Y_k)\in  N_\ell(X)^k;
   \textit{ not leading type}\right\}\right]$. We assume $k\geq 2$ firstly.
  We can assume that $Y_1\cup\cdots\cup Y_k$ fill the union $S_1\cup\cdots\cup S_s$ of subusrfaces of $X$ with geodesic boundaries, where every $S_i$ is connected and has a disjoint interior from each other, and $1\leq s\leq k$. We have $$
  \sum_{i=1}^s \ell(\partial S_i)\leq \sum_{i=1}^k \ell(\partial Y_i)\leq k\ell 
  $$ 
and
$$
\area(Y_1)<\sum_{i=1}^s \area(S_i)\leq 2\pi k+2\sum_{i=1}^k\ell(\partial Y_i)\leq 2\pi k+2k\ell.
$$
by the Isoperimetric Inequality on hyperbolic surfaces. So $$
2\leq \sum_{i=1}^s \left|\chi(S_i)\right|\leq k+\frac{k\ell}{\pi}=O(\log g).
$$
By Lemma \ref{lemma uniform L+6}, the map$$
(Y_1,\cdots,Y_k)\mapsto \cup_{i=1}^s S_i
$$
 is at most $\left(\left(k+\frac{k\ell}{\pi}\right)\frac{ e^{\ell+6}+3}{2}\right)^{3k}$ to $1$. Assume $S_i\simeq S_{g_i,n_i+a_i}$, where $n_i$ is the number of boundary geodesics and $a_i$ is the number of cusps for $S_i$. Let $S=\{(g_i,n_i,a_i)\}_{i=1}^s$. We view $\cup_{i=1}^s S_i\subset X$ as an element $\mathcal{F}$ in $\mathcal{B}_{S}^{s,(g,n)}$.
 Let $\Gamma_0=\cup_{i=1}^s f_i(\partial S_{g_i,n_i+a_i})$.
 If $(Y_1,\cdots, Y_k)$ is not of leading type, we have $\{(g_i,n_i,a_i)\}_{i=1}^s\neq (0,1,2)^s$. So for $k\geq 2$ we have \begin{equation}\label{lemma N ell k without H main inequality not leading}\begin{aligned}
     &\#\left\{(Y_1,\cdots,Y_k)\in  N_\ell(X)^k;
   \textit{ not leading type}\right\}\\
 \prec&\left(\left(k+\frac{k\ell}{\pi}\right)\frac{e^{\ell+6}+3}{2} \right)^{3k}\sum_{1\leq s\leq k}\sum_{S\in \textbf{Ind}_0}\\
 &  \sum_{F\in \mathcal{B}_S^{s,(g,n)}}\sum_{\Gamma\in \overline{\Mod}_{g,n}\cdot \Gamma_0}\textbf{1}_{[0,k\ell]}(\ell(\Gamma)),
 \end{aligned}
 \end{equation}
where $S$ is over \begin{align*}
\textbf{Ind}_0=\{&S=\{(g_i,n_i,a_i)\}_{i=1}^s\neq (0,1,2)^s; 1\leq 2g_i+n_i+a_i-2,\\
 &1\leq n_i,\,\, 2\leq  \sum_{i=1}^s (2g_i+n_i+a_i-2)\leq k+\frac{k\ell}{\pi}\}.
   \end{align*}

\begin{lemma}\label{lemma N ell k without H integral part}
   Fix $k\geq 2$ and $1\leq s\leq k$. For any $S\in \textbf{Ind}_0$, $F\in\mathcal{B}_S^{s,(g,n)}$, and $\Gamma_0=\cup_{i=1}^sf_i(\partial S_{g_i,n_i+a_i})$, the following estimate holds: \begin{align*}
    &\Egn\left[\sum_{\Gamma\in \overline{\Mod}_{g,n}\cdot \Gamma_0}\textbf{1}_{[0,k\ell]}(\ell(\Gamma))\right]\\
        \prec &\frac{{n\choose a_1,\cdots,a_s,b_1,\cdots b_r}}{V_{g,n}}\prod_{i=1}^s V_{g_i,n_i+a_i}\prod_{j=1}^r V_{h_j,m_j+b_j} \cdot e^{2k\ell}.
    \end{align*}
\end{lemma}
\begin{proof}
    We can divide $\overline{\Mod}_{g,n}\cdot \Gamma_0$ into ${n\choose a_1,\cdots,a_s,b_1,\cdots,b_r}$ mapping class group orbits and apply Mirzakhani's Integration Formula Theorem \ref{thm mir int formula} to them. Let $x_p$ be the length of the simple closed geodesic $\gamma_p$ in $\Gamma$ for $p$ in the index set $I$. We have $\sum_{p\in I}x_p\leq k\ell$, and 
\begin{align*}
    &\Egn\left[\sum_{\Gamma\in \overline{\Mod}_{g,n}\cdot \Gamma_0}\textbf{1}_{[0,k\ell]}(\ell(\Gamma))\right]\\
    \prec&\frac{{n\choose a_1,\cdots,a_s,b_1,\cdots,b_r}}{V_{g,n}}\int_{\sum_{p\in I}x_p\leq k\ell}  \prod_{i=1}^s V_{g_i,n_i+a_i}(x_{p^i_1},\cdots ,x_{p^{i}_{n_i}},0^{a_i})\\
    \nonumber\cdot&\prod_{j=1}^r V_{h_j,m_j+b_j}(x_{p^{j}_1},\cdots,x_{p^{j}_{m_j}},0^{b_j})\prod_{p\in I} x_p\cdot\prod_{p\in I}dx_p,
\end{align*}where $p_1^i,\cdots,p_{n_i}^i$ are the labels of geodesics on $\partial S_i$, and $q_1^j,\cdots,q_{m_j}^j$ are the labels of geodesics on $\partial V_j$. Since each $\gamma_p$ serves as a boundary geodesic for some domains cut by $\Gamma$ twice, by the estimates \eqref{over N ell k case 1  eq 3} and \eqref{over N ell k case 1  eq 5}, we have 
    \begin{align*}
    &\Egn\left[\sum_{\Gamma\in \overline{\Mod}_{g,n}\cdot \Gamma_0}\textbf{1}_{[0,k\ell]}(\ell(\Gamma))\right]\\
    \prec&\frac{{n\choose a_1,\cdots,a_s,b_1,\cdots,b_r}}{V_{g,n}}\prod_{i=1}^s V_{g_i,n_i+a_i}\prod_{j=1}^r V_{h_j,m_j+b_j}
    \int_{\sum_{p\in I}x_p\leq k\ell} e^{\sum_{p\in I}x_p} \prod_{p\in I} x_p\prod_{p\in I}dx_p\\
    \prec &\frac{{n\choose a_1,\cdots,a_s,b_1,\cdots,b_r}}{V_{g,n}}\prod_{i=1}^s V_{g_i,n_i+a_i}\prod_{j=1}^r V_{h_j,m_j+b_j}\cdot e^{2k\ell},
\end{align*}
which ends the proof.
\end{proof}
\begin{lemma}\label{lemma N ell k final without H kgeq 2}
    If $n=o(g^{\frac{1}{2}-\epsilon})$ for some $\epsilon>0$, then for any $k\geq 2$, \begin{align*}
         &\Egn\left[\#\left\{(Y_1,\cdots,Y_k)\in  N_\ell(X)^k;
   \textit{ not leading type}\right\}\right]\\
 \nonumber    \prec& \left(\left(k+\frac{k\ell}{\pi}\right)\frac{ e^{\ell+6}+3}{2}\right)^{3k} \cdot e^{2k\ell} \cdot (k\ell)^{3k+3}\cdot\frac{\log^{14} g(\log g+n)^3}{g^2}.
    \end{align*}
\end{lemma}
\begin{proof}
By \eqref{lemma N ell k without H main inequality not leading} and Lemma \ref{lemma N ell k without H integral part},
we have \begin{equation}\label{lemma N ell k final without H kgeq 2 eq 1}
\begin{aligned}
     &\Egn\left[\#\left\{(Y_1,\cdots,Y_k)\in  N_\ell(X)^k;
   \textit{ not leading type}\right\}\right]\\
   \prec&\left(\left(k+\frac{k\ell}{\pi}\right)\frac{ e^{\ell+6}+3}{2}\right)^{3k} \cdot e^{2k\ell}\cdot\sum_{1\leq s\leq k}\sum_{S\in \textbf{Ind}_0}\\
 &  \sum_{F\in \mathcal{B}_S^{s,(g,n)}}\frac{{n\choose a_1,\cdots,a_s,b_1,\cdots b_r}}{V_{g,n}}\prod_{i=1}^s V_{g_i,n_i+a_i}\prod_{j=1}^r V_{h_j,m_j+b_j}.
\end{aligned}
\end{equation}
    Since $\textbf{Ind}_0=\textbf{Ind}_1$ defined in Subsection \ref{subsec est N ell k with H case 1}, we can apply the estimates \eqref{lemma N ell k final for case 1 eq 2} and \eqref{lemma N ell k final for case 1 eq 3} to the summation in \eqref{lemma N ell k final without H kgeq 2 eq 1} over $S\in \textbf{Ind}_0$. So we have \begin{equation}\label{lemma N ell k final without H kgeq 2 eq 2}
    \begin{aligned}
      &  \sum_{1\leq s\leq k}\sum_{S\in \textbf{Ind}_0}
   \sum_{F\in \mathcal{B}_S^{s,(g,n)}}\frac{{n\choose a_1,\cdots,a_s,b_1,\cdots b_r}}{V_{g,n}}\prod_{i=1}^s V_{g_i,n_i+a_i}\prod_{j=1}^r V_{h_j,m_j+b_j}\\
   \prec&(k\ell)^{3k+3}\cdot\frac{\log^{14} g(\log g+n)^3}{g^2}.
    \end{aligned}
    \end{equation}
    The lemma follows from \eqref{lemma N ell k final without H kgeq 2 eq 1} and \eqref{lemma N ell k final without H kgeq 2 eq 2}.
\end{proof}

For $k=1$ we calculate $\Egn\left[ \# N_\ell(X)\right]$ more precisely.
For $Y\in N_\ell$, if $Y$ has two cusps, $Y\subset X$ is of the leading type, and the expectation of such $Y$ is computed in Lemma \ref{lemma N ell k without H leading term}.
\begin{lemma}\label{lemma N ell k final without H k equal 1 one cusp}
If $n=o(\sqrt{g})$, then
    \begin{align*}
        &\Egn\left[\#\left\{Y\in  N_\ell(X); Y
   \textit{ has one cusp}\right\}\right]\\
   =&\frac{n}{4\pi^2 g}\int_{x+y\leq\ell} \sinh\frac{x}{2}\sinh\frac{y}{2}dxdy \left(1+O\left(\frac{n\ell^2}{g}\right)\right)+O\left(\frac{n^3\cdot \ell\cdot e^{\frac{\ell}{2}}}{g^2}\right).
    \end{align*}
\end{lemma}
\begin{proof}
    For $Y\in N_\ell$ with one cusp, if $X\setminus Y$ is connected, then $X\setminus Y\simeq S_{g-1,n+1}$. If $X\setminus Y$ is not connected, we can assume $X\setminus Y\simeq S_{g_1,n_1+1}\cup S_{g_2,n_2+1}$ with $n_1+n_2=n-1$ and $g_1+g_2=g$. By Mirzakhani's Integration Formula Theorem \ref{thm mir int formula}, Lemma \ref{lemma NWX vgn(x)} and Lemma \ref{lemma weak vgn}, we have
    \begin{align*}
          &\Egn\left[\#\left\{Y\in  N_\ell(X); Y
   \textit{ has one cusp, } Y\setminus X\textit{ is connected}\right\}\right]\\
  =&\frac{1}{2}\frac{n}{V_{g,n}}\int_{x+y\leq \ell}V_{0,3}(0,x,y)V_{g-1,n+1}(x,y,0^{n-1})xydxdy\\
 =&\frac{n}{2}\frac{V_{g-1,n+1}}{V_{g,n}}\int_{x+y\leq \ell} 4\sinh\frac{x}{2}\sinh\frac{y}{2}\left(1+O\left(\frac{n(x^2+y^2)}{g}\right)\right)dxdy\\
 =&\frac{n}{4\pi^2 g}\int_{x+y\leq\ell} \sinh\frac{x}{2}\sinh\frac{y}{2}dxdy \left(1+O\left(\frac{n\ell^2}{g}\right)\right),
 \end{align*}
 where $\frac{1}{2}$ is $C_\Gamma$ in the integration formula, and $n$ is the number of mapping class orbits for $\partial Y$. Similarly, we have
 \begin{align*} 
 &\Egn\left[\#\left\{Y\in  N_\ell(X); Y
   \textit{ has one cusp, } Y\setminus X\textit{ is not connected}\right\}\right]\\
   \prec&\sum_{\substack{n_1+n_2=n-1\\
   g_1+g_2=g\\
   2g_i+n_i\geq 2}}\frac{{n\choose 1,n_1,n_2}}{V_{g,n}}\int_{x,y\leq \ell}V_{0,3}(0,x,y)V_{g_1,n_1+1}(x,0^{n_1}) V_{g_2,n_2+1}(y,0^{n_2})    xydxdy\\
   \prec&\sum_{\substack{n_1+n_2=n-1\\
   g_1+g_2=g\\
   2g_i+n_i\geq 2}}\frac{{n\choose 1,n_1,n_2}}{V_{g,n}} V_{g_1,n_1+1}V_{g_2,n_2+1}\int_{x+y\leq \ell}\sinh\frac{x}{2}\sinh\frac{y}{2}dxdy\\
  \prec&\frac{n}{g}\cdot \ell\cdot e^{\frac{\ell}{2}} \cdot\sum_{\substack{n_1+n_2=n-1\\
   g_1+g_2=g\\
   2g_i+n_i\geq 2}}\frac{{n-1\choose n_1}}{V_{g,n-1}} V_{g_1,n_1+1}V_{g_2,n_2+1}\\
  \prec&\frac{n^3}{g^2}\cdot \ell\cdot e^{\frac{\ell}{2}},
    \end{align*}
    where ${n\choose1,n_1,n_2}$ is the number of mapping class group orbits of $\partial Y$ for given $g_i,n_i$, and the last inequality is given by Lemma \ref{lemma in hide appendix gen k}.
\end{proof}
\begin{lemma}\label{lemma N ell k final without H k equal 1 S11}
If $n=o(\sqrt{g})$, then
    \begin{align*}
        &\Egn\left[\#\left\{Y\in  N_\ell(X); Y\simeq S_{1,1}
  \right\}\right]\\
   =&\frac{1}{8\pi^2g }\int_{0}^\ell V_{1,1}(x) \sinh\frac{x}{2}dx\left(1+O\left(\frac{(1+n)\ell^2}{g}\right)\right).
    \end{align*}
\end{lemma}
\begin{proof}
If $Y\simeq S_{1,1}$, then $X\setminus Y\simeq S_{g-1,n+1}$ must be connected. By Mirzakhani's Integration Formula Theorem \ref{thm mir int formula}, Lemma \ref{lemma NWX vgn(x)} and Lemma \ref{lemma weak vgn}, we have \begin{align*}
    &\Egn\left[\#\left\{Y\in  N_\ell(X); Y\simeq S_{1,1}
  \right\}\right]\\
   =&\frac{1}{2}\cdot\frac{1}{V_{g,n}}\int_{0}^\ell V_{1,1}(x)V_{g-1,n+1}(x,0^n)xdx\\
   =&\frac{V_{g-1,n+1}}{2V_{g,n}}\int_{0}^\ell V_{1,1}(x) 2\sinh\frac{x}{2}\left(1+O\left(\frac{(1+n)x^2}{g}\right)\right)dx\\
   =&\frac{1}{8\pi^2g }\int_{0}^\ell V_{1,1}(x) \sinh\frac{x}{2}dx\left(1+O\left(\frac{(1+n)\ell^2}{g}\right)\right),
\end{align*}
which ends the proof.
\end{proof}

\begin{lemma}\label{lemma N ell k final without H k equal 1 no cusp}
If $n=o(\sqrt{g})$, then 
    \begin{align*}
        &\Egn\left[\#\left\{Y\in  N_\ell(X); Y\simeq S_{0,3}\textit{ with no cusp}
  \right\}\right]\\
   =&\frac{1}{6\pi^2g}\int_{\sum_{i=1}^3 x_i\leq \ell} \sinh\frac{x_1}{2}\sinh\frac{x_2}{2}\sinh\frac{x_3}{2}dx_1dx_2dx_3\cdot\left(1+O\left(\frac{(1+n)\ell^2}{g}\right)\right)\\
   +&O\left(\ell^2\cdot e^{\frac{\ell}{2}}\cdot \frac{1+n^2}{g^2}\right).
    \end{align*}
\end{lemma}
\begin{proof}
    For $Y\simeq S_{0,3}$ with no cusp, if $X\setminus Y$ is connected, then $Y\setminus X\simeq S_{g-2,3}$. If $X\setminus Y$ has two connected components, we can assume $X\setminus Y\simeq S_{g_1,n_1+1}\cup S_{g_2,n_2+2}$ with $n_1+n_2=n$ and $g_1+g_2=g-1$, where $n_i$ is the number of cusps for the corresponding component. If $X\setminus Y$ has three components, we can assume $X\setminus Y\simeq S_{g_1,n_1+1}\cup S_{g_2,n_2+1}\cup S_{g_3,n_3+1}$ with $n_1+n_2+n_3=n$ and $g_1+g_2+g_3=g$.
By Mirzakhani's Integration Formula Theorem \ref{thm mir int formula}, Lemma \ref{lemma NWX vgn(x)} and Lemma \ref{lemma weak vgn}, we have \begin{align*}
     &\Egn\left[\#\left\{Y\in  N_\ell(X); Y\simeq S_{0,3},\textit{ with no cusp, } X\setminus Y\textit{ is connected } 
  \right\}\right]\\
  =&\frac{1}{6}\frac{1}{V_{g,n}}\int_{\sum_{i=1}^3 x_i\leq \ell} 
\!\! V_{0,3}(x_1,x_2,x_3)V_{g-2,n+3}(x_1,x_2,x_3,0^n) x_1x_2x_3dx_1dx_2dx_3\\
  =&\frac{V_{g-2,n+3}}{6V_{g,n}}\int_{x_1+x_2+x_3\leq \ell} 8\sinh\frac{x_1}{2}\sinh\frac{x_2}{2}\sinh\frac{x_3}{2}\\
  \cdot&\left(1+O\left(\frac{(1+n)(x_1^2+x_2^2+x_3^2)}{g}\right)\right)dx_1dx_2dx_3\\
  =&\frac{1}{6\pi^2g}\!\int_{\sum_{i=1}^3 x_i\leq \ell}\!\! \sinh\frac{x_1}{2}\sinh\frac{x_2}{2}\sinh\frac{x_3}{2}dx_1dx_2dx_3\!\left(1\!+\!O\left(\frac{(1+n)\ell^2}{g}\right)\right).
\end{align*}
    Also by Mirzakhani's Integration Formula Theorem \ref{thm mir int formula}, Lemma \ref{lemma NWX vgn(x)}, Lemma \ref{lemma weak vgn}, and the estimate \eqref{type 6 dis connected 3}, we have 
\begin{align*}
    &\Egn\left[\#\left\{Y\in  N_\ell(X); Y\simeq S_{0,3},\textit{ with no cusp, } X\setminus Y\textit{ has two components } 
  \right\}\right]\\
  \prec&\sum_{\substack{
n_1+n_2=n,
  g_1+g_2=g-1\\
  2g_1+n_1\geq 2,2g_2+n_2\geq 1
  }}\frac{{n\choose n_1}}{V_{g,n}}\int_{\sum_{i=1}^3x_i\leq \ell}V_{0,3}(x_1,x_2,x_3)V_{g_1,n_1+1}(x_1,0^{n_1})\\
  \cdot& V_{g_2,n_2+2}(x_2,x_3,0^{n_2})\cdot x_1x_2x_3 dx_1dx_2dx_3\\
  \prec&\sum_{\substack{
n_1+n_2=n,
  g_1+g_2=g-1\\ 2g_1+n_1\geq 2,2g_2+n_2\geq 1
  }}\frac{{n\choose n_1}}{V_{g,n}} V_{g_1,n_1+1}V_{g_2,n_2+2}
  \int_{\sum_{i=1}^3x_i\leq \ell} e^{\frac{x_1+x_2+x_3}{2}}dx_1dx_2dx_3\\
  \prec&\ell^2\cdot e^{\frac{\ell}{2}}\cdot \frac{1+n^2}{g^2}.
\end{align*}
   Similarly we have 
    \begin{align*}
&\Egn\left[\#\left\{Y\in  N_\ell(X); Y\simeq S_{0,3},\textit{ with no cusp, } X\setminus Y\textit{ has three components } 
  \right\}\right]\\
    \prec&   \sum_{\substack{n_1+n_2+n_3=n\\
    g_1+g_2+g_3=n\\
    2g_i+n_i\geq 2}}\frac{{n\choose n_1,n_2,n_3}}{V_{g,n}}\int_{\sum_{i=1}^3x_i\leq \ell} V_{0,3}(x_1,x_2,x_3) V_{g_1,n_1+1}(x_1,0^{n_1})\\
    \cdot&V_{g_2,n_2+1}(x_2,0^{n_2})V_{g_3,n_3+1}(x_3,0^{n_3})x_1x_2x_3dx_1dx_2dx_3\\
    \prec& \sum_{\substack{n_1+n_2+n_3=n\\
    g_1+g_2+g_3=n\\
    2g_i+n_i\geq 2}}\frac{{n\choose n_1,n_2,n_3}}{V_{g,n}}V_{g_1,n_1+1}V_{g_2,n_2+1}V_{g_3,n_3+1}\\
    \cdot&\int_{\sum_{i=1}^3x_i\leq \ell}\sinh\frac{x_1}{2}\sinh\frac{x_2}{2}\sinh\frac{x_3}{2}dx_1dx_2dx_3\\
    \prec&\ell^2\cdot e^{\frac{\ell}{2}}\cdot \frac{1+n^4}{g^3},
    \end{align*}
    where the last inequality is given by \eqref{type 6 dis connected 2}.
\end{proof}

\subsection{Proof of Lemma \ref{thm Nellk without H k geq 2} and \ref{thm Nellk without H k equal 1}}  
Lemma \ref{thm Nellk without H k geq 2} is a combination of Lemma \ref{lemma N ell k without H leading term} and Lemma \ref{lemma N ell k final without H kgeq 2}. Lemma \ref{thm Nellk without H k equal 1} is a combination of Lemma \ref{lemma N ell k without H leading term}, Lemma \ref{lemma N ell k final without H k equal 1 one cusp}, Lemma \ref{lemma N ell k final without H k equal 1 S11}, and Lemma \ref{lemma N ell k final without H k equal 1 no cusp}. Here we use the fact that $V_{1,1}(x)\prec 1+x^2$ by Theorem \ref{mir07 poly}. 

\section{Proof of Theorem \ref{thm main-1}}\label{sec main proof}

Now we will prove the main theorem in this paper.
We assume that $n=O(g^\alpha)$ for $\alpha<\frac{1}{2}$.
By Lemma \ref{thm Nellk without H k equal 1} , we have \begin{equation}\label{last section eq 1}
\begin{aligned}
    &\Probgn\left(X\in \mathcal{N}_\ell^c\right)\leq \Egn\left[\#N_\ell\right]\\
   \prec & \frac{(1+n^2)\ell^2 e^{\frac{\ell}{2}}}{g}\prec \frac{g^{2\alpha+\frac{\kappa}{2}}\log^2 g}{g}=o(1)
\end{aligned}
\end{equation}
as $g\to \infty$ if $\ell=\kappa\log g$ for $\kappa<2-4\alpha$.
We have established the inequality \eqref{eq ineq on subset T}. Now we repeat it:
\begin{equation}\label{last section eq 2}
\begin{aligned}
      &\Egn\left[C(\epsilon_1)Te^{T(1-\epsilon_1)\sqrt{\frac{1}{4}-\lambda_1(X)}}\cdot\textbf{1}_{X\in\mathcal{N}_\ell}\cdot  \textbf{1}_{\lambda_1(X)\leq \frac{1}{4}}\right] \\
\leq&\underbrace{\Egn\left[1_{\mathcal{N}_\ell}\cdot\sum_{\gamma\in \mathcal{P}_{nsep}^{s}(X)}H_{X,1}(\gamma)\right]}_{\mathrm{Int}_{nsep}}-\underbrace{\Probgn\left(\mathcal{N}_\ell\right)\hat{f}_T(\frac{i}{2})}_{\mathrm{Int}_{\lambda_0}}\\
+&\underbrace{\Egn\left[1_{\mathcal{N}_\ell}\cdot\sum_{\gamma\in \mathcal{P}^{ns}(X)}H_{X,1}(\gamma)\right]}_{\mathrm{Int}_{ns}}\\
+&O\left(\frac{g}{T}+g\log^2g+\log^{3a+4} g\right)
    \end{aligned}
    \end{equation}
    for any $\epsilon_1>0$ and $T=a\log g$ with $a>0$.
    By Lemma \ref{lemma inc-ex}, the term $\mathrm{Int}_{nsep}-\mathrm{Int}_{\lambda_0}$ can be written as 
    \begin{equation}\label{last section eq 3}
    \begin{aligned}
&\Egn\left[1_{\mathcal{N}_\ell}\cdot\sum_{\gamma\in \mathcal{P}_{nsep}^{s}(X)}H_{X,1}(\gamma)\right]-\Probgn\left(\mathcal{N}_\ell\right)\hat{f}_T(\frac{i}{2})\\
        =&\Egn\left[\sum_{\gamma\in \mathcal{P}_{nsep}^{s}(X)}H_{X,1}(\gamma)\right]-\hat{f}_T(\frac{i}{2})
+\sum_{k=1}^j(-1)^k\\\cdot&\left(\Egn\left[ (\#N_\ell)_k\cdot\sum_{\gamma\in \mathcal{P}_{nsep}^{s}(X)}H_{X,1}(\gamma)\right]-\hat{f}_T(\frac{i}{2})\cdot\Egn\left[ (\#N_\ell)_k\right]\right)\\
+&O\!\left(\Egn\!\left[(\#N_\ell)_{j+1}\cdot\!\!\!\!\!\!\!\sum_{\gamma\in \mathcal{P}_{nsep}^{s}(X)}H_{X,1}(\gamma)\right]\!+\!\hat{f}_T(\frac{i}{2})\!\cdot\!\Egn\left[ (\#N_\ell)_{j+1}\right]\right)
    \end{aligned}
    \end{equation}
for any $j\geq 0$.

\begin{proof}[Proof of Theorem \ref{thm main-1}]
We take $T=6(1-\alpha)\log g$ and fix $j=j_0\geq \frac{2-3\alpha}{1-2\alpha}-1\geq 1$ in \eqref{last section eq 3}. 
By Lemma \ref{lemma nsep zero eigen}, we have
   \begin{equation}\label{last section eq 4}
   \begin{aligned}
        &\Egn\left[\sum_{\gamma\in \mathcal{P}_{nsep}^{s}(X)}H_{X,1}(\gamma)\right]-\hat{f}_T(\frac{i}{2})\\
     =&\frac{1}{\pi^2g}\int_0^Tf_T(x)\left(\left(1-\frac{n}{2}\right)x-\frac{x^2}{4}\right)e^\frac{x}{2}  dx +O\left( (\log g)^6\cdot  g\right).
    \end{aligned}
    \end{equation}

Since ${n\choose 2,\cdots,2,n-2k}\prec n^{2k}\prec g^{2\alpha k}$,
by Lemma \ref{thm Nellk with H k geq 2},
for $T=6(1-\alpha)\log g$ we have \begin{equation}\label{last section eq 5}
    \begin{aligned}&\Egn\left[(\#N_\ell)_{j_0+1}\cdot\sum_{\gamma\in \mathcal{P}_{nsep}^{s}(X)}H_{X,1}(\gamma)\right]\\
   \prec &\frac{g^{2\alpha(j_0+1)}}{g^{j_0+1}}g^{\frac{1}{2}(j_0+1)\kappa}e^{\frac{T}{2}}+(\log g)^{366j_0+748} \cdot\frac{g^{3\alpha+5(j_0+1)\kappa}}{g^2}e^{\frac{T}{2}}\\
    +&\log^{6j_0+7} g ^{5(j_0+1)\kappa}+(\log g)^{6j_0+499} \cdot  \frac{g^{62\alpha+5(j_0+1)\kappa}}{g^{61}} e^{\frac{9}{2}T}\\
 \prec&(\log g)^{366j_0+748}\cdot g^{1+5(j_0+1)\kappa}.
\end{aligned}
\end{equation}
And by Lemma \ref{thm Nellk without H k geq 2} for $T=6(1-\alpha)\log g$, we have \begin{equation}\label{last section eq 6}
\begin{aligned}&\Egn\left[(\#N_\ell)_{j_0+1}\right]\cdot \hat{f}_T(\frac{i}{2})\\
      \prec& \frac{g^{2\alpha(j_0+1)}}{g^{j_0+1}}g^{\frac{1}{2}(j_0+1)\kappa} e^{\frac{T}{2}}+(\log g)^{6j_0+26}g^{3\alpha-2+5(j_0+1)\kappa} e^\frac{T}{2}\\
    \prec&(\log g)^{6j_0+26}\cdot g^{1+5(j_0+1)\kappa}.
\end{aligned}
\end{equation}
For $2\leq k\leq j_0$, by Lemma \ref{thm Nellk with H k geq 2} and Lemma \ref{thm Nellk without H k geq 2} we have \begin{equation}\label{last section eq 7}
\begin{aligned}&\left|\Egn\left[(\#N_\ell)_{k}\cdot\sum_{\gamma\in \mathcal{P}_{nsep}^{s}(X)}H_{X,1}(\gamma)\right]-\Egn\left[(\#N_\ell)_{k}\right]\cdot \hat{f}_T(\frac{i}{2})\right|\\
     =&\frac{1}{k!}{n\choose 2,\cdots,2,n-2k}\!\!\left(\! \frac{\cosh\frac{\ell}{2}-1}{2\pi^2g}\!\right)^k \left|\int_0^T 
\!\! 2\sinh\frac{y}{2}f_T(y)dy-\hat{f}_T(\frac{i}{2})\right|\\
      +&O\left(\frac{1}{k!}{n\choose 2,\cdots,2,n-2k}\left( \frac{\cosh\frac{\ell}{2}-1}{2\pi^2g}\right)^k \cdot e^{\frac{T}{2}}\cdot\frac{n\log^2 g}{g}\right)\\
      +&O\left((\log g)^{366k+382}\cdot g^{1+5k\kappa}\right)\\
    \prec&\frac{n^{2k}}{g^k}\cdot g^{\frac{k}{2}\kappa}+\frac{n^{2k+1}}{g^{k+1}}\log^2 g \cdot g^{3-3\alpha+\frac{k}{2}\kappa}+(\log g)^{366k+382}\cdot g^{1+5k\kappa}\\
    \prec &(\log g)^{366j_0+382}\cdot g^{1+5j_0\kappa}.
\end{aligned}
\end{equation}
While for $k=1$, by Lemma \ref{thm Nellk with H k equal 1} and Lemma \ref{thm Nellk without H k equal 1} we have \begin{equation}\label{last section eq 8}
\begin{aligned}&\left|\Egn\left[(\#N_\ell)\cdot\sum_{\gamma\in \mathcal{P}_{nsep}^{s}(X)}H_{X,1}(\gamma)\right]-\Egn\left[\#N_\ell\right]\cdot \hat{f}_T(\frac{i}{2})\right|\\
       =&\frac{1}{\pi^2g}\left[ \frac{n(n-1)}{4}\left(\cosh\frac{\ell}{2}-1\right)+\frac{n}{4}\int_{x_1+x_2\leq \ell}\sinh\frac{x_1}{2}\sinh\frac{x_2}{2}dx_1dx_2\right.\\
        +&\frac{1}{6} \int_{x_1+x_2+x_3\leq \ell}\sinh\frac{x_1}{2}\sinh\frac{x_2}{2}\sinh\frac{x_3}{2}dx_1dx_2dx_3 \\
      +& \left.\frac{1}{8}\int_{0}^\ell V_{1,1}(x)\sinh\frac{x}{2}dx
  \right] \cdot \left|\int_{0}^T 2\sinh\frac{y}{2}f_T(y)dy-\hat{f}_T(\frac{i}{2})  \right|\\
  +&O\left((\log g)^{748} g^{1+5 \kappa}  \right)\\
  \prec&O\left((\log g)^{748} g^{1+5 \kappa}  \right).
\end{aligned}
\end{equation}
Combining \eqref{last section eq 3}, \eqref{last section eq 4}, \eqref{last section eq 5}, \eqref{last section eq 6}, \eqref{last section eq 7}, and \eqref{last section eq 8}, we have \begin{equation}\label{last section eq 9}
\begin{aligned}
&\mathrm{Int}_{nsep}-\mathrm{Int}_{\lambda_0}\\
   =& \Egn\left[1_{\mathcal{N}_\ell}\cdot\sum_{\gamma\in \mathcal{P}_{nsep}^{s}(X)}H_{X,1}(\gamma)\right]-\Probgn\left(\mathcal{N}_\ell\right)\hat{f}_T(\frac{i}{2})\\
  =&\frac{1}{\pi^2g}\int_0^Tf_T(x)\left(\left(1-\frac{n}{2}\right)x-\frac{x^2}{4}\right)e^\frac{x}{2}  dx\\
 +&O\left((\log g)^{366j_0+748}\cdot g^{1+5(j_0+1)\kappa}\right).
\end{aligned}
\end{equation}

By Lemma \ref{lemma total int ns for all type} we have 
\begin{equation}\label{last section eq 10}
\begin{aligned}
      &\mathrm{Int}_{ns}\\
       = &\Egn\left[1_{N_\ell}\cdot \sum_{\gamma\in \mathcal{P}^{ns}(X)}H_{X,1}(\gamma)\right]\\
      \leq&\frac{1}{\pi^2g}\int_0^\infty f_T(x) e^\frac{x}           {2}\left(\frac{x^2}{4}+\left(\frac{n}   
            {2}-1\right)x\right)dx \\
      +O&\left((\log g)^3\cdot g^{6\delta_0(1-\alpha)+\frac{7}{2}\alpha-\frac{5}{2}} 
          +(\log g)^{230}\cdot g^{1+6\epsilon_0(1-\alpha)}
    \right)
    \end{aligned}
    \end{equation}
for any $\delta_0>\frac{1}{2}$ and $\epsilon_0>0$, which will cancel out the main term in \eqref{last section eq 9}.

Now we take \eqref{last section eq 9} and \eqref{last section eq 10}
into \eqref{last section eq 2}, we have \begin{equation}\label{last section eq 11}
\begin{aligned}
&C(\epsilon_1) \log g \cdot g^{6(1-\alpha)(1-\epsilon_1)\sqrt{\left(\frac{1}{6(1-\alpha)}\right)^2+\epsilon}} \\
\cdot &\Probgn\left(X\in \M_{g,n}:\lambda_1(X)\leq \frac{1}{4}-\left(\frac{1}{6(1-\alpha)}\right)^2-\epsilon,X\in \mathcal{N}_\ell\right)\\
\prec &(\log g)^{366j_0+748}\cdot g^{1+5(j_0+1)\kappa}
+(\log g)^3\cdot g^{6\delta_0(1-\alpha)+\frac{7}{2}\alpha-\frac{5}{2}} \\
          +&(\log g)^{230}\cdot g^{1+6\epsilon_0(1-\alpha)},
\end{aligned}
\end{equation}
for any $\ell=\kappa\log g$,  $j_0\geq \frac{2-3\alpha}{1-2\alpha}-1\geq 1$, $\epsilon,\epsilon_1,\epsilon_0>0$, $0<\kappa<2-4\alpha$, and $\delta_0>\frac{1}{2}$. Given $\epsilon>0$ and $0\leq \alpha< \frac{1}{2}$, we can fix $j_0\geq 1$. Then we can take $\epsilon_1,\epsilon_0,\kappa,(\delta_0-\frac{1}{2})$ 
 small enough such that $\kappa<2-4\alpha$ and \begin{align*}
     &6(1-\alpha)(1-\epsilon_1)\sqrt{\left(\frac{1}{6(1-\alpha)}\right)^2+\epsilon}\\
     >&\max\{6\delta_0(1-\alpha)+\frac{7}{2}\alpha-\frac{5}{2},1+6\epsilon_0(1-\alpha),1+5(j_0+1)\kappa\}.
 \end{align*}
In this case by \eqref{last section eq 1} and \eqref{last section eq 11} we have \begin{align*}
  \lim_{g\to\infty}  \Probgn\left(X\in \M_{g,n}:\lambda_1(X)\leq \frac{1}{4}-\left(\frac{1}{6(1-\alpha)}\right)^2-\epsilon,X\in \mathcal{N}_\ell\right)=0,
\end{align*}
and \begin{align*}
     \lim_{g\to\infty}  \Probgn\left(X\notin \mathcal{N}_\ell\right)=0.
\end{align*}
So we have \begin{align*}
    \lim_{g\to\infty}  \Probgn\left(X\in \M_{g,n}:\lambda_1(X)> \frac{1}{4}-\left(\frac{1}{6(1-\alpha)}\right)^2-\epsilon\right)=1,
\end{align*}
for any $\epsilon>0$ if $n=O(g^\alpha)$.
\end{proof}

\appendix
\section{Count filling geodesics}\label{appendix counting}
In this Appendix, we will briefly sketch the proof of Theorem \ref{thm double filling count} and Theorem \ref{thm mono length}. The proof of Theorem \ref{thm double filling count} is basically the same as Theorem \ref{thm filling count}. See \cite{wx22-3/16, WX2022prime} for more details. The proof of Theorem \ref{thm mono length} we provide in this paper is generalized from \cite{parlier2005}.

For simplicity, we use the same letters for geodesics and their lengths in this Appendix.

\subsection{Proof of Theorem \ref{thm mono length}}

In \cite{parlier2005}, Parlier showed Theorem \ref{thm mono length} for only simple closed curves. His proof can be extended to all closed curves in the following way.

\begin{proof}[Proof of Theorem \ref{thm mono length}]
    Choose a pants decomposition of $X$. Define $Y_i$ by using the same pants decomposition and Fenchel-Nielsen coordinate as $X$ but changing the boundary length $(x_1,\cdots,x_n)$ into $(y_1,\cdots,y_i, x_{i+1},\cdots,x_n)$, and $Y=Y_n$. For any closed curve $\eta\sbs S_{g,n}$, let $\eta_X$ (resp. $\eta_Y$) be the geodesic in $X$ (resp. $Y$) homotopic to $\eta$. We will construct piecewise geodesics $\eta'_{Y_i}$ and smooth geodesic $\eta_{Y_i}$ homotopic to $\eta$ in each $Y_i$ and show that $$\ell(\eta_X)\geq\ell(\eta'_{Y_1})\geq\ell(\eta_{Y_1})\geq\cdots\geq\ell(\eta'_{Y_n})\geq\ell(\eta_{Y_n})=\ell(\eta_Y).$$

For a pants with boundary $(a,b,x)$, denote the three perpendiculars between the boundaries as $(\alpha,\beta,\gamma)$. Cut the pants along $\alpha$ and $\beta$. Then we will get an octagon as shown in Figure \ref{fig pants}.

\begin{figure}[h]
    \centering
    \includegraphics[width=3.5 in]{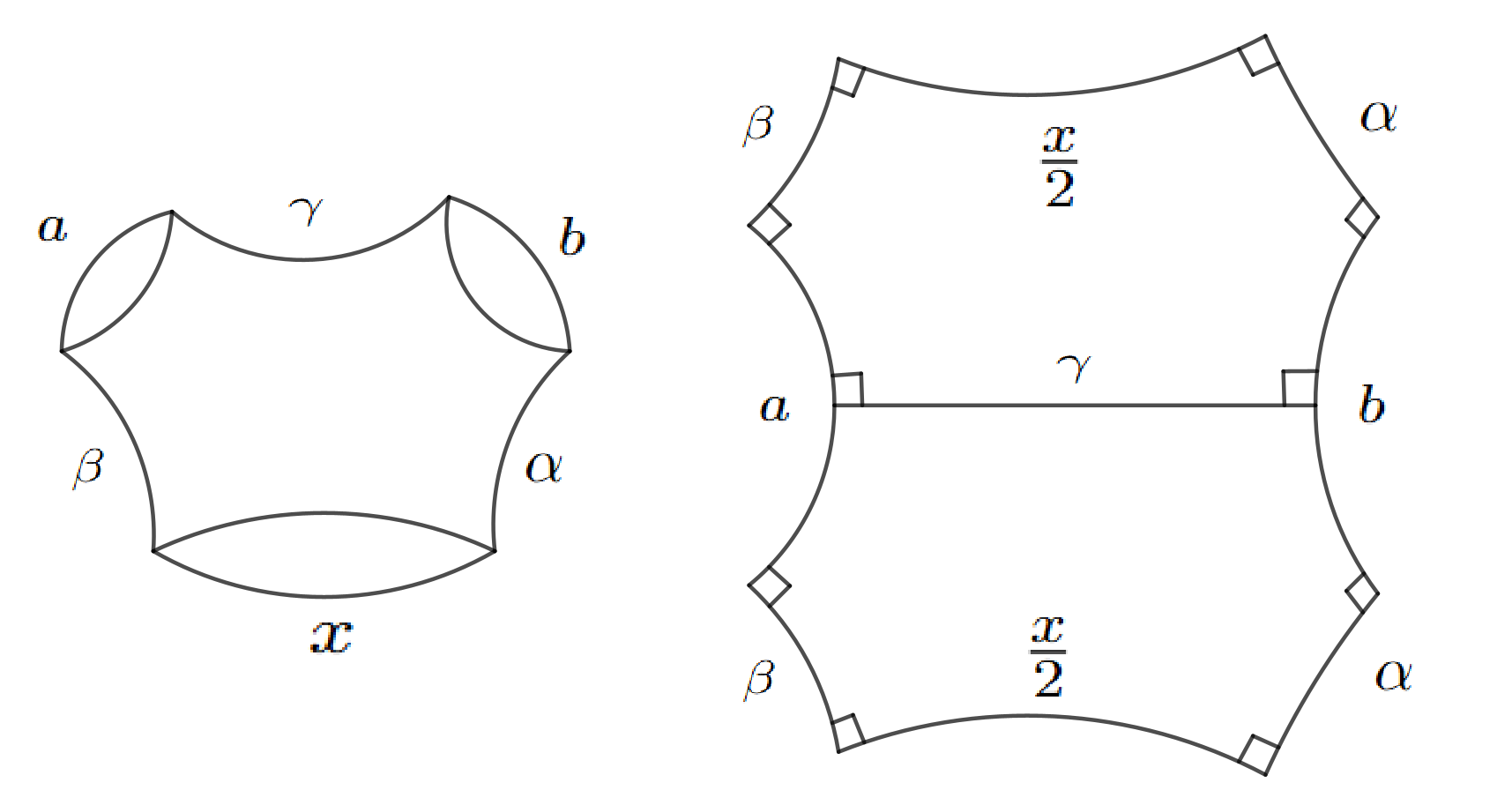}
    \caption{Pants and octagon}
    \label{fig pants}
\end{figure}

Hyperbolic trigonometry says that 
\begin{equation}
    \cosh\alpha = \frac{\cosh\frac{a}{2}+\cosh\frac{b}{2}\cosh\frac{x}{2}}{\sinh\frac{b}{2}\sinh\frac{x}{2}},
\end{equation}
\begin{equation}
    \cosh\beta = \frac{\cosh\frac{b}{2}+\cosh\frac{a}{2}\cosh\frac{x}{2}}{\sinh\frac{a}{2}\sinh\frac{x}{2}},
\end{equation}
\begin{equation}
    \cosh\gamma = \frac{\cosh\frac{x}{2}+\cosh\frac{a}{2}\cosh\frac{b}{2}}{\sinh\frac{a}{2}\sinh\frac{b}{2}}.
\end{equation}
So when fixing $a,b$ and decreasing $x$, then $\gamma$ decreases and $\alpha,\beta$ increase. 

In the pants, for a closed geodesic or a geodesic segment with endpoints on $a$ or $b$, cut it whenever it intersects with $\alpha$ or $\beta$. Then it will be cut into pieces of segments (denoted as $d$ in Figure \ref{fig 9cases}) of the following 9 cases. For example, the red geodesic in Figure \ref{fig e.g. cut geod} will be cut into segments of cases (4), (3), and (2) in sequence (starting from the end point on $a$).

\renewcommand{\thesubfigure}{\arabic{subfigure}}

\begin{figure}[htbp]
  \centering
  \begin{subfigure}[b]{0.3\linewidth}
      \centering
      \includegraphics[width=\linewidth]{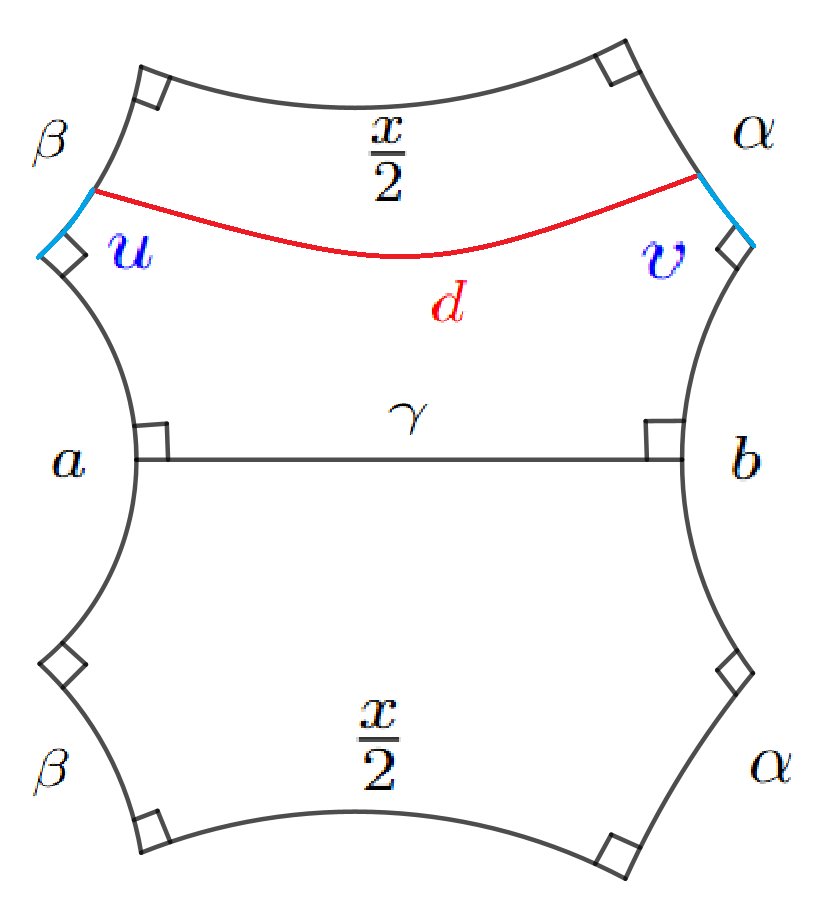}
      \caption{}
      \label{}
  \end{subfigure}
  \hfill
  \begin{subfigure}[b]{0.3\linewidth}
      \centering
      \includegraphics[width=\linewidth]{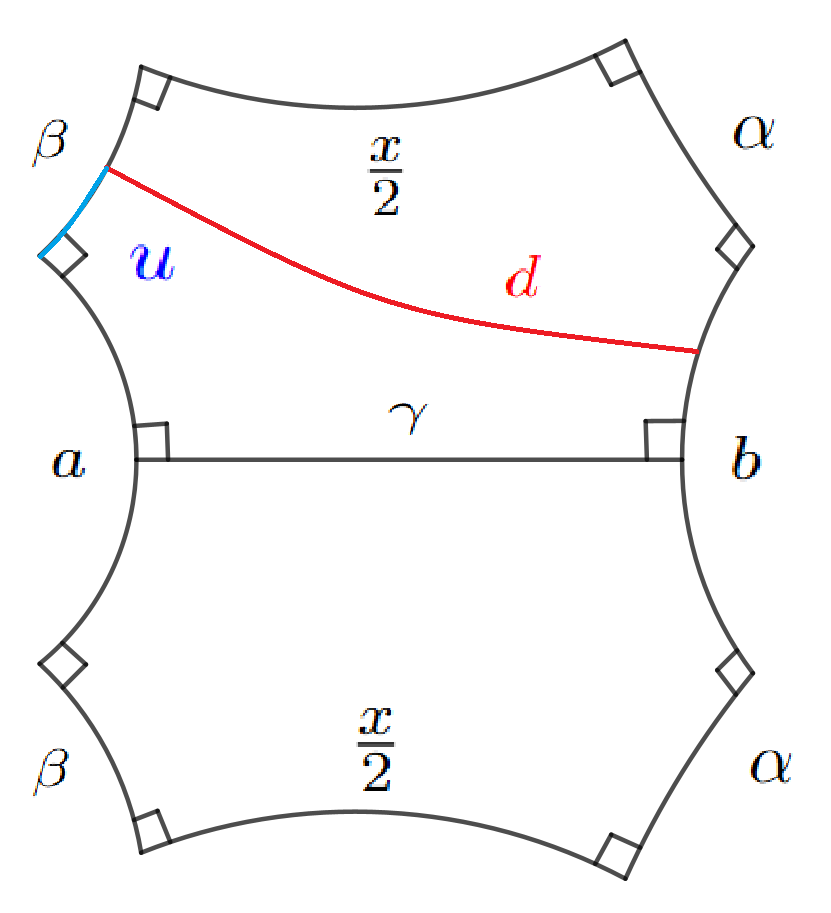}
      \caption{}
      \label{}
  \end{subfigure}
  \hfill
  \begin{subfigure}[b]{0.3\linewidth}
      \centering
      \includegraphics[width=\linewidth]{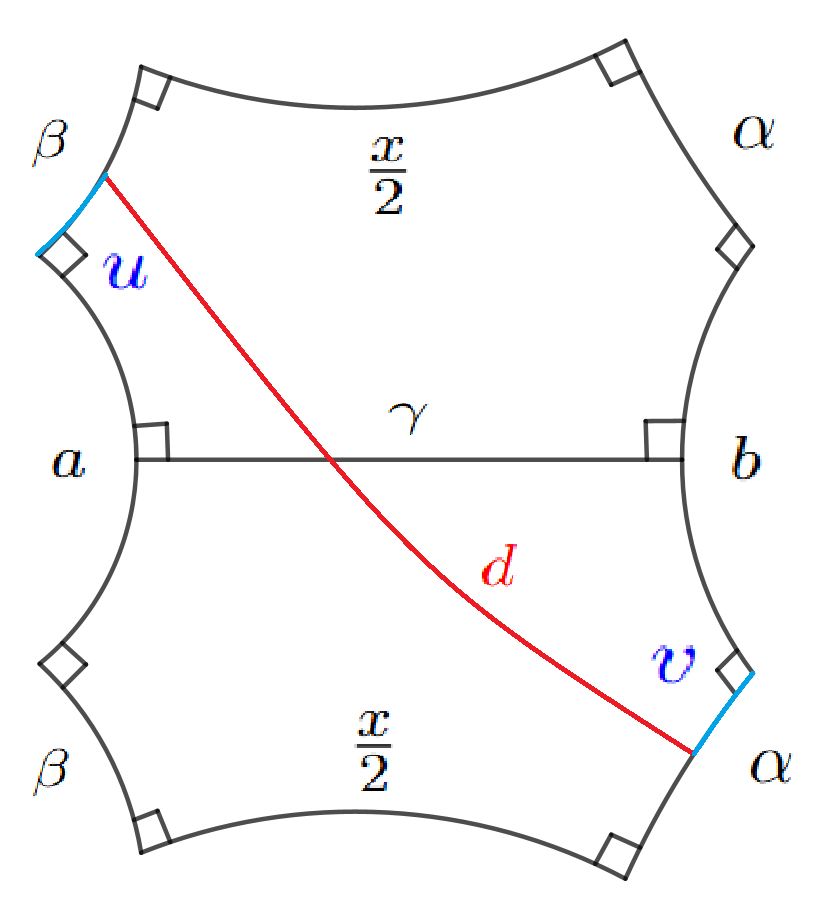}
      \caption{}
      \label{}
  \end{subfigure}
  \\
  \begin{subfigure}[b]{0.3\linewidth}
      \centering
      \includegraphics[width=\linewidth]{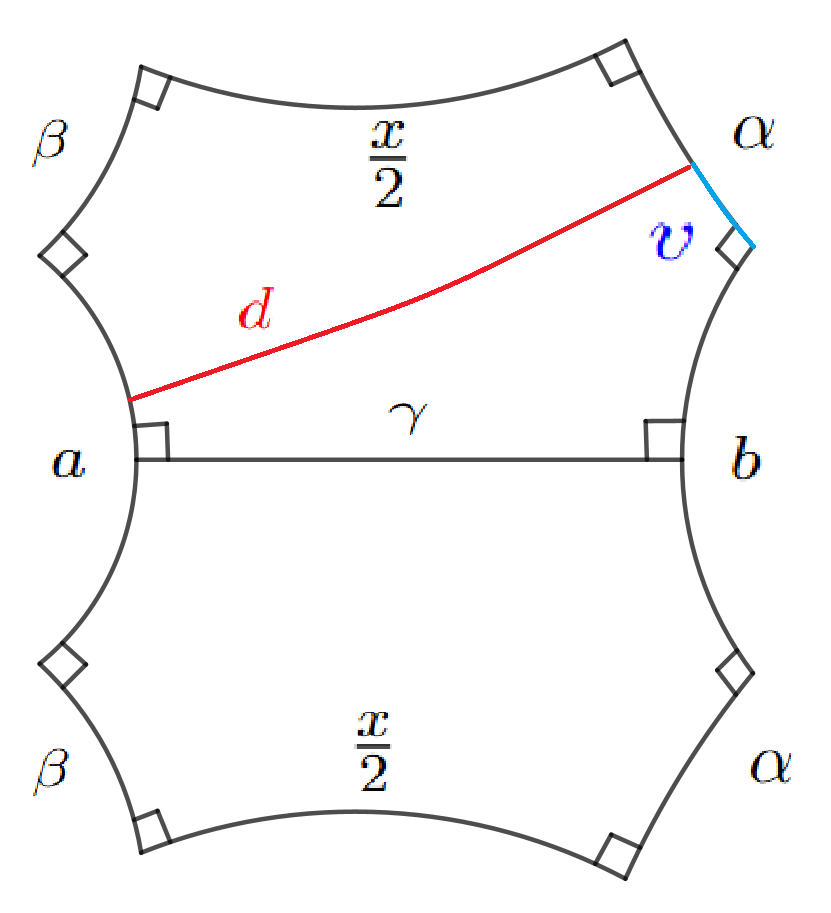}
      \caption{}
      \label{}
  \end{subfigure}
  \hfill
  \begin{subfigure}[b]{0.3\linewidth}
      \centering
      \includegraphics[width=\linewidth]{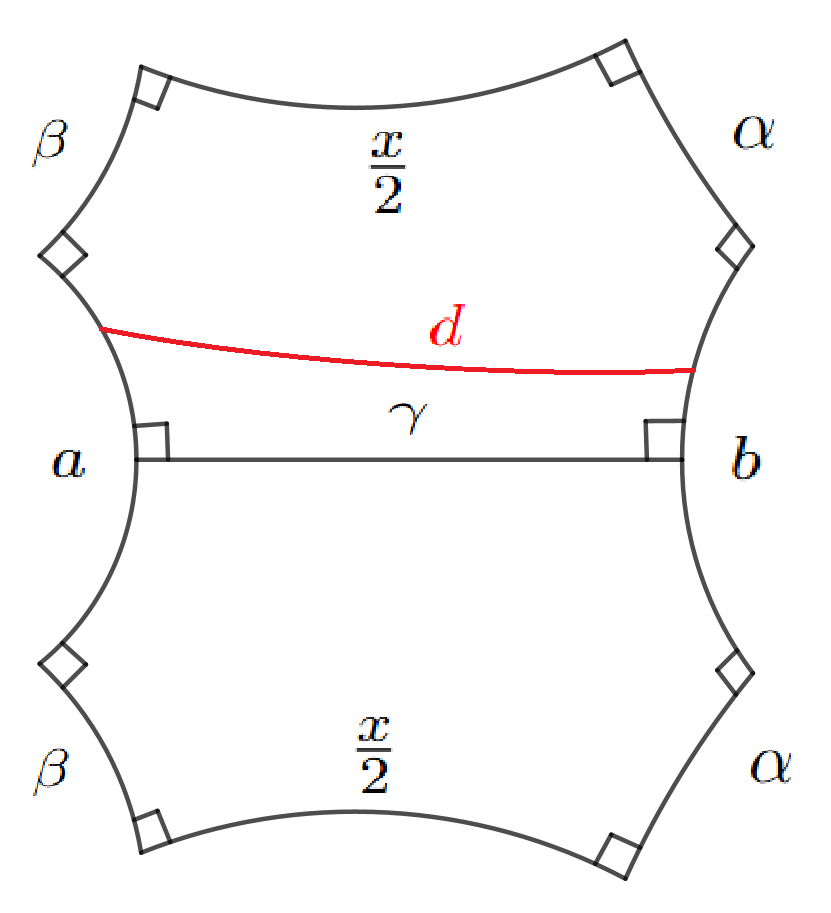}
      \caption{}
      \label{}
  \end{subfigure}
  \hfill
  \begin{subfigure}[b]{0.3\linewidth}
      \centering
      \includegraphics[width=\linewidth]{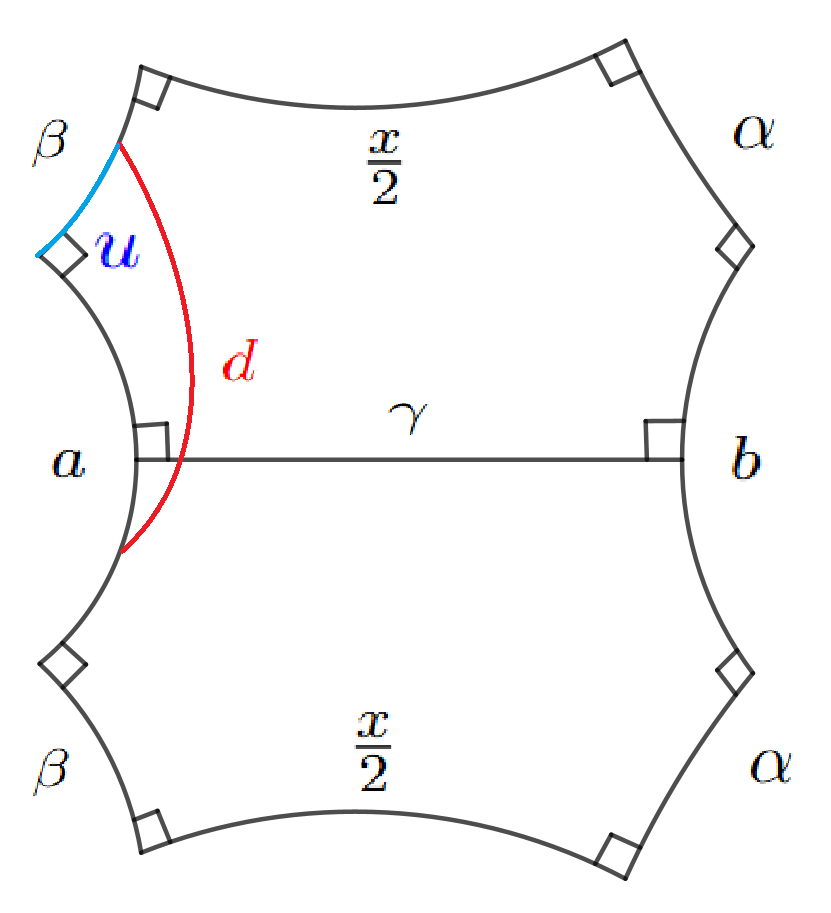}
      \caption{}
      \label{}
  \end{subfigure}
  \\\begin{subfigure}[b]{0.3\linewidth}
      \centering
      \includegraphics[width=\linewidth]{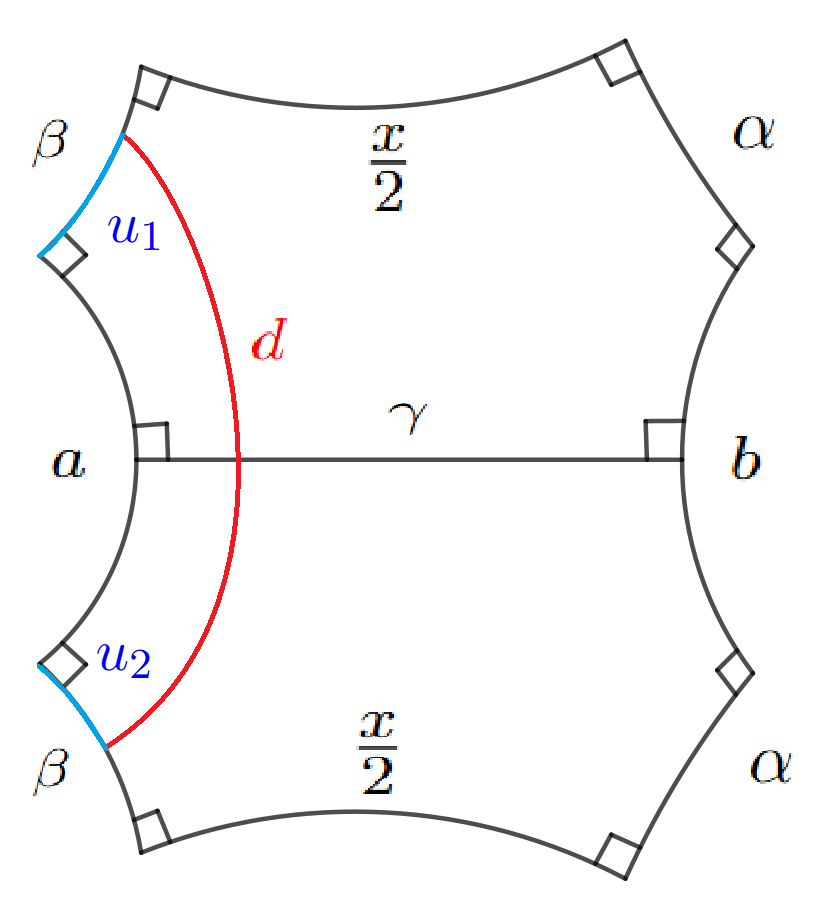}
      \caption{}
      \label{}
  \end{subfigure}
  \hfill
  \begin{subfigure}[b]{0.3\linewidth}
      \centering
      \includegraphics[width=\linewidth]{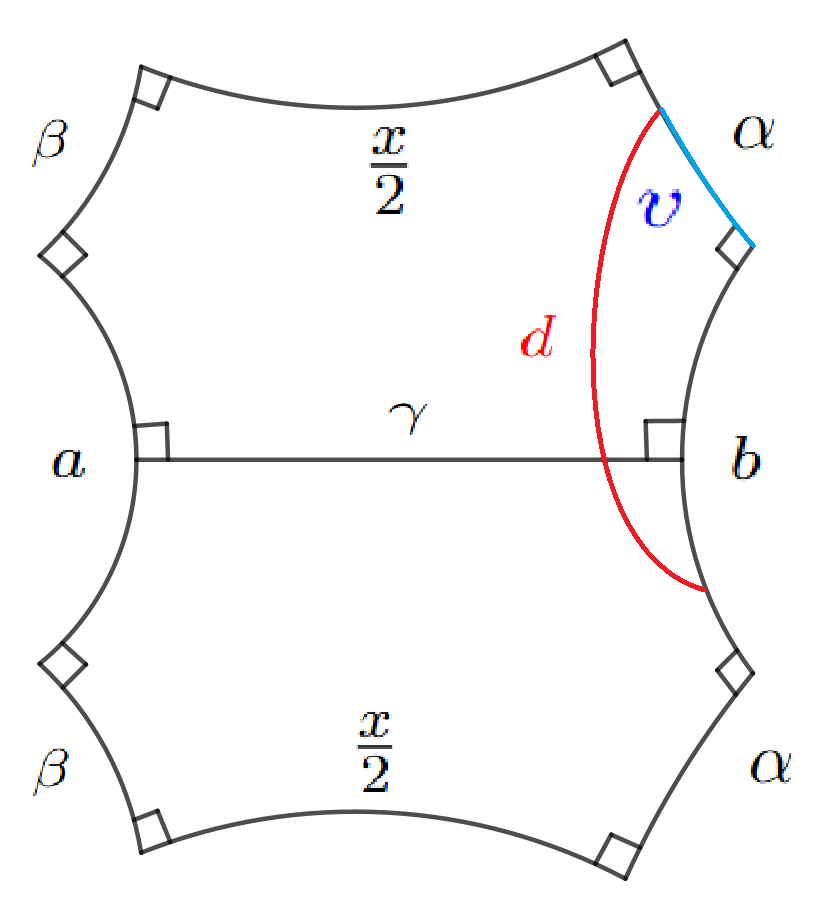}
      \caption{}
      \label{}
  \end{subfigure}
  \hfill
  \begin{subfigure}[b]{0.3\linewidth}
      \centering
      \includegraphics[width=\linewidth]{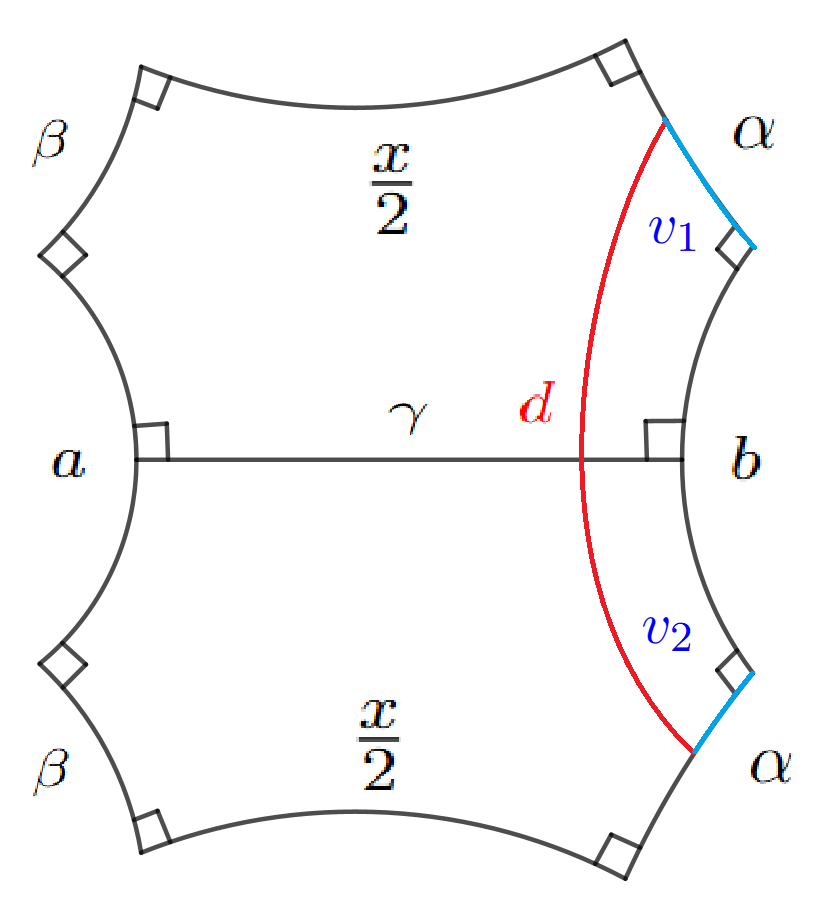}
      \caption{}
      \label{}
  \end{subfigure}
  \caption{Cases of geodesic segments cut by $\alpha$ and $\beta$.}
    \label{fig 9cases}
\end{figure}

\begin{figure}[h]
    \centering
    \includegraphics[width=2 in]{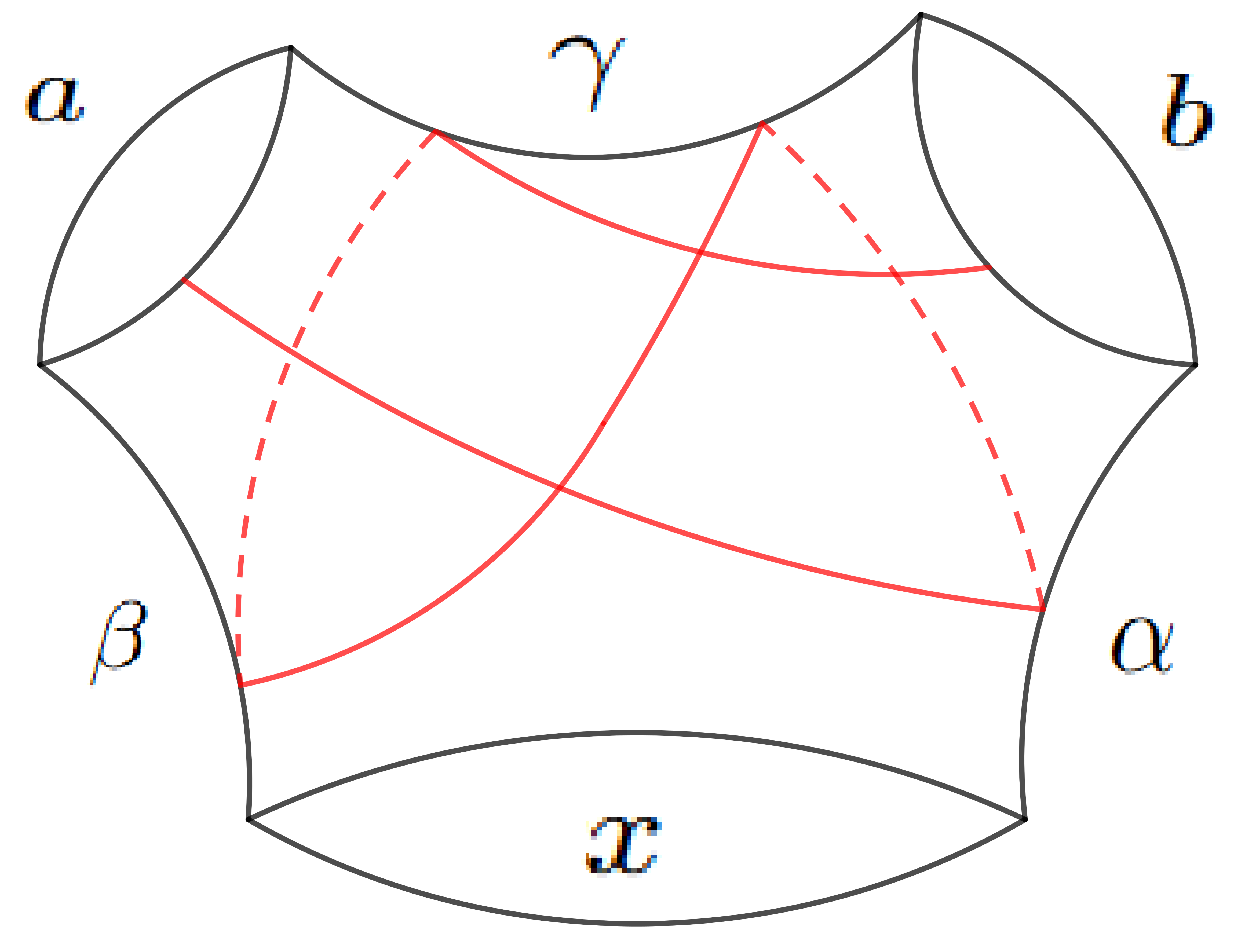}
    \caption{Example: this red geodesic will be cut into segments of cases (4), (3) and (2) in sequence.}
    \label{fig e.g. cut geod}
\end{figure}

In each case, denote by $u$ (resp. $v$) the part of $\beta$ (resp. $\alpha$) which is between $d$ and $a$ (resp. $b$), as shown the blue segments in Figure \ref{fig 9cases}. 

Now we construct the piecewise geodesic $\eta'_{Y_1}$ and smooth geodesic $\eta_{Y_1}$ in $Y_1$. By the definition of $Y_1$, $Y_1$ and $X$ have differences only in the pants that contains the boundary $x_1$. So we may define $\eta'_{Y_1}$ to be the same geodesic segments as $\eta_X$ outside this pants. And in this boundary pants, $\eta_X$ can be split into some pieces of geodesic segments of 9 cases shown in Figure \ref{fig 9cases}, with the cutting points given by every time it intersects with $\alpha$ and $\beta$ (here we assume $x$ in Figure \ref{fig pants} and \ref{fig 9cases} is the boundary $x_1$). When switching $X$ to $Y_1$, $x$ and $\gamma$ decrease, $\alpha$ and $\beta$ increase, and $a$ and $b$ are fixed. For each geodesic segment $d$ in $X$ as above, we fix its endpoints (fix the length $u,v$ if they are on $\beta,\alpha$) and get the corresponding geodesic segment in $Y_1$. By connecting these geodesic segments in $Y_1$, we get a piecewise geodesic homotopic to $\eta$, denoted by $\eta'_{Y_1}$. Let the corresponding smooth geodesic in $Y_1$ homotopic $\eta$ to be $\eta_{Y_1}$. To prove $\ell(\eta_X)\geq\ell(\eta'_{Y_1})\geq\ell(\eta_{Y_1})$, we only need to show that in each case of  Figure \ref{fig 9cases}, $d$ decreases when $x$ decreases and $a,b,u,v$ are fixed.

In cases (6), (7), (8), and (9), the length of $d$ is fixed obviously since in the triangle or quadrilateral that contains $d$, other edges are fixed. In the first 5 cases, draw perpendiculars from the two endpoints of $d$ to $\gamma$ respectively, denoted by $h_1$ and $h_2$. Denote by $p$ and $q$ the left and right pieces of $\gamma$ cut by the foots of $h_1$ and $h_2$. See Figure \ref{fig compute d}. If any endpoint of $d$ is on $a$ or $b$, then $h_i$ and $p$ or $q$ of its side disappear, which makes the computation simpler. Since $a,b,u,v$ are fixed, by standard hyperbolic trigonometry we know $h_1,h_2,p,q$ are fixed. And 
\begin{equation}\label{eq rectangle}
    \cosh d = \cosh h_1 \cosh h_2 \cosh(\gamma-p-q) - \sinh h_1 \sinh h_2.
\end{equation}
In equation \eqref{eq rectangle} $h_1, h_2$ are allowed to be negative, where $h_1\cdot h_2 <0$ means they are in the different sides of $\gamma$ and $d$ intersects with $\gamma$. Since $h_1,h_2,p,q$ are fixed and $\gamma$ decreases, we know that $d$ decreases. 

\begin{figure}[h]
    \centering
    \includegraphics[width=2.5 in]{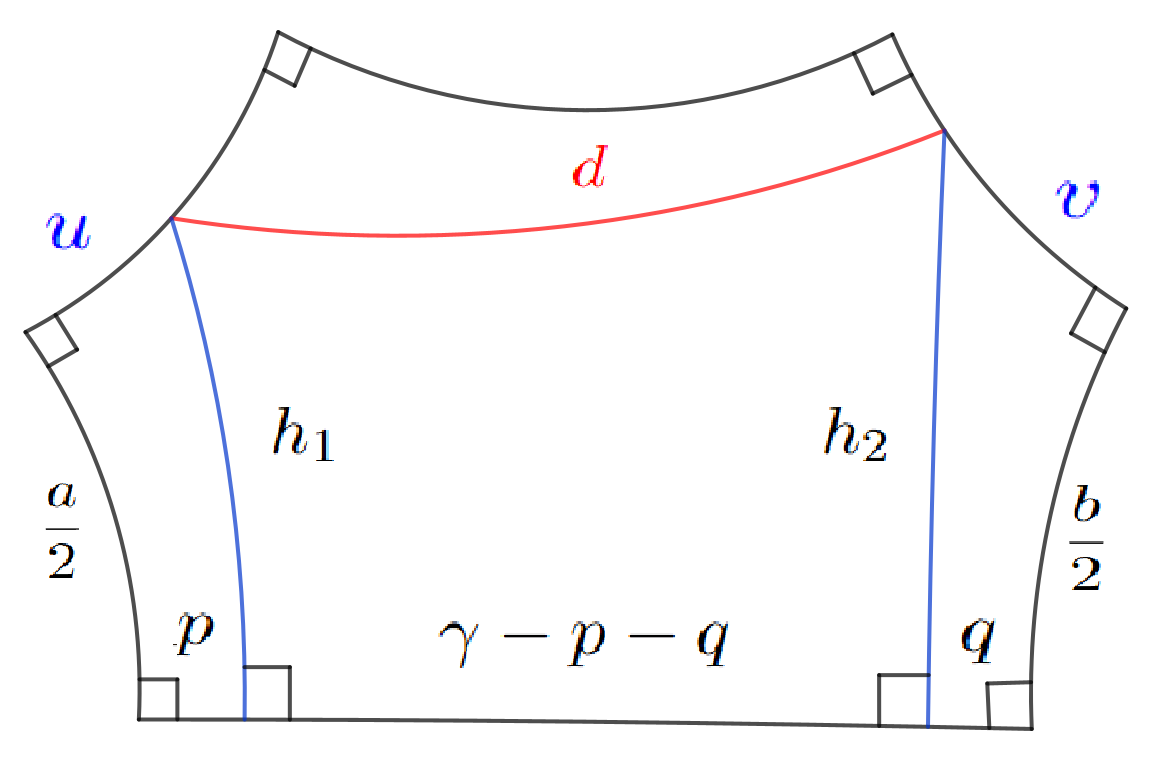}
    \caption{}
    \label{fig compute d}
\end{figure}

In every cases, $d$ decreases or remain fixed. So by the construction, we have 
$$\ell(\eta_X)\geq\ell(\eta'_{Y_1})\geq\ell(\eta_{Y_1}).$$
And in the same way, based on $\eta_{Y_1}$, we may construct piecewise geodesic $\eta'_{Y_2}$ and smooth geodesic $\eta_{Y_2}$ on $Y_2$, and $\ell(\eta_{Y_1})\geq\ell(\eta'_{Y_2})\geq\ell(\eta_{Y_2})$. By conduction, we may construct $\eta'_{Y_i},\eta_{Y_i}$ on $Y_i$ for $i=1,2,\cdots,n$ with $\eta_Y=\eta_{Y_n}$ and
$$\ell(\eta_X)\geq\ell(\eta'_{Y_1})\geq\ell(\eta_{Y_1})\geq\cdots\geq\ell(\eta'_{Y_n})\geq\ell(\eta_{Y_n})=\ell(\eta_Y).$$
So for any closed curve $\eta\sbs S_{g,n}$, we have $\ell_\eta(X)\geq\ell_\eta(Y)$.
\end{proof}

\subsection{Proof of Theorem \ref{thm double filling count}}
\begin{lemma} \label{thm double filling length}
There exists a universal constant $L_0>0$ such that for any $0<\eps<1$, $m=2g-2+n\geq 1$, $\Delta\geq 0$ and $\sum_{i=1}^n x_i \geq \Delta +L_0\frac{mn}{\eps}$, the following holds: for any hyperbolic surface $X\in \T_{g,n}(x_1,\cdots,x_n)$, one can always find a new hyperbolic surface $Y\in \T_{g,n}(y_1,\cdots,y_n)$ satisfying
\begin{enumerate}
\item $y_i\leq x_i$ for all $1\leq i\leq n$;
\item $\sum_{i=1}^n x_i - \sum_{i=1}^n y_i = \Delta$;
\item for any double-filling $k$-tuple $\Gamma=(\gamma_1,\cdots,\gamma_k)$ in $S_{g,n}$, we have
$$\sum_{i=1}^k\ell_{\gamma_i}(X) - \sum_{i=1}^k\ell_{\gamma_i}(Y) \geq (1-\eps)\Delta.$$ 
\end{enumerate}
\end{lemma}

\begin{proof}
The surface $Y$ and the proof are identical to the one of \cite[Theorem 38]{wx22-3/16} by replacing filling curves with double-filling k-tuples. Here we briefly recall how $Y$ is constructed and the places where the proof change. One may refer to \cite[Section 8]{wx22-3/16} and \cite[Appendix A]{WX2022prime} for more details.

Consider a longest boundary component of $X$ and denote it and its length by $2b$. Let $w$ be the width of the maximal embedded half-collar of $2b$ in $X$, that is, 
$$w=\sup \left\{d>0;\ 
\begin{aligned}
    &\{p\in X;\ \dist(p,2b)<d\} \\ 
    &\text{is homeomorphic to a cylinder and} \\
    &\text{does not intersect with} \ (\partial X\setminus 2b)
\end{aligned} \right\}.$$
Then the closed curve $\{p\in X;\ \dist(p,2b)=w\}$ either touches another component of $\partial X$ (type 2 in \cite{wx22-3/16}) or has non-empty self-intersection (type 1 in \cite{wx22-3/16}). For both types, it induces a pair of pants, denoted by $\mathcal{P}$, in which $2a,2b,2c$ are boundary curves (see \cite[Page 382 Construction]{wx22-3/16} for details). 

Then for any small $\delta>0$, decrease the boundary length $(2a,2b,2c)$ of $\mathcal{P}$ to be $(2a,2b-\delta,2c)$ for type 1 and to be $(2a,2b-\frac{1}{2}\delta,2c-\frac{1}{2}\delta)$ for type 2. Glue the new pants back to $X\setminus\mathcal{P}$ (with the twists unchanged) and then get a surface $X_\delta$ whose total boundary length decreases by $\delta$ compared to $X$ (see \cite[Page 383 Construction]{wx22-3/16} for details). \cite[Proposition 40]{wx22-3/16} showed that when $\ell(\partial X)$ is large enough, such construction of $X_\delta$ exists and for any filling closed curve $\eta$ in $S_{g,n}$ has 
\begin{equation}\label{eq filling length decrease}
    \ell_\eta(X)-\ell_\eta(X_\delta)\geq \frac{1}{2}\left(1-O(\frac{mn}{\sum x_i})\right)\delta.
\end{equation}
Here for any double-filling $k$-tuple $\Gamma=(\gamma_1,\cdots,\gamma_k)$ in $S_{g,n}$, with the same proof we may erase the $\frac{1}{2}$ in \eqref{eq filling length decrease}: 
\begin{equation}\label{eq double filling length decrease}
    \sum_{i=1}^k\ell_{\gamma_i}(X)-\sum_{i=1}^k\ell_{\gamma_i}(X_\delta)\geq \left(1-O(\frac{mn}{\sum x_i})\right)\delta.
\end{equation}
Then repeat the above procedure of reducing the total boundary length by $\delta$ (See \cite[Page 384 Proof of Theorem 38]{wx22-3/16} for the exact procedure). Finally we will get the surface $Y$ with total boundary length reduced by $\Delta$ and 
\begin{equation}
    \sum_{i=1}^k\ell_{\gamma_i}(X)-\sum_{i=1}^k\ell_{\gamma_i}(Y)\geq \left(1-\eps\right)\Delta.
\end{equation}
To prove \eqref{eq double filling length decrease}, the only two differences from the proof of \eqref{eq filling length decrease} are to modify \cite[Lemma 46, Lemma 56]{wx22-3/16} as follow:
\begin{lemma}\label{thm modify lem 46 56 wx22-3/16}
\ 
\begin{enumerate}
    \item If pants $\mathcal{P}$ is of type 1, then $\Gamma\cap\mathcal{P}$ must contain at least two segments of types 1.1, 1.2, 1.5, 1.6, 1.7, 1.8, 1.13 and 1.14 as shown in \cite[Page 393, Figure 9]{wx22-3/16}.
    \item If pants $\mathcal{P}$ is of type 2, then $\Gamma\cap\mathcal{P}$ must contain at least four segments of types 2.2, 2.3 and 2.4 as shown in \cite[Page 403, Figure 12]{wx22-3/16}.
\end{enumerate}
\end{lemma}
For type 1, only using the filling property of $\Gamma$, \cite[Lemma 46]{wx22-3/16} says that $\Gamma\cap\mathcal{P}$ must contain at least one segment of types 1.1, 1.2, 1.5, 1.6, 1.7, 1.8, 1.13 or 1.14. However, here we assume that $\Gamma$ is double-filling in addition. So, if there is only one such segment, the inner-boundary of cylinder component of $X\setminus\Gamma$ containing $2b$ will self-overlap on that segment, which contradict with the double-filling property. 

For type 2, since $\Gamma$ is filling, it must intersect with $\alpha$. And in addition, if $\Gamma$ only intersects with $\alpha$ for once, the inner-boundary of cylinder components of $X\setminus\Gamma$ containing $2b$ and $2c$ will overlap on a segment passing that intersection point, which contradicts with the double-filling property. And at each time that $\Gamma$ intersects with $\alpha$, on both sides of the intersection point, there is one segment of type 2.2 or 2.3 or 2.4. Then there are at least four such segments.

Hence Lemma \ref{thm modify lem 46 56 wx22-3/16} is true. 

Then same as the proof in \cite{wx22-3/16}, for type 1 that decreases $(2a,2b,2c)$ to be $(2a,2b-\delta,2c)$, each segment in Lemma \ref{thm modify lem 46 56 wx22-3/16} (1) has length reducing $\frac{1}{2}\left(1-O(\frac{m^2 n^2}{(\sum x_i)^2})\right)\delta$; and for type 2 that decreases $(2a,2b,2c)$ to be $(2a,2b-\frac{1}{2}\delta,2c-\frac{1}{2}\delta)$, each segment in Lemma \ref{thm modify lem 46 56 wx22-3/16} (2) has length reducing $\frac{1}{4}\left(1-O(\frac{mn}{\sum x_i})\right)\delta$. So by Lemma \ref{thm modify lem 46 56 wx22-3/16} for both types the length of $\Gamma$ reduce at least $\left(1-O(\frac{mn}{\sum x_i})\right)\delta$, which shows \eqref{eq double filling length decrease}, and finishes the proof of Lemma \ref{thm double filling length}.
\end{proof}

\begin{proof}[Proof of Theorem \ref{thm double filling count}]
    If the total boundary length $\sum x_i \leq L_0 \frac{mn}{\eps}$ where $L_0$ is the constant in Lemma \ref{thm double filling length}, then \cite[Lemma 64]{WX2022prime} directly shows that 
    \begin{eqnarray*}
        N_k^{2-fill}(X,T) &\prec_k & m^k (1+T)^{k-1} e^{T} \\
        &\prec_{k,\eps,m}& (1+T)^{k-1} e^{T-(1-\eps)\sum x_i}.
    \end{eqnarray*}

    If the total boundary length $\sum x_i > L_0 \frac{mn}{\eps}$, let $\Delta= \sum x_i - L_0\frac{mn}{\eps}$ and $Y$ be the surface obtained in Lemma \ref{thm double filling length}. Since every double-filling $k$-tuple $\Gamma=(\gamma_1,\cdots,\gamma_k)$ has
    $$\sum_{i=1}^k\ell_{\gamma_i}(X) - \sum_{i=1}^k\ell_{\gamma_i}(Y) \geq (1-\eps)\Delta,$$ 
    the numbers of double-filling $k$-tuple have the inequality 
    $$N_k^{2-fill}(X,T) \leq N_k^{2-fill}(Y,T-(1-\eps)\Delta).$$
    Then again applying \cite[Lemma 64]{WX2022prime} for $Y$ we have 
    \begin{eqnarray*}
        N_k^{2-fill}(Y,T-(1-\eps)\Delta) &\prec_k & m^k (1+T-(1-\eps)\Delta)^{k-1} e^{T-(1-\eps)\Delta} \\
        &\prec_{k,\eps,m}& (1+T)^{k-1} e^{T-(1-\eps)\sum x_i}.
    \end{eqnarray*}
    Hence, we complete the proof.
\end{proof}

\section{asymptotic of Weil-Petersson volumes and intersection numbers}\label{appendix wp volume}
Firstly, based on the acknowledgment of several rough summation estimates,
we will prove the following three asymptotic estimates of Weil-Petersson volumes when $n=o(\sqrt{g})$. Following the algorithm in \cite{MZ15}, very similar calculations can be found for fixed $n$ in \cite{MZ15, NWX23, AM22-JMP}, for $n=o(\sqrt{g})$ in \cite{hide2023spectral, shenwu2022arbitrarily}, and  for  $g\geq n^c$ with $c\succ 1$ in \cite{hide2025spectral}. At the end of this appendix, we prove several rough estimates.

\begin{lemma}[Lemma \ref{lemma NWX vgn(x)}]\label{appendix Vgn(x)/Vgn cn/g}
  For $n=n(g)=o(\sqrt{g})$, there exists a uniform constant $c>0$ such that $$
\prod_{i=1}^n \frac{\sinh(x_i/2)}{x_i/2}\left(1-cn\frac{x_1^2+\cdots+x_n^2}{g}\right)\leq \frac{V_{g,n}(x_1,\cdots,x_n)}{V_{g,n}}\leq  \prod_{i=1}^n \frac{\sinh(x_i/2)}{x_i/2}.
$$
\end{lemma}

\begin{theorem}[Theorem  \ref{thm vgn/vgn+1 n^2+1/g^2}]\label{appendix vgn/vgn+1 n^2+1/g^2}
    For  $n=n(g)=o(\sqrt{g})$ :
    \begin{align}
       \label{appendix vgn/vgn+1 n^2+1/g^2 eq1} &\frac{V_{g,n+1}}{8\pi^2g V_{g,n}}=1+ \left(\left(\frac{1}{2}-\frac{1}{\pi^2}\right)n-\frac{5}{4}+\frac{2}{\pi^2}\right)\cdot\frac{1}{g}+O\left(\frac{1+n^2}{g^2}\right),\\
    \label{appendix vgn/vgn+1 n^2+1/g^2 eq2} &\frac{V_{g-1,n+2}}{V_{g,n}}=1+\frac{3-2n}{\pi^2}\cdot\frac{1}{g}+O\left(\frac{1+n^2}{g^2}\right).
 \end{align}
    The implied constants here are independent of $n$ and $g$.
\end{theorem}

\begin{theorem}[Theorem \ref{thm vgn(x)/vgn 1+n^3/g^2}]\label{appendix vgn(x)/vgn 1+n^3/g^2} For any fixed integer $k\geq 1$, if $n=n(g)=o(\sqrt{g})$ and $n\geq k$, then  for any $\textbf{x}=(x_1,\cdots,x_k,\cdots, 0^{n-k})\in \mathbb{R}_{\geq 0}^n$, as $g\to\infty$,\begin{align*}
&\frac{V_{g,n}(\textbf{x})}{V_{g,n}}=\prod_{i=1}^n \frac{\sinh(x_i/2)}{x_i/2}
+\frac{f_n^1(\textbf{x})}{g}\\
+&O_k\left(\frac{n^3(1+x_1+\cdots+x_k)^4}{g^2}\exp{\left(\frac{x_1+\cdots+x_k}{2}\right)}\right),
\end{align*}
where \begin{align*}
f_n^1(\textbf{x})=&\frac{1}{\pi^2}\sum_{i=1}^n\left[\cosh(x_i/2)+1-(\frac{x_i^2}{16}+2)\frac{\sinh(x_i/2)}{x_i/2}\right]\prod_{l\neq i}\frac{\sinh(x_l/2)}{x_l/2}\\
-&\frac{1}{2\pi^2}\sum_{1\leq i<j\leq n}\left[\cosh(x_i/2)\cosh(x_j/2)+1-2\frac{\sinh(x_i/2)}{x_i/2}\frac{\sinh(x_j/2)}{x_j/2}\right]\\
\cdot&\prod_{l\neq i,j}\frac{\sinh(x_l/2)}{x_l/2},
\end{align*} 
with $x_{k+1}=\cdots=x_n=0$, 
and the implied constant here is independent of $n,g,x_1,\cdots,x_k$ but related to $k$.
\end{theorem}

\subsection{Intersections of tautological classes on $\overline{\M}_{g,n}$ and recursive formulas.} 
In this subsection, we recall some well-known facts on tautological classes on the moduli space $\overline{\M}_{g,n}$ of stable $n$-pointed curves and their intersection numbers. For each $1\leq i\leq n$  there is a tautological line bundle $\sL_i$ on $\overline{\M}_{g,n}$ whose fiber at $(C,x_1,\cdots,x_n)$ is the cotangent line to $C$ at $x_i$. Take $\psi_i=c_1(\sL_i)\in H_2(\overline{\M}_{g,n},\mathbb{Q})$.  The first Mumford class is defined as $\kappa_1=p_*\psi_{n+1}^2\in H_2(\overline{\M}_{g,n},\mathbb{Q})$ via the forgetful map $p:\overline{\sM}_{g,n+1}\to \overline{\sM}_{g,n}$. We have $\kappa_1=\frac{\omega_{wp}}{2\pi^2}$. See e.g. \cite{AC96-JAG} or \cite{harris2006moduli} for more details.

For $\textbf{d}=(d_1,\cdots,d_n)$ with $d_i\in \mathbb{N}_{\geq 0}$ and $|\textbf{d}|=\sum_{i=1}^n d_i$, if $|\textbf{d}|\leq 3g-3+n$, we set $d_0=3g-3+n-|\textbf{d}|$, and $$
[\tau_{d_1}\cdots\tau_{d_n}]_{g,n}=\left[\prod_{i=1}^n\tau_{d_i}\right]_{g,n}=\frac{(2\pi^2)^{d_0}}{d_0^2}\prod_{i=1}^n 2^{2d_i}(2d_i+1)!!\int_{\overline{\sM}_{g,n}}\kappa_1^{d_0}\psi_1^{d_1}\cdots \psi_n^{d_n}.
$$
For $|\textbf{d}|>3g-3+n$, we set $[\tau_{d_1}\cdots\tau_{d_n}]_{g,n}=0.$
Then by \cite[Theorem 1.1]{mir2007-JAMS}, \begin{equation}\label{eq-vgn(L)-intersect number}
    V_{g,n}(2L_1,\cdots,2L_n)=\!\!\sum_{|\textbf{d}|\leq 3g-3+n}\left[\prod_{i=1}^n \tau_{d_i}\right]_{g,n}\frac{L_1^{2d_1}}{(2d_1+1)!}\cdots \frac{L_n^{2d_n}}{(2d_n+1)!},
\end{equation}
and $V_{g,n}=[\tau_0^n]_{g,n}$ in particular.

For $|\textbf{d}|\leq 3g-3+n$, the following recursive formulas hold: 
 \begin{equation}\label{recursive I}\tag{\textbf{\textrm{I}}}
\begin{aligned}
\left[\tau_0\tau_1\prod_{i=1}^n\tau_{d_i}\right]_{g,n+2}=&\left[\tau_0^4\prod_{i=1}^n\tau_{d_i}\right]_{g-1,n+4}\\
+&6\sum_{g_1+g_2=g
    \atop  I\sqcup J=\{1,\cdots,n\}}  \left[\tau_0^2 \prod_{i\in I}\tau_{d_i} \right]_{g_1,|I|+2}\left[\tau_0^2 \prod_{i\in J}\tau_{d_i}\right]_{g_2,|J|+2}.
    \end{aligned} 
    \end{equation} 
     \begin{equation}\label{recursive II}\tag{\textbf{\textrm{II}}}
    (2g-2+n)\left[\prod_{i=1}^n\tau_{d_i}\right]_{g,n}=\frac{1}{2}\sum_{L=1}^{3g-2+n}\frac{(-1)^{L-1} L\pi^{2L-2}}{(2L+1)!}\left[ \tau_{L}\prod_{i=1}^n\tau_{d_i}\right]_{g,n+1}.
     \end{equation}
     If $a_0=\frac{1}{2}$ and $a_i=\zeta(2i)(1-2^{1-2i})$ for $i\geq 1$, then \begin{equation}\label{recursive III}\tag{\textbf{\textrm{III}}}
         \left[\prod_{g,n}\tau_{d_i}\right]_{g,n}=\sum_{j=2}^nA_{\textbf{d}}^j+B_{\textbf{d}}+C_{\textbf{d}},
     \end{equation}
     where \begin{equation}\label{int-Adj}
    A_{\textbf{d}}^j=8\sum_{L=0}^{d_0}(2d_j+1)a_L\left[\tau_{L+d_1+d_j-1}\prod_{i\neq 1,j}\tau_{d_i} \right]_{g,n-1},
    \end{equation}
    \begin{equation}\label{int-Bd}
    B_{\textbf{d}}=16\sum_{L=0}^{d_0}\sum_{k_1+k_2=L+d_1-2}a_L\left[ \tau_{k_1}\tau_{k_2}\prod_{i\neq 1}\tau_{d_i}\right]_{g-1,n+1},
    \end{equation}
    \begin{equation}\label{int-Cd}
    \begin{aligned}
    C_{\textbf{d}}
    =&16\sum_{g_1+g_2=g\atop I\sqcup J=\{ 2,\cdots,n\}}\sum_{L=0}^{d_0}\\
    &\sum_{k_1+k_2=L+d_1-2}a_L\left[\tau_{k_1}\prod_{i\in I}\tau_{d_i} \right]_{g_1,|I|+1}\left[\tau_{k_2}\prod_{i\in J}\tau_{d_i} \right]_{g_1,|J|+1}.
    \end{aligned}
    \end{equation}

Where \eqref{recursive I} is a special case of \cite[Proposition 3.3]{LX09}, \eqref{recursive II} is shown in \cite{DN09, LX09}, and \eqref{recursive III} is proved in \cite{mir07-Inv}.

\subsection{Basic estimates}
\begin{lemma}\label{lemma-ai-MZ15}
\cite[Lemma 2.2]{MZ15}
  The sequence $\{a_i\}_{i=0}^\infty$ satisfies \begin{enumerate}
      \item $\lim_{n\to\infty} a_i=1$
      \item $\sum_{i=0}^\infty i(a_{i+1}-a_i)=\frac{1}{4}. $
      \item there exists $C>0$ such that for all $i\geq 0$,
      $$\frac{1}{C\cdot 2^{2i}}\leq a_{i+1}-a_i\leq  \frac{C}{2^{2i}}.$$
  \end{enumerate}  
\end{lemma}
\begin{lemma}\label{appendix  mir13 [tau_d]<vgn}
\cite[Lemma 3.2]{mir13}    
For any $\textbf{d}=(d_1,\cdots,d_n)$, $$
[\tau_{d_1+1}\tau_{d_2}\cdots\tau_{d_n}]_{g,n}\leq [\tau_{d_1}\tau_{d_2}\cdots\tau_{d_n}]_{g,n}\leq [\tau_0^n]_{g,n}=V_{g,n}.
$$
\end{lemma}
\begin{lemma}\label{MZ15-universal}\cite[Lemma 5.1]{MZ15} There exists a constant $C>0$ such that 
for any $g,n$ with $2g-2+n>0$, $$
\left|\frac{(2g-2+n)V_{g,n}}{V_{g,n+1}}-\frac{1}{4\pi^2}\right|\leq C\cdot \frac{n+1}{2g-2+n}.
$$
\end{lemma}

The following lemma is directly proved by the same argument as in \cite[Page 286]{mir13}, or \cite[Lemma 22]{NWX23}, and or \cite[Lemma 24]{shenwu2022arbitrarily}.

\begin{lemma}\label{appendix tau_d nd+d^2/g}
If $n=o(\sqrt{g})$,  for any $|\textbf{d}|\leq 3g-3+n$, there exists $C>0$ such that as $g\to\infty$, $$
0\leq 1-\frac{[\tau_{d_1}\cdots\tau_{d_n}]_{g,n}}{V_{g,n}}\leq C\frac{|\textbf{d}|^2+n|\textbf{d}|}{g}.
$$
\end{lemma}

\begin{rem*}
    Lemma \ref{appendix tau_d nd+d^2/g} also holds for $|\textbf{d}|>3g-3+n$ by setting $\left[\tau_{d_1}\cdots\tau_{d_n}\right]=0.$
\end{rem*}
We need to prove the following weak version of Theorem \ref{appendix vgn/vgn+1 n^2+1/g^2}. 
\begin{lemma}\label{lemma weak vgn}
    For $n=n(g)=o(\sqrt{g})$, as $g\to\infty$:
    \begin{equation}\label{lemma weak vgn 1}
        \frac{(2g-2+n)V_{g,n}}{V_{g,n+1}}=\frac{1}{4\pi^2}+O\left( \frac{1+n}{g}\right),
    \end{equation}
    and 
    \begin{equation}\label{lemma weak vgn 2}
        \frac{V_{g-1,n+2}}{V_{g,n}}=1+O\left( \frac{1+n}{g}\right),
    \end{equation}
    where the implied constants are independent of $n,g$.
\end{lemma}
\begin{rem*}
    Lemma \ref{lemma weak vgn} is encompassed by \cite[Theorem 4.2]{hide2025spectral}, which has a very complex statement and proof. For the sake of convenience, we still provide an outline of the proof.
\end{rem*}

\begin{proof}[Proof of Lemma \ref{lemma weak vgn}]
    The estimation \eqref{lemma weak vgn 1} is just Lemma \ref{MZ15-universal}, which is proved by taking Lemma \ref{appendix tau_d nd+d^2/g} into the recursive formula \eqref{recursive II} for $\textbf{d}=(0,\cdots,0)$.

For \eqref{lemma weak vgn 2}, we can assume that $n\geq 2$. If $n=0,1$, it holds by Lemma \ref{mir13 vgn}.
Notice that $$
\frac{V_{g,n}}{V_{g-1,n+2}}=\frac{V_{g,n}}{\left[\tau_1\tau_0^{n-1}\right]_{g,n}}\cdot \frac{\left[\tau_1\tau_0^{n-1}\right]_{g,n}}{V_{g-1,n+2}}.
$$
 By Lemma \ref{appendix tau_d nd+d^2/g} we have $$
 \frac{V_{g,n}}{\left[\tau_1\tau_0^{n-1}\right]_{g,n}}=1+O(\frac{n+1}{g}).
 $$
For $\frac{\left[\tau_1\tau_0^{n-1}\right]_{g,n}}{V_{g-1,n+2}}$, by Corollary \ref{cor asymp vgn} and the recursive formula \eqref{recursive I} we have
 \begin{align*}
     &\frac{\left[\tau_1\tau_0^{n-1}\right]_{g,n}}{V_{g-1,n+2}}=1+6\sum_{g_1+g_2=g\atop I\sqcup J=\{1,\cdots,n-2\}}\frac{V_{g_1,|I|+2}V_{g_2,|J|+2}}{V_{g-1,n+2}}\\
     =&1+O\left(\sum_{g_1+g_2=g\atop i+j=n-2}\frac{V_{g_1,i+2}V_{g_2,j+2}}{V_{g,n}} {{n-2}\choose i  }\right).
 \end{align*}
By Lemma \ref{appendix product lemma  2i+j geq k} for $k=1$, we have $$\sum_{g_1+g_2=g\atop i+j=n-2}\frac{V_{g_1,i+2}V_{g_2,j+2}}{V_{g,n}} {{n-2}\choose i  }=O\left(\frac{1+n}{g}\right).$$
So $$
\frac{\left[\tau_1\tau_0^{n-1}\right]_{g,n}}{V_{g-1,n+2}}=1+O\left(\frac{1+n}{g}\right),
$$ and then \eqref{lemma weak vgn 2} follows.
\end{proof}

\subsection{Proof of Lemma \ref{appendix Vgn(x)/Vgn cn/g}}
The upper bound of $\frac{V_{g,n}(x_1,\cdots,x_n)}{V_{g,n}}$ is given in Lemma \ref{lemma NWX vgn(x) old version}. For the lower bound, by \eqref{eq-vgn(L)-intersect number} and Lemma \ref{appendix tau_d nd+d^2/g}, we have \begin{equation}\label{appendix pf lemma Vgn(x)/Vgn cn/g eq1}
    \begin{aligned}
        &\frac{V_{g,n}(2x_1,\cdots,2x_n)}{V_{g,n}}=\sum_{\textbf{d}\in \mathbb{N}_{\geq 0}^n}\left[\prod_{i=1}^n \tau_{d_i}\right]_{g,n}\frac{x_1^{2d_1}}{(2d_1+1)!}\cdots \frac{x_n^{2d_n}}{(2d_n+1)!}\\
        \geq &\sum_{\textbf{d}\in \mathbb{N}_{\geq 0}^n}\left(1-C_0\frac{|\textbf{d}|^2+n|\textbf{d}|}{g}\right)\frac{x_1^{2d_1}}{(2d_1+1)!}\cdots \frac{x_n^{2d_n}}{(2d_n+1)!}
    \end{aligned}
\end{equation}
for some uniform $C_0>0$. Notice that \begin{align*}
    \prod_{i=1}^n \frac{\sinh x_i}{x_i}=\sum_{\textbf{d}\in \mathbb{N}_{\geq 0}^n}\frac{x_1^{2d_1}}{(2d_1+1)!}\cdots \frac{x_n^{2d_n}}{(2d_n+1)!},
\end{align*}
and \begin{align*}
    &\left(x_1^2+\cdots +x_n^2\right) \prod_{i=1}^n \frac{\sinh x_i}{x_i}\\
    =&\sum_{\textbf{d}\in \mathbb{N}_{\geq 0}^n}\left(\sum_{i=1}^n2d_i(2d_i+1)\right)   \frac{x_1^{2d_1}}{(2d_1+1)!}\cdots \frac{x_n^{2d_n}}{(2d_n+1)!}.
\end{align*}
Then we have \begin{equation}\label{appendix pf lemma Vgn(x)/Vgn cn/g eq2}
    \begin{aligned}
        \sum_{\textbf{d}\in \mathbb{N}_{\geq 0}^n} \left|\textbf{d}\right|  \frac{x_1^{2d_1}}{(2d_1+1)!}\cdots \frac{x_n^{2d_n}}{(2d_n+1)!}\leq \frac{1}{2}\left(x_1^2+\cdots+x_n^2\right)\prod_{i=1}^n\frac{\sinh x_i}{x_i},
    \end{aligned}
\end{equation}
and by the Cauchy-Schwarz inequality, we have \begin{equation}\label{appendix pf lemma Vgn(x)/Vgn cn/g eq3}
    \begin{aligned}
        \sum_{\textbf{d}\in \mathbb{N}_{\geq 0}^n} \left|\textbf{d}\right|^2  \frac{x_1^{2d_1}}{(2d_1+1)!}\cdots \frac{x_n^{2d_n}}{(2d_n+1)!}\leq \frac{n}{4}\left(x_1^2+\cdots+x_n^2\right)\prod_{i=1}^n\frac{\sinh x_i}{x_i}.
    \end{aligned}
\end{equation}
Taking \eqref{appendix pf lemma Vgn(x)/Vgn cn/g eq2} and \eqref{appendix pf lemma Vgn(x)/Vgn cn/g eq3} into \eqref{appendix pf lemma Vgn(x)/Vgn cn/g eq1}, we will get \begin{align*}
    \frac{V_{g,n}(2x_1,\cdots,2x_n)}{V_{g,n}}\geq \prod_{i=1}^n \frac{\sinh x_i}{x_i}\left(1-\frac{3C_0n}{4}\frac{x_1^2+\cdots+x_n^2}{g}\right).
\end{align*}
So the Lemma follows if we set $C=\frac{3C_0}{16}$.

\subsection{Estimation of intersection numbers}
In this subsection, we will show the following asymptotic property of intersection numbers when $n=n(g)=o(\sqrt{g}).$ When $n$ is any fixed number, it has been shown by \cite{MZ15}. 
Since the proofs will heavily depend on the algorithm developed in \cite{MZ15}, we present them briefly and only focus on the remainder terms.

\begin{proposition}\label{asy tau dn n^3d^4/g^2}
    If $n\geq 1$ and $n=o(\sqrt{g})$, then as $g\to\infty$,\begin{align*}
        \left[\tau_{d_1},\cdots, \tau_{d_n} \right]_{g,n}&=V_{g,n}\\
        &-4V_{g,n-1}\cdot \mathbbm{1}_{n\geq 2}\sum_{1\leq i<j\leq n}\!\!\left((2d_i+1)(2d_j+1)+\mathbbm{1}_{d_i=d_j=0}-2\right)\\
        &+8V_{g-1,n+1}\mathbbm{1}_{g\geq 1}\sum_{i=1}^n\left( 2d_i-1+\mathbbm{1}_{d_i=0}-\frac{(2d_i+1)d_i}{2}\right)\\
        &+O\left(V_{g,n} \frac{n^2|\mathbf{d}|^3+|\mathbf{d}|^4}{g^2}\right).
    \end{align*}
\end{proposition}
\begin{proposition}\label{prop delta taud}
    For $n=n(g)=o(\sqrt{g})$, as $g\to\infty$,\begin{align*}
        &\left[\tau_{d_1}\tau_{d_2}\cdots\tau_{d_n} \right]_{g,n}-\left[\tau_{d_1+1}\tau_{d_2}\cdots\tau_{d_n} \right]_{g,n}\\
        =&4(4d_1-1+2\mathbbm{1}_{d_1=0})V_{g-1,n+1}\mathbbm{1}_{g\geq 1}\\
        +&4\sum_{j=2}^n(4d_j+2-\mathbbm{1}_{d_1=d_j=0})V_{g,n-1}\mathbbm{1}_{n\geq 2}\\
        +O&\left(V_{g,n} \frac{(1+n^2)(1+|\mathbf{d}|^2)+|\mathbf{d}|^3}{g^2}\right).
    \end{align*}
\end{proposition}
Set $\textbf{d}^\prime=(d_1+1,d_2,\cdots,d_n)$.
By \eqref{recursive III} we can write \begin{equation}\label{eq decomp delta tau d}
\begin{aligned}
&\left[\tau_{d_1}\tau_{d_2}\cdots\tau_{d_n} \right]-\left[\tau_{d_1+1}\tau_{d_2}\cdots\tau_{d_n}\right]\\
=&\sum_{j=2}^n(A_{\textbf{d}}^j-A_{\textbf{d}^\prime}^j)
+(B_{\textbf{d}}-B_{\textbf{d}^\prime})+(C_{\textbf{d}}-C_{\textbf{d}^\prime}).
\end{aligned}
\end{equation}

\begin{lemma}\label{adj-d^2+nd/g^2}
For $2\leq j\leq n$ and $n=o(\sqrt{g})$, as $g\to\infty$,
   \begin{align*}
    A_{\textbf{d}}^j-A_{\textbf{d}^\prime}^j=&4V_{g,n-1}\mathbbm{1}_{n\geq 2}(4d_j+2-\mathbbm{1}_{d_1=d_j=0})\\
    +&O\left(V_{g,n}\frac{(d_j+1)(|\mathbf{d}|^2+n(|\mathbf{d}|+1))}{g^2}\right).
\end{align*}
\end{lemma}
\begin{proof}
For $n=1$, this is true. Now we assume $n\geq 2$.
By \eqref{int-Adj} we have 
\begin{equation}\label{adj-d^2+nd/g^2 eq1}
    \begin{aligned}
        A_{\textbf{d}}^j-A_{\textbf{d}^\prime}^j
        =&8(2d_j+1)\sum_{L=0}^\infty 
        a_L\left[  \tau_{d_1+d_j+L-1}\prod_{i\neq 1,j}\tau_{d_i}  \right]_{g,n-1}\\
        -&8(2d_j+1)\sum_{L=0}^\infty 
        a_L\left[  \tau_{d_1+d_j+L}\prod_{i\neq 1,j}\tau_{d_i}  \right]_{g,n-1}\\
        =&8(2d_j+1)a_0\left[  \tau_{d_1+d_j-1}\prod_{i\neq 1,j}\tau_{d_i}  \right]_{g,n-1}\\
       +&8(2d_j+1)\sum_{L=1}^\infty (a_L-a_{L-1})\left[  \tau_{d_1+d_j+L-1}\prod_{i\neq 1,j}\tau_{d_i}  \right]_{g,n-1}.\\
    \end{aligned}
    \end{equation}
If $d_1=d_j=0$, then $\left[  \tau_{d_1+d_j-1}\prod_{i\neq 1,j}\tau_{d_i}  \right]_{g,n-1}=0$. If $d_1+d_j >0$,
by Lemma \ref{appendix mir13 vg-1,n+2 vg,n+1} and Lemma \ref{appendix tau_d nd+d^2/g}, we have 
\begin{equation}\label{adj-d^2+nd/g^2 eq2}\begin{aligned}
    &8(2d_j+1)a_0\left[  \tau_{d_1+d_j-1}\prod_{i\neq 1,j}\tau_{d_i}  \right]_{g,n-1}\\
    =&4(2d_j+1)V_{g,n-1}\left(1+O\left(\frac{|\textbf{d}|^2+n|\textbf{d}|}{g}\right)\right)\\
    =&4(2d_j+1)V_{g,n-1}+O\left(V_{g,n}\frac{(2d_j+1)(|\textbf{d}|^2+n|\textbf{d}|)}{g^2}\right).
\end{aligned}
\end{equation}
By Lemma \ref{appendix mir13 vg-1,n+2 vg,n+1},  Lemma \ref{lemma-ai-MZ15}, and Lemma \ref{appendix tau_d nd+d^2/g}, we have
\begin{equation}\label{adj-d^2+nd/g^2 eq3}
\begin{aligned}
    &8(2d_j+1)\sum_{L=1}^\infty (a_L-a_{L-1})\left[  \tau_{d_1+d_j+L-1}\prod_{i\neq 1,j}\tau_{d_i}  \right]_{g,n-1}\\
=&8(2d_j+1)\sum_{L=1}^\infty (a_L-a_{L-1})V_{g,n-1}\!\left(1\!+\!O\!\left(\!\frac{(|\textbf{d}|+L)^2\!+\!n(|\textbf{d}|+L)}{g}\!\right)\!\right)\\
=&4(2d_j+1)V_{g,n-1}\left(1+O\left(\frac{|\textbf{d}|^2+n(|\textbf{d}|+1)}{g}\right)\right)\\
=&4(2d_j+1)V_{g,n-1}+O\left((2d_j+1)V_{g,n}\frac{|\textbf{d}|^2+n(|\textbf{d}|+1)}{g^2}\right).
\end{aligned}
\end{equation}
Taking \eqref{adj-d^2+nd/g^2 eq2} and \eqref{adj-d^2+nd/g^2 eq3} into \eqref{adj-d^2+nd/g^2 eq1}, we can get the Lemma.
\end{proof}

\begin{lemma}\label{Bd-d^3+nd^2/g^2}
For $n=o(\sqrt{g})$, as $g\to\infty$,
   \begin{align*}
    B_{\textbf{d}}-B_{\textbf{d}^\prime}=&4V_{g-1,n+1}\mathbbm{1}_{g\geq 1}(4d_1-1+2\mathbbm{1}_{d_1=0})\\
    +&O\left(V_{g,n}\frac{|\textbf{d}|^3+n(|\textbf{d}|^2+1)+1}{g^2}\right).
\end{align*}
\end{lemma}
\begin{proof}
   By \eqref{int-Bd} we have 
   \begin{equation}\label{Bd-d^3+nd^2/g^2 eq1}
   \begin{aligned}
         B_{\textbf{d}}-B_{\textbf{d}^\prime}=&16\sum_{k_1+k_2=d_1-2}a_0\left[\tau_{k_1}\tau_{k_2}\prod_{i\neq 1}\tau_{d_i}\right]_{g-1,n+1}\\
         +&16\sum_{L=1}^\infty\sum_{k_1+k_2=d_1+L-2}(a_L-a_{L-1})\left[\tau_{k_1}\tau_{k_2}\prod_{i\neq 1}\tau_{d_i}\right]_{g-1,n+1}.
   \end{aligned}
   \end{equation}
   If $d_1\leq 1$ then $16\sum_{k_1+k_2=d_1-2}a_0\left[\tau_{k_1}\tau_{k_2}\prod_{i\neq 1}\tau_{d_i}\right]_{g-1,n+1}=0$. So
   by Lemma \ref{appendix mir13 vg-1,n+2 vg,n+1} and Lemma \ref{appendix tau_d nd+d^2/g}, we have \begin{equation}\label{Bd-d^3+nd^2/g^2 eq2}
   \begin{aligned}
  &16\sum_{k_1+k_2=d_1-2}a_0\left[\tau_{k_1}\tau_{k_2}\prod_{i\neq 1}\tau_{d_i}\right]_{g-1,n+1}\\
  =&8\cdot\mathbbm{1}_{d_1\geq 2}\sum_{k_1+k_2=d_1-2} V_{g-1,n+1}\left(1+O\left(\frac{|\textbf{d}|^2+(n+1)|\textbf{d}|}{g}\right) \right) \\
  =&8\cdot\mathbbm{1}_{d_1\neq 0}(d_1-1)V_{g-1,n+1}+O\left(V_{g,n}\frac{|\textbf{d}|^3+(n+1)|\textbf{d}|^2}{g^2}\right).
   \end{aligned}
   \end{equation}
   While by Lemma \ref{appendix mir13 vg-1,n+2 vg,n+1}, Lemma \ref{lemma-ai-MZ15}, and Lemma \ref{appendix tau_d nd+d^2/g}, we have \begin{equation}\label{Bd-d^3+nd^2/g^2 eq3}
   \begin{aligned}
       &16\sum_{L=1}^\infty\sum_{k_1+k_2=d_1+L-2}(a_L-a_{L-1})\left[\tau_{k_1}\tau_{k_2}\prod_{i\neq 1}\tau_{d_i}\right]_{g-1,n+1}\\
       =&16\sum_{L=1}^\infty(L+d_1-1)(a_L-a_{L-1})V_{g-1,n+1}\\
       \cdot&\left(1+O\left(  \frac{(|\textbf{d}|+L)^2+(n+1)(|\textbf{d}|+L)    }{g}  \right)   \right)\\
       =&(8d_1+4)V_{g-1,n+1}+O\left(V_{g-1,n+1}\cdot\frac{|\textbf{d}|^3+(n+1)(|\textbf{d}|^2+1)}{g} \right)\\
       =&(8d_1+4)V_{g-1,n+1}+O\left(V_{g,n}\cdot\frac{|\textbf{d}|^3+(n+1)(|\textbf{d}|^2+1)}{g^2} \right).
   \end{aligned}
   \end{equation}
   Taking \eqref{Bd-d^3+nd^2/g^2 eq2} and \eqref{Bd-d^3+nd^2/g^2 eq3} into \eqref{Bd-d^3+nd^2/g^2 eq1}, we can get the Lemma.
\end{proof}

\begin{lemma}\label{Cd-d^3+nd^2/g^2}
For $n=o(\sqrt{g})$, as $g\to\infty$,
   \begin{align*}
    C_{\textbf{d}}-C_{\textbf{d}^\prime}=O\left(V_{g,n}\frac{(1+|\textbf{d}|)(1+n^2)}{g^2}\right).
\end{align*}
\end{lemma}
\begin{proof}
    By \eqref{int-Cd}, Lemma \ref{appendix  mir13 [tau_d]<vgn}, and Lemma \ref{lemma-ai-MZ15}, we have \begin{align*}
         &C_{\textbf{d}}-C_{\textbf{d}^\prime}\\
        =&16\sum_{g_1+g_2=g\atop I\cup J=\{2,\cdots,n\}}\left( \sum_{k_1+k_2=d_1-2}a_0\left[\tau_{k_1}\prod_{i\in I}\tau_{d_i} \right]_{g_1,|I|+1}\left[\tau_{k_2}\prod_{i\in J}\tau_{d_i} \right]_{g_2,|J|+1}\right. \\
        +&\left.\sum_{L=1}^\infty \sum_{k_1+k_2=d_1+L-2} (a_L-a_{L-1})\left[\tau_{k_1}\prod_{i\in I}\tau_{d_i} \right]_{g_1,|I|+1}\left[\tau_{k_2}\prod_{i\in J}\tau_{d_i} \right]_{g_2,|J|+1}\right)\\
        \prec&\left(\sum_{g_1+g_2=g\atop I\cup J=\{2,\cdots,n\}}V_{g_1,|I|+1}V_{g_2,|J|+1}\right)\cdot\left(d_1+\sum_{L=1}^\infty (d_1+L-1) (a_L-a_{L-1})\right )\\
        \prec& \left(\sum_{g_1+g_2=g\atop i+j=n-1}V_{g_1,i+1}V_{g_2,j+1}{n-1\choose i} \right)\cdot (d_1+1).
    \end{align*}

    By Lemma \ref{lemma weak vgn} and Lemma \ref{lemma in hide appendix gen k} (whose proof is independent of this Lemma, although we show it later to make the layout more concise), we finish the proof. 
\end{proof}
\begin{proof}[Proof of Proposition \ref{prop delta taud}]
It follows from \eqref{eq decomp delta tau d},  Lemma \ref{adj-d^2+nd/g^2}, Lemma \ref{Bd-d^3+nd^2/g^2} and Lemma \ref{Cd-d^3+nd^2/g^2}.
\end{proof}
\begin{proof}[Proof of Proposition \ref{asy tau dn n^3d^4/g^2}]
    Applying Proposition \ref{prop delta taud} to the identity \begin{align*}
        &\left[\tau_{d_1}\cdots \tau_{d_n} \right]_{g,n}\\
        =&V_{g,n}-\sum_{i=1}^n\sum_{k=0}^{d_i-1}\left(\left[\tau_k\tau_0^{i-1}\tau_{d_{i+1}}\cdots\tau_{d_n}\right]_{g,n}-\left[\tau_{k+1}\tau_0^{i-1}\tau_{d_{i+1}}\cdots\tau_{d_n}\right]_{g,n}\right),
    \end{align*}
    which contains the summation of $\sum_{i=1}^nd_i$ terms of first differences,
     we have \begin{align*}
        \left[\tau_{d_1},\cdots, \tau_{d_n} \right]_{g,n}&=V_{g,n}\\
        &-4V_{g,n-1}\cdot \mathbbm{1}_{n\geq 2}\sum_{1\leq i<j\leq n}\left((2d_i+1)(2d_j+1)+\mathbbm{1}_{d_i=d_j=0}-2\right)\\
        &+8V_{g-1,n+1}\mathbbm{1}_{g\geq 1}\sum_{i=1}^n\left( 2d_i-1+\mathbbm{1}_{d_i=0}-\frac{(2d_i+1)d_i}{2}\right)\\
    &+O\left(V_{g,n}\left(\sum_{i=1}^n d_i\right) \frac{(1+n^2)(1+|\mathbf{d}|^2)+|\mathbf{d}|^3}{g^2}\right)\\
        &=V_{g,n}\\
         &-4V_{g,n-1}\cdot \mathbbm{1}_{n\geq 2}\sum_{1\leq i<j\leq n}\left((2d_i+1)(2d_j+1)+\mathbbm{1}_{d_i=d_j=0}-2\right)\\
        &+8V_{g-1,n+1}\mathbbm{1}_{g\geq 1}\sum_{i=1}^n\left( 2d_i-1+\mathbbm{1}_{d_i=0}-\frac{(2d_i+1)d_i}{2}\right)\\
        &+O\left(V_{g,n} \frac{n^2|\mathbf{d}|^3+|\mathbf{d}|^4}{g^2}\right).
    \end{align*}
\end{proof}

\subsection{Proof of Theorem \ref{appendix vgn/vgn+1 n^2+1/g^2}}
Firstly by \eqref{recursive II} we have \begin{equation}\label{vgn/vgn+1 n^2+1/g^2-eq-1}
    \begin{aligned}
        (2g-2+n)V_{g,n}=\frac{1}{2}\sum_{L=1}^\infty(-1)^{L-1}L\frac{\pi^{2L-2}}{(2L+1)!}\left[ 
         \tau_{L}\tau_0^n \right]_{g,n+1}.
    \end{aligned}
\end{equation}
By Proposition \ref{asy tau dn n^3d^4/g^2}, for $L\geq 1$ we have \begin{align*}
\left[ \tau_{L}\tau_0^n \right]_{g,n+1}&=V_{g,n+1}-4V_{g,n}\mathbbm{1}_{n\geq 1}n(2L-1)\\
&+8V_{g-1,n+2}\mathbbm{1}_{g\geq 1}\left( (2L-1)-\frac{(2L+1)L}{2} \right)\\
&+O\left(V_{g,n+1}\frac{(n+1)^2L^3+L^4 }{g^2}\right).
\end{align*}
It follows from Lemma \ref{lemma weak vgn} that 
    \begin{equation}\label{vgn/vgn+1 n^2+1/g^2-eq-2}
    \begin{aligned}
        \frac{\left[ \tau_{L}\tau_0^n \right]_{g,n+1}}{V_{g,n+1}}=&1-4n(2L-1)\frac{V_{g,n}}{V_{g,n+1}}\\
         +&8\left( (2L-1)-\frac{(2L+1)L}{2} \right)\frac{V_{g-1,n+2}}{V_{g,n+1}}\\
         +&O\left(\frac{(n+1)^2L^3+L^4 }{g^2}\right)\\
         =&1-4n(2L-1)\left(\frac{1}{8\pi^2g}+O\left(\frac{1+n}{g^2}\right)\right)\\
        +&8\left( (2L-1)-\frac{(2L+1)L}{2} \right)\left(\frac{1}{8\pi^2g}+O\left(\frac{1+n}{g^2}\right)\right)\\
        +&O\left(\frac{(n+1)^2L^3+L^4 }{g^2}\right)\\
        =&1-\frac{1}{2\pi^2g}\left( n(2L-1) +2L^2-3L+2\right)\\
        +&O\left(\frac{(1+n^2)(1+L^4)}{g^2}\right).
    \end{aligned}
    \end{equation}
Combining \eqref{vgn/vgn+1 n^2+1/g^2-eq-1} and \eqref{vgn/vgn+1 n^2+1/g^2-eq-2}, we have \begin{align*}
    &\frac{(2g-2+n)V_{g,n}}{V_{g,n+1}}\\
    =&\frac{1}{2}\sum_{L=1}^\infty (-1)^{L-1}L\cdot\frac{\pi^{2L-2}}{(2L+1)!}\left( 1-\frac{1}{2\pi^2g}\left( n(2L-1) +2L^2-3L+2\right)\right)\\
    +&O\left( \sum_{L=1}^\infty \frac{L\pi^{2L-2}}{(2L+1)!}\frac{(1+n^2)(1+L^4)}{g^2} \right)\\
    =&\frac{1}{4\pi^2}+\frac{4n+\pi^2-8}{16\pi^4g}+\left(\frac{1+n^2}{g^2}\right),
\end{align*}
which is equivalent to \begin{equation}
    \frac{8\pi^2g V_{g,n}}{V_{g,n+1}}=1+\left(\left(\frac{1}{\pi^2}-\frac{1}{2}\right)n+\frac{5}{4}-\frac{2}{\pi^2}\right)\frac{1}{g}+O\left(\frac{1+n^2}{g^2}\right),
\end{equation}
and implies \eqref{appendix vgn/vgn+1 n^2+1/g^2 eq1}.

Secondly, by \eqref{recursive I} and  \eqref{vgn/vgn+1 n^2+1/g^2-eq-2} we have \begin{align*}
    \frac{V_{g-1,n+4}}{V_{g,n+2}}=&\frac{\left[\tau_1\tau_0^{n+1} \right]_{g,n+2}}{V_{g,n+2}}-6\sum_{g_1+g_2=g\atop I\cup J=\{1,\cdots,n\}}\frac{V_{g_1,|I|+2}V_{g_2,|J|+2}}{V_{g,n+2}}\\
 =&1-\frac{n+2}{2\pi^2g}+O\left(\frac{1+n^2}{g^2}\right)-12n\frac{V_{g,n+1}}{V_{g,n+2}}\\
 +&O\left(\sum_{\substack{i+j=n\\ g_1+g_2=g \\ 2g_1+i\geq 2\\ 2g_2+j\geq 2}} {n\choose i}\frac{V_{g_1,i+2}V_{g_2,j+2}}{V_{g,n+2}}\right).
\end{align*}
Then by Lemma \ref{lemma weak vgn} and Lemma \ref{appendix product lemma  2i+j geq k} for $k=2$ (whose proof is independent of this Lemma, although we show it later to make the layout more concise), we have 
\begin{align*}
     \frac{V_{g-1,n+4}}{V_{g,n+2}}=&1-\frac{n+2}{2\pi^2g}-\frac{12n}{8\pi^2g}\left(1+O\left(\frac{1+n}{g}\right)\right)+O\left(\frac{1+n^2}{g^2}\right).\\
     =&1-\frac{2n+1}{\pi^2g}+O\left(\frac{1+n^2}{g^2}\right).
\end{align*}
This proves \eqref{appendix vgn/vgn+1 n^2+1/g^2 eq2} for $n\geq 2$ and $n=n(g)=o(\sqrt{g})$. The cases of $n=0,1$ have been proved in \cite{MZ15}.

\subsection{Proof of Theorem \ref{appendix vgn(x)/vgn 1+n^3/g^2}} Proposition \ref{asy tau dn n^3d^4/g^2} holds for any $\textbf{d}\in \mathbb{N}_{\geq 0}^n$ by extending $[\tau_{d_1}\cdots \tau_{d_n}]_{g,n}=0$ for $|\textbf{d}|>3g-3+n$. Taking it into \eqref{eq-vgn(L)-intersect number}, we have 
    \begin{equation}\label{vgn(x)/vgn 1+n^3/g^2 eq1}
    \begin{aligned}
         &\frac{V_{g,n}(\textbf{x})}{V_{g,n}}\\
         =&\sum_{\textbf{d}\in\mathbb{N}_{\geq 0}^n}\left[\prod_{i=1}^n\frac{x_i^{2d_i}}{2^{2d_i}(2d_i+1)!}\right.\\
   -&8\frac{V_{g-1,n+1}}{V_{g,n}}
  \cdot\sum_{i=1}^n\left(\frac{(2d_i+1)d_i}{2}-2d_i+1-\mathbbm{1}_{d_i=0} \right)\prod_{j=1}^n\frac{x_j^{2d_j}}{2^{2d_j}(2d_j+1)!}\\
   -&\left.4\!\frac{V_{g,n-1}}{V_{g,n}}\!\mathbbm{1}_{n\geq 2}\!\sum_{i<j}\!\left((2d_i\!+\!1)(2d_j\!+\!1)\!-\!2\!+\!\mathbbm{1}_{d_i=d_j=0}\right)\!\prod_{j=1}^n\!\frac{x_j^{2d_j}}{2^{2d_j}(2d_j\!+\!1)!}\right]\\
   +&O\left(  \sum_{\textbf{d}\in\mathbb{N}_{\geq 0}^n}\frac{n^2|\textbf{d}|^3+|\textbf{d}|^4}{g^2}\prod_{i=1}^n\frac{x_i^{2d_i}}{2^{2d_i}(2d_i+1)!} \right).
    \end{aligned}
    \end{equation}
Notice that $$
\sum_{d=0}^\infty\frac{x^{2d}}{2^{2d}(2d+1)!}=\frac{\sinh(x/2)}{x/2},
$$
and $$
\sum_{d=0}^\infty\frac{x^{2d}}{2^{2d}(2d)!}=\cosh(x/2).
$$
Then \begin{equation}\label{vgn(x)/vgn 1+n^3/g^2 eq2}
    \sum_{\textbf{d}\in\mathbb{N}_{\geq 0}^n}\prod_{i=1}^n\frac{x_i^{2d_i}}{2^{2d_i}(2d_i+1)!}=\prod_{i=1}^n \frac{\sinh(x_i/2)}{x_i/2}.
\end{equation}
By Lemma \ref{lemma weak vgn}, we have
\begin{equation}\label{vgn(x)/vgn 1+n^3/g^2 eq3}
\begin{aligned}
   &8\frac{V_{g-1,n+1}}{V_{g,n}}\!\!\!\sum_{\textbf{d}\in\mathbb{N}_{\geq 0}^n}\!\! \sum_{i=1}^n\left(\!\frac{(2d_i+1)d_i}{2}\!-\!2d_i\!+\!1\!-\!\mathbbm{1}_{d_i=0}\! \right)\!\prod_{j=1}^n\frac{x_j^{2d_j}}{2^{2d_j}(2d_j+1)!}\\
=&\sum_{i=1}^n\left[\left(\frac{x_i^2}{16}+2\right)\frac{\sinh(x_i/2)}{x_i/2}\!-\!\cosh(x_i/2)\!-\!1\right]\prod_{j\neq i}\frac{\sinh(x_j/2)}{x_j/2}\\
\cdot&\left(\frac{1}{\pi^2g}+O\left(\frac{1+n}{g^2}\right)\right)\\
  =&\frac{1}{\pi^2g}\sum_{i=1}^n\left[\left(\frac{x_i^2}{16}+2\right)\frac{\sinh(x_i/2)}{x_i/2}\!-\!\cosh(x_i/2)\!-\!1\right]\prod_{j\neq i}\frac{\sinh(x_j/2)}{x_j/2}\\
   +&O\left(\frac{(1+n)\left(n+\sum_{i=1}^n x_i^2\right)}{g^2}\exp\left(\frac{x_1+\cdots +x_n}{2}\right)\right),
  \end{aligned}
  \end{equation}
and
    \begin{equation}\label{vgn(x)/vgn 1+n^3/g^2 eq4}\begin{aligned}
         &4\frac{V_{g,n-1}}{V_{g,n}}\!\!\sum_{\textbf{d}\in\mathbb{N}_{\geq 0}^n} \sum_{i<j}\left((2d_i\!+\!1)(2d_j\!+\!1)\!-\!2\!+\!\mathbbm{1}_{d_i=d_j=0}\right)\prod_{j=1}^n\frac{x_j^{2d_j}}{2^{2d_j}(2d_j+1)!}\\
          =&\sum_{i<j} \left(\cosh(x_i/x)\cosh(x_j/2)\!-\!2\frac{\sinh(x_i/2)}{x_i/2}\frac{\sinh(x_j/2)}{x_j/2}\!+\!1\right)\\
          \cdot&\prod_{l\neq i,j}\frac{\sinh(x_l/2)}{x_l/2}\cdot\left(\frac{1}{2\pi^2g}+O\left(\frac{1+n}{g^2}\right)\right)\\
          =&\frac{1}{2\pi^2g}\sum_{i<j} \left(\cosh(x_i/x)\cosh(x_j/2)\!-\!2\frac{\sinh(x_i/2)}{x_i/2}\frac{\sinh(x_j/2)}{x_j/2}\!+\!1\right)\\
          \cdot&\prod_{l\neq i,j}\frac{\sinh(x_l/2)}{x_l/2}\\
          +&O\left(\mathbbm{1}_{n\geq 2}\frac{n^3}{g^2}\exp\left(\frac{x_1+\cdots +x_n}{2}\right)\right).
    \end{aligned}
    \end{equation}
Notice that we assume $\textbf{x}=(x_1,\cdots,x_k,0^{n-k})$, so \begin{equation}\label{vgn(x)/vgn 1+n^3/g^2 eq5}
\begin{aligned}
 &\sum_{\textbf{d}\in\mathbb{N}_{\geq 0}^n}\frac{n^2|\textbf{d}|^3+|\textbf{d}|^4}{g^2}\prod_{i=1}^n\frac{x_i^{2d_i}}{2^{2d_i}(2d_i+1)!}\\
 =&\sum_{\textbf{d}\in\mathbb{N}_{\geq 0}^k}\frac{n^2|\textbf{d}|^3+|\textbf{d}|^4}{g^2}\prod_{i=1}^k\frac{x_i^{2d_i}}{2^{2d_i}(2d_i+1)!}.
 \end{aligned}
 \end{equation}
Since for any $m\geq 0$
$$
x_i^{2m}\sum_{\textbf{d}\in\mathbb{N}_{\geq 0}^k}\prod_{j=1}^k\frac{x_j^{2d_j}}{2^{2d_j}(2d_j+1)!}=4^m\sum_{\textbf{d}\in\mathbb{N}_{\geq 0}^k}\prod_{t=1}^{2m }(2d_i+2-t)\prod_{j=1}^k\frac{x_j^{2d_j}}{2^{2d_j}(2d_j+1)!}  ,
$$
by Cauchy-Schwarz inequality 
we have 
    \begin{equation} \label{vgn(x)/vgn 1+n^3/g^2 eq6}
    \begin{aligned}
&\sum_{\textbf{d}\in\mathbb{N}_{\geq 0}^k}|\textbf{d}|^3\prod_{i=1}^k\frac{x_i^{2d_i}}{2^{2d_i}(2d_i+1)!}\\
\leq&\sum_{\textbf{d}\in\mathbb{N}_{\geq 0}^k}|\textbf{d}|^4\prod_{i=1}^k\frac{x_i^{2d_i}}{2^{2d_i}(2d_i+1)!}\\
\prec&\sum_{\textbf{d}\in\mathbb{N}_{\geq 0}^k}\sum_{i=1}^k d_i^4\prod_{i=1}^k\frac{x_i^{2d_i}}{2^{2d_i}(2d_i+1)!} \\
\prec&\sum_{\textbf{d}\in\mathbb{N}_{\geq 0}^k}\sum_{i=1}^k\left(1+ x_i^2 +x_i^4  \right)\prod_{i=1}^k\frac{x_i^{2d_i}}{2^{2d_i}(2d_i+1)!}\\
\prec &(1+x_1+\cdots +x_k)^4\exp\left(\frac{x_1+\cdots+x_k}{2}\right),
\end{aligned}
\end{equation}
where the implied constant depends only on $k$ and is independent of $g$,$n$ and $x_1,\cdots,x_k$.
Now by \eqref{vgn(x)/vgn 1+n^3/g^2 eq1}, \eqref{vgn(x)/vgn 1+n^3/g^2 eq2}, \eqref{vgn(x)/vgn 1+n^3/g^2 eq3}, \eqref{vgn(x)/vgn 1+n^3/g^2 eq4}, \eqref{vgn(x)/vgn 1+n^3/g^2 eq5}, and \eqref{vgn(x)/vgn 1+n^3/g^2 eq6}, we have 
\begin{align*}
&\frac{V_{g,n}(\textbf{x})}{V_{g,n}}=\prod_{i=1}^n \frac{\sinh(x_i/2)}{x_i/2}
+\frac{f_n^1(\textbf{x})}{g}\\
+&O_k\left(\frac{n^3(1+x_1+\cdots+x_k)^4}{g^2}\exp{\left(\frac{x_1+\cdots+x_n}{2}\right)}\right),
\end{align*}
where \begin{align*}
f_n^1(\textbf{x})=&\frac{1}{\pi^2}\sum_{i=1}^n\left[\cosh(x_i/2)+1-(\frac{x_i^2}{16}+2)\frac{\sinh(x_i/2)}{x_i/2}\right]\prod_{l\neq i}\frac{\sinh(x_l/2)}{x_l/2}\\
-&\frac{1}{2\pi^2}\sum_{1\leq i<j\leq n}\left[\cosh(x_i/2)\cosh(x_j/2)+1-2\frac{\sinh(x_i/2)}{x_i/2}\frac{\sinh(x_j/2)}{x_j/2}\right]\\
\cdot&\prod_{l\neq i,j}\frac{\sinh(x_l/2)}{x_l/2},
\end{align*} 
and the implied constant is related to $k$.

\subsection{Several rough estimates}

\begin{lemma}\label{lemma in hide appendix gen k}
    For any $k\geq 2$,
    if $n=o(\sqrt{g})$, then \begin{align*}
         \sum_{\substack{
0\leq i\leq g,0\leq j\leq n\\ 
        k\leq 2i+j\leq 2g+n-k
        }}{n\choose j}\frac{V_{i,j+1}V_{g-i,n-j+1}}{V_{g,n}}\prec \frac{1+n^k}{g^{k-1}}.
    \end{align*}
\end{lemma}
\begin{rem*}
    The case of $k=2$ is given in \cite[Lemma A.4]{hide2023spectral}. We slightly extend it.
\end{rem*}

\begin{proof}[Proof of Lemma \ref{lemma in hide appendix gen k}]
If $n$ is bounded, this Lemma  follows (3) in Lemma \ref{mir13 vgn}. Now we assume $n\geq 2k+2$.
By Lemma \ref{appendix mir13 vg-1,n+2 vg,n+1}, Theorem \ref{thm mz15 asymp}, and Corollary \ref{cor vgn+2 leq vg+1n}, for $j\geq 0$ we have \begin{align*}
        V_{i,j+1}\prec V_{i+[\frac{j}{2}],j+1-2[\frac{j}{2}]}\prec \frac{1}{\sqrt{i+j}}(2i+j-2)!(4\pi^2)^{2i+j-2}.
    \end{align*}
It follows that \begin{equation}\label{lemma in hide appendix gen k-1}
\begin{aligned}
    & \sum_{\substack{
0\leq i\leq g,0\leq j\leq n\\ 
        k\leq 2i+j\leq 2g+n-k
        }}{n\choose j}\frac{V_{i,j+1}V_{g-i,n-j+1}}{V_{g,n}}\\
    \prec& \sum_{\substack{
0\leq i\leq g,0\leq j\leq n\\ 
        k\leq 2i+j\leq 2g+n-k
        }}\frac{n!}{j!(n-j)!}\frac{(2i+j-2)!(2g+n-2i-j-2)!}{(2g+n-3)!}\\
        \cdot&\frac{\sqrt{g}}{\sqrt{i+j}\sqrt{2g+n-2i-j}}\\
    \prec& \sum_{\substack{
0\leq i\leq g,0\leq j\leq n\\ 
        k\leq 2i+j\leq 2g+n-k
        }}\frac{n!}{j!(n-j)!}\frac{(2i+j-2)!(2g+n-2i-j-2)!}{(2g+n-3)!}.
\end{aligned}
\end{equation}
   Let $L=2i+j$.
    If $2k+2\leq L\leq n$, then \begin{equation}\label{lemma in hide appendix gen k-2}
    \begin{aligned}
      & \frac{n!}{j!(n-j)!}\frac{(2i+j-2)!(2g+n-2i-j-2)!}{(2g+n-3)!}
       \prec \frac{n^j (L-2)!}{g^{L-1}}\\
    \prec& \frac{n^2\sqrt{L}}{g}\left(\frac{nL}{eg}\right)^{L-2}
       \prec\sqrt{L}\left(\frac{nL}{eg}\right)^{L-2k-2}L^{2k} \left(\frac{n}{ge}\right)^{2k}\\
      \prec& L^{2k+\frac{1}{2}} g^{-k}\left(\frac{nL}{eg}\right)^{L-2k-2}
       \prec  L^{2k+\frac{1}{2}} g^{-k} e^{2k+2-L}\prec g^{-k}.
    \end{aligned}
    \end{equation}
  If $n+1\leq L\leq \frac{2g+n}{2}$, then \begin{equation}\label{lemma in hide appendix gen k-3}
  \begin{aligned}
         & \frac{n!}{j!(n-j)!}\frac{(2i+j-2)!(2g+n-2i-j-2)!}{(2g+n-3)!}\\
        \leq & 2^n \frac{(L-2)!(2g+n-L-2)!}{(2g+n-3)!}
         \prec2^n\frac{(n-1)!(2g-3)!}{(2g+n-3)!}\prec\left(\frac{2n}{g}\right)^n\\
      \prec& g^{-k}.
    \end{aligned}
    \end{equation}
  By the symmetry, for $ \frac{2g+n}{2}\leq L\leq 2g+n-2k-2$ we also have \begin{align}\label{lemma in hide appendix gen k-4}
        \frac{n!}{j!(n-j)!}\frac{(2i+j-2)!(2g+n-2i-j-2)!}{(2g+n-3)!}\prec g^{-k}.
    \end{align}
 So by \eqref{lemma in hide appendix gen k-1}, \eqref{lemma in hide appendix gen k-2}, \eqref{lemma in hide appendix gen k-3}, and \eqref{lemma in hide appendix gen k-4} we have    \begin{align*}
    & \sum_{\substack{
0\leq i\leq g,0\leq j\leq n\\ 
        k\leq 2i+j\leq 2g+n-k
        }}{n\choose j}\frac{V_{i,j+1}V_{g-i,n-j+1}}{V_{g,n}}\\
        \prec & \sum_{\substack{
0\leq i\leq g,0\leq j\leq n\\ 
        2k+2\leq 2i+j\leq 2g+n-2k-2
        }}\frac{n!}{j!(n-j)!}\frac{(2i+j-2)!(2g+n-2i-j-2)!}{(2g+n-3)!}\\
        +&\sum_{\substack{
0\leq i\leq g,0\leq j\leq n\\ 
        k\leq 2i+j\leq 2k+1
        \textit{ or}\\{2g+n-2k-1\leq  2i+j\leq 2g+n-k}
        }} \frac{n!}{j!(n-j)!}\frac{(2i+j-2)!(2g+n-2i-j-2)!}{(2g+n-3)!}\\
        \prec&\sum_{\substack{
0\leq i\leq g,0\leq j\leq n\\ 
        2k+2\leq 2i+j\leq 2g+n-2k-2
        }}\frac{1}{g^{k}}+\frac{n^k}{g^{k-1}}\\
        \prec&\frac{1+n^k}{g^{k-1}},
 \end{align*}
as required.
\end{proof}

For $k=2$, if we do not allow $(i,j)=(0,2)$ or $(g,n-2)$, then the following estimate is direct from the proof of Lemma \ref{lemma in hide appendix gen k}.
\begin{lemma}\label{lemma in hide appendix not s03}
     If $n=o(\sqrt{g})$, then \begin{align*}
        \sum_{\substack{
        0\leq i\leq g, 
        0\leq j\leq n\\
        2\leq 2i+j\leq 2g+n-2\\
        (i,j)\neq (0,2)\textit{ or }(g,n-2)
        }}{n\choose j}\frac{V_{i,j+1}V_{g-i,n-j+1}}{V_{g,n}}\prec \frac{1+n}{g}.
    \end{align*}
\end{lemma}

We also need the following similar estimate:
\begin{lemma}\label{appendix product lemma  2i+j geq k}
For any $k\geq 1$,
    if $n=o(\sqrt{g})$, then \begin{align*}
         \sum_{\substack{
0\leq i\leq g,0\leq j\leq n\\ 
        k\leq 2i+j\leq 2g+n-k
        }}{n\choose j}\frac{V_{i,j+2}V_{g-i,n-j+2}}{V_{g,n+2}}\prec \frac{1+n^k}{g^k}.
    \end{align*}
\end{lemma}
\begin{proof}
If $n$ is bounded, the lemma follows from (3) in Lemma \ref{mir13 vgn}. Now we can assume $n\geq 2k+2$.
    By Lemma \ref{appendix mir13 vg-1,n+2 vg,n+1} and Theorem \ref{thm mz15 asymp}, for $j\geq 0$ we have \begin{align*}
        V_{i,j+2}\leq V_{i+[\frac{j}{2}],j+2-2[\frac{j}{2}]}\prec \frac{1}{\sqrt{i+j}}(2i+j-1)!(4\pi^2)^{2i+j-1}.
    \end{align*}
By the similar estimate as \eqref{lemma in hide appendix gen k-1} we have \begin{equation}\label{appendix lemma product 1-1}
\begin{aligned}
    & \sum_{\substack{
0\leq i\leq g,0\leq j\leq n\\ 
        k\leq 2i+j\leq 2g+n-k
        }}{n\choose j}\frac{V_{i,j+2}V_{g-i,n-j+2}}{V_{g,n+2}}\\
    \prec& \sum_{\substack{
0\leq i\leq g,0\leq j\leq n\\ 
        k\leq 2i+j\leq 2g+n-k
        }}\frac{n!}{j!(n-j)!}\frac{(2i+j-1)!(2g+n-2i-j-1)!}{(2g+n-1)!}.
\end{aligned}
\end{equation}
    Let $L=2i+j$.
    If $2k+2\leq L\leq n$, then \begin{equation}\label{appendix lemma product 1-2}
    \begin{aligned}
      & \frac{n!}{j!(n-j)!}\frac{(2i+j-1)!(2g+n-2i-j-1)!}{(2g+n-1)!}
       \prec \frac{n^j L!}{g^L}\\
       \prec& \sqrt{L}\left(\frac{nL}{eg}\right)^L
      \prec \sqrt{L}\left(\frac{nL}{eg}\right)^{L-2k-2}L^{2k+2} \left(\frac{n}{ge}\right)^{2k+2}\\
      \prec& L^{2k+\frac{5}{2}} g^{-k-1}\left(\frac{nL}{eg}\right)^{L-2k-2}
      \prec  L^{2k+\frac{5}{2}} g^{-k-1} e^{2k+2-L}\prec g^{-k-1}.
    \end{aligned}
    \end{equation}
    If $n+1\leq L\leq \frac{2g+n}{2}$, then \begin{equation}\label{appendix lemma product 1-3}
    \begin{aligned}
         & \frac{n!}{j!(n-j)!}\frac{(2i+j-1)!(2g+n-2i-j-1)!}{(2g+n-1)!}\\
         \leq & 2^n \frac{(L-1)!(2g+n-L-1)!}{(2g+n-1)!}
         \prec2^n\frac{n!(2g-2)!}{(2g+n-1)!}\\
         \prec&\left(\frac{2n}{g}\right)^n \prec g^{-k-1}.
    \end{aligned}
    \end{equation}
    By the symmetry, for $ \frac{2g+n}{2}\leq L\leq 2g+n-2k-2$ we also have \begin{align}\label{appendix lemma product 1-4}
        \frac{n!}{j!(n-j)!}\frac{(2i+j-1)!(2g+n-2i-j-1)!}{(2g+n-1)!}\prec g^{-k-1}.
    \end{align}

 So by \eqref{appendix lemma product 1-1}, \eqref{appendix lemma product 1-2}, \eqref{appendix lemma product 1-3}, and \eqref{appendix lemma product 1-4} we have     \begin{align*}
    & \sum_{\substack{
0\leq i\leq g,0\leq j\leq n\\ 
        k\leq 2i+j\leq 2g+n-k
        }}{n\choose j}\frac{V_{i,j+2}V_{g-i,n-j+2}}{V_{g,n+2}}\\
        \prec & \sum_{\substack{
0\leq i\leq g,0\leq j\leq n\\ 
        2k+2\leq 2i+j\leq 2g+n-2k-2
        }}\frac{n!}{j!(n-j)!}\frac{(2i+j-1)!(2g+n-2i-j-1)!}{(2g+n-1)!}\\
        +&\sum_{\substack{
0\leq i\leq g,0\leq j\leq n\\ 
        k\leq 2i+j\leq 2k+1
        \textit{ or}\\{2g+n-2k-1\leq  2i+j\leq 2g+n-k}
        }} \frac{n!}{j!(n-j)!}\frac{(2i+j-1)!(2g+n-2i-j-1)!}{(2g+n-1)!}\\
        \prec&\sum_{\substack{
0\leq i\leq g,0\leq j\leq n\\ 
        2k+2\leq 2i+j\leq 2g+n-2k-2
        }}\frac{1}{g^{k+1}}+\frac{n^k}{g^k}\\
        \prec&\frac{1+n^k}{g^k},
 \end{align*}
    as desired.
\end{proof}

\begin{lemma}\label{lemma admissible triple in hide}\cite[Lemma A.5]{hide2023spectral}
    Let $n=o(\sqrt{g})$ and let $g_0,a_0,n_0$ and $k$ be given with $m=2g_0+a_0+k-2=O(\log g).$ For $1\leq q\leq k-2n_0$,$$
    \sum_{\substack{
    \{(g_i,a_i,n_i)\}_{i=1}^q\in\mathcal{A}
    }} \frac{n!}{a_0!\cdots a_q!}\cdot \frac{V_{g_0,k+a_0} V_{g_1,n_1+a_1}\cdots V_{g_q,n_q+a_q}}{V_{g,n}}\prec (m-1)!\frac{1+n^{a_0}}{g^m}
    $$
    when $g\to\infty$, and the implied constant is independent of $n,g,g_0,a_0,q,k$ for $g$ large enough. Here $\mathcal{A}$ is the set of admissible tribles $\{ \{(g_i,a_i,n_i)\}_{i=1}^q\}$ for $g_i,a_i\geq 0$, $n_i\geq 1$ and  $2g_i+a_i+n_i\geq 3$ with \begin{enumerate}
        \item $\sum_{i=1}^q(2g_i+n_i+a_i-2)=2g+n-2-m$;
        \item $\sum_{i=1}^q n_i=k-2n_0$;
        \item $\sum_{i=1}^q a_i=n-a_0$.
    \end{enumerate}
\end{lemma}
\begin{rem*}
    In \cite{hide2023spectral} Lemma \ref{lemma admissible triple in hide} requires $m\leq 3\log g-2$. Indeed, It is easy to check that the proof in \cite{hide2023spectral} works for $m=O(\log g)$ after replacing the inequality $(A.12)$ in \cite{hide2023spectral} by \begin{align*}
   & \frac{C_1^g\sqrt{g}}{\prod_{i=1}^q \sqrt{g_i+\max\{[\frac{a_i+n_i-2}{2}],0\}}}\frac{n!}{\prod_{i=0}^na_i!}\frac{\prod_{i=1}^q(2g_i+a_i+n_i-3)!}{(2g+n-3)!}\\
    \prec &\frac{1+n^{a_0}}{g^m}\cdot g^{-\frac{7}{2}q}.
\end{align*}
\end{rem*}
We will need a generalization of Lemma \ref{lemma admissible triple in hide}. Recall the definition in section \ref{susbsection n ell k}.
\begin{def*}
Fix $S=\{(g_i,n_i,a_i)\}_{i=1}^s$.
For $2g_i+a_i+n_i-2\geq 1$ and $n_i\geq 1$, assume $S_i\simeq S_{g_i,n_i+a_i}$ is a topology surface with cusps and nonempty boundaries, where $n_i$ is the number of boundary loops and $a_i$ is the number of punctures. We also use $(g_i,n_i,a_i)=(0,2,0)$ to represent that $S_i$ is a simple curve.
Then $\mathcal{B}_S^{s,(g,n)}$ consists of  $F=\left(\{f_i\}_{i=1}^s,\{V_j\}_{j=1}^r,\{h_j,m_j,b_j\}_{j=1}^r\right)$, where  
 \begin{enumerate}
     \item If $2g_i+a_i+n_i-2\geq 1$, the local homeomorphism $f_i:S_{g_i,n_i+a_i}\to X_{g,n}$ is injective on $\mathring{S}_{g_i,n_i+a_i}$, and $f_i(\partial S_{g_i,n_i+a_i})$ consists of simple closed geodesics in $X_{g,n}$.
     \item If $(g_i,n_i,a_0)=(0,2,0)$, the image of the embedding $f_i:S_{g_i,n_i+a_i}\to X_{g,n}$ is a simple closed geodesic. 
     \item For different $i$, $f_i( S_{g_i,n_i+a_i})$ has disjoint interior. 
     \item The complementary $X_{g,n}\setminus\cup_{i=1}^s f_i( S_{g_i,n_i+a_i})$ has $r$ components, denoted by $V_j$ for $j=1,\cdots,r$, where $V_j$ has genus $h_j$, $b_j$ cusps and $m_j$ geodesic boundary components.
     \item $2h_i+m_i+b_i\leq 2h_j+m_j+b_j$ if $i< j$.
     \item $F=\left(\{f_i\}_{i=1}^s,\{V_j\}_{j=1}^r,\{h_j,m_j,b_j\}_{j=1}^r\right)$ is said to be equivalent to $E=\left(\{\tilde{f}_i\}_{i=1}^s,\{\tilde{V}_j\}_{j=1}^r,\{\tilde{h}_j,\tilde{m}_j,\tilde{b}_j\}_{j=1}^r\right)$ if there is a homeomorphism $H: X_{g,n}\to X_{g,n}$ which may not fix the label of cusps and permutation elements $\sigma_1\in S_s,\sigma_2\in S_r$ such that $H\circ f_i=\tilde{f}_{\sigma_1(i)}$ and $H(V_i)=V_{\sigma_2(i)}$.
 \end{enumerate}
  We say $f_i(S_{g_i,n_i+a_i})$ for $i=1,\cdots,s$ and $V_j$ for $j=1,\cdots,r$ are all the components subordinated to $F$.
 \end{def*}
 \begin{lemma}\label{appendix lemma k Si product}
If $n=o(g^{\frac{1}{2}-\epsilon})$ for some $\epsilon>0$ and $s\geq 1$ is fixed, then for any $S=\{(g_i,n_i,a_i)\}_{i=1}^s$, with $\sum_{i=1}^s \left(2g_i+a_i+n_i-2\right)\leq A \log g$, there exists some $C>1$ depending on $A$ such that 
\begin{align*}
  &\frac{1}{V_{g,n}}  \sum_{F\in \mathcal{B}_S^{s,(g,n)} }{n\choose a_1,\cdots,a_s,b_1,\cdots,b_r} \prod_{(g_i,n_i,a_i)\neq (0,2,0)}V_{g_i,n_i+a_i}\prod_{j=1}^rV_{h_j,m_j+b_j}\\
  \prec&\prod_{i=1}^s\min\left\{\left((2g_i+n_i+a_i)!\right)^2 \left(C\log g\right)^{2n_i}\frac{(Cn+C\log g)^{a_i}}{g^{2g_i+n_i+a_i-2}},1\right\},
\end{align*}
where the implied constant is also related to $A$. 
 \end{lemma}
 \begin{proof}
Since  $n_i\geq 1$, $g_i,n_i,a_i=O(\log g)$ and $n=o(g^{\frac{1}{2}-\epsilon})$, when $2g_i+n_i+a_i-2\geq 1$, we have $$
\left((2g_i+n_i+a_i)!\right)^2 \left(C\log g\right)^{2n_i}\frac{(Cn+C\log g)^{a_i}}{g^{2g_i+n_i+a_i-2}}\prec 1.
$$
While for $(g_i,n_i,a_i)=(0,2,0)$,$$
\left((2g_i+n_i+a_i)!\right)^2 \left(C\log g\right)^{2n_i}\frac{(Cn+C\log g)^{a_i}}{g^{2g_i+n_i+a_i-2}}\asymp \log^4 g.
$$

     We prove the Lemma by induction on $s$. For $s=1$, if $2g_1+n_1+a_1-2\geq 1$, the estimate is given by taking Lemma \ref{lemma admissible triple in hide} over $n_0\leq \frac{k-1}{2}$ and $q\leq k-2n_0$ with $k=n_1$.
     If $(g_1,n_1,a_1)=(0,2,0)$, the estimate is given by Lemma \ref{lemma in hide appendix gen k} for separating orbits and by Theorem \ref{thm mz15 asymp} for non-separating orbits.

     Now we assume $s\geq 2$. For any $F\in\mathcal{B}_S^{s,(g,n)}$, notice that ${n\choose a_1,\cdots,a_s,b_1\cdots,b_r}$ is the number of possible cusp labeling for each component subordinated to $F$.
     Since $\sum_{i=1}^s(2g_i+a_i+n_i-2)=O(\log g)$, we have \begin{align}\label{estimate sum ni}
     r\leq \sum_{i=1}^s n_i\prec \log g.
     \end{align}
     Since $2h_i+m_i+b_i\leq 2h_j+m_j+b_j$ if $i< j$, we have $$
        (2h_r+m_r+b_r-2)\geq \frac{1}{r}\left(2g-2+n-\sum_{i=1}^s(2g_i+a_i+n_i-2)\right)\succ \frac{g}{\log g}. $$
     Since $m_r\leq \sum_{i=1}^s n_i=O(\log g)$ and $b_r\leq n$, we have \begin{align}\label{lower hr}
    h_r\succ\frac{g}{\log g}.
     \end{align}
     Now we assume that $V_r$ shares common boundary geodesics with $f_s(S_{g_s,n_s+a_s})$.
     Set $S^\prime=\{(g_i,n_i,a_i)\}_{i=1}^{s-1}$, and let $F^\prime=\left(\{f_i\}_{i=1}^{s-1},\{V_j^\prime\}_{j=1}^{r^\prime},\{h_j^\prime,m_j^\prime, b_j^\prime\}_{j=1}^{r^\prime} \right)\in \mathcal{B}_{S^\prime}^{s-1}(g,n)$. That is, we glue $f_s(S_{g_s,n_s+a_s})$ with all $V_j$'s that share common boundary geodesics with $f_s(S_{g_s,n_s+a_s})$ to get $V_{r^\prime}^\prime\simeq S_{h_{r^\prime}^{\prime},m_{r^\prime}^{\prime}+b_{r^\prime}^{\prime}}$, and keep the rest $f_i(S_{g_i,n_i+a_i})$ invariant for $i=1,\cdots,s-1$.
     An labeling of cusps to each component subordinated to $F$ induces an labeling of cusps to each component subordinated to $F^\prime$.
     From $F^\prime$ and the induced labeling, we can get $F\in \mathcal{B}_S^{s,(g,n)}$ and the orignal labeling by finding the labeled orbit of $S_s\simeq S_{g_s,n_s+a_s}$ in $V_{r^\prime}^\prime\simeq S_{h_{r^\prime}^{\prime},m_{r^\prime}^{\prime}+b_{r^\prime}^{\prime}}$, which is determined by $f_s:S_{g_s,n_s+a_s}\to V_{r^\prime}^\prime$. All connected components of the complementary $V_{r^\prime}^\prime\setminus f_s(S_{g_s,n_s+a_s})$
 along with $V_1^\prime,\cdots,V_{r^\prime-1}^\prime$ forms $\{V_j\}_{j=1}^r$. 
    If we deform all the boundary geodesics of $V_{r^\prime}^\prime$ to cusps, then $f_S: S_{g_s,n_s+a_s} \to V_{r^\prime}^\prime$ corresponds to an element $G=\left(\tilde{f}_s, \{V_j^G\}_{j=1}^{r_G},\{h_j^G,m_j^G,b_j^G \}_{j=1}^{r_G} \right)$ in $\mathcal{B}_{(g_s,n_s-k,a_s+k)}^{1,(h_{r^\prime}^\prime,m_{r^\prime}^\prime+b_{r^\prime}^\prime)}$ for some $k\leq \min\{n_s,m_{r^\prime}^\prime\}$. We must have $k=0$ in the case of $(g_s,n_s,a_s)=(0,2,0)$. Here we require the representation $\tilde{f}_s$ satisfying that $\tilde{f}_s(S_{g_s,n_s+a_s})$ contains exactly $k$ cusps that belong to the $m_{r^\prime}^\prime$ boundary geodesics of $V_{r^\prime}^\prime$ before the deformation.
It follows that \begin{equation}\label{appendix ineq for s}
\begin{aligned}
 &\frac{1}{V_{g,n}}  \sum_{\substack{F\in \mathcal{B}_S^{s,(g,n)}\\
 V_r\cap f_s(S_{g_s,n_S+a_s})\neq \emptyset
 }}{n\choose a_1,\cdots,a_s,b_1,\cdots,b_r}\\
 \cdot&\prod_{(g_i,n_i,a_i)\neq (0,2,0)}V_{g_i,n_i+a_i}\prod_{j=1}^rV_{h_j,m_j+b_j}\\
  \leq &\frac{1}{V_{g,n}}\sum_{E\in B_{S^\prime}^{s-1,(g,n)}}{n\choose a_1,\cdots,a_{s-1},b_1^\prime,\cdots,b_{r^\prime}^\prime}
\prod_{\substack{(g_i,n_i,a_i)\neq (0,2,0)\\
i=1,\cdots,s-1}}V_{g_i,n_i+a_i}\\
\cdot& \prod_{j=1}^{r^\prime}V_{h_j^\prime,m_j^\prime+b_j^\prime}
\sum_{\substack{
k=0,\cdots,\min\{n_s,m_{r^\prime}^\prime\}\\
\textit{if }(2g_s+n_s+a_s-2)\geq 1;\\
k=0 \textit{ if }(g_s,n_s,a_s)=(0,2,0)
}}
\frac{1}{V_{h_{r^\prime}^\prime,m_{r^\prime}^\prime+b_{r^\prime}^\prime}}
\cdot\sum_{G\in \mathcal{B}_{(g_s,n_s-k,a_s+k)}^{1,(h_{r^\prime}^\prime,m_{r^\prime}^\prime+b_{r^\prime}^\prime)}}\\
\cdot&{m_{r^\prime}^\prime\choose k}{b_{r^\prime}^\prime+m_{r^\prime}^\prime-k \choose a_s,b_1^G,\cdots,b_{r_G}^G}
\cdot \prod_{\substack{(g_s,n_s,a_s)\neq (0,2,0)}}V_{g_s,n_s+a_s}\cdot\prod_{l=1}^{r_G} V_{h_l^G,m_l^G+b_l^G}.
\end{aligned}
\end{equation}
Based on the induced hypothesis for the $s-1$ case, we have\begin{equation}\label{hyp on s-1}\begin{aligned}
    &\frac{1}{V_{g,n}}\sum_{E\in B_{S^\prime}^{s-1,(g,n)}}{n\choose a_1,\cdots,a_{s-1},b_1^\prime,\cdots,
    b_{r^\prime}^\prime}\\   
    \cdot&
\prod_{\substack{(g_i,n_i,a_i)\neq (0,2,0)\\
i=1,\cdots,s-1}}V_{g_i,n_i+a_i}
\cdot \prod_{j=1}^{r^\prime}V_{h_j^\prime,m_j^\prime+b_j^\prime}\\
\prec&\prod_{i=1}^{s-1}\min\left\{\left((2g_i+n_i+a_i)!\right)^2 \left(C_1\log g\right)^{2n_i}\frac{(C_1 n+C_1\log g)^{a_i}}{g^{2g_i+n_i+a_i-2}},1\right\}
\end{aligned}
\end{equation}
for some $C_1>1$.
     Now we estimate the size of $V_{r^\prime}^\prime$. Firstly, by \eqref{lower hr} we have \begin{align}\label{estimate hrprime}
     h_{r^\prime}^\prime\geq h_r\succ\frac{g}{\log g}.
     \end{align}
     Secondly, by \eqref{estimate sum ni} we have \begin{align}\label{estimate mrprime}
     m_{r^\prime}^\prime\leq \sum_{i=1}^s n_i\prec \log g,
     \end{align}
     and by the definition we have \begin{align}\label{estimate brprime}
     b_{r^\prime}^\prime\leq n=O\left(g^{\frac{1}{2}-\epsilon}\right).
     \end{align}
     So we have \begin{align}\label{estimate mrprime+brprime}
m_{r^\prime}^\prime+b_{r^\prime}^\prime=o\left(\left(h_{r^\prime}^\prime\right)^{\frac{1}{2}-\frac{\epsilon}{2}}\right).
     \end{align}
If $(g_s,n_s,a_s)=(0,2,0)$, we have \begin{align}\label{estimate s part for (0,2,0)}
\frac{1}{V_{h_{r^\prime}^\prime,m_{r^\prime}^\prime+b_{r^\prime}^\prime}}
\sum_{G\in \mathcal{B}_{(g_s,n_s,a_s)}^{1,(h_{r^\prime}^\prime,m_{r^\prime}^\prime+b_{r^\prime}^\prime)}}{b_{r^\prime}^\prime+m_{r^\prime}^\prime \choose a_s,b_1^G,\cdots,b_{r_G}^G}
\cdot\prod_{l=1}^{r_G} V_{h_l^G,m_l^G+b_l^G}\prec 1
\end{align}
by Lemma \ref{lemma in hide appendix gen k} for separating orbits and Theorem \ref{thm mz15 asymp} for non-separating orbits.
If $2g_s+n_s+a_s-2\geq 1$, for $k\leq \min\{n_s,m_{r^\prime}^\prime\}$, we have \begin{equation}\label{ineq for the s part}
\begin{aligned}
   & \frac{{m_{r^\prime}^\prime\choose k}{b_{r^\prime}^\prime+m_{r^\prime}^\prime-k \choose a_s,b_1^G,\cdots,b_{r_G}^G}}{{b_{r^\prime}^\prime+m_{r^\prime}^\prime\choose k+a_s,b_1^G,\cdots,b_{r_G}^G}}=\frac{m_{r^\prime}^\prime!}{k!(m_{r^\prime}^\prime-k)!}\frac{(b_{r^\prime}^\prime+m_{r^\prime}^\prime-k)!(k+a_s)!}{(b_r^\prime+m_{r^\prime}^\prime)!a_s!}\\
    \leq &\frac{\left(m_{r^\prime}^\prime\right)^k}{\max\left\{\left(b_r^\prime\right)^k,1\right\}}{k+a_s\choose k}\leq \frac{\left(m_{r^\prime}^\prime\right)^k}{\max\left\{\left(b_r^\prime\right)^k,1\right\}}(2g_s+n_s+a_s)!.
\end{aligned}
\end{equation}
By \eqref{estimate mrprime} and \eqref{estimate mrprime+brprime}, we can apply Lemma \ref{lemma admissible triple in hide} and \eqref{ineq for the s part} to get
\begin{equation}\label{estimate s part for |chi|>0}
\begin{aligned}
  &  \sum_{k=0}^{\min\{n_s,m_{r^\prime}^\prime\}}
\frac{1}{V_{h_{r^\prime}^\prime,m_{r^\prime}^\prime+b_{r^\prime}^\prime}}
\cdot\sum_{G\in \mathcal{B}_{(g_s,n_s-k,a_s+k)}^{1,(h_{r^\prime}^\prime,m_{r^\prime}^\prime+b_{r^\prime}^\prime)}}\\
\cdot&{m_{r^\prime}^\prime\choose k}{b_{r^\prime}^\prime+m_{r^\prime}^\prime-k \choose a_s,b_1^G,\cdots,b_{r_G}^G}
\cdot \prod_{\substack{(g_s,n_s,a_s)\neq (0,2,0)}}V_{g_s,n_s+a_s}\cdot\prod_{l=1}^{r_G} V_{h_l^G,m_l^G+b_l^G}\\
\prec & \sum_{k=0}^{\min\{n_s,m_{r^\prime}^\prime\}}\frac{\left(m_{r^\prime}^\prime\right)^k}{\max\left\{\left(b_r^\prime\right)^k,1\right\}}\left((2g_s+n_s+a_s)!\right)^2\frac{1+\left(m_{r^\prime}^\prime+b_{r^\prime}^\prime\right)^{a_s+k}}{{h_{r^\prime}^\prime}^{2g_s+a_s+n_s-2}}\\
\prec&\left((2g_s+n_s+a_s)!\right)^2 \left(m_{r^\prime}^\prime+2\right)^{2n_s}\frac{1+\left(m_{r^\prime}^\prime+b_{r^\prime}^\prime\right)^{a_s}}{{h_{r^\prime}^\prime}^{2g_s+a_s+n_s-2}}.
\end{aligned}
\end{equation}
By \eqref{estimate hrprime}, \eqref{estimate mrprime} and \eqref{estimate brprime}, there exists some $C_2>1$ such that \begin{align*}
    m_{r^\prime}+2\leq C_2\log g,\quad b_{r^\prime}\leq C_2 n, \quad h_{r^\prime}^\prime\geq \frac{g}{C_2\log g}.
\end{align*}
So by \eqref{estimate s part for (0,2,0)} and \eqref{estimate s part for |chi|>0} we have \begin{equation}\label{estimate for s part final}
\begin{aligned}
    &\sum_{\substack{
k=0,\cdots,\min\{n_s,m_{r^\prime}^\prime\}\\
\textit{if }(2g_s+n_s+a_s-2)\geq 1;\\
k=0 \textit{ if }(g_s,n_2,a_s)=(0,2,0)
}}
\frac{1}{V_{h_{r^\prime}^\prime,m_{r^\prime}^\prime+b_{r^\prime}^\prime}}
\cdot\sum_{G\in \mathcal{B}_{(g_s,n_s-k,a_s+k)}^{1,(h_{r^\prime}^\prime,m_{r^\prime}^\prime+b_{r^\prime}^\prime)}}\\
\cdot&{m_{r^\prime}^\prime\choose k}{b_{r^\prime}^\prime+m_{r^\prime}^\prime-k \choose a_s,b_1^G,\cdots,b_{r_G}^G}
\cdot \prod_{\substack{(g_s,n_s,a_s)\neq (0,2,0)}}V_{g_s,n_s+a_s}\cdot\prod_{l=1}^{r_G} V_{h_l^G,m_l^G+b_l^G}\\
\prec &\min\left\{\left((2g_s+n_s+a_s)!\right)^2 \left(C_2\log g\right)^{2n_s}\frac{(C_2 n+C_2\log g)^{a_s}}{g^{2g_s+n_s+a_s-2}},1\right\}.
\end{aligned}
\end{equation}
Now the lemma follows from \eqref{appendix ineq for s}, \eqref{hyp on s-1}, and \eqref{estimate for s part final} by taking $C$ to be $\max\{C_1,C_2\}$.
 \end{proof}

\bibliographystyle{amsalpha}
\bibliography{ref}

\end{document}